\documentclass[reqno,final]{amsart}
\usepackage{latexsym,amssymb,amsmath,amsthm,epsfig}

\usepackage{fancyhdr} %Headers
\usepackage[usenames]{color} %Color definition
\usepackage{hyperref} %Internal and external links
\usepackage{graphicx} %Graphics inclusion
\usepackage[english]{babel}
\usepackage[utf8]{inputenc}

\definecolor{aleacolor}{rgb}{0.16,0.59,0.78}

\hypersetup{
breaklinks,
colorlinks=true,
linkcolor=aleacolor,
urlcolor=aleacolor,
citecolor=aleacolor}

\renewcommand{\headheight}{11pt}
\makeatother

\theoremstyle{plain}
\newtheorem{theorem}{Theorem}[section]                                          
\newtheorem{proposition}[theorem]{Proposition}                          
\newtheorem{lemma}[theorem]{Lemma}
\newtheorem{corollary}[theorem]{Corollary}
\newtheorem{conjecture}[theorem]{Conjecture}
\theoremstyle{definition}
\newtheorem{definition}[theorem]{Definition}
\theoremstyle{remark}
\newtheorem{remark}[theorem]{Remark}

\makeatletter \@addtoreset{equation}{section} \makeatother
\renewcommand\theequation{\thesection.\arabic{equation}}
\renewcommand\thefigure{\thesection.\arabic{figure}}
\renewcommand\thetable{\thesection.\arabic{table}}

\newcommand{\myparagraph}[1]{\medskip\noindent\textbf{#1}}

\newcommand{\be}[1]{\begin{equation}\label{#1}}
\newcommand{\ee}{\end{equation}}

\newcommand{\bt}[1]{\begin{theorem}\label{#1}}
\newcommand{\et}{\end{theorem}}

\newcommand{\bl}[1]{\begin{lemma}\label{#1}}
\newcommand{\el}{\end{lemma}}

\newcommand{\bp}[1]{\begin{proposition}\label{#1}}
\newcommand{\ep}{\end{proposition}}

\newcommand{\bd}[1]{\begin{definition}\label{#1}}
\newcommand{\ed}{\end{definition}}

\newcommand{\br}[1]{\begin{remark}\label{#1}}
\newcommand{\er}{\end{remark}}

\newcommand{\bcj}[1]{\begin{conjecture}\label{#1}}
\newcommand{\ecj}{\end{conjecture}}

\newcommand{\bcor}[1]{\begin{corollary}\label{#1}}
\newcommand{\ecor}{\end{corollary}}

\newcommand{\bpr}{\begin{proof}}
\newcommand{\epr}{\end{proof}}

\newcommand{\wt}{\widetilde}
\newcommand{\dd}{\mathrm{d}}
\newcommand{\eee}{\mathrm{e}}
\renewcommand{\P}{\mathbb{P}}
\DeclareMathOperator{\var}{Var}

\def\CC{\mathcal{C}}
\def\CD{\mathcal{D}}
\def\CM{\mathcal{M}}

\def\CE{\mathcal{E}}
\def\CF{\mathcal{F}}

\def\CL{\mathcal{L}}

\def\CP{\mathcal{P}}

\def\B{\mathbb{B}}
\def\C{\mathbb{C}}
\def\E{\mathbb{E}}

\def\N{\mathbb{N}}

\def\R{\mathbb{R}}

\def\Z{\mathbb{Z}}

\newcommand{\uc}{\underline{c}}

\newcommand{\ux}{\underline{x}}

\newcommand{\ud}{\underline{d}}
\newcommand{\udster}{\underline{d}^*}
\newcommand{\umuu}{\underline{\mu}}

\newcommand{\umu}{\underline{m}}
\newcommand{\uL}{\underline{\Lambda}}
\newcommand{\ul}{\underline{\lambda}}

\newcommand{\uth}{\underline{\theta}}
\newcommand{\unu}{\underline{0}}

\newcommand{\jto}{{_{\D \Longrightarrow \atop j \to \infty}}}
\newcommand{\tto}{{_{\D \Longrightarrow \atop t \to \infty}}}

\newcommand{\Ntoo}{{_{\D \Longrightarrow \atop N \to \infty}}}

\newcommand{\ttO}{{_{\D \longrightarrow \atop t \to \infty}}}

\newcommand{\D}{\displaystyle}

\newcommand{\eea}{\end{eqnarray}}
\newcommand{\bean}{\begin{eqnarray*}}
\newcommand{\eean}{\end{eqnarray*}}

\newcommand{\intl}{\int\limits}
\newcommand{\prodl}{\prod\limits}
\newcommand{\liml}{\lim\limits}
\newcommand{\suml}{\sum\limits}

\newcommand{\ve}{\varepsilon}

\newcommand{\Ntwo}{\N \backslash \{1\}}

\allowdisplaybreaks

\begin{document}

\title[Renormalisation of hierarchically interacting Cannings processes]{Renormalisation of hierarchically interacting\\
Cannings processes}

\author{A.\ Greven}

\author{F.\ den Hollander}

\author{S.\ Kliem}

\author{A.\ Klimovsky}

\address{Department Mathematik\newline
Universit\"at Erlangen-N\"urnberg\newline
Cauerstra{\ss}e $11$\newline
D-91058 Erlangen, Germany
}
\email{greven@mi.uni-erlangen.de}
\urladdr{\url{http://www.mathematik.uni-erlangen.de/~greven}}

\address{Mathematical Institute\newline
Leiden University\newline
P.O.\ Box 9512\newline
NL-2300 RA Leiden, The Netherlands
}
\email{denholla@math.leidenuniv.nl}
\urladdr{\url{http://www.math.leidenuniv.nl/~denholla}}

\address{AG Wahrscheinlichkeitstheorie\newline
Fakultät für Mathematik\newline
Universität Duisburg-Essen\newline
D-45117 Essen, Germany
}
\email{sandra.kliem@uni-due.de}
\urladdr{\url{https://www.uni-due.de/~hm0135}}

\address{AG Wahrscheinlichkeitstheorie\newline
Fakultät für Mathematik\newline
Universität Duisburg-Essen\newline
D-45117 Essen, Germany
}
\email{ak@aklimovsky.net}
\urladdr{\url{http://www.aklimovsky.net/}}

\thanks{The idea for this paper arose from
discussions with P.\ Pfaffelhuber and A.\  Wakolbinger during an Oberwolfach
meeting on ``Random Trees'' in January 2009. FdH thanks J.\ Goodman, R.J.\
Kooman and E.\ Verbitskiy for discussions on M\"obius-transformations. AG was
supported by the Deutsche Forschungsgemeinschaft  (grant DFG-GR 876/15-1), FdH
by the European Research Council (Advanced Grant VARIS-267356), and AK by the
European Commission (project PIEF-GA-2009-251200) and by NWO (grant
613.000.913).  SK held a postdoctoral position at EURANDOM from the Summer of
2009 until the  Summer of 2011. AK was Postdoctoral Marie Curie Fellow at
EURANDOM from  December 2009 until July 2012, and held a postdoctoral position
in Leiden from  August 2012 until May 2013. AK was a guest of the Hausdorff
Research Institute for  Mathematics  in Bonn in the Fall of 2010 (Junior
Trimester Program on Stochastics). SK and AK were guests of the Institute for
Mathematical Sciences, National University of Singapore, during its 2011
programme on Probability and Discrete Mathematics in Mathematical Biology. The
authors thank the anonymous referees for a careful reading of the
manuscript and for helpful suggestions on the presentation.}


\begin{abstract}
{\small
In order to analyse universal patterns in the large space-time behaviour of
interacting multi-type stochastic populations on countable geographic spaces, 
a key approach has been to carry out a renormalisation analysis in the hierarchical 
mean-field limit. This has provided considerable insight into the structure of 
interacting systems of finite-dimensional \emph{diffusions}, such as Fisher-Wright 
or Feller diffusions, and their infinite-dimensional analogues, such as 
Fleming-Viot or Dawson-Watanabe superdiffusions.

The present paper brings a new class of interacting \emph{jump processes}  into
focus. We start from a single-colony $C^\Lambda$-process, which arises as  the
continuum-mass limit of a $\Lambda$-Cannings individual-based population  model,
where $\Lambda$ is a finite non-negative measure that describes the  offspring
mechanism, i.e., how individuals in a single colony are replaced via 
resampling. The key feature of the $\Lambda$-Cannings individual-based
population  model is that the offspring of a single individual can be a positive
fraction of  the total population. After that we introduce a system of
\emph{hierarchically interacting} $C^\Lambda$-processes, where the interaction
comes from migration  and reshuffling-resampling on \emph{all} hierarchical
space-time scales \emph{simultaneously}. More precisely, individuals live in
colonies labelled by  the hierarchical group $\Omega_N$ of order $N$, and are
subject to \emph{migration}  based on a sequence of migration coefficients
$\uc=(c_k)_{k\in\N_0}$ and to  \emph{reshuffling-resampling} based on a sequence
of resampling measures  $\uL= (\Lambda_k)_{k\in\N_0}$, both acting in
$k$-\emph{macro-colonies}, for all $k\in\N_0$.  The reshuffling is linked to the
resampling: before resampling in a macro-colony takes  place all individuals in
that macro-colony are relocated uniformly, i.e., resampling is  done in a locally
``panmictic'' manner.  We refer to this system as the  $C_N^{\uc,\uL}$-process.
The dual process of the $C^\Lambda$-process is the  $\Lambda$-coalescent,
whereas the dual process of the $C_N^{\uc,\uL}$-process is  a spatial coalescent
with multi-scale non-local coalescence.

For the above system, we carry out a \emph{full renormalisation analysis} in the 
\emph{hierarchical mean-field limit} $N \to \infty$. Our main result is that, in
the limit as $N\to\infty$, on each hierarchical scale $k\in\N_0$, the $k$-macro-colony  
averages of the $C_N^{\uc,\uL}$-process at the macroscopic time scale $N^k$
(= the volume of the $k$-macrocolony) converge to a random process that is a 
superposition of a $C^{\Lambda_k}$-process and a Fleming-Viot process, the latter 
with a volatility $d_k$ and with a drift of strength $c_k$ towards the limiting  
$(k+1)$-macro-colony average. It turns out that $d_k$ is a function of $c_l$ and 
$\Lambda_l$  for all $0\leq l <k$. Thus, it is through the volatility that the renormalisation
manifests itself. We investigate how $d_k$ scales as $k\to\infty$, which requires 
an analysis of compositions of certain M\"obius-transformations, and leads to four 
different \emph{regimes}.

We discuss the implications of the scaling of $d_k$ for the behaviour on large 
space-time scales of the $C_N^{\uc,\uL}$-process. We compare the outcome 
with what is known from the renormalisation analysis of  hierarchically interacting
Fleming-Viot diffusions, pointing out several new features. In particular, we
obtain a \emph{new classification} for when the process exhibits
\emph{clustering}  (= develops spatially expanding mono-type regions),
respectively, exhibits \emph{local coexistence} (= allows for different types to
live next to each other with positive probability). Here, the simple dichotomy
of recurrent  versus transient migration for hierarchically interacting
Fleming-Viot diffusions, namely, $\sum_{k\in\N_0} (1/c_k) = \infty$ versus
$<\infty$, is  replaced by a dichotomy that expresses a trade-off between
migration and  reshuffling-resampling, namely, $\sum_{k\in\N_0} (1/c_k)
\sum_{l=0}^k \Lambda_l ([0,1])=\infty$ versus $<\infty$. Thus, while recurrent
migrations still only  give rise to clustering, there now are transient migrations 
that do the same when the non-local resampling is strong enough, namely,
$\sum_{l\in\N_0} \Lambda_l ([0,1])=\infty$. Moreover, in the clustering regime
we find a richer scenario  for the \emph{cluster formation} than for
Fleming-Viot diffusions. In the local-coexistence regime, on the other hand, we
find that the types initially present only survive with a positive probability,
not with probability one as for Fleming-Viot diffusions. Finally, we show that
for finite $N$ the same dichotomy between clustering and local coexistence
holds as for $N\to\infty$, even though we lack proper control on the cluster
formation, respectively, on the distribution of the types that survive. 
}
\end{abstract}

\maketitle

{\small
\tableofcontents
}

\section{Introduction and main results}
\label{s.intro}

\subsection{Outline}
\label{s.introoutline}

Section~\ref{ss.motive} provides the background for the paper. Section~\ref{ss.cannings} 
defines the single-colony and the multi-colony $C^\Lambda$-process, as well as the 
so-called McKean-Vlasov $C^\Lambda$-process, a single-colony $C^\Lambda$-process with 
immigration and emigration from and to a cemetery state arising in the context of 
the scaling limit of the multi-colony $C^\Lambda$-process with mean-field interaction.
Section~\ref{ss.exthier} defines a {\em new} process, the $C_N^{\uc,\uL}$-process, where
the countably many colonies are labelled by the hierarchical group $\Omega_N$ of order 
$N$, and the migration and the reshuffling-resampling on successive hierarchical space-time
scales are governed by a sequence $\uc=(c_k)_{k\in\N_0}$ of migration coefficients and 
a sequence $\uL=(\Lambda_k)_{k\in\N_0}$ of resampling measures. Section~\ref{ss.mainth} 
introduces {\em multiple space-time scales} and a collection of {\em renormalised 
systems}. It is shown that, in the hierarchical mean-field limit $N\to\infty$, the 
block averages of the $C_N^{\uc,\uL}$-process on hierarchical space-time scale $k$ 
converge to a McKean-Vlasov process that is a superposition of a single-colony 
$C^{\Lambda_k}$-process and a single-colony Fleming-Viot process with a volatility
$d_k$ that is a function of $c_l$ and $\Lambda_l$ for all $0\leq l<k$, and a drift 
of strength $c_k$ towards the limiting $(k+1)$-st block average. The scaling 
of $d_k$ as $k\to\infty$ turns out to have several \emph{universality classes}. 
The implications of this scaling for the behaviour of the $C_N^{\uc,\uL}$-process 
on large space-time scales is discussed in detail, and the outcome is compared 
with what is known for hierarchically interacting Fleming-Viot diffusions.

A key feature of the $C_N^{\uc,\uL}$-process is that it has a spatial $\uL$-coalescent
with block migration and multi-scale non-local coalescence as a dual process. This duality, 
which is of intrinsic  interest, and the properties of the dual process are worked out in 
Section~\ref{s.spatcoal}. The proofs of the main theorems are given in 
Sections~\ref{s.wpmp}--\ref{s.prcor}. To help the reader, a list of the main symbols 
used in the paper is added in Section~\ref{s.notation}.

\subsection{Background}
\label{ss.motive}

\subsubsection{Population dynamics}
\label{sss.popdyn}

For the description of spatial populations subject to migration and to neutral 
stochastic evolution (i.e., resampling without selection, mutation or recombination), 
it is common to use variants of interacting Fleming-Viot diffusions (Dawson~\cite{D93},  
Donnelly and Kurtz~\cite{DK99}, Etheridge~\cite{E00,E11}). These are processes taking 
values in $\CP(E)^I$, where $I$ is a countable Abelian group playing the role of a 
\emph{geographic space} labelling the colonies of the population (e.g.\ $\Z^d$, the
$d$-dimensional integer lattice, or $\Omega_N$, the hierarchical group  of order
$N$), $E$ is a compact Polish space playing the role of a \emph{type space} encoding 
the possible types of the individuals living in these colonies (e.g., $[0,1]$), and 
$\CP(E)$ is the set of probability measures on $E$. An element in $\CP(E)^I$ specifies 
the frequencies of the types in each of the  colonies in $I$.

Let us first consider the (locally finite) populations of individuals from which
the above processes arise as continuum-mass limits. Assume that the individuals 
\emph{migrate} between the colonies according to independent continuous-time random 
walks on $I$. Inside each colony, the evolution is driven by a change of generation 
called \emph{resampling}. Resampling, in its simplest form (Moran model), means that
after exponential waiting times a pair of individuals  (``the parents'') is replaced 
by a new pair of individuals (``the children''), who randomly and independently adopt 
the type of one of the parents. The process of type \emph{frequencies} in each of 
the colonies as a result of the migration and the resampling is a jump process taking 
values in $\CP(E)^I$.

If we pass to the {\em continuum-mass limit} of the frequencies by letting the number 
of individuals per colony tend to infinity, then we obtain a system of \emph{interacting 
Fleming-Viot diffusions} (Dawson, Greven and Vaillancourt~\cite{DGV95}). By picking 
different resampling mechanisms, occurring at a rate that depends on the state of the 
colony, we obtain variants of interacting Fleming-Viot diffusions with a state-dependent 
resampling rate~\cite{DM95}. In this context, key questions are: To what extent does 
the behaviour on large space-time scales depend on the precise form of the resampling 
mechanism? In particular, to what extent is this behaviour \emph{universal}? For 
Fleming-Viot models and a small class of state- and type-dependent Fleming-Viot 
models, this question has been answered in \cite{DGV95}.

If we consider resampling mechanisms where, instead of a pair of individuals,  a
positive fraction of the local population is replaced (an idea due to
Cannings~\cite{C74,C75}), then we enter the world of {\em jump processes}. In
this paper, we will focus on jump processes that are parametrised by a measure
$\Lambda$ on $[0,1]$ that models the random proportion of offspring in the
population generated by a single individual in a resampling event. It has been
argued by many authors that such jump processes are suitable for describing
situations with \emph{little biodiversity}. For instance, the jumps may account
for selective sweeps, or for extreme reproduction events (occurring on smaller
time scales and in a random manner, so that an effectively neutral evolution
results), such as those observed in certain marine organisms, e.g., {\em Atlantic 
cod} or {\em Pacific oyster} (Eldon and Wakeley~\cite{EW06}). It is argued in 
Der, Epstein and Plotkin~\cite{DEP11} that mixtures of diffusive dynamics and 
Cannings dynamics provide a better fit to generation-by-generation empirical data 
from {\em Drosophila} populations. Birkner and Blath~\cite{BB08,BB09}  
treat the issue of statistical inference on the genealogies corresponding to a 
one-parameter family of Cannings dynamics. None of these models includes 
the effect of geography.

Our goal is to describe the effect of jumps in a \emph{spatial} setting with a 
volatile reproduction. To that end, we add two ingredients: (1) a geographic 
space with a migration mechanism; (2) a spatially structured reproduction 
mechanism. As a result, we obtain a system of {\em interacting Cannings 
processes}. 

As geographic space, we choose a hierarchically structured lattice: the
hierarchical group, i.e., we study a system of hierarchically interacting
Cannings processes. The interaction is chosen in such a way that the
geographic space mimics the {\em two}-dimensional Euclidean space, 
with the migration of individuals given by independent random walks.

On top of migration and single-colony resampling, we add \emph{multi-colony 
resampling} by carrying out a Cannings-type resampling in all blocks
simultaneously, combined with a \emph{reshuffling} of the individuals inside
the block before the resampling is done. This is a {\em first attempt} to account
for the fact that the volatility the Cannings model tries to capture results 
from {\em catastrophic events} on a smaller time scale (with a geographic 
structure). In this view, the reshuffling mimics the fact that in reproduction 
the local geographic interaction typically takes place on a smaller time scale, 
in a random manner, and effectively results in a Cannings jump and in 
a complete geographic redistribution of individuals during a single observation 
time. To carry out this idea fully, the mechanism should actually be modelled 
by specifying a random environment. In this work, however, we concentrate 
on the case of spatially homogeneous parameters. The case of spatially 
inhomogeneous parameters (modelled via a random environment) is left 
for future work. On a technical level, we will see that in our model the 
reshuffling substantially simplifies the analysis. 

The idea to give reproduction a non-local geographic structure, in particular,
in two dimensions, was exploited by Barton, Etheridge and V\'eber~\cite{BEV10}
and by Berestycki, Etheridge and V\'eber~\cite{BEV13} also\footnote{In the
literature, there is an alternative terminology -- ``generalised
$\Lambda$-Fleming-Viot process" or ``jump-type Fleming-Viot process" -- which
refers to the continuum-mass limit of the original discrete individual-based 
Cannings model. In this paper, we stick to the name ``Cannings process" 
also for the continuum-mass limit.}. There, the process lives on the
torus of sidelength $L$ and is constructed via its dual, and it is shown that a
limiting process on $\R^2$ exists as $L \to \infty$. In \cite{BEV10,BEV13}, it is
assumed that  the individual lineages are compound Poisson processes.
Freeman~\cite{Fpr} considers a particular case of the spatially structured
Cannings model with a continuum self-similar geographic space, where
\textit{all} individuals in a block are updated upon resampling. The latter
set-up does not require compensation for small jumps and allows for their
accumulation.

\subsubsection{Renormalisation}
\label{sss.renorm}

A key approach to understand universality in the behaviour of interacting systems 
has been a \emph{renormalisation analysis of block averages on successive space-time 
scales} combined with a \emph{hierarchical mean-field limit}. In this setting, one 
replaces $I$ by the hierarchical group $\Omega_N$ of order $N$ and passes to the 
limit $N\to\infty$ (``the hierarchical mean-field limit'')\footnote{Actually, this set-up 
provides an {\em approximation} for the geographic space $I=\Z^2$, on which simple 
random walk migration is critically recurrent (Dawson, Gorostiza and Wakolbinger~\cite{DGW}). 
We will comment on this issue in Section~\ref{sss.mighg}.}.
With the limiting dynamics obtained through the hierarchical mean-field limit 
one associates a (nonlinear) \emph{renormalisation transformation} $\CF_c$ (which
depends on the migration rate $c$), acting on the resampling rate function $g$ 
driving the diffusion in single colonies. One studies the {\em orbit} $(\CF^{[k]}
(g))_{k\in\N}$, with $\CF^{[k]} = \CF_{c_{k-1}} \circ \cdots \circ \CF_{c_0}$, 
characterising the behaviour of the system on an increasing sequence of space-time 
scales, where $(c_k)_{k \in \N}$ represents the sequence of migration coefficients, 
with the index $k$ labelling the hierarchical distance. The \emph{universality classes} 
of the system are associated with the fixed points (or the fixed  shapes) of $\CF_c$, 
i.e., $g$ with $\CF_c(g) = ag$ with $a=1$ (or $a=a(c) \in (0,\infty)$).

The above \emph{renormalisation program} was developed for various choices of
the single-colony  state space. Each such choice gives rise to a different
universality class with specific features  for the large space-time behaviour.
For the \emph{stochastic part} of the renormalisation program (i.e., the derivation
of the limiting renormalised dynamics), see Dawson and Greven~\cite{DG93a},
\cite{DG93b}, \cite{DG93c}, \cite{DG96}, \cite{DG99}, \cite{DG03}, Dawson,
Greven and Vaillancourt~\cite{DGV95}, and Cox, Dawson and Greven~\cite{CDG04}.
For the \emph{analytic part} (i.e., the study of the renormalisation map
$\CF$), see Baillon, Cl\'ement, Greven and den Hollander~\cite{BCGH95},
\cite{BCGH97}, den Hollander and Swart~\cite{HS98}, and Dawson, Greven, den 
Hollander, Sun and Swart~\cite{DGHSS08}.

So far, two important classes of single-colony processes could not be treated:
{\em Anderson  diffusions} \cite{GH07} and {\em jump processes}. In the present 
paper, we focus on the second class, in particular, on so-called $C^\Lambda$-processes. 
In all previously treated models, the renormalisation transformation was a map 
$\CF_c$ acting on the set $M(E)$ of measurable functions on $E$, 
the single-component state space, while the function $g$ was a branching rate, 
a resampling rate or other, defining a diffusion function $x \mapsto xg(x)$ on 
$[0,\infty)$ or $x \mapsto x(1-x)g(x)$ on $[0,1]$, etc. In the present paper, 
however, we deal with jump processes that are characterised by a sequence of 
finite measures $\uL=(\Lambda_k)_{k\in\N_0}$ on $[0,1]$, and we obtain a 
renormalisation map $\CF_c$ acting on a pair $(g,\uL)$, where $g \in M(E)$ 
characterises diffusive behaviour and $\uL$ characterises resampling behaviour. 
It turns out that the orbit of this map is of the form  
\be{ag0}
(d_kg^*,(\Lambda_l)_{l \geq k})_{k\in\N_0},  
\ee
where $g^* \equiv 1$ and $d_k$ depends on $d_{k-1}$, $c_{k-1}$ and the total mass 
of $\Lambda_{k-1}$. Here, as before, $\uc=(c_k)_{k\in\N_0}$ is the sequence of 
migration coefficients. The reason behind this reduction is that our single-colony 
process is a superposition of a $C^\Lambda$-process and a Fleming-Viot process with 
\emph{state-independent} resampling rates and that both these processes renormalise 
to a multiple of the latter. It turns out that $d_k$ can be expressed in terms of
compositions of certain {\em M\"obius-transformations} with parameters changing 
from composition to composition. It is through these compositions that the 
\emph{renormalisation} manifests itself.

If the single-colony process would be a superposition of a $C^\Lambda$-process
and a  Fleming-Viot process with \emph{state-dependent} resampling rate, i.e.,
$g$ would {\em not} be a constant but a function of the state, then the renormalisation
transformation would be much more complicated. It remains a challenge to deal with 
this generalisation.

\subsection{The Cannings model}
\label{ss.cannings}

The $\Lambda$-Cannings model involves a finite non-negative measure 
$\Lambda\in\CM_f([0,1])$. Below, we often assume that
\be{ag1}
\Lambda(\{0\})=0
\ee
and $\Lambda$ satisfying the so-called \emph{dust-free condition}
\be{ag2}
\int_{(0,1]} \frac{\Lambda(\dd r)}{r} = \infty.
\ee
Condition \eqref{ag1} excludes the well-studied case of \textit{Fleming-Viot
diffusions}. In this paper, we are primarily interested in the new effects
brought by the \textit{pure jump} case in the $\Lambda$-Cannings model. These
effects were not studied using renormalisation techniques previously. Besides
the pure jump case, later on, we allow for superpositions of Fleming-Viot
diffusion and pure-jump $\Lambda$-Cannings models (cf.\ Sections~\ref{sss.MV} and
\ref{sss.hierarCan}). Condition~\eqref{ag2} excludes cases where the jump sizes
do not accumulate. Moreover, this condition is needed to have well-defined
proportions of the different types in the population in the infinite-population
limit (Pitman~\cite[Theorem~8]{P99}), and also to be able to define a
genealogical tree for the population (Greven, Pfaffelhuber and
Winter~\cite{GPW09})\footnote{Condition~\eqref{ag2} is relevant for some of the
questions addressed in this paper, though not for all. We comment on this issue
we go along. Another line of research would be to work with the most general
Cannings models that allow for simultaneous multiple resampling events. We do
not pursue such a generalisation here.}.

In Sections~\ref{sss.scolonycannings}--\ref{sss.MV}, we build up the Cannings model
in three steps: single-colony $C^\Lambda$-process, multi-colony $C^\Lambda$-process,
and $C^\Lambda$-process with immigration-emigration (McKean-Vlasov limit).

\subsubsection{Single-colony $C^\Lambda$-process}
\label{sss.scolonycannings}

We recall the definition of the $\Lambda$-Cannings model in its simplest form.
This model describes the evolution of allelic types of finitely many individuals
living in a single colony. Let $M \in \N$ be the number of individuals, and let
$E$ be a compact Polish space encoding the types (a typical choice is $E=[0,1]$). The
evolution of the population, whose state space is $E^M$, is as follows.  
\begin{itemize}
\item
The number of individuals stays fixed at $M$ during the evolution.
\item
Initially, i.i.d.\ types are assigned to the individuals according to a
given distribution
\be{ag5}
\theta \in \CP(E).
\ee
\item
Let $\Lambda^*\in\CM([0,1])$ be the $\sigma$-finite non-negative measure defined as
\be{ag3}
\Lambda^*(\{0\}) = 0, \qquad \Lambda^*(\dd r) = \frac{ \Lambda(\dd r)}{r^2}, \quad r \in (0,1].
\ee
Consider an inhomogeneous Poisson point process on $[0,\infty) \times [0,1]$ with
intensity measure
\be{ag4}
\dd t \otimes \Lambda^*(\dd r).
\ee
For each point $(t,r)$ in this process, we carry out the following transition at
time $t$. {\em Mark} each of the $M$ individuals independently with a $1$ or $0$ 
with probability $r$, respectively, $1-r$. \emph{All} individuals marked by a $1$ 
are killed and are replaced by copies of a \emph{single} individual (= ``parent") 
that is uniformly chosen at random among all the individuals marked by a $1$ (see 
Fig.~\ref{fig-canevol}).
\end{itemize}
In this way, we obtain a pure-jump Markov process, which is called the
\emph{$\Lambda$-Cannings model} with measure $\Lambda$ and population size $M$.

\begin{figure}[htbp]
\centering
\includegraphics[width=0.9\textwidth]{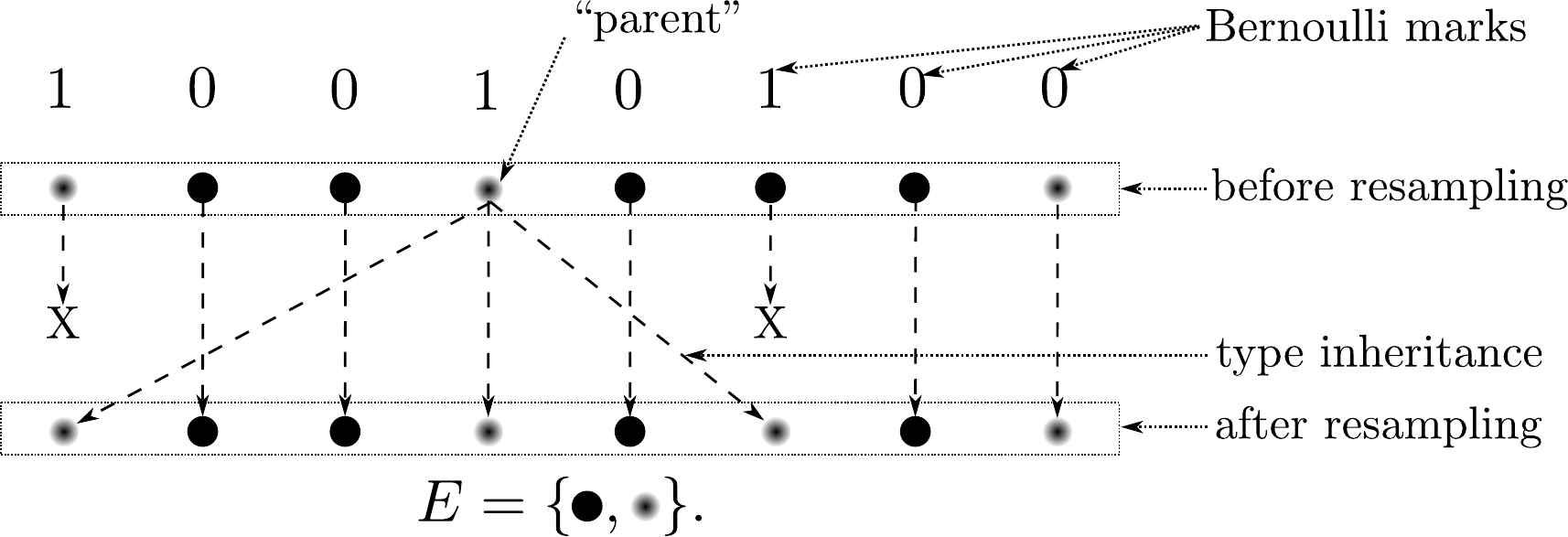}
\caption{Cannings resampling event in a colony of $M=8$ individuals of 
two types. Arrows indicate type inheritance, X indicates death.}
\label{fig-canevol}
\end{figure}

Note that, for a jump to occur, at least two individuals marked by a $1$ are needed. 
Hence, for finite $M$, the rate at which some pair of individuals is marked is
\be{ag11}
\int_{(0,1]} \frac{\Lambda(\dd r)}{r^2} \tfrac12 M(M-1)\,r^2
= \tfrac12 M(M-1)\,\Lambda ((0,1]) < \infty,
\ee
and so only finitely many jumps occur in any finite time interval.

By observing the frequencies of the types, i.e., the number of individuals with
a  given type divided by $M$, we obtain a measure-valued pure-jump Markov
process on  $\CP(E)$. Equip $\CP(E)$ with the topology of weak convergence of
probability measures. Letting $M\to\infty$, we obtain a limiting process
$X=(X(t))_{t \geq 0}$,  called the \emph{$C^\Lambda$-process}, which is a strong
Markov jump process with  paths in $D([0,\infty),\CP(E))$ (the set of c\`adl\`ag
paths in $\CP(E)$ endowed  with the Skorokhod $J_1$-topology) and can be
characterised as the solution of a well-posed martingale problem (Donnelly and
Kurtz~\cite{DK99}). This process has countably many jumps in any finite time
interval when $\Lambda((0,1])>0$.

Note that the limiting case $\Lambda = \delta_0$ is the Fleming-Viot diffusion 
(cf.\ Section~\ref{sss.MV}). It is well known that this limiting case is obtained 
as a scaling limit of the Moran model.

\subsubsection{Multi-colony $C^\Lambda$-process: mean-field version}
\label{sss.mcolonycannings}

Next, we consider the {\em spatial} $\Lambda$-Cannings model in its standard
mean-field version. Consider as geographic space a block of sites $\{0,\ldots,N-1\}$ 
and assign $M$ individuals to each site (= colony). The evolution of the population,
whose state space is $(E^M)^N$, is defined as the following pure-jump Markov process.
\begin{itemize}

\item
The total number of individuals stays fixed at $NM$ during the evolution.

\item
At the start, each individual is assigned a type that is drawn from $E$ according
to some prescribed exchangeable law. 

\item
Individuals \emph{migrate} between colonies at rate $c>0$, jumping according to the
uniform distribution on $\{0,\ldots,N-1\}$ (see Fig.~\ref{fig-mfcan}).

\item
Individuals \emph{resample} within each colony according to the $\Lambda$-Cannings 
model with population size corresponding to the current size of the colony.

\end{itemize}
By considering the frequencies of the types in each of the colonies, we obtain a
pure-jump Markov process taking values in $\CP(E)^N$.

\begin{figure}[htbp]
\centering
\includegraphics[width=0.9\textwidth]{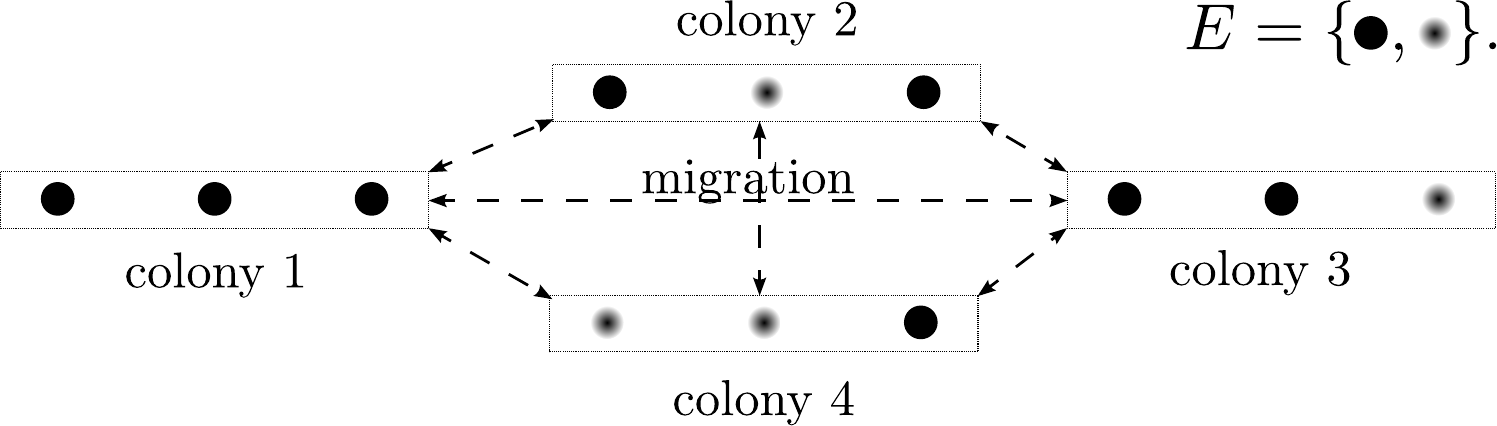}
\caption{\small Possible one-step migration paths between $N=4$ colonies with $M=3$ 
individuals of two types in the mean-field version.}
\label{fig-mfcan}
\end{figure}

Letting $M\to\infty$, we pass to the continuum-mass limit and we obtain a system of 
$N$ interacting $C^\Lambda$-processes, denoted by
\be{ag7}
X^{(N)} = \big(X^{(N)}(t)\big)_{t \geq 0} \quad \text{ with } \quad
X^{(N)}(t) = \big\{X^{(N)}_i(t)\big\}_{i=0}^{N-1}
\in \CP(E)^N.
\ee
The process $X^{(N)}$ can be characterised as the solution of a well-posed martingale
problem on $D([0,\infty),\CP(E)^N)$ with the product topology on $\CP(E)^N$. To this 
end, we have to consider an algebra $\CF \subset C_{\mathrm{b}}(\CP(E)^N,\R)$ of test 
functions, and a linear operator $L^{(N)}$ on $C_{\mathrm{b}}(\CP(E)^N,\R)$ with domain 
$\CF$, playing the role of the generator in the martingale problem. Here, we let $\CF$ 
be the algebra of functions $F$ of the form
\be{ak:test-functions}
\begin{aligned}
&F(x) = \int_{E^n} \left(\bigotimes_{m=1}^n x_{i_m}(\dd u^m)\right)
\varphi\big(u^1,\ldots,u^n\big),
\quad x=(x_0,\ldots,x_{N-1})\in\CP(E)^N,\\
&n\in\N,\,\varphi\in C_{\mathrm{b}}(E^n,\R),\,i_1,\ldots,i_n \in
\{0,\ldots,N-1\}.
\end{aligned}
\ee
The generator
\be{ag18}
L^{(N)}\colon\,\CF \to C_\mathrm{b}\big(\CP(E)^N,\R\big)
\ee
has two parts,
\be{ag20a}
L^{(N)} =  L^{(N)}_{\mathrm{mig}} +  L^{(N)}_{\mathrm{res}}.
\ee

The \emph{migration operator} is given by
\be{ak:migration-operator}
(L^{(N)}_{\mathrm{mig}}F)(x)
= \frac{c}{N} \sum_{i,j=0}^{N-1} \int_E (x_j - x_i)(\dd a)\,
\frac{\partial F(x)}{\partial x_i}[\delta_a],
\ee
where
\be{Gatder}
\frac{\partial F(x)}{\partial x_i}[\delta_a] = \lim_{h \downarrow 0} \frac{1}{h}
\,\Big[F(x_0,\ldots,x_{i-1},x_i+h\delta_a,x_{i+1},\ldots,x_{N-1})-F(x)\Big]
\ee
is the G\^ateaux-derivative of $F$ with respect to $x_i$ in the direction $\delta_a$
(this definition requires that in \eqref{ak:test-functions} we extend $\CP(E)$ to
the set of finite signed measure on $E$). Note that the total derivative in the direction
$\nu \in \CP(E)$ is the integral over $\nu$ of the expression in (\ref{Gatder}),
since $\CP(E)$ is a Choquet simplex and $F$ is continuously differentiable.

The \emph{resampling operator} is given by (cf.~the verbal description of the
single-colony $C^\Lambda$-process in Section~\ref{sss.scolonycannings})
\be{ak:one-level-generator}
\begin{aligned}
(L^{(N)}_{\mathrm{res}}F)(x) = & \sum_{i=0}^{N-1} \int_{(0,1]} 
\Lambda^*(\dd r) \int_E x_i(\dd a)\\
&
\quad \quad \times \Big[F\big(x_0,\ldots,x_{i-1},
(1-r)x_i+r\delta_a,x_{i+1},\ldots,x_{N-1}\big)- F(x)\Big].
\end{aligned}
\ee
Note that, by the law of large numbers, in the limit $M \to \infty$ the evolution 
in (\ref{ag5}--\ref{ag4}) results in the transition $x \to (1-r)x + r\delta_a$ with 
type $a$ drawn from distribution $x$. This gives rise to \eqref{ak:one-level-generator}.

\bp{P.wp}{\bf [Multi-colony martingale problem]}\\ 
Without assumption \eqref{ag2}, for every $x\in\CP(E)^N$, the martingale problem for 
$(L^{(N)},\CF,\delta_x)$ is well-posed. The unique solution is a strong Markov 
process with the Feller property. 
\ep

\noindent
The proof of Proposition~\ref{P.wp} is given in Section~\ref{ss.mppropproofs}.

\subsubsection{$C^\Lambda$-process with immigration-emigration: McKean-Vlasov limit}
\label{sss.MV}

The $N\to\infty$ limit of the $N$-colony model defined in Section~\ref{sss.mcolonycannings} 
can be described in terms of an independent and identically distributed family of 
$\CP(E)$-valued processes indexed by $\N$. Let us describe the distribution of a single 
member of this family, which can be viewed as a spatial variant of the model in 
Section~\ref{sss.scolonycannings} when we add immigration-emigration to/from a 
cemetery state, with the immigration given by a source that is constant in time. 
Such processes are of interest in their own right. They are referred to as 
\emph{McKean-Vlasov processes} for $(c,d,\Lambda,\theta)$, $c,d \in (0,\infty)$, 
$\Lambda \in \CM_f([0,1])$, $\theta \in \CP(E)$, or $C^\Lambda$-processes with 
immigration-emigration at rate $c$ with source $\theta$ and volatility constant $d$.

Let $\mathcal{F} \subseteq C_\mathrm{b}(\CP(E),\R)$ be the algebra of functions $F$ 
of the form
\be{ak:multi-level-test-functions*}
F(x) = \int_{E^n} x^{\otimes n}(\dd u)\,\varphi(u),
\qquad x \in \CP(E),\, n \in \N,\,\varphi \in C_{\mathrm{b}}(E^n,\R).
\ee
Define the second G\^ateaux-derivative of $F$ with respect to $x$ as
\be{ak-second-variation}
\frac{\partial^2 F(x)}{\partial x^2 }[\delta_u,\delta_v]
=
\frac{\partial}{\partial x} \left( \frac{\partial F(x)}{\partial x} [\delta_u] \right)[\delta_v]
,
\quad u,v \in E
.
\ee
For $c,d \in [0,\infty)$, $\Lambda\in\CM_f([0,1])$ subject to (\ref{ag1}--\ref{ag2})
and $\theta\in\CP(E)$, let $L_\theta^{c,d,\Lambda}\colon\,\mathcal{F} \to
C_\mathrm{b}(\CP(E),\R)$ be the linear operator
\be{generic-interaction-operator}
L_\theta^{c,d,\Lambda} = L^c_\theta + L^d + L^\Lambda
\ee
acting on $F\in \mathcal{F}$ as
\be{Ldefs}
\begin{aligned}
(L^c_\theta F)(x) &= c \int_E \left(\theta-x\right)(\dd a)\,
\frac{\partial F(x)}{\partial x}[\delta_a],\\
(L^dF)(x) &= d \int_E \int_E Q_x(\dd u,\dd v)\,
\frac{\partial^2 F(x)}{\partial x^2}[\delta_u,\delta_v],\\
(L^\Lambda F)(x) &= \int_{(0,1]} \Lambda^*(\dd r) \int_E x(\dd a)\,
\big[F\big((1-r)x+ r\delta_a\big)-F(x)\big],
\end{aligned}
\ee
where
\be{flemming-viot-kernel}
Q_x(\dd u,\dd v) = x(\dd u)\,\delta_u(\dd v) - x(\dd u)\,x(\dd v)
\ee
is the Fleming-Viot diffusion coefficient. The three parts of $L_\theta^{c,d,\Lambda}$ 
correspond to: a \emph{drift} towards $\theta$ of strength $c$ (immigration-emigration), 
a \emph{Fleming-Viot diffusion} with \emph{volatility} $d$ (Moran resampling), 
and a $C^\Lambda$-process with resampling measure $\Lambda$ (Cannings resampling).
This model arises as the $M \to\infty$ limit of an individual-based model with $M$ 
individuals at a single site with {\em immigration} from a constant source with type 
distribution $\theta \in \CP(E)$ and emigration to a cemetery state, both at rate $c$, 
in addition to the $\Lambda$-resampling.
 
\bp{prop:McKean-Vlasov-well-posedness}
{\bf [McKean-Vlasov martingale problem]}\\
Without assumption \eqref{ag2}, for every $x \in \CP(E)$, the martingale problem for 
$(L_\theta^{c,d,\Lambda},\mathcal{F},\delta_x)$ is well-posed. The unique solution 
is a strong Markov process with the Feller property.
\ep

\noindent
The proof of Proposition~\ref{prop:McKean-Vlasov-well-posedness} is given in 
Section~\ref{ss.mppropproofs}.

Denote by
\be{Zdef}
Z_\theta^{c,d,\Lambda} = \big(Z_\theta^{c,d,\Lambda}(t)\big)_{t\geq 0},
\quad Z^{c,d,\Lambda}_\theta (0) = \theta,
\ee
the solution of the martingale problem in
Proposition~\ref{prop:McKean-Vlasov-well-posedness}  for the special choice
$x=\theta$. This is called the \emph{McKean-Vlasov process}\footnote{The
terminology stems from the fact that this process describes the limiting
behaviour of an interacting particle system for which propagation of chaos
holds. The physics terminology is related to the fact that the system of
independent components is more random (= more chaotic) than the one with dependent
components. In our context, in the mean-field limit ($N \to \infty$), the components of the system become independent of each
other. Therefore, ``chaos propagates".} with  parameters $c,d,\Lambda$ and
initial state $\theta$.

\subsection{The hierarchical Cannings process}
\label{ss.exthier}

The model described in Section~\ref{sss.mcolonycannings} has a {\em finite} geographical 
space, an interaction that is mean-field, and a resampling of individuals at the same 
site. In this section, we introduce two new features into the model:
\begin{itemize} 
\item[(1)] 
We consider a {\em countably infinite} geographic space, namely, the hierarchical group 
$\Omega_N$ of order $N$, with a migration mechanism that is block-wise exchangeable.
\item[(2)] 
We allow resampling between individuals not only at the same site but also in blocks 
around a site, which we view as {\em macro-colonies}.
\end{itemize}
Both the migration rates and the resampling rates for macro-colonies decay as the
distance between the macro-colonies grows. Feature (1) is introduced in 
Sections~\ref{sss.hg}--\ref{sss.mighg}, feature~(2) in Section~\ref{sss.reshg}.
The hierarchical model is defined in Section~\ref{sss.hierarCan}.

\subsubsection{Hierarchical group of order $N$}
\label{sss.hg}

The {\em hierarchical group $\Omega_N$} of order $N$ is the set
\be{ag30_a}
\Omega_N = \Big\{\eta=(\eta^l)_{l\in\N_0} \in\{0,1,\ldots, N-1\}^{\N_0}
\colon\, \sum_{l\in\N_0} \eta^l < \infty\Big\}, \qquad N\in\N\backslash\{1\},
\ee
endowed with the addition operation $+$ defined by $(\eta+\zeta)^l=\eta^l+\zeta^l
\textrm{ (mod $N$)}$, $l\in\N_0$ (see Fig.~\ref{fig-hierargr} for the case $N=3$). 
In other words, $\Omega_N$ is the direct sum of the cyclical group of order $N$, 
a fact that is important for the application of Fourier analysis. The group $\Omega_N$ 
is equipped with the ultrametric distance $d(\cdot,\cdot)$ defined by
\be{ag31}
d(\eta,\zeta)=d(0,\eta-\zeta)
= \min\{k\in\N_0 \colon\, \eta^l=\zeta^l, \text{ for all } l \geq k\},
\qquad \eta,\zeta\in\Omega_N.
\ee

Let
\be{block-definition}
B_k(\eta) = \{\zeta\in\Omega_N\colon\, d(\eta,\zeta) \leq k\},
\qquad \eta\in\Omega_N,\,k\in\N_0,
\ee
denote the $k$-block around $\eta$, which we think of as a {\em macro-colony}.
The geometry of $\Omega_N$ is explained in 
Fig.~\ref{fig-hierargr}).

\begin{figure}[htbp]
\centering
\includegraphics[width=\textwidth]{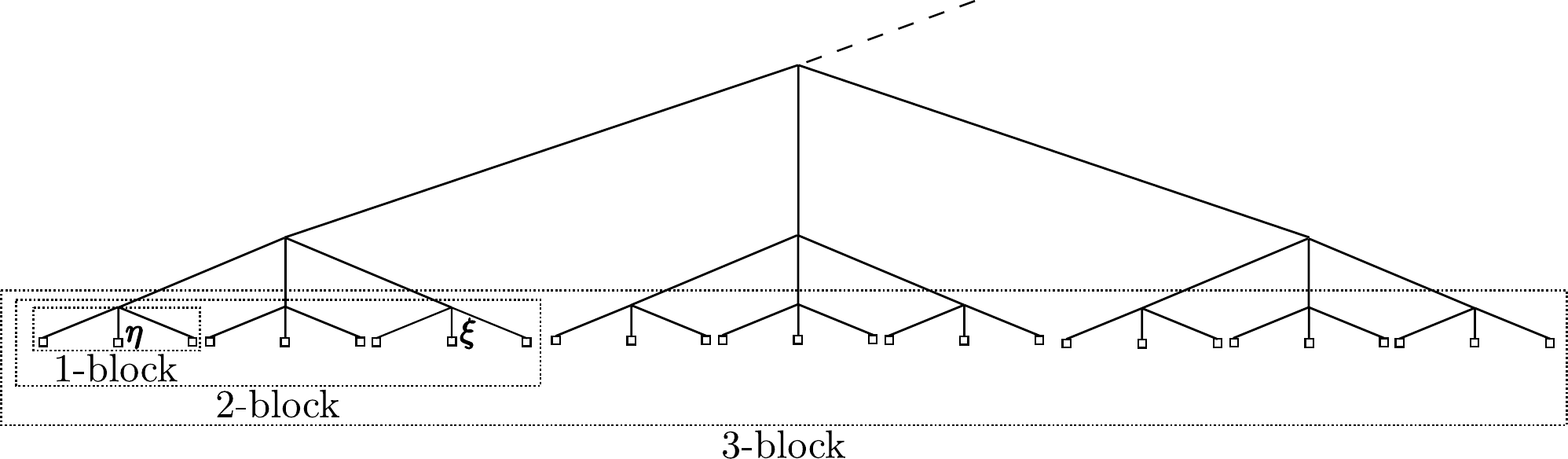}
\caption{\small Close-ups of a 1-block, a 2-block and a 3-block in the hierarchical 
group of order $N=3$. The elements of the group are the leaves of the tree ($\Box$). 
The hierarchical distance between two elements is the graph distance to the most 
recent common ancestor: $d(\xi,\eta) = 2$ for $\xi$ and $\eta$ in the picture.}
\label{fig-hierargr}
\end{figure}

We construct a process 
\be{xdefhg}
X^{(\Omega_N)} = \big(X^{(\Omega_N)}(t)\big)_{t\geq 0} \quad \text{ with } \quad
X^{(\Omega_N)}(t)=\big\{X^{(\Omega_N)}_\eta(t)\big\}_{\eta\in\Omega_N}\in\CP(E)^{\Omega_N},
\ee
by using the same evolution mechanism as for the multi-colony system in
Section~\ref{sss.mcolonycannings}, except that we replace the migration on
$\{0,\dots,N-1\}$ by a migration on $\Omega_N$, and the resampling acting in
each colony by a resampling in each of the macro-colonies. On $\CP(E)^{\Omega_N}$, 
we again choose the product of the weak topology on $\CP(E)$ as the basic topology.

\subsubsection{Block migration}
\label{sss.mighg}

We introduce migration on $\Omega_N$ through a random walk kernel. For that purpose, we 
introduce a sequence of migration rates
\be{ckdef}
\uc = (c_k)_{k\in\N_0} \in (0,\infty)^{\N_0},
\ee
and we let the individuals \emph{migrate} as follows:
\begin{itemize}
\item
Each individual, for every $k\in\N$, chooses at rate $c_{k-1}/N^{k-1}$ the block of 
radius $k$ around its present location and jumps to a location uniformly chosen at 
random in that block.
\end{itemize}
The transition kernel of the random walk that is thus performed by each individual are
\be{32b}
a^{(N)}(\eta,\zeta) = \sum_{k \geq d(\eta,\zeta)} \frac{c_{k-1}}{N^{2k-1}},
\qquad \eta,\zeta\in\Omega_N,\,\eta\neq\zeta, \qquad a^{(N)}(\eta,\eta)=0.
\ee
As shown in Dawson, Gorostiza and Wakolbinger~\cite{DGW05}, this random walk is recurrent
if and only if $\sum_{k\in\N_0} (1/c_k) = \infty$. For the special case where $c_k=c^k$, 
it is strongly recurrent for $c<1$, critically recurrent for $c=1$, and transient for
$c>1$\footnote{Loosely speaking, the behaviour is like that of simple random walk on $\Z^d$ 
with $d<2$, $d=2$ and $d>2$, respectively. More precisely, with the help of potential 
theory it is possible to associate with the random walk a dimension as a function of 
$c$ and $N$ that for $N \to \infty$ converges to 2. This shows that, in the limit as $N\to\infty$, 
the potential theory of the hierarchical random walk given by \eqref{32b} with $c=1$ is similar 
to that of simple random walk on $\Z^2$.}.

Throughout the paper, we assume that\footnote{In Section~\ref{ss.dichoXN}, we will analyse 
the case $N<\infty$, where \eqref{ak:recurrence-cond} must be replaced by $\limsup_{k\to\infty} 
\tfrac{1}{k} \log c_k < \log N$.}
\be{ak:recurrence-cond}
\limsup_{k\to\infty} \tfrac{1}{k} \log c_k < \infty.
\ee
This guarantees that the total migration rate per individual is bounded (at least for 
sufficiently large $N$).

\subsubsection{Block reshuffling-resampling}
\label{sss.reshg}

As we saw in Section~\ref{ss.cannings}, the idea of the Cannings model is to allow
reproduction with an offspring that is of a size comparable to the whole population.
Since we have introduced a spatial structure, we now allow, on all hierarchical 
levels $k$ simultaneously, a reproduction event where each individual treats the 
$k$-block around its present location as a \emph{macro-colony} and uses it for its 
resampling. More precisely, we choose a sequence of finite non-negative resampling 
measures
\be{ag41}
\uL = \big(\Lambda_k)_{k\in\N_0} \in \CM_f([0,1])^{\N_0},
\ee
each subject to \eqref{ag1}. Assume in addition that
\begin{equation}
\label{ak1000}
\int_{(0,1]} \Lambda^*_k(\dd r)  < \infty, \qquad k \in \N, 
\end{equation}
and that $\Lambda_0$ satisfies \eqref{ag2}. The condition in (\ref{ak1000}) is needed to 
guarantee that in finite time a colony is affected by finitely many reshuffling-resampling 
events only, since otherwise this transition cannot be defined (see Remark~\ref{rem:r-expansion} 
at the end of Section \ref{ss.exthier}). The condition in (\ref{ag2}) guarantees that the population
has a well-defined genealogy and most of the population at a site goes back to a finite number
of ancestors after a positive finite time. 

Set
\be{lambda-total-masses}
\lambda_k = \Lambda_k([0,1]), \qquad \lambda^*_k = \Lambda^*_k([0,1]), \qquad k\in\N_0.
\ee
We let individuals \emph{reshuffle-resample} by carrying out the following two steps 
at once (the formal definition requires the use of a suitable Poisson point process: 
cf.~(\ref{ag3}--\ref{ag4}) and \eqref{ag22b}):
\begin{itemize}
\item
For every $\eta\in\Omega_N$ and $k\in\N_0$, choose the block $B_k(\eta)$ at rate
$1/N^{2k}$.
\item
Each individual in $B_k(\eta)$ is first moved to a uniformly chosen random location 
in $B_k(\eta)$, i.e., a reshuffling takes place (see Fig.~\ref{fig-reshuffle}). 
After that, $r$ is drawn according to the intensity measure $\Lambda^*_k$ (recall 
\eqref{ag3}), and with probability $r$ each of the individuals in $B_k(\eta)$ is 
replaced by an individual of type $a$, with $a$ drawn according to the type distribution 
in $B_k(\eta)$, i.e.,
\be{def-blocks}
y_{\eta,k} \equiv N^{-k} \sum_{\zeta\in B_k(\eta)} x_\zeta.
\ee
\end{itemize}
Note that the reshuffling-resampling affects all the individuals in a macro-colony 
simultaneously and in the same manner. The reshuffling-resampling occurs at 
all levels $k\in\N_0$, at a rate that is fastest in single colonies and gets slower 
as the level $k$ of the macro-colony increases.\footnote{Because the reshuffling 
is done first, the resampling always acts on a uniformly distributed state (``panmictic 
resampling''). }

\begin{figure}[htbp]
\label{fig:reshuffling} 
\centering
\includegraphics[width=\textwidth]{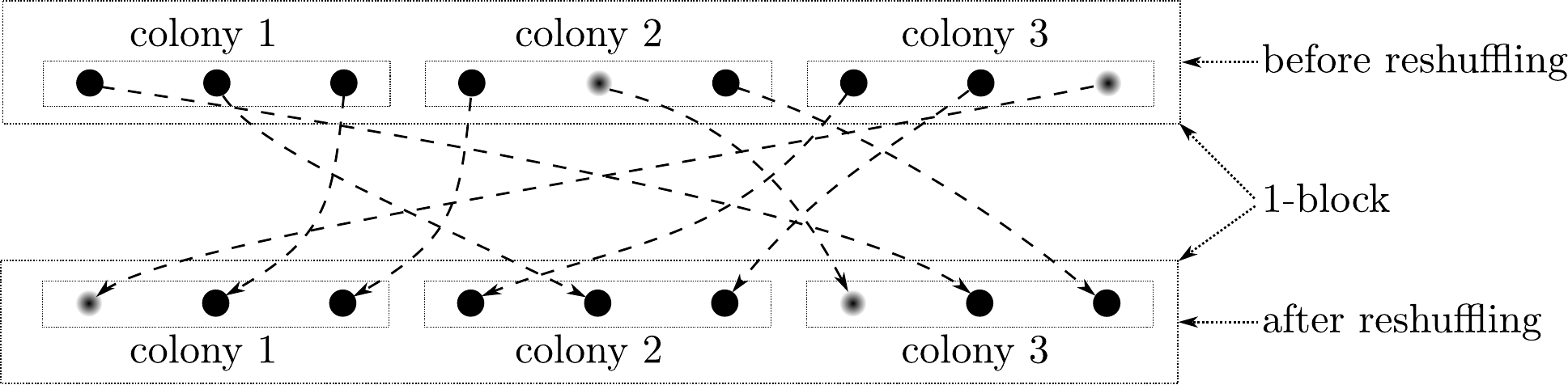} 
\caption{\small
Random reshuffling in a 1-block on the hierarchical lattice of order $N=3$
with $M=3$ individuals of two types per colony.}
\label{fig-reshuffle}
\end{figure}

Throughout the paper, we assume that $\ul^*=(\lambda^*_k)_{k\in\N_0}$ (recall the 
definition of $\lambda^*_k$ from \eqref{lambda-total-masses}) satisfies\footnote{In 
Section~\ref{ss.dichoXN}, we will analyse the case $N<\infty$, where
\eqref{ak:lambda-growth-condition} must be replaced by $\limsup_{k\to\infty} 
\tfrac{1}{k}\,\log \lambda^*_k < \log N$.}
\be{ak:lambda-growth-condition}
\limsup_{k\to\infty} \tfrac{1}{k}\,\log \lambda^*_k < \infty.
\ee
Note that each of the $N^k$ colonies in a $k$-block can trigger reshuffling-resampling 
in that block, and for each colony the block is chosen at rate $N^{-2k}$. Therefore
\eqref{ak:lambda-growth-condition} guarantees that the total resampling rate per 
individual is bounded.

In the continuum-mass limit, the reshuffling-resampling operation, when it acts on 
the states in the colonies, takes the form
\be{ag12}
{ x_\zeta } \text{ is replaced by }
(1-r)y_{\eta,k}+r\delta_a \text{ for all } \zeta\in B_k(\eta)
\ee
with $a \in E$ drawn from $y_{\eta,k}$ (the type distribution in $B_k(\eta)$ (cf.\ \eqref{def-blocks}). 
Note that in the mean-field case and in the single-colony case of Section~\ref{sss.scolonycannings}, 
$a \in E$ is drawn from $x_\zeta$ (cf.\ \eqref{ak:one-level-generator} and the comment following 
it)\footnote{Reshuffling is a parallel update affecting all individuals in a macro-colony simultaneously. 
Therefore it cannot be seen as a migration of individuals equipped with independent clocks.}.

\subsubsection{Hierarchical Cannings process}
\label{sss.hierarCan}

We are now ready to formally define our system of \emph{hierarchically interacting 
$C^\Lambda$-processes} in terms of a martingale problem. This is the continuum-mass 
limit ($M \rightarrow \infty$) of the individual-based model that we described in
Sections~\ref{sss.hg}--\ref{sss.reshg}. Recall that so far we have considered block 
migration and non-local reshuffling-resampling on the hierarchical group of fixed order 
$N$, starting with $M$ individuals at each site. 

We equip the set $\CP(E)^{\Omega_N}$ with the product topology to get a state space 
that is Polish. Let $\mathcal{F} \subset C_\mathrm{b}\big(\CP(E)^{\Omega_N},\R\big)$ 
be the algebra of functions of the form
\be{ak:multi-level-test-functions}
\begin{aligned}
&F(x) = \int_{E^n} \left(\bigotimes_{m=1}^n x_{\eta_m}\big(\dd u^m\big)\right)
\varphi\big(u^1,\ldots,u^n\big),
\quad x = (x_{\eta})_{\eta\in\Omega_N}\in\CP(E)^{\Omega_N},\\
&n \in \N, \quad \varphi \in C_{\mathrm{b}}(E^n,\R), 
\quad \eta_1,\ldots,\eta_n \in \Omega_N.
\end{aligned}
\ee
The linear operator for the martingale problem
\be{ak:multi-level-generator}
L^{(\Omega_N)}\colon\,\mathcal{F} \to C_\mathrm{b}\big(\CP(E)^{\Omega_N},\R\big)
\ee
again has two parts,
\be{ak:multi-level-generator-decomposition}
L^{(\Omega_N)}=L^{(\Omega_N)}_{\mathrm{mig}}+L^{(\Omega_N)}_{\mathrm{res}}.
\ee
The \emph{migration operator} is given by
\be{ak:multi-level-migration}
(L^{(\Omega_N)}_{\mathrm{mig}} F)(x)
= \sum_{\eta,\zeta \in \Omega_N} a^{(N)}(\eta,\zeta)
\int_E (x_{\zeta} - x_{\eta})(\dd a)\,\frac{\partial F(x)}{\partial x_\eta}[\delta_a]
\ee
and the \emph{reshuffling-resampling operator} by
\be{ak:multi-level-resampling-global}
\begin{aligned}
(L^{(\Omega_N)}_{\mathrm{res}} F)(x)
= \sum_{\eta \in \Omega_N} 
&
\left(
(L^{d_0}_{\eta} F)(x)
+
\int_{(0,1]} \Lambda^*_0 (\dd r) \int_E x_{\eta}(\dd a)
\left[F\left(\Phi_{r,a,\{\eta\}}(x)\right)-F(x)\right]
\right.
\\
&
\left.
+
\sum_{k\in\N} N^{-2k}
\int_{(0,1]} \Lambda^*_k (\dd r) \int_E y_{\eta,k}(\dd a)
\left[F\left(\Phi_{r,a,B_k(\eta)}(x)\right)-F(x)\right]
\right)
,
\end{aligned}
\ee
where $\Phi_{r,a,B_k(\eta)}\colon\,\CP(E)^{\Omega_N}\to\CP(E)^{\Omega_N}$ is the
\emph{reshuffling-resampling map} acting as
\be{ak:resampling-mapping-global}
\Big[\big(\Phi_{r,a,B_k(\eta)}\big)(x)\Big]_\zeta =
\begin{cases}
(1-r) y_{\eta,k} + r \delta_a, &\zeta \in B_k(\eta),\\
x_{\zeta}, &\zeta \notin B_k(\eta),
\end{cases}
\ee
where $r \in [0,1]$, $a \in E$, $k\in\N_0$, $\eta\in\Omega_N$, and
$L^{d_0}_{\eta}$ is the Fleming-Viot diffusion operator with volatility $d_0$
(see~\eqref{Ldefs}) acting on the colony $x_\eta$ with
\be{ak:d-0}
d_0  \geq 0.
\ee

\begin{remark}
\label{rem:r-expansion}
{\rm (1)}
If $d_0 = 0$, then the operator in \eqref{ak:multi-level-resampling-global} is pure-jump.\\
{\rm (2)}
The right-hand side of \eqref{ak:multi-level-resampling-global} is well-defined because 
of \eqref{ak1000}. Indeed, by Taylor-expanding the inner integral in 
\eqref{ak:multi-level-resampling-global} in powers of $r$, we get
\begin{align}
\label{ak:9999}
\int_E y_{\eta,k}(\dd a)
\left[F\left(\Phi_{r,a,B_k(\eta)}(x)\right)-F(x)\right]
=  F(y_{\eta,k})-F(x)+O(r^2),
\quad \text{as } r \downarrow 0.
\end{align}
To have a well-defined resampling operator \eqref{ak:multi-level-resampling-global}, 
the expression in \eqref{ak:9999} must be integrable with respect to $\Lambda^*_k (\dd r)$, 
which is equivalent to assumption~\eqref{ak1000}.
\end{remark}

\begin{proposition}
\label{P.vecLambda}
{\bf [Hierarchical martingale problem]}\\
Without assumption \eqref{ag2}, for every $\Theta \in \CP(E)^{\Omega_N}$, the 
martingale problem for $(L^{(\Omega_N)},\CF,\delta_{\Theta})$ is well-posed\footnote{As 
a part of the definition of the martingale problem, we always require that the solution has 
c\`adl\`ag paths and is adapted to a natural filtration.}. 
The unique solution is a strong Markov process with the Feller property. 
\end{proposition}

\noindent
The proof of Proposition~\ref{P.vecLambda} is given in Section~\ref{ss.mppropproofs}.

The Markov process arising as the solution of the above martingale problem is denoted 
by $X^{(\Omega_N)}=(X^{(\Omega_N)}(t))_{t\geq 0}$, and is referred to as the 
$C_N^{\uc,\uL}$-process on $\Omega_N$.

\medskip\noindent
{\bf Remark:} 
For the analysis of the $C_N^{\uc,\uL}$-process, the following auxiliary models will 
be important later on. Given $K \in \N_0$, consider the finite geographical space
\be{ag17}
G_{N,K} = \{0,\ldots, N-1\}^K,
\ee
which is a truncation of the hierarchical group $\Omega_N$ after $K$ levels. Equip 
$G_{N,K}$ with coordinate-wise addition modulo $N$, which turns it into a finite 
Abelian group. By restricting the migration and the resampling to $G_{N,K}$ (i.e., 
by setting $c_k=0$ and $\Lambda_k=0$ for $k \geq K$), we obtain a Markov process 
with geographic space $G_{N,K}$ that can be characterised by a martingale problem 
as well. In the limit as $K\to\infty$, this Markov process can be used to approximate 
the $C_N^{\uc,\uL}$-process. This approximation of $X^{(\Omega_N)}$ by $X^{(G_{N,K})}$ 
is made rigorous in Proposition~\ref{P.Kapprox}

\medskip\noindent
{\bf Remark:} Similarly to the mean-field Cannings process $X^{(N)}$ from Section~\ref{sss.mcolonycannings}, the hierarchical Cannings process $X^{(\Omega_N)}$ can be obtained as a $M \to \infty$ limit of the finite $M$ individual-based models.

\subsection{Main results}
\label{ss.mainth}

Our main results concern a multiscale analysis of the $C_N^{\uc,\uL}$-process on 
$\Omega_N$, $X^{(\Omega_N)}$ (cf. below Proposition~\ref{P.vecLambda}) in the
limit as $N\to\infty$. To that end, we introduce {\em renormalised} systems with the 
proper {\em space-time} scaling.

For each $k \in \N_0$, we look at the \emph{$k$-block averages} defined by
\be{k-block-average}
Y_{\eta,k}^{(\Omega_N)}(t) = \frac{1}{N^k}
\sum_{\zeta\in B_k(\eta)} X_{\zeta}^{(\Omega_N)}(t), 
\qquad \eta \in \Omega_N,
\ee
which constitute a {\em renormalisation of space} where the component $\eta$ is 
replaced by the average in $B_k(\eta)$. The corresponding {\em renormalisation 
of time} is to replace $t$ by $tN^k$, i.e., $t$ is the associated macroscopic time 
variable. For each $k \in \N_0$ and $\eta \in \Omega_N$, we can thus introduce 
a {\em renormalised} interacting system
\be{a3}
\left(\left(Y^{(\Omega_N)}_{\eta,k}
(t N^k)\right)_{\eta \in \Omega_N}\right)_{t \geq 0},
\ee
which is constant in $B_k(\eta)$ and can be viewed as an interacting system indexed 
by the set $\Omega^{(k)}_N$ that is obtained from $\Omega_N$ by dropping the first 
$k$-entries of $\eta \in \Omega_N$ (recall \eqref{ag30_a}). This provides us with a 
\emph{sequence of renormalised interacting systems}, which for fixed $N$ are however 
{\em not} Markov.

Our main results are stated in Sections~\ref{sss.scaling}--\ref{sss.extmsp}. In 
Section~\ref{sss.scaling}, we state the scaling behaviour of the renormalised
interacting system in \eqref{a3} as $N\to\infty$ for fixed $k\in\N_0$. In 
Section~\ref{sss.extmsp}, we look at the interaction chain that captures the 
scaling behaviour on all scales simultaneously. In Section~\ref{ss.dichoXN}, we 
take a look at our system $X^{(\Omega_N)}$ for finite $N$. In 
Section~\ref{sss.univclass}, we compare the result with the hierarchical 
Fleming-Viot process. In Sections~\ref{sss.polregvar}--\ref{sss.expregvar},
we identify the different regimes for $k\to\infty$ and in Section~\ref{sec:cluster-foramtion} 
we investigate cluster formation.

\subsubsection{The hierarchical mean-field limit}
\label{sss.scaling}

Our first main theorem identifies the scaling behaviour of $X^{(\Omega_N)}$ 
as $N\to\infty$ (the so-called hierarchical mean-field limit) for every 
fixed block scale $k\in\N_0$. We assume that, for each $N$, the law of 
$X^{(\Omega_N)}(0)$ is the restriction to $\Omega_N$ of a random field $X$ 
indexed by $\Omega_\infty=\bigoplus_\N \N$ that is taken to be i.i.d.\ 
with a single-site mean $\theta$ for some $\theta\in\CP(E)$.

Recall \eqref{lambda-total-masses} and \eqref{ak:d-0}. Let $\ud=(d_k)_{k\in\N_0}$ 
be the sequence of \emph{volatility constants} defined recursively as
\be{diffusion-constants}
d_{k+1} = \frac{c_k(\tfrac12\lambda_k +d_k)}
{c_k + (\tfrac12\lambda_k + d_k)},
\quad k \in \N_0.
\ee
Let $\mathcal{L}$ 
denote law, let $\Longrightarrow$ 
denote weak convergence on path 
space, and recall (\ref{Zdef}).

\begin{theorem}
\label{mainth}
{\bf [Hierarchical mean-field limit and renormalisation]}\\
For every $k\in\N$, uniformly in $\eta\in\Omega_\infty$,
\be{macroscopic-behaviour}
\mathcal{L}
\left[\left(Y_{\eta,k}^{(\Omega_N)}(t N^k)\right)_{t\geq 0}\right]
\Ntoo \mathcal{L}
\left[\left(Z_\theta^{c_k,d_k,\Lambda_k}(t)\right)_{t\geq 0}\right].
\ee 
For $k=0$, \eqref{macroscopic-behaviour} is still true, but the McKean-Vlasov 
process must be started from $Z(0) = X^{(\Omega_N)}_\eta(0)$ instead of 
$Z(0)=\theta$ (cf.~\eqref{Zdef}).
\end{theorem}

\noindent
The proof of Theorem~\ref{mainth} is given in Section~\ref{s.proof1stth}. The limiting 
process in \eqref{macroscopic-behaviour} is a \textit{McKean-Vlasov process} with 
drift constant $c=c_k$ and resampling measure $d_k \delta_0 + \Lambda_k$ (cf.\ 
\eqref{Ldefs}). This shows that the class of Cannings models with block resampling 
is preserved under the renormalisation. 

\myparagraph{Heuristics.}  In order to understand the origin of the recursion
relation in \eqref{diffusion-constants},  let us start by explaining where
$d_1=c_0 \lambda_0/(2 c_0 + \lambda_0)$ comes from.  Consider two
\textit{lineages}\footnote{The fact that we consider coalescing lineages as 
opposed to type distributions is actually the essence of the \emph{duality} 
approach to the study of the dynamics of interacting particle systems. In 
the present context, duality is developed in Section~\ref{s.spatcoal}.} 
drawn at random from a macro-colony of order $1$, say $B_1(\eta)$ 
for some $\eta \in \Omega_N$. Due to migration, both lineages are
uniformly distributed  over the macro-colony after the first migration step. For
each lineage, marking the  migration steps that result in being in the same
colony, we get a Poisson process with  rate $2c_0$ on timescale $N t$. For every
such mark, the rate to coalesce is $\lambda_0 N$  (on time scale $N t$), while
the rate to migration away is $2 c_0 N$. Hence, the probability  that the two
lineages coalesce before they migrate away is $\lambda_0/(2c_0+\lambda_0 )$. 
Therefore, thinning the Poisson process with rate $2 c_0$, we see that the two
lineages  coalesce at rate $2c_0\lambda_0/(2c_0+\lambda_0)$. Since the
coalescence rate is  twice the diffusion coefficient (cf., Section~\ref{ss.1st2nd}), this gives a heuristic
explanation for $d_1$. Note that three  lineages are within the same colony only
after a time of order $N^2$, so three lineages  do not coalesce on time scale
$Nt$.

To understand the generic step of the recursion relation, i.e., $d_{k+1}$, consider a
macro-colony of order $k+1$, say $B_{k+1}(\eta)$ for some $\eta \in \Omega_N$, 
and two lineages drawn at random from this macro-colony. Consider only migration 
on level $k$, i.e., migration events between the macro-colonies of order $k$, which 
occur at rate $2c_k N^{-k}$. For every such event, the rate of coalescence is $2d_k
+\lambda_k$, while migration of one of them occurs at rate $2c_k$. Hence, the
probability that the two lineages coalesce before one of them migrates is $(2d_k+
\lambda_k)/(2c_k+2d_k +\lambda_k)$. After speeding up time by a factor $N$, 
we see that the coalescence rate is $2c_k(2d_k+\lambda_k )/(2c_k+2d_k+\lambda_k)$. 
Since the coalescence rate is twice the diffusion coefficient, this gives a heuristic 
explanation for $d_k$. Again, three or more lineages do not coalesce on the same 
time scale.

\subsubsection{Multi-scale analysis: the interaction chain}
\label{sss.extmsp}

\myparagraph{Multi-scale behaviour.}
Our second main theorem looks at the implications of the scaling behaviour of $d_k$ as 
$k\to\infty$, to be described in Theorems~\ref{T.beh}--\ref{dcases} in 
Section~\ref{sss.univclass}--\ref{sss.polregvar}, for which we must extend Theorem~\ref{mainth} 
to include \emph{multi-scale renormalisation}. This is done by considering {\em two} indices $(j,k)$ 
and introducing an appropriate multi-scale limiting process, called the {\em interaction chain}
\be{inch}
M^{(j)}=(M^{(j)}_k)_{k=-(j+1),\ldots,0}, \qquad j \in \N_0,
\ee 
which describes all the block averages of size $N^{|k|}$ indexed by $k=-(j+1),\ldots,0$ 
\emph{simultaneously} at time $N^j t$ with $j\in\N_0$ fixed. Formally, the interaction 
chain is defined as the time-inhomogeneous Markov chain with a prescribed initial 
state at time $-(j+1)$, 
\be{ag61}
M^{(j)}_{-(j+1)} = \theta \in \CP(E),
\ee
and with transition kernel
\be{ag14}
K_k(x,\cdot) = \nu^{c_k,d_k,\Lambda_k}_x(\cdot),
\quad x \in \CP(E),\,k \in \N_0,
\ee
for the transition from time $-(k+1)$ to time $-k$ (for $k=j,\ldots,0$). Here, 
$\nu^{c,d,\Lambda}_x$ is the unique equilibrium of the McKean-Vlasov process 
$Z^{c,d,\Lambda}_x$ defined in \eqref{Ldefs} of Section~\ref{sss.MV} (see 
Section~\ref{s.mcv_imem} for details). 

\begin{theorem}   
\label{Tscb}
{\bf [Multi-scale behaviour]}\\   
Let $(t_N)_{N\in\N}$ be such that 
\begin{equation}
\begin{aligned}
\lim_{N\to\infty} t_N=\infty \text{ and } \lim_{N\to\infty}t_N/N=0.
\end{aligned}
\end{equation}
Then, for every $j\in\N_0$, uniformly in $\eta\in \Omega_\infty$ and $u_k \in (0,\infty)$,
\be{ag16}
\begin{aligned}
&\CL \left[\left(Y^{(\Omega_N)}_{\eta,k} (N^j t_N + N^k u_k)\right)_{k=j,\ldots,0}\right]
\Ntoo
\CL \left[\left(M^{(j)}_{-k}\right)_{k=j,\ldots,0}\right],\\
&\CL \left[Y^{(\Omega_N)}_{\eta,j+1} (N^j t_N)\right] \Ntoo \delta_\theta,
\end{aligned}
\ee
where $\theta \in \CP(E)$ is the single-site mean of the initial distribution
$X^{(\Omega_N)}(0)$, cf.~Section~\ref{sss.scaling}.

\end{theorem}

\noindent
The proof of Theorem~\ref{Tscb} is given in Section~\ref{s.msana}.

Theorem~\ref{Tscb} says that, as $N \to \infty$, the system is in a {\em quasi-equilibrium}
$\nu_x^{c_k, d_k, \Lambda_k}$ on time scale $N^j t_N +N^k u$, with $u \in (0,\infty)$ 
the macroscopic time parameter on level $k$, when $x$ is the average on level 
$k+1$. 

\myparagraph{Heuristics.} The effect described in Theorem~\ref{Tscb} results from
the fact that on the smaller time scale $uN^k$ a $k$-block average evolves
effectively like a single component of the $N-1$ other $k$-block averages with a
mean-field migration mechanism. This leads to \textit{propagation of chaos},
i.e., convergence to a system of independently evolving components that 
interact only because they feel the overall type density in the $(k+1)$-block. 
Since we look at the system at a late time $N^j t_N$, we see that the dynamics 
at scale $N^ku$, which is $o(N^j t_N)$, has already reached equilibrium, as 
is clear from a restart argument that absorbs an order-$N^k$ term into
$N^j t_N$.

\myparagraph{The basic dichotomy.}
We next let the index in the multi-scale renormalisation scheme tend to 
infinity and identify how the limit depends on the parameters $(\uc,\uL)$. Indeed, 
Theorem~\ref{Tscb}, in combination with Theorems~\ref{T.beh}--\ref{dcases} in
Sections~\ref{sss.univclass}--\ref{sss.polregvar}, allows us to study the universality 
properties on large space-time scales when we first let $N \to \infty$ and then 
$j \to \infty$\footnote{For several previously investigated systems, the limit as 
$j\to\infty$ was shown to be interchangeable (Dawson, Greven and Vaillancourt~\cite{DGV95}, 
Fleischmann and Greven~\cite{FG94}).)}.

The interaction chain exhibits a {\em dichotomy}, as will be seen in Theorem~\ref{T.dicho} 
below, in the sense that
\be{jinflim}
\CL\left[M^{(j)}_0\right] \jto \nu_\theta \in \CP(\CP(E)),
\ee
with $\nu_\theta$ either (I) of the form of a random single-atom measure, i.e., 
\be{form1}
\nu_\theta = \CL[\delta_U], \text{ for some random } U \in E
\mbox{ with } \CL[U]=\theta,
\ee 
or (II) $\nu_\theta$ spread out.
To be more specific, define
\be{vardef2}
\var_x(\psi)
= \int_{E \times E} [x(\dd u)\delta_u(\dd v) - x(\dd u) x(\dd v)]\,\psi(u)\psi(v).
\ee
Then, $\nu_\theta$ is spread out iff
\be{spread}
\sup_{\psi \in B_1} \E_{\nu_\theta} [\var_{\cdot}(\psi)] > 0,
\ee
where $B_1 \equiv C_\mathrm{b} (E,\R) \cap \{\psi\colon\,|\psi| \leq 1\}$ and the
expectation is taken with respect to the parameter $x$ in \eqref{vardef2}, i.e.,
\be{vardef1}
\E_{\nu_\theta} [\var_{\cdot}(\psi)] = \int_{\CP(E)} \nu_\theta(\dd x)
\var_x(\psi)
.
\ee
Case (I) is called the {\em clustering regime}, since it indicates the formation 
of large mono-type regions, while case (II) is called the {\em local coexistence 
regime}, since it indicates the formation of multi-type local equilibria under which
different types can live next to each other with a positive probability. In the local 
coexistence regime, a remarkable difference occurs comparing with the hierarchical 
Fleming-Viot process: mono-type regions for $M^{(j)}_0$ as $j\to\infty$ have a 
probability in the open interval $(0,1)$ rather than probability $0$ (see 
Proposition~\ref{P.mkv}(b) below). The latter is referred to in \cite{DGV95} 
by saying that the system is in the {\em stable regime} (which is stronger than 
local coexistence). In the present paper, we do not identify the conditions on 
$\uc$ and $\ul$ that correspond to the stable regime. The dichotomy can be 
conveniently rephrased as follows: There is either a trivial or a non-trivial 
{\em entrance law} for the interaction chain with initial state $\theta \in \CP(E)$ 
at time $-\infty$\footnote{Recall that an entrance law for a sequence of transition kernels 
$(K_k)_{k=-\infty}^0$ and an entrance state $\theta$ is any law of a Markov 
chain $(Y_k)_{k=-\infty}^0$ with these transition kernels such that $\lim_{k\to -\infty} 
Y_k = \theta$.}.

\myparagraph{Explicit dichotomy criterion.} The large-scale behaviour of $X^{(\Omega_N)}$ 
is determined by the sequence $\umu=(m_k)_{k\in\N_0}$ with
\be{ag46}
m_k = \frac{\mu_k + d_k}{c_k}, \mbox{ where } \mu_k=\tfrac12\lambda_k
\ee
(recall $c_k$ from \eqref{ckdef}, $\lambda_k$ from \eqref{lambda-total-masses} 
and $d_k$ from \eqref{diffusion-constants}). We will argue that the dichotomy
\be{ag46b} 
\sum_{k\in\N_0} m_k=\infty  \quad \text{vs.} \quad \sum_{k\in\N_0} m_k<\infty
\ee
represents qualitatively different situations for the interacting system $X^{(\Omega_N)}$ 
corresponding to, respectively,
\begin{itemize}
\item
\emph{clustering} (= formation of large mono-type regions), 
\item
\emph{local coexistence} (= convergence to multi-type equilibria).
\end{itemize}
In the clustering regime, the scaling behaviour of $d_k$ is independent of $d_0$, 
while in the local coexistence regime it depends on $d_0$. In \eqref{ag58} of 
Section~\ref{ss.1st2nd}, we will show that 
\be{ag50}
\E_{\mathcal{L}[M^{(j)}_0]} \left[\var_{\cdot}(\psi)\right] 
= \left[\prod_{k=0}^j \frac{1}{1+m_k}\right]\,\var_\theta(\psi), 
\quad j\in\N_0,\,\psi \in C_\mathrm{b}(E,\R),\,\theta \in \CP(E).
\ee
This implies that the entrance law is trivial when $\sum_{k\in\N_0} m_k=\infty$ and 
non-trivial when $\sum_{k\in\N_0} m_k<\infty$. Our third main theorem identifies the
dichotomy.

\begin{theorem} 
\label{T.dicho}
{\bf [Dichotomy of the entrance law]}
\begin{itemize}
\item[\textup{(a)}] 
The interaction chain converges to an entrance law:
\be{ak9}
\begin{cases}
\CL\left[
\left(M^{(j)}_k\right)_{k= -(j+1),\ldots,0}\right] \jto
\CL\left[\left(M^{(\infty)}_k\right)_{k=-\infty,\dots,0}\right],&\\
M^{(\infty)}_{-\infty}=\theta.&
\end{cases}
\ee
\item[\textup{(b)}] 
{\bf[Clustering]}
If $\sum_{k\in\N_0} m_k=\infty$, then $\CL[M^{(j)}_0] \jto
\CL[\delta_U]$ with $\CL[U]=\theta$.
\item[\textup{(c)}] 
{\bf[Local coexistence]}
If $\sum_{k\in\N_0} m_k<\infty$, then 
\be{ak:201402081402}
\sup_{\psi \in C_\mathrm{b}(E,\R)} \E_{\mathcal{L}[M^{(\infty)}_0]}[\var_\cdot(\psi)]>0.
\ee
\end{itemize}
\end{theorem}

\noindent
The proof of Theorem~\ref{T.dicho} is given in Section~\ref{ss.dichointeract}.

Theorem~\ref{T.dicho}, in combination with Theorem~\ref{T.beh}(c) in Section~\ref{sss.univclass},
says that, like for Fleming-Viot diffusions, we have a clear-cut criterion for the two 
regimes in terms of the migration coefficients and the resampling coefficients. 

\myparagraph{Heuristics.}
If the resampling happens only locally, i.e., $\lambda_k =0$, for $k \in \N$, we simply obtain the two regimes
depending on whether two ancestral lines coalesce with probability 1 or $<1$, giving after a long time
monotype or coexistence, if and only if they meet with probability 1 or $<1$. Now, the ancestral lines can coalesce
due to the reshuffling-resampling in a $k$-ball and hence the occupation time of two ancestral lines in
the distances $k$ weighted by the $\lambda_k$ is the relevant quantity.

\subsubsection{Main results for finite $N$}
\label{ss.dichoXN}

In this section, we take a look at our system $X^{(\Omega_N)}$ ($C_N^{\uc,\uL}$-process 
on $\Omega_N$, cf. below Proposition~\ref{P.vecLambda}) for finite $N$, i.e., 
without taking the hierarchical mean-field limit. We ask whether this system also
exhibits a dichotomy of clustering versus local coexistence, i.e., for fixed
$N$ and $t\to\infty$, does $\CL[X^{(\Omega_N)}(t)]$ converge to a mono-type state, 
where the type is distributed according to $\theta$, or to an equilibrium state, 
where different types live next to each other? 

As it will turn out below, in the finite-$N$ case there is the dichotomy and, moreover, the quantitative
criterion is the same as in the $N \to \infty$ limit.

Concretely, let $P_t(\cdot,\cdot)$ denote the transition kernel of the random walk on $\Omega_N$ 
with migration coefficients
\be{a2}
\bar{c}_k(N) = c_k + N^{-1} \lambda_{k+1}, \quad k \in \N_0
\ee 
starting at 0 (cf.~Section~\ref{sss.mighg}). Let
\be{stable-ergodicity-condition}
\bar H_N = \sum_{k\in\N_0} \lambda_k N^{-k} 
\int_0^\infty P_{2s}(0,B_k(0))\,\dd s,
\ee
where $B_k(0)$ is the $k$-block in $\Omega_N$ around 0 (recall
\eqref{block-definition}) and $P_t(0,B_k(0))\equiv\sum_{\zeta\in B_k(0)}
P_t(0,\eta)$. We will see in  Section~\ref{sss.dichosv} that $\bar H_N$ in
\eqref{stable-ergodicity-condition} is the \emph{expected hazard} for two 
partition elements in the spatial $\uL$-coalescent with non-local coalescence 
to coalesce. Note in particular that the second summand in \eqref{a2} is 
induced by the reshuffling in the spatial $\uL$-coalescent with
non-local coalescence.

Our next three main theorems identify the ergodic behaviour for finite $N$.

\begin{theorem}
\label{dich_coclu}
{\bf [Dichotomy for finite $N$]}\\
The following dichotomy holds for every $N\in\N\backslash\{1\}$ fixed:
\begin{itemize}
\item[\textup{(a)}] 
{\bf [Local coexistence]}
If $\bar H_N<\infty$, then
\be{ag52a}
\liminf_{t \to \infty} \sup_{\psi \in B_1} 
\E_{X^{(\Omega_N)}_\eta(t)}[ \var_\cdot(\psi)] > 0,
\quad \text{for all }\eta\in\Omega_N.
\ee
\item[\textup{(b)}] 
{\bf [Clustering]}
If $\bar H_N=\infty$, then
\be{ag52b}
\lim_{t \to \infty} \sup_{\psi \in B_1} 
\E_{X^{(\Omega_N)}_\eta(t)} [\var_\cdot(\psi)] = 0,
\quad \text{for all }\eta\in\Omega_N.
\ee
\end{itemize}
\end{theorem}

\noindent
The proof of Theorem~\ref{dich_coclu} is given in Section~\ref{ss.cluscoex}. 

The dichotomy can be sharpened by using duality theory and the complete longtime 
behaviour of $X^{(\Omega_N)}$ can be identified.

\begin{theorem}
\label{P.ergodic}
{\bf [Ergodic behaviour for finite $N$]}\\
The following dichotomy holds:
\begin{itemize}
\item[\textup{(a)}]
{\bf [Local coexistence]}
If $\bar H_N<\infty$, then for every $\theta \in \CP(E)$ and every $X^{(\Omega_N)}(0)$ 
whose law is stationary and ergodic w.r.t.\ translations in $\Omega_N$ and has a 
single-site mean $\theta$,
\be{stable-ergodicity}
\CL\left[ X^{(\Omega_N)}(t)\right] \tto \nu^{(\Omega_N),\uc,\ul}_\theta
\in \CP(\CP(E)^{\Omega_N})
\ee
for some unique law $\nu^{(\Omega_N),\uc,\ul}_\theta$ that is stationary and ergodic 
w.r.t.\ translations in $\Omega_N$ and has single-site mean $\theta$.
\item[\textup{(b)}]
{\bf [Clustering]} If $\bar H_N =\infty$, then, for every $\theta \in \CP(E)$, 
\be{nonstable-ergodicity}
\CL\left[X^{(\Omega_N)}(t)\right] \tto \int_0^1 \theta(\dd u) 
\delta_{(\delta_u)^{\Omega_N}} \in \CP(\CP(E)^{\Omega_N}).
\ee
\end{itemize}
\end{theorem}

\noindent
The proof of Theorem~\ref{P.ergodic} is given in Section~\ref{ss.cluscoex}.

\begin{theorem}
\label{dichodual}
{\bf [Agreement of dichotomy for $N < \infty$ and $N=\infty$]}\\
Under the weak regularity condition
\be{dichreg}
\text{ either } \quad \limsup_{k\to\infty} \frac{\lambda_{k+1}}{c_k} < \infty \quad \text{ or } \quad 
\liminf_{k\to\infty} \left(\frac{\lambda_{k+1}}{c_k} \wedge \frac{\lambda_k}{\lambda_{k+1}}\right)>0,
\ee
the dichotomies in Theorems~{\rm \ref{T.dicho}} and {\rm\ref{P.ergodic}} coincide  i.e., 
$\sum_{k\in\N_0} m_k=\infty$ if and only if $\bar{H}_N =\infty$. 
\end{theorem}

\noindent
The proof of Theorem~\ref{dichodual} is given in Section~\ref{ss.gencomp}.

\subsubsection{Comparison with the dichotomy for the hierarchical Fleming-Viot process}
\label{sss.univclass}

We return to the case $N=\infty$. For the classical case of hierarchically interacting Fleming-Viot 
diffusions (i.e., in the absence of non-local reshuffling-resampling), the dichotomy was analysed 
in \cite{DGV95}. It was shown there that the dichotomy in \eqref{ag46b} reduces to 
\begin{equation}
\label{dichclass}
\sum_{k\in\N_0} (1/c_k)=\infty \quad \mbox{vs.} \quad 
\sum_{k\in\N_0} (1/c_k)<\infty,
\end{equation} 
corresponding to the random walk with migration coefficients $\uc=(c_k)_{k\in\N_0}$ 
being recurrent, respectively, transient. Moreover, it is known that in the clustering
regime $\lim_{k\to\infty} \sigma_k d_k =1$ with $\sigma_k = \sum_{l=0}^{k-1} (1/c_l)$
for all $d_0$. 

Our next main theorem provides a comparison of the clustering vs.\ coexistence 
dichotomy with the one for the hierarchical Fleming-Viot process. Let 
\be{a1}
\udster=(d^*_k)_{k\in\N_0}
\ee
be the sequence of volatility constants when $\mu_0>0$ and $\mu_k=0$ for all 
$k\in\N$ ($\mu_k=\tfrac12\lambda_k$, see \eqref{ag46}), i.e., there is resampling 
in single colonies but not in macro-colonies. By \eqref{diffusion-constants}, this 
sequence has initial value $d^*_0=0$ and satisfies the recursion relation
\be{dbarkrec}
d^*_1=d_1=\frac{c_0\mu_0}{c_0+\mu_0}, 
\qquad \frac{1}{d^*_{k+1}} = \frac{1}{c_k} + \frac{1}{d^*_k}, \quad k\in\N,
\ee
whose solution is 
\be{dbarksol}
d^*_k = \frac{\mu_0}{1+\mu_0\sigma_k}, \quad k\in\N,
\quad \mbox{ with } \sigma_k = \sum_{l=0}^{k-1} \frac{1}{c_l}.
\ee

\begin{theorem}
\label{T.beh}
{\bf [Comparison with hierarchical Fleming-Viot]}\\
The following hold for $(d_k)_{k \in \N_0}$ as in \eqref{diffusion-constants} (also recall \eqref{ag46}):
\begin{itemize}
\item[\textup{(a)}]
The maps $\uc\mapsto\ud$ and $\umuu\mapsto\ud$ are component-wise non-decreasing.
\item[\textup{(b)}]
$d_k \geq d^*_k$ for all $k \in \N$.
\item[\textup{(c)}]
$\sum_{k\in\N_0} m_k=\infty$ if and only if $\sum_{k\in\N_0} (1/c_k) \sum_{l=0}^k \mu_l
= \infty$.
\item[\textup{(d)}]
If $\lim_{k\to\infty} \sigma_k=\infty$ and $\sum_{k\in\N} \sigma_k\mu_k<\infty$, then
$\lim_{k\to\infty} \sigma_kd_k=1$.
\end{itemize}
\end{theorem}

\noindent
The proof of Theorem~\ref{T.beh} is given in Section~\ref{ss.gencomp}.

In words, (a) and (b) say that both migration and reshuffling-resampling increase volatility 
(recall (\eqref{ag46}--\ref{ag46b})), (c) says that the dichotomy in (\ref{dichclass}) due to 
migration is affected by reshuffling-resampling only when the latter is strong enough, i.e., 
when $\sum_{k\in\N_0} \mu_k = \infty$, while (d) says that the scaling behaviour of $d_k$ 
in the clustering regime is unaffected by the reshuffling-resampling when the latter is weak 
enough, i.e., when $\sum_{k\in\N} \sigma_k\mu_k<\infty$. Note that the criterion in (c) shows 
say that migration tends to inhibit clustering while reshuffling-resampling tends to enhance clustering.

We will see in the last paragraph of Section~\ref{ss.gencomp} that in the \emph{local coexistence regime} 
$d_k \sim \sum_{l=0}^k \mu_l$ as $k\to\infty$ when this sum diverges and $d_k \to 
\sum_{l\in\N_0} \mu_l/\prod_{j=l}^\infty (1+m_j) \in (0,\infty)$ when it converges.
Thus, in the local coexistence regime the scaling of $d_k$ is determined the 
resampling-reshuffling. 

In the regime, where the system \emph{clusters}, i.e., $\sum_{k\in\N_0} m_k = \infty$, 
it is important to be able to say more about the behaviour of $m_k$ as $k \to \infty$ 
in order to understand the patterns of cluster formation. For this the key is the 
behaviour of $d_k$ as $k \to \infty$, which we study in 
Sections~\ref{sss.polregvar}--\ref{sss.expregvar} for polynomial, respectively, 
exponential growth of the coefficients $c_k$ and $\lambda_k$.

\medskip\noindent
{\bf Heuristics.}
The recursion relation in \eqref{diffusion-constants} has the shape $d_{k+1}=f_k(d_k)$
with $f_k\colon\,\R\to\R$ a M\"obius-transformation (see Section~\ref{ss.Mobtr}). Thus, to 
obtain the asymptotics of $d_k$ as $k\to\infty$ we must study \emph{inhomogeneous 
iterates of M\"obius-transformations}. For each $k\in\N$, $f_k$ is hyperbolic with two fixed 
points: a repulsive fixed point $x_k^-<0$ and an attractive fixed point $x_k^+>0$. Depending
on the scaling of the coefficients $c_k$ and $\lambda_k$, the scaling of $x_k^+$ exhibits
four regimes. For three of the regimes, it turns out that $d_k \sim x_k^+$ as $k \to\infty$,
i.e., the iterates of the M\"obius-transformations attract towards the fixed point of the
last one. The fourth regime is different. In Section~\ref{sss.polregvar} we deal with 
polynomial coefficients, in Section~\ref{sss.expregvar} with exponential coefficients. In
order to obtains sharp results, the coefficients $c_k$ and $\lambda_k$ must satisfy
certain \emph{regularity conditions}.

\subsubsection{Scaling in the clustering regime: polynomial coefficients}
\label{sss.polregvar}

The following main theorem identifies the scaling behaviour of $d_k$ as $k\to\infty$ in
four different regimes, defined by the relative size of the migration coefficient 
$c_k$ versus the block resampling coefficient $\lambda_k$. The necessary 
\emph{regularity conditions} are stated in (\ref{regvarass}--\ref{regvarex}) 
below.

Define
\be{Klimex}
\lim_{k\to\infty} \frac{\mu_k}{c_k} = K \in [0,\infty] \text{ and, if }
K=0, \text{ also } \lim_{k\to\infty} k^2\frac{\mu_k}{c_k} = L \in [0,\infty].
\ee

\begin{theorem}
\label{dcases}
{\bf [Scaling of the volatility in the clustering regime: polynomial coefficients]}\\
Assume that the regularity conditions {\rm (\ref{regvarass}--\ref{regvarex})} hold.
\begin{itemize}
\item[\textup{(a)}] 
If $K=\infty$, then
\be{ag49}
\lim_{k\to\infty} \frac{d_k}{c_k} = 1.
\ee
\item[\textup{(b)}] 
If $K \in (0,\infty)$, then
\be{ag48}
\lim_{k\to\infty} \frac{d_k}{c_k} =  M \text{ with } 
M = \tfrac12 K\left[-1+\sqrt{1+(4/K)}\right] \in (0,1).
\ee
\item[\textup{(c)}] 
If $K=0$ and $L=\infty$, then
\be{ag53}
\lim_{k\to\infty} \frac{d_k}{\sqrt{c_k\mu_k}} = 1. 
\ee
\item[\textup{(d)}] 
If $K=0$, $L<\infty$ and $ a \in (-\infty,1)$, then
\be{ag47}
\lim_{k\to\infty} \sigma_k d_k = M^* \mbox{ with }
M^* = \tfrac12\left[1+\sqrt{1+4L/(1-a)^2}\right]
\in [1,\infty).
\ee
\end{itemize}
\end{theorem}

\noindent
The proof of Theorem~\ref{dcases} in given Section~\ref{ss.dcases}. The meaning of the four 
regimes for the evolution of the population will be explained in Corollary~\ref{C.scalemk}.
Case (a) can be termed ``reshuffling-resampling dominated'', cases (c) and (d) ``migration
dominated'', and case (b) ``balanced''. 

\myparagraph{Regularity conditions.} In Theorem~\ref{dcases}, we need to impose some mild 
regularity conditions on $\uc$ and $\umuu$, which we collect in 
(\ref{regvarass}--\ref{regvarex}) below. We require that both $c_k$ and $\mu_k$ 
are \emph{regularly varying at infinity}, i.e., there exist $a, b \in \R$ such that
\be{regvarass}
c_k \sim L_c(k)k^a, \qquad \mu_k \sim L_\mu(k)k^b,
\qquad k\to\infty, 
\ee
with $L_c,L_\mu$ slowly varying at infinity (Bingham, Goldie and 
Teugels~\cite[Section 1.9]{BGT87}). The numbers $a,b$ are referred to as the
\emph{indices} of $\uc$ and $\umuu$\footnote{Regular variation is typically defined 
with respect to a continuous instead of a discrete variable. However, every regularly 
varying sequence can be embedded into a regularly varying function.}. Note that
\eqref{dichreg} is satisfied.

To handle the boundary cases, where $c_k$, $\mu_k$, $\mu_k/c_k$ and/or $k^2\mu_k/c_k$ 
are slowly varying, we additionally require that for specific choices of the indices 
the following functions are \emph{asymptotically monotone}:   
\be{regvarassmon}
\begin{array}{lll}
&a=0: &k \mapsto \Delta L_c(k)/L_c(k),
\,k \mapsto k\Delta L_c(k)/L_c(k),\\[0.2cm]
&b=0: &k \mapsto \Delta L_\mu(k)/L_\mu(k),
\,k \mapsto k\Delta L_\mu(k)/L_\mu(k),
\end{array}
\ee 
and the following functions are \emph{bounded}:
\be{regvarassbd}
\begin{array}{lll}
&a=0: &k \mapsto k\Delta L_c(k)/L_c(k),\\[0.2cm]
&b=0: &k \mapsto k\Delta L_\mu(k)/L_\mu(k),
\end{array}
\ee
where $\Delta L(k) = L(k+1)-L(k)$. To ensure the \emph{existence} of the limits in 
(\ref{Klimex}), we also need the following functions to be \emph{asymptotically 
monotone}:
\be{regvarex}
\begin{array}{lll}
&a=b: &k \mapsto L_\mu(k)/L_c(k),\\[0.2cm]
&a=b-2: &k \mapsto k^2L_\mu(k)/L_c(k).
\end{array}
\end{equation}

\myparagraph{Scaling of the variance.}
The next corollary shows what the scaling of $d_k$ in Theorem~\ref{dcases} implies 
for the scaling of $m_k$ and hence of the variance in \eqref{ag50} (we will see in 
Section~\ref{ss.dcases} that the conditions for Case (d) imply that $\lim_{k\to\infty}
\mu_k\sigma_k=0$ and $\lim_{k\to\infty} c_k\sigma_k=\infty$). 

\begin{corollary}
\label{C.scalemk}
{\bf [Scaling behaviour of $m_k$]}\\
The following asymptotics of $m_k$ for $k \to \infty$ holds in the four cases of
\textup{Theorem}~{\rm \ref{dcases}}:
\begin{equation}
\label{dkmk}
\begin{aligned}
&\textup{(a)} \quad m_k \sim \frac{\mu_k}{c_k} \to \infty,
\quad
\textup{(b)} \quad m_k \to K+M,\\
\quad
&\textup{(c)} \quad m_k \sim \sqrt{\frac{\mu_k}{c_k}} \to 0,
\quad
\textup{(d)} \quad  m_k \sim \frac{M^*}{c_k\sigma_k} \to 0.
\end{aligned}
\end{equation}
\end{corollary}

\medskip\noindent 
All four cases fall in the clustering regime. For the variance in \eqref{ag50}
they imply: (a) superexponential decay; (b) exponential decay, (c--d) 
subexponential decay.  

Note that Case (d) also falls in the {\em clustering regime} because it assumes that
$a \in (-\infty,1)$, which implies that $\lim_{k\to\infty} \sigma_k=\infty$. Indeed, 
$1/c_k\sigma_k=(\sigma_{k+1}-\sigma_k)/\sigma_k$, and in Section~\ref{ss.gencomp} we
will see that
\be{iffimplearly} 
\lim_{k\to\infty} \sigma_k = \infty 
\quad \Longleftrightarrow \quad \sum_{k\in\N} \frac{1}{c_k\sigma_k}=\infty.
\ee
Combining Cases (a--d), we conclude the following:
\begin{itemize}
\item
The regime of \emph{weak block resampling} (for which the scaling behaviour of $d_k$ is 
the same as if there were no block resampling) coincides with the choice $K=0$ and 
$L<\infty$.
\item 
The regime of \emph{strong block resampling} (for which the scaling behaviour of $d_k$ 
is different) coincides with $K=0$ and $L=\infty$ or $K>0$. 
\end{itemize}

Note that $M \uparrow 1$ as $K\to\infty$, so that Case (b) connects up with Case (a). 
Further note that $M \sim \sqrt{K}$ as $K \downarrow 0$, so that Case (b) also connects 
up with Case (c). Finally, note that $\sqrt{c_k\mu_k} \sim \sqrt{L}c_k/k$ as $k\to\infty$
for Case (d) by \eqref{Klimex}, while $c_k\sigma_k \sim k/(1-a)$ as $k\to\infty$ when 
$a\in (-\infty,1)$ by \eqref{regvarass}. Hence, Case (d) connects up with Case (c) as 
well.

\subsubsection{Scaling in the clustering regime: exponential coefficients}
\label{sss.expregvar}

We briefly indicate how Theorem~\ref{dcases} extends when $c_k$ and $\mu_k$ satisfy
\be{expregvar}
\begin{aligned}
&c_k = c^k\bar{c}_k,\,\mu_k= \mu^k\bar{\mu}_k \mbox{ with } c,\mu \in (0,\infty)
\mbox{ and } (\bar{c}_k),\,(\bar{\mu}_k) \mbox{ regularly varying at infinity},\\
&\bar{K} = \lim_{k \to \infty} \frac{\bar{\mu}_k}{\bar{c}_k} \in [0, \infty],
\end{aligned}
\ee 
and the analogues of (\ref{regvarassmon}--\ref{regvarex}) apply to the 
regularly varying parts. Again, note that \eqref{dichreg} is satisfied.

\begin{theorem}
\label{dcases-exp}
{\bf [Scaling of the volatility in the clustering regime: exponential coefficients]}\\
Assume that \eqref{expregvar} holds. Recall the cases \textup{(a--d)} from 
\textup{Theorem}~\ref{dcases}. Then:
\begin{itemize}

\item[\textup{(A)}]
{\bf [scaling like Case (a)]}
$c<\mu$ or $c=\mu$, $\bar{K}=\infty$: $\lim_{k\to\infty} d_k/c_k=1/c$.

\item[\textup{(B)}]
{\bf [scaling like Case (b)]}
$c=\mu$, $\bar{K} \in (0,\infty)$: $\lim_{k\to\infty} d_k/c_k = \bar{M}$ with
\be{Mexp}
\bar{M} = \frac{1}{2c}\left[-(c(\bar{K}+1)-1) + \sqrt{(c(\bar{K}+1)-1)^2+4c\bar{K}}\right] 
.
\ee

\item[\textup{(C)}]
The remainder $c>\mu$ or $c=\mu$, $\bar{K}=0$ splits into three cases:
\begin{itemize}

\item[\textup{(C1)}]
{\bf [scaling like Case (d)]}
$1>c>\mu$ or $1=c>\mu$, $\lim_{k\to\infty} \sigma_k = \infty$: $\lim_{k\to\infty} 
\sigma_k d_k = 1$. 

\item[\textup{(C2)}]
{\bf [scaling like Case (b)]}
$c=\mu<1$, $\bar{K}=0$: $\lim_{k\to\infty} d_k/c_k=(1-c)/c$.

\item[\textup{(C3)}]
{\bf [scaling like Case \textup{(c)}]}
$c=\mu>1$, $\bar{K}=0$: $\lim_{k\to\infty} d_k/\mu_k = 1/(\mu-1)$.
\end{itemize}
\end{itemize}
\end{theorem}

\begin{remark}
The analogue of $L$ (cf., \eqref{Klimex} and \textup{Theorem}~\ref{dcases}) no longer plays
a role for exponential coefficients (cf., \textup{Theorem}~\ref{dcases-exp}).
\end{remark}

\noindent
The proof of Theorem~\ref{dcases-exp} is given in Section~\ref{ss.exp}. The choices 
$1=c>\mu$, $\lim_{k\to\infty} \sigma_k < \infty$ and $c>1$, $c>\mu$ correspond to 
local coexistence (and so does $c=\mu>1$, $\bar{K}=0$, $\sum_{k\in\N_0} 
\bar{\mu}_k/\bar{c}_k<\infty$).

\subsubsection{Cluster formation} 
\label{sec:cluster-foramtion}

In the clustering regime, it is of interest to study the size of the mono-type
regions  as a function of time, i.e., how fast do the clusters grow? To that end, we
look at the  interaction chain $M^{(j)}_{-k(j)}$ for $j\to\infty$, where the
\textit{level scaling} function  $k\colon\,\N\to\N$ with $\lim_{j\to\infty} k(j)=\infty$ 
is suitably chosen such that we obtain a \textit{nontrivial clustering limiting law}, i.e.,
\be{ak:interaction-chain-scaling-limit}
\lim_{j \to \infty} 
\mathcal{L} \left[ M^{(j)}_{-k(j)} \right]
= 
\mathcal{L} \left[ \hat{\theta} \right]
,
\ee
where the limiting random measure $\hat{\theta}$ satisfies
\be{ak:clustering-interaction-chain}
\P \{ \hat{\theta} = \delta_U, \text{ for some } U \in E \} < 1
.
\ee
For example, in Dawson and Greven~\cite{DG93b} such a question was 
answered in the case of the interacting Fleming-Viot processes with critically
recurrent migration $\uc$. There, different types of limit laws and different types 
of scaling can occur, corresponding to different clustering regimes. Following Dawson,
Greven and  Vaillancourt~\cite{DGV95} and Dawson and Greven~\cite{DG96}, it is
natural to consider  a whole family of scalings $k_\alpha\colon\,\N \to \N$,
$\alpha \in [0,1]$ satisfying \eqref{ak:interaction-chain-scaling-limit}. We
single  out \emph{fast}, \emph{diffusive} and \emph{slow} clustering regimes:
\begin{itemize}
\item[\textup{(i)}] {\bf Fast clustering:} 
$\lim_{j\to\infty} k_\alpha(j)/j=1$ for all $\alpha$.
\item[\textup{(ii)}] {\bf Diffusive clustering:} 
In this regime,
$\lim_{j\to\infty} k_\alpha(j)/j=\kappa(\alpha)$ for all $\alpha$, where $\alpha\mapsto
\kappa(\alpha)$ is continuous and non-increasing with $\kappa(0)=1$ and 
$\kappa(1)=0$. 
\item[\textup{(iii)}] {\bf Slow clustering:} 
$\lim_{j\to\infty} k_\alpha(j)/j=0$ for all $\alpha$. This regime borders with the 
regime of local coexistence.
\end{itemize}

\medskip\noindent
{\bf Remark:}
Diffusive clustering similar to (ii) was previously found for the voter model on $\Z^2$ 
by Cox and Griffeath~\cite{CG86}, where the radii of the clusters of opinion ``all 1'' 
or ``all 0'' scale as $t^{\alpha/2}$ with $\alpha \in [0,1)$, i.e., clusters occur on 
all scales $\alpha \in [0,1)$. This is different from what happens on $\Z^d$, $d \geq 3$, 
where clusters occur only on scale $\alpha =1$. For the model of hierarchically interacting 
Fleming-Viot diffusions with $c_k \equiv 1$ (= critically recurrent migration), Fleischmann 
and Greven~\cite{FG94} showed that, for all $N \in \N \setminus \{1\}$ and all $\eta \in 
\Omega_N$, 
\be{allim}
\CL\left[
\left(Y^{(\Omega_N)}_{\eta,\lfloor(1-\alpha)t\rfloor}(N^t)\right)_{\alpha \in [0,1)}
\right] 
\tto
\CL\left[
\left(Y^{(\Omega_N)}\left(\log\left(\frac{1}{1-\alpha}\right)\right)
\right)_{\alpha \in [0,1)}
\right],
\ee 
where $(Y^{(\Omega_N)}(t))_{t \in (0,1]}$ is a Fleming-Viot diffusion on $\CP(E)$. 
A similar behaviour occurs for other models, e.g., for branching models 
(Dawson and Greven~\cite{DG96}). 

\medskip
Our last two main theorems show which type of clustering occurs for the various 
scaling regimes of the coefficients $\uc$ and $\umuu$ identified in 
Theorems~\ref{dcases}--\ref{dcases-exp}. Polynomial coefficients allow for fast 
and diffusive clustering only. Exponential coefficients allow for fast, diffusive 
and slow clustering, with the latter only in a narrow regime.

\begin{theorem}
\label{T.dicho2}
{\bf[Clustering regimes for polynomial coefficients]}\\
Recall the scaling regimes of Theorem~{\rm \ref{dcases}}.
\begin{itemize}
\item[\textup{(i)}] 
{\bf [Fast clustering]} In cases \textup{(a-c)}, the system exhibits fast clustering.
\item[\textup{(ii)}] {\bf [Diffusive clustering]} 
In case \textup{(d)}, the system exhibits diffusive clustering, i.e.,
\be{ak11}
\CL\left[\left(M^{(j)}_{-\lfloor (1-\alpha) j \rfloor}\right)_{\alpha \in [0,1)} \right]
\jto \CL\left[\left(Z^{0,1,0}_{\theta} 
\left(\log\left(\frac{1}{1-\alpha^R}\right)\right)\right)_{\alpha \in [0,1)}\right],
\ee
where $R=M^*(1-a)$ with $M^*$ defined in \eqref{ag47} and $a$ the exponent in
\eqref{regvarass}.
\end{itemize}
\end{theorem}

\begin{theorem}
\label{T.clustering-exp}
{\bf[Clustering regimes for exponential coefficients]}\\
Recall the scaling regimes of Theorem~{\rm \ref{dcases-exp}}.
\begin{itemize}
\item[\textup{(i)}] {\bf [Fast clustering]} 
In cases \textup{(A, B, C1, C2)}, and case \textup{(C3)} with 
\be{ak:201402081404}
\lim_{k\to\infty} 
k\bar{\mu}_k/\bar{c}_k=\infty, 
\ee
the system exhibits fast clustering.
\item[\textup{(ii)}] {\bf [Diffusive clustering]} 
In case \textup{(C3)} with $\lim_{k\to\infty} k\bar{\mu}_k/\bar{c}_k=C$, the system 
exhibits diffusive clustering, i.e., \eqref{ak11} holds with $R = C/(\mu-1)$.
\item[\textup{(iii)}] {\bf [Slow clustering]} 
In case \textup{(C3)} with $k\bar{\mu}_k/\bar{c}_k \asymp 1/(\log k)^\gamma$,
$\gamma \in (0,1)$, the system exhibits slow clustering. 
\end{itemize}
\end{theorem}

\noindent
The proofs of Theorems~\ref{T.dicho2}--\ref{T.clustering-exp} are given in 
Section~\ref{ss.ich}. Note that \eqref{allim} is a statement valid for {\em all} 
$N \in \N \setminus \{1\}$. In contrast, Theorems~\ref{T.dicho2}--\ref{T.clustering-exp} 
are valid in the hierarchical mean-field limit $N \to \infty$ only.

\subsection{Discussion}
\label{ss.outline}

\myparagraph{Summary.} We have constructed the $C^{\uc,\uL}_N$-process in Section \ref{sss.hierarCan}, describing
hierarchically interacting  Cannings processes, and have identified its
space-time scaling behaviour in the hierarchical  mean field limit $N\to\infty$
(interaction chain, cf.~Theorem~\ref{Tscb}). We have fully classified the clustering  vs.\ local
coexistence dichotomy in terms of the parameters $\uc,\uL$ of the model (cf.~Theorem~\ref{T.dicho}), and 
found different regimes of cluster formation (cf.~Theorems~\ref{T.dicho2}, \ref{T.clustering-exp}). Moreover, we have verified the
dichotomy also  for finite $N$ (cf.~Theorems~\ref{dich_coclu}--\ref{dichodual}). Our results provide a full generalisation of
what was known for hierarchically interacting diffusions, and show that Cannings
resampling leads to new phenomena (cf.~Theorem~\ref{T.beh} and comment following it). 

\myparagraph{Diverging volatility of the Fleming-Viot part and local coexistence.}
The growth of the block resampling rates $(\lambda_k)_{k \in \N}$ can lead to a
situation, where, as we pass to larger block averages, the volatility of the
Fleming-Viot part of the asymptotic limit dynamics diverges, even though on the
level of a single component the system exhibits local coexistence (recall Theorem~\ref{T.dicho}(c)). This requires
that the migration rates are (barely) transient  and the block resampling rate
decays very slowly. An example of such a situation is the choice
$c_k=k(\log k)^3$ and $\mu_k=1/k$ which leads to $d_k \sim \log k$ and  $m_k\sim
1/k(\log k)^2$ as $k\to\infty$. Thus, the system may be in the local coexistence
 regime and yet have a diverging volatility on large space-time scales.

\myparagraph{Open problems.} The results of Section~\ref{ss.mainth} and
suggest that a dichotomy  between clustering and local
coexistence also holds for a suitably defined Cannings  model with non-local
resampling on $\Z^d$, $d \geq 3$. In addition, a continuum limit to  the
geographic space $\R^2$ ought to arise as well, cf.~\cite{BEV10}. The latter
may be easier to investigate  in the limit $N\to\infty$, following the approach
outlined in Greven~\cite{G05}. Another  open problem concerns the different
ways in which cluster formation can occur. Here, the  limit $N\to\infty$ could
already give a good picture of what is to be expected for finite  $N$. A further
task is to investigate the genealogical structure of the model, based on  the
work in Greven, Klimovsky and Winter~\cite{GKWpr} for the model without
multi-colony  Cannings resampling (i.e., $\Lambda_k = \delta_0$ for $k \in \N$).

\myparagraph{Outline of the remainder of the paper.}
Section~\ref{s.spatcoal} introduces the spatial $\uL$-coalescent with block
coalescence and derives  some of its key properties.
Sections~\ref{s.wpmp}--\ref{s.prcor} use the results in Section~\ref{s.spatcoal}
to prove the propositions and the theorems stated in
Sections~\ref{ss.cannings}--\ref{ss.mainth}. Here is a \emph{roadmap}:
\begin{itemize}

\item 
Section~\ref{s.wpmp} handles all  issues related to the well-posedness of 
martingale problems. The proofs of Propositions~\ref{P.wp}--\ref{P.vecLambda} 
are in Section~\ref{ss.mppropproofs}.

\item
Section~\ref{s.mcv_imem}  deals with the properties of the McKean-Vlasov
process, including its equilibrium distribution. 

\item Section~\ref{s.stratth} outlines  the strategy behind the proofs of the
scaling results for the hierarchical Cannings  process, which are worked out in
Sections~\ref{s.mflambda}--\ref{s.msana} as follows: Theorem~\ref{mainth} is
proved in Section~\ref{s.proof1stth} with preparatory work being done in
Sections~\ref{s.mflambda}--\ref{s.hiermod}, Theorem~\ref{Tscb} is proved in
Section~\ref{ss.intchain}, Theorem~\ref{T.dicho} in
Section~\ref{ss.dichointeract}, and
Theorems~\ref{T.dicho2}--\ref{T.clustering-exp} in Section~\ref{ss.ich}.

\item 
Section~\ref{ss.cluscoex} proves the scaling results for the interaction
chain stated in Theorems~\ref{dich_coclu} and \ref{P.ergodic}.

\item 
Section~\ref{s.prcor} derives the scaling results for the volatility constant: 
Theorems~\ref{dichodual} and \ref{T.beh} are proved in Section~\ref{ss.gencomp}, M\"obius-transformations are introduced in Section~\ref{ss.Mobtr}, Theorem~\ref{dcases} 
is proved in Section~\ref{ss.dcases}, and Theorem~\ref{dcases-exp} in Section~\ref{ss.exp}. 

\item 
Section~\ref{s.notation} collects the notation. 
\end{itemize}

\section{Spatial $\Lambda$-coalescent with non-local coalescence}
\label{s.spatcoal}

In this section, we introduce a new class of spatial $\uL$-coalescent processes,
namely, processes where coalescence of partition elements at distances larger 
than or equal to zero can occur. This is a generalisation of the spatial coalescent 
introduced by Limic and Sturm~\cite{LS06}, which allows for the coalescence 
of the partition elements (= families = lineages) residing at the same location only. Informally, the spatial
$\uL$-coalescent with non-local coalescence is the process that encodes the
\emph{family structure} of a sample from the currently alive population in the
$C^{\uc,\uL}_N$-process, i.e., it is the process of \emph{coalescing lineages} 
that occur when the evolution of the spatial $C^{\uc,\uL}_N$-Cannings process is 
traced backwards in time up to a common ancestor. In what follows, we denote this 
backwards-in-time process by $\mathfrak{C}^{\uc, \uL}_{N}$.

Recall that two Markov processes $X$ and $Y$ with Polish state spaces $\CE$ and $\CE^\prime$ 
are called {\em dual w.r.t.\ the duality function} $H\colon\,\CE \times \CE^\prime 
\to \R$ if
\be{ag13}
\E_{X_0} [H(X_t, Y_0)] = \E_{Y_0} [H(X_0, Y_t)],
\qquad \text{for all}\,(X_0, Y_0) \in \CE \times \CE^\prime,
\ee
and if the family $\{H(\cdot,Y_0) \colon\, Y_0 \in \CE^\prime\}$ uniquely determines 
a law on $\CE$. Typically, the key point of a duality relation is to translate 
questions about a complicated process into questions about a simpler process. 
This translation often allows for an analysis of the long-time behaviour of the 
process, as well as a proof of existence and uniqueness for associated martingale 
problems. If $H(\cdot, \cdot) \in C_\mathrm{b}(\CE \times \CE^\prime)$, and if $H(\cdot,Y_0)$ 
and $H(X_0,\cdot)$ are in the domain of the generator of $X$, respectively, $Y$ for 
all $(X_0, Y_0) \in \CE \times \CE^\prime$, then it is possible to establish 
duality by just checking a {\em generator relation} (see Remark~\ref{R2.9} below
and also Liggett~\cite[Section II.3]{L85}).

The analysis of the processes on their relevant time scales will lead us to
study  a number of auxiliary processes on geographic spaces different from
$\Omega_N$.  The duality will be crucial for the proof of
Propositions~\ref{P.wp}--\ref{P.vecLambda}  (martingale well-posedness) in
Section~\ref{s.wpmp}, and also for statements about  the long-time behaviour of
the processes and the qualitative properties of their  equilibria. In
Section~\ref{ss.spatcoal}, we define the spatial $\Lambda$-coalescent  with
local coalescence. In Section~\ref{ss.blockcoal}, we add non-local coalescence. 
In Section~\ref{ss.dual}, we formulate and prove the duality relation between
the  $C^{\uc,\uL}_N$-process and the spatial $\Lambda$-coalescent with non-local
coalescence. In Section~\ref{ss.ltbvecL},  we look at the long-time behaviour of
the spatial $\Lambda$-coalescent with non-local coalescence.

\subsection{Spatial $\Lambda$-coalescent with local coalescence}
\label{ss.spatcoal}

In this section, we briefly recall the definition of the spatial
$\Lambda$-coalescent on a countable geographic space $G$ as introduced by Limic
and Sturm~\cite{LS06}. (For a general discussion of exchangeable coalescents,
see Berestycki~\cite{B09}.) Here, we do not need assumption \eqref{ag1} on
measure $\Lambda$. In Section~\ref{ss.blockcoal}, we will add {\em non-local
coalescence}, i.e., coalescence of individuals not necessarily located at the
same site.

The following choices of the geographic space $G$ will be needed later on:
\be{Gchoices}
G_{N,K}=\{0,\ldots,N-1\}^K,\,K,N\in\N, \qquad G=\Omega_N,\,N\in\N,
\qquad G = \{0,*\}.
\ee
The choices in \eqref{Gchoices} correspond to geographic spaces that are needed,
respectively, for finite approximations of the hierarchical group, for the 
hierarchical group, for a single-colony with immigration-emigration, and for
the McKean-Vlasov limit. We define the basic transition mechanisms and characterise 
the process by a martingale problem in order to be able to verify duality and to 
prove convergence properties. In Section~\ref{sss.statevol} we define the state 
space and the evolution rules, in Section~\ref{sss.mart} we formulate the martingale 
problem, while in Section~\ref{sss.special-spatial-coalescents} we introduce 
coalescents with immigration-emigration.

\subsubsection{State space, evolution rules, graphical construction and entrance law}
\label{sss.statevol}

\myparagraph{State space.} 
As with non-spatial exchangeable coalescents, it is convenient to start with finite 
state spaces and subsequently extend to infinite state spaces via exchangeability. 
Given $n \in \N$, consider the set
\be{ag19a}
[n] = \{1,\ldots,n\}
\ee
and the set $\Pi_n$ of its partitions into {\it families}:
\be{ag19}
\Pi_n = \text{ set of all partitions  }  \pi = \{ \pi_i \subset [n] \}_{i=1}^b
\text{ of set $[n]$ into disjoint families $\pi_i$, $i\in[b]$.}
\ee
That is, for any $\pi = \{ \pi_i \}_{i=1}^b \in \Pi_n$, we have $[n] = \bigcup_{i=1}^b 
\pi_i $ and $\pi_i \cap \pi_j = \emptyset$ for $i,j \in [b]$ with $i \neq j$. In 
what follows, we denote by 
\be{add30}  
b = b(\pi) \in [n]
\ee
the number of families in $\pi \in \Pi_n$.

\begin{remark}[Notation]
\label{rem:coalescent-block-mapping}
By a slight abuse of notation, we can associate with $\pi \in \Pi_n$ the mapping 
$\pi\colon\,[n] \to [b]$ defined as $\pi(i)=k$, where $k \in [b]$ is such that 
$i \in \pi_k$. In words, $k$ is the label of the unique family containing $i$.
\end{remark}

\noindent 
Abbreviate
\be{ak-999}
\pi^{-1}(k) = \min \{ i \in [n] \colon  \pi(i) = k \}, \quad k \in [b].
\ee

The state space of the spatial coalescent is the set of \emph{$G$-labelled partitions} 
defined as
\be{labelled-partitions}
\Pi_{G,n} =
\Big\{
\pi_{G} = \{ (\pi_1, g_1), (\pi_2, g_2), \ldots, (\pi_b, g_b) \}
\colon\,\{ \pi_1, \ldots, \pi_b \} \in \Pi_n,\, g_1, \ldots, g_b \in G, b \in [n] \Big\}.
\ee
For definiteness, we assume that the families of $\pi_{G} \in \Pi_{G,n}$ are indexed 
in the increasing order of each family's smallest element, i.e., the enumeration is 
such that $\min \pi_i < \min \pi_j$ for all $i,j\in [b]$ with $i \neq j$.

Let $S_{G,n} \in \Pi_{G,n}$ denote the labelled partition of $[n]$ into
\emph{singletons}, i.e.,
\be{labelled-singletons}
S_{G,n} =
\Big\{(\{1\}, g_1), (\{2\}, g_2), \ldots, (\{n\}, g_n)\colon\,
g_i \in G, i \in [n]\Big\}.
\ee
With each $\pi_G \in \Pi_{G,n}$ we can naturally associate the partition $\pi
\in \Pi_n$ by \emph{removing the labels}, i.e., with
\be{20a}
\pi_G = \{ (\pi_1, g_1), (\pi_2, g_2), \ldots, (\pi_b, g_b) \}
\ee
we associate $\pi = \{ \pi_1, \ldots, \pi_b \} \in \Pi_n$. With each $\pi_G \in 
\Pi_{G,n}$ we also associate the set of its
\emph{labels}
\be{20b}
L(\pi_G)= \{g_1,\ldots,g_b\} \subset G.
\ee

In addition to the finite-$n$ sets $\Pi_n$ and $\Pi_{G,n}$ considered above,
consider their infinite versions 
\be{ag20}
\Pi = \{\text{partitions of } \N\},
\quad
\Pi_G = \{G\text{-labelled partitions of } \N\},
\ee
and introduce the set of standard initial states
\begin{equation}
\label{eq:initial-state}
S_G = \big\{ \{(\{i\}, g_i)\}_{i \in \N} \colon\, g_i \in G, i \in \N \big\}.
\end{equation}
Equip $\Pi_G$ with the following topology. First, equip the set $\Pi_{G,n}$ with 
the {\em discrete topology}. In particular, this implies that $\Pi_{G,n}$ is a
{\em Polish space}. We say that the sequence of labelled partitions 
$\{ \pi_G^{(k)} \in\Pi_G \}_{k \in \N}$ converges to the labelled partition 
$\pi_G \in \Pi_G$ if the sequence $\{ \pi_G^{(k)}\vert_n \in \Pi_{G,n} \}_{k\in\N}$ 
converges to $\pi_G\vert_n \in \Pi_{G,n}$ for all $n \in \N$. This topology makes 
the space $\Pi_G$ Polish, too.

\myparagraph{Evolution rules.} 
Assume that we are given transition rates (= ``migration rates'') on $G$
\be{ag21}
a^*\colon\, G^2 \to \R,
\quad
a^\ast (g,f) = a(f,g),
\ee
where $a(\cdot,\cdot)$ is the migration kernel of the corresponding $C^\Lambda$-process with 
geographic space $G$ as in \eqref{Gchoices}. The spatial \emph{$n$-$\Lambda$-coalescent} is the continuous-time 
Markov process $\mathfrak{C}^{(G),\mathrm{loc}}_n = (\mathfrak{C}^{(G),\mathrm{loc}}_n
(t) = \pi_G(t) \in \Pi_{G,n})_{t \geq 0}$ with the following dynamics. Given the 
current state $\pi_G = \mathfrak{C}^{(G),\mathrm{loc}}_n(t-) \in \Pi_{G,n}$, the process 
$\mathfrak{C}^{(G),\mathrm{loc}}_n$ evolves via:
\begin{itemize}
\item
\emph{Coalescence}.
Independently, at each site $g \in G$, the families of $\pi_G$ with label $g$
\emph{coalesce} according to the mechanism of the non-spatial $n$-$\Lambda$-coalescent. 
In other words, given that in the current state of the spatial $\Lambda$-coalescent
there are $b = b(\pi_G,g) \in [n]$ families with label $g$, among these $i \in[2,b] 
\cap \N$ \emph{fixed} families coalesce into one family with label $g$ at rate 
$\lambda^{(\Lambda)}_{b,i}$, where
\be{coalescence-rates}
\lambda^{(\Lambda)}_{b,i} = \int_{[0,1]} \Lambda^*(\dd r) r^{i}(1-r)^{b-i},
\quad i \in [2,b] \cap \N,
\ee
with $\Lambda^*$ given by \eqref{ag3}.
\item
\emph{Migration}.
Families migrate independently at rate $a^\ast$, i.e., for any ordered pair of 
labels $(g,g^\prime) \in G^2$, a family of $\pi_G$ with label $g \in G$ 
changes its label (= ``migrates'') to $g^\prime \in G$ at rate $a^*(g,g^\prime)$.
\end{itemize}

\myparagraph{Graphical construction.}
Next, we recall the explicit construction of the above described spatial 
$n$-$\Lambda$-coalescent via Poisson point processes (see also Limic and 
Sturm~\cite{LS06}).

Consider the family $\mathfrak{P} = \{\mathfrak{P}_g \}_{g \in G} $ of i.i.d.\
Poisson point processes on $[0,\infty) \times [0, 1] \times \{0, 1\}^\N$ defined 
on the filtered probability space $(\Omega,\CF,(\CF_t)_{t\geq 0},\P)$ with intensity 
measure
\be{ppp-coalescence}
\dd t \otimes \left[\Lambda^\ast(\dd r)
(r \delta_1 + (1-r) \delta_0)^{\otimes \N}\right](\dd \omega),
\ee
where $\omega = (\omega_i)_{i \in \N} \subset \{0,1\}^\N$. We assume that 
point processes $\mathfrak{P}$ are adapted to filtration $(\CF_t)_{t\geq 0}$. 

Note that the second
factor of the intensity measure in \eqref{ppp-coalescence} is not a product measure 
on $[0,1] \times \{0,1\}^\N$, in particular, it is not the same as 
\be{a4}
\big[\Lambda^\ast (\dd r)(r \delta_1 +(1-r)\delta_0)\big]^{\otimes \N}(\dd\omega).
\ee

Given $J \subset [n]$ and $g \in G$, define the \emph{labelled coalescence
map} $\textrm{coal}_{J,g}\colon\,\Pi_{G,n} \to \Pi_{G,n}$, which coalesces 
the blocks with indices specified by $J$ and locates the new-formed block at $g$,
as follows:
\be{labelled-coalescence}
\textrm{coal}_{J,g}(\pi_{G,n}) =
\left(\bigcup_{i \in J \cap [b(\pi)]}\pi_i, g\right)
\cup
\left(\pi_{G,n}\,\setminus\,\bigcup_{i \in J \cap [b(\pi)]}(\pi_i, g_i)\right),
\quad \pi_{G,n} \in \Pi_{G,n}.
\ee
Using $\mathfrak{P}$, we construct the standard spatial $n$-$\Lambda$-coalescent 
$\mathfrak{C}^{(G),\mathrm{loc}}_n = (\mathfrak{C}^{(G),\mathrm{loc}}_n(t))_{t \geq 0}$ 
as a Markov $\Pi_{G,n}$-valued process with the following properties:
\begin{itemize}
\item
\emph{Initial state.} Assume $\mathfrak{C}^{(G),\mathrm{loc}}_n(0) \in S_{G,n}$.
\item
\emph{Coalescence.} For each $g \in G$ and each point $(t, r, \omega)$ of the 
Poisson point process $\mathfrak{P}_g$ satisfying $\sum_{i \in \N} \omega_i \geq 2$, 
all families $(\pi_i(t-),g_i(t-)) \in \mathfrak{C}^{(G),\mathrm{loc}}_n(t-)$ such 
that $g_i(t-)= g$ and $\omega_i = 1$ coalesce into a new family labelled by $g$, i.e.,
\be{no1}
\mathfrak{C}^{(G),\mathrm{loc}}_n(t)=\textrm{coal}_{\{i \in [n]\colon\,
\omega_i = 1, g_i(t-) = g\},g}(\mathfrak{C}^{(G),\mathrm{loc}}_n(t-)).
\ee
\item
\emph{Migration.} Between the coalescence events, the labels of all partition elements 
of $\mathfrak{C}^{(G),\mathrm{loc}}_n(t)$ perform independent random walks with 
transition rates $a^*$\footnote{The adjective ``between'' is well defined because the 
set of points $(t,r,\omega)$ of $\mathfrak{P}_g$ satisfying the condition $\sum_{i \in \N} 
\omega_i \geq 2$ is topologically discrete, and hence can be ordered w.r.t.\ the first 
coordinate (= time).}.
\end{itemize}

In what follows, we denote by $\cdot\vert_n\colon\, \Pi_{G,m} \to \Pi_{G,n}$, for $m \geq n$,
(respectively, $\cdot\vert_n\colon\,\Pi_{G} \to \Pi_{G,n}$) the operation of 
projection of all families in $[m]$ (respectively, $\N$) onto $[n]$.

\myparagraph{Entrance law.}
Note that, by construction, the spatial $n$-$\Lambda$-coalescent satisfies the 
following consistency property:
\be{no2}
\CL\left[\mathfrak{C}^{(G),\mathrm{loc}}_m\vert_n\right]
= \CL\left[\mathfrak{C}^{(G),\mathrm{loc}}_n\right],
\qquad n,m \in \N,\,n \leq m.
\ee
Therefore, by the Kolmogorov extension theorem, there exists a process
\be{add1}
\mathfrak{C}^{(G),\mathrm{loc}} = (\mathfrak{C}^{(G),\mathrm{loc}}(t) \in \Pi_G)_{t \geq 0} 
\ee
such that $\mathfrak{C}^{(G),\mathrm{loc}}\vert_n = \mathfrak{C}^{(G),\mathrm{loc}}_n$.

\begin{definition}[\cite{LS06}]
Call the process $\mathfrak{C}^{(G),\mathrm{loc}}$ the spatial $\Lambda$-coalescent 
corresponding to the migration rates $a^*$ and the coalescence measure $\Lambda$.
\end{definition}

\subsubsection{Martingale problem}
\label{sss.mart}

In this section, we characterise the spatial $\Lambda$-coalescent as the unique
solution of the corresponding well-posed martingale problem.

Let $\CC_G$ be the algebra of bounded continuous functions $F\colon\,\Pi_G \to \R$
such that for all $F \in \CC_G$ there exists an $n \in \N$ and a bounded function
\be{ghkk1}
F_n\colon\, \Pi_{G,n} \to \R
\ee 
with the property that $F(\cdot)=F_n(\cdot\vert_n)$. In words, $F$ only depends 
on the family structure of a finite number of individuals. It is easy to check 
that $\CC_G$ separates points on $\Pi_G$. Given $f,g \in G$ and $i \in [n]$, define 
the \emph{migration map} $\textrm{mig}_{f \to g, i}\colon\,\Pi_{G,n} \to \Pi_{G,n}$ as
\be{migration-operation}
\textrm{mig}_{f \to g, i}(\pi_{G,n}) =
\begin{cases}
\left(\pi_i, g\right) \cup \left(\pi_{G,n}\setminus(\pi_i, f)\right),
&(\pi_i, f) \in \pi_{G,n},\\
\pi_{G,n},
&(\pi_i, f) \notin \pi_{G,n},
\end{cases}
\quad
\pi_{G,n} \in \Pi_{G,n},
\ee
describing the jump in which the family labelled $i$ migrates from colony $f$ to 
colony $g$.

Consider the linear operator $L^{(G)*}$ defined as
\be{spatial-coalescent-generator}
L^{(G)*} = L^{(G)*}_{\mathrm{mig}} + L^{(G)*}_{\mathrm{coal}},
\ee
where the operators $L^{(G)*}_{\mathrm{mig}},L^{(G)*}_{\mathrm{coal}}\colon\,
\mathcal{C}_G \to \mathcal{C}_G$ are defined for $\pi_G \in \Pi_G$ and 
$F \in \mathcal{C}_G$ as
\be{no3}
(L^{(G)*}_{\mathrm{mig}} F) (\pi_G) =
\sum_{i=1}^{b(\pi_{G}|_n)} \sum_{g,f \in G} a^*(g,f)
\big[F_n(\textrm{mig}_{g \to f,i}(\pi_{G}\vert_n))-F(\pi_G)\big],
\ee
\be{add2_a}
(L^{(G)*}_{\mathrm{coal}} F) (\pi_G) =
\sum_{g \in G} \sum_{\substack{J\subset\{i \in [n]:\,g_i = g\},\\|J| \geq 2}}
\lambda^{(\Lambda)}_{b(\pi_{G}|_n,g),|J|}
\big[F_n(\textrm{coal}_{J,g}(\pi_{G}\vert_n))-F(\pi_G)\big]
\ee
(recall definitions \eqref{add30}, \eqref{ag21}, \eqref{coalescence-rates} and \eqref{labelled-coalescence}).

\begin{proposition}
\label{prop:spatial-coalescent-martingale-problem-1}
{\bf [Martingale problem for the spatial $\Lambda$-coalescent with local coalescence]}\\
The spatial $\Lambda$-coalescent with local coalescence defined in Section~{\rm \ref{sss.statevol}} 
solves the well-posed martingale problem for $(L^{(G)*},C_{\mathrm{b}}(\Pi_G),
\delta_{S_G})$ with $S_G$ as in \eqref{eq:initial-state}. 
\end{proposition}

\begin{proof}
A straightforward inspection of the graphical construction yields the existence. 
The uniqueness is immediate because we have a duality relation, as we will see 
in Section~\ref{ss.dual}.
\end{proof}

\begin{remark}
Note that, instead of the singleton initial condition in
Proposition~{\rm \ref{prop:spatial-coalescent-martingale-problem-1}} (and in the 
graphical construction of Section~{\rm \ref{sss.statevol}}), we can use any other 
initial condition in $\Pi_G$.
\end{remark}

\subsubsection{Mean-field and immigration-emigration $\Lambda$-coalescents}
\label{sss.special-spatial-coalescents}

Some special spatial $\Lambda$-coalescents will be needed in the course of our 
analysis of the hierarchically interacting Cannings process. We define 
the \emph{mean-field $\Lambda$-coalescent} as the spatial $\Lambda$-coalescent 
with geographic space $G = \{0, \ldots, N-1 \} $ and migration kernel $a(i,j) = c/N$  
for all $i,j \in G$ with $i \neq j$. Furthermore, we define the \emph{$\Lambda$-coalescent 
with immigration-emigration} as the spatial $\Lambda$-coalescent with geographic space
$G = \{0, \ast\}$ and migration kernel $a(0,\ast) = c$, $a(\ast,0) = 0$.
In other words, $\ast$ is a cemetery migration state.

\subsection{Spatial $\Lambda$-coalescent with non-local coalescence}
\label{ss.blockcoal}

In this section, we construct a new type of spatial coalescent process based on 
a sequence $\uL =(\Lambda_k)_{k \in  \N_0}$ of finite measures on $[0,1]$ as in \eqref{ag41}, namely, the 
\emph{spatial $\uL$-coalescent} on $G=\Omega_N$ {\em with non-local
coalescence}. For each $k\in\N$, we introduce two additional transition
mechanisms: (1) a  \emph{block reshuffling} of all partition elements in a ball
of radius $k$;  (2) a \emph{non-local $\Lambda$-coalescence} of partition
elements in a ball of radius $k$. 

In this section, we assume that, for all $k \in \N$, measure $\Lambda_k$ satisfy \eqref{ag1}. But we do not assume that measure $\Lambda_0$ satisfies \eqref{ag1}. Denote  
\be{ak987}
d_0 = \Lambda_0\{0\}
.
\ee

In Section~\ref{sss.evpois}, we  give
definitions, in Section~\ref{sss.martingale-problem-block-coalescence} we 
formulate the martingale problem.

\subsubsection{The evolution rules and the Poissonian construction}
\label{sss.evpois}

In what follows, we consider $G=\Omega_N$. We start by extending the graphical construction from Section~\ref{sss.statevol} 
to incorporate the additional transition mechanisms of non-local reshuffling and 
coalescence.

Given the filtered probability space $(\Omega, \CF, (\CF_t)_{t\geq 0}, \P)$, consider Poisson point processes $\mathfrak{P}^{(\Omega_N)}$ on
\be{ag22}
[0,\infty) \times \Omega_N \times \N_0 \times [0,1] \times \{0, 1\}^{\N} 
\ee
having intensity measure
\be{ag22b}
\dd t \otimes
\dd \eta \otimes
\left( N^{-2k}\dd k \left[\Lambda^\ast_k(\dd r) 
\big(r \delta_1 + (1-r) \delta_0\big)^{\otimes{\N}}\right](\dd \omega) \right),
\ee
where $\omega = (\omega_i)_{i \in \N} \subset \{0,1\}^{\N}$, $(t,\eta,k,r,\omega) 
\in [0,\infty) \times \Omega_N \times \N_0 \times [0,1] \times \{0,1\}^{\N}$,
$\dd k$ is counting measure on $\N$ and $\dd \eta$ is counting measure on $\Omega_N$.  
Again, note that the third factor in \eqref{ag22b} is not a product measure 
(compare \eqref{a4}).

Given $\Sigma \Subset \Omega_N$ (i.e., $\Sigma$ is a finite subset of $\Omega_N$) 
and $\xi = \{\xi_i\}_{i=1}^{|\Sigma|}, \xi_i \in \Sigma$, let $\textrm{resh}_{\Sigma,
\xi}\colon\,\Pi_{\Omega_N} \to \Pi_{\Omega_N}$ be the \emph{reshuffling map} 
that for all $i$ moves families from $\eta_i \in \Sigma$ to $\xi_i \in \Sigma$:
\be{reshuffling-mapping}
\textrm{resh}_{\Sigma,\xi}(\pi_{\Omega_N})_i
= \begin{cases}
(\pi_i,\eta_i), &\eta_i \notin \Sigma,\\
(\pi_i,\xi_i), &\eta_i \in \Sigma,
\end{cases}
\quad
\pi_{\Omega_N}\in\Pi_{\Omega_N},\,i \in [b(\pi_{\Omega_N})].
\ee
Let
\be{ag23}
U_{\Sigma} = \{U_{\Sigma}(\xi) \}_{\xi \in \Sigma}
\ee
be a collection of independent $\Sigma$-valued random variables uniformly distributed on $\Sigma$.
We construct the \emph{standard spatial
$n$-$\Lambda$-coalescent with non-local coalescence} $\mathfrak{C}^{(\Omega_N)}_n =
(\mathfrak{C}^{(\Omega_N)}_n(t) \in \Pi_{\Omega_N,n})_{t \geq 0}$ as the
$\Pi_{\Omega_N,n}$-valued Markov process with the following properties:

\begin{itemize}
\item
\emph{Initial state.}
Assume $\mathfrak{C}^{(\Omega_N)}_n(0) \in S_{\Omega_N,n}$ (recall \eqref{labelled-singletons}).
\item
\emph{Coalescence with reshuffling.}
For each point $(t,\eta,k,r,\omega)$ of the Poisson 
point process $\mathfrak{P}^{(\Omega_N)}$ (cf. \eqref{ag22}--\eqref{ag22b}), all families $(\pi_i,\eta_i) \in 
\mathfrak{C}^{(\Omega_N)}_n(t-)$ such that $\omega_i = 1$ and $\eta_i \in B_k(\eta)$
coalesce into a new family with label $\eta$. Subsequently, all families with labels 
$\zeta \in B_k(\eta)$ obtain a new label that is drawn independently and uniformly 
from $B_k(\eta)$. In a formula (recall \eqref{labelled-coalescence}, \eqref{reshuffling-mapping}--\eqref{ag23}):
\be{block-coalescence-transition}
\mathfrak{C}^{(\Omega_N)}_n(t)
= \textrm{resh}_{B_k(\eta),U_{B_k(\eta)}} \circ 
\textrm{coal}_{\{i \in [n]\colon\,\omega_i = 1, 
\eta_i(t-) \in B_k(\eta)\},\eta}(\mathfrak{C}^{(\Omega_N)}_n(t-)).
\ee
Note that, in contrast with the spatial coalescent with local coalescence from Section~\ref{ss.spatcoal},
the coalescence mechanism in \eqref{block-coalescence-transition} is no longer 
local: all families whose labels are in $B_k(\eta)$, $k \in \N$, are involved 
in the coalescence event at site $\eta \in \Omega_N$.
\item
\emph{Migration.}
Independently of the coalescence events, the labels of all partition elements of
$\mathfrak{C}^{(\Omega_N)}_n(t)$ perform independent random walks with transition
rates $a^{(N)}(\cdot,\cdot)$ (recall \eqref{32b} and \eqref{ag21}).
\end{itemize}

As in Section~\ref{ss.spatcoal}, the consistency-between-restrictions property
allows us to apply the Kolmogorov extension theorem to the family 
$\{\mathfrak{C}^{(\Omega_N)}_n\}_{n \in \N}$ to construct the Markov process
\be{ag23b}
\mathfrak{C}^{(\Omega_N)}
\ee
taking values in $\Pi_{\Omega_N}$. 

\begin{definition}
\label{def-spatial-coal}
The process $\mathfrak{C}^{(\Omega_N)}$ is called the spatial $\uL$-coalescent with 
non-local coalescence corresponding to the resampling measures $(\Lambda_k)_{k \in \N_0}$ 
(recall \eqref{ag41}) and the migration coefficients $(c_k)_{k\in\N_0}$ (recall 
\eqref{ckdef}).
\end{definition}

\begin{proposition}
{\bf [Feller property]}\\
The process $\mathfrak{C}^{(\Omega_N)}$ is a càdlàg strong Markov process with the 
Feller property.
\end{proposition}

\begin{proof}
This is an immediate consequence of the Poissonian construction.
\end{proof}

\subsubsection{Martingale problem}
\label{sss.martingale-problem-block-coalescence}

In this section, we characterise the spatial $\uL$-coalescent with non-local
coalescence as the solution of the corresponding martingale problem.

Given $\pi_{\Omega_N,n} \in \Pi_{\Omega_N,n}$ and $\eta \in \Omega_N$, 
denote the number of families of
$\pi_{\Omega_N,n}$ with labels in $B_k(\eta)$ (recall~\eqref{block-definition}) by
\be{no4}
b(\eta) = b(\pi_{\Omega_N,n},B_k(\eta)) = |\{ (\pi_i,\eta_i) \in \pi_{\Omega_N,n}\colon\, 
\eta_i \in B_k(\eta) \}| \in \N.
\ee
Recall the definition of the algebra of test functions $\CC_G$ from Section~\ref{sss.mart}. 
Let $\pi_{\Omega_N} = \{(\pi_i, \eta_i)\}_{i \in \N } \in \Pi_{\Omega_N}$,  $F \in 
\mathcal{C}_{\Omega_N}$ and $F(\cdot) = F_n(\cdot\vert_n)$ (recall \eqref{ghkk1}). Consider the linear operator 
$L^{(\Omega_N)*}$ defined as
\be{spatial-block-coalescent-generator}
L^{(\Omega_N)*} = L^{(\Omega_N)*}_{\mathrm{mig}} 
+ L^{(\Omega_N)*}_{\mathrm{coal}},
\ee
where the linear operators $L^{(\Omega_N)*}_{\mathrm{mig}}$ and 
$L^{(\Omega_N)*}_{\mathrm{coal}}$ are defined as follows (recall \eqref{ghkk1}). 
The migration operator is\footnote{Note that $a^{(N)} =  a^{(N)\ast}$ for the hierarchical 
random walk (cf.~\eqref{no3}).}
\be{coalescent:multi-level-migration}
\left(L^{(\Omega_N)*}_{\mathrm{mig}} F\right) (\pi_{\Omega_N})
= \sum_{i=1}^{b(\pi_{\Omega_N}\vert_n)}
\sum_{\eta,\zeta \in \Omega_N}a^{(N)\ast}(\eta,\zeta)
\left[F_n(\textrm{mig}_{\eta \to \zeta,i}(\pi_{\Omega_N}\vert_n))
- F(\pi_{\Omega_N})\right],
\ee
and the block-coalescence-reshuffling operator is (recall \eqref{coalescence-rates}, \eqref{labelled-coalescence}, \eqref{reshuffling-mapping} and \eqref{no4})
\be{coalescent:multi-level-coal}
\begin{aligned}
&\left(L^{(\Omega_N)*}_{\mathrm{coal}} F\right) (\pi_{\Omega_N})
= \sum_{\eta \in \Omega_N} \sum_{k \in \N_0} N^{-2k} \sum_{{\xi_1} \in B_k(\eta)}
N^{-k} \sum_{{\xi_2} \in B_k(\eta)} N^{-k} \cdots
\sum_{{\xi_{|B_k(\eta)|}} \in B_k(\eta)}N^{-k}\\
& \times \sum_{\substack{J \subset [b(\eta)],\\|J| \geq 2}}
\lambda^{(\Lambda_k)}_{b(\eta),|J|}
\left[F_n(\textrm{resh}_{B_k(\eta),\xi}\circ
\textrm{coal}_{\{i \in J\colon\, \eta_i \in
B_k(\eta)\},\eta}(\pi_{\Omega_N}\vert_n)) - F(\pi_{\Omega_N})\right].
\end{aligned}
\ee

\begin{proposition}
\label{prop:spatial-coalescent-martingale-problem-2}
{\bf [Martingale problem: Spatial $\uL$-coalescent with non-local coalescence]}\\
The spatial $\uL$-coalescent with non-local coalescence $\mathfrak{C}^{(\Omega_N)}$
defined in Section~{\rm \ref{sss.evpois}} solves the well-posed martingale problem
$(L^{(\Omega_N)*},\CC_{\Omega_N},\delta_{S_{\Omega_N}})$ with $S_{\Omega_N}$ as in \eqref{eq:initial-state}.
\end{proposition}

\begin{proof}
A straightforward inspection of the graphical construction in Section~\ref{sss.evpois} 
yields the existence of a solution. Uniqueness on finite geographic spaces is clear: 
this follows in the same way as for the single-site case. Once we have well-posedness 
for finite geographic spaces, we can show uniqueness for $G=\Omega_N$ via approximation. 
The approximation via finite geographic spaces follows from the fact that the occupation 
numbers of the sites are stochastically smaller than in the case of pure random walks
(see Liggett and Spitzer~\cite{LS81}).
\end{proof}

\begin{remark}
Note that, instead of the singleton initial condition in
Proposition~{\rm \ref{prop:spatial-coalescent-martingale-problem-2}} (and in the 
graphical construction of Section~{\rm \ref{sss.evpois}}), we can use any other 
initial condition in $\Pi_{\Omega_N}$.
\end{remark}

\subsection{Duality relations}
\label{ss.dual}

We next formulate and prove the duality relation between the
$C^{\uc,\uL}_N$-process from Section~\ref{sss.hierarCan} and the spatial
$\uL$-coalescent with non-local coalescence $\mathfrak{C}^{(\Omega_N)}$
described so far. This follows a general pattern for all choices of the
geographic space $G$ in \eqref{Gchoices}. We only give the proof for the case $G
= \Omega_N$.

Recall \eqref{ag13}. The construction of the duality function $H(\cdot,\cdot)$ requires 
some new ingredients. For $n \in \N$ and $\varphi \in C_{\mathrm{b}}(E^{n},\R)$, 
consider the \emph{bivariate function} 
$H^{(n)}_\varphi\colon\,\CP(E)^G \times\Pi_{G,n} \to \R$ of the form
\be{duality-test-functions}
\begin{aligned}
H^{(n)}_\varphi(x , \pi_{G,n})
& = \int_{E^{b}}
\left(\bigotimes_{i=1}^{b} x_{\eta_{\pi^{-1}(i)}}\big(\dd u_i\big)\right)
\varphi\big(u_{\pi(1)},u_{\pi(2)},\ldots,u_{\pi(n)}\big),
\end{aligned}
\ee
where $x = (x_{\eta})_{\eta \in G}\in\CP(E)^{G}$, $\pi_{G,n} \in \Pi_{G,n}$, $b 
= b(\pi_{G,n}) = |\pi_{G,n}|$ (cf. \eqref{add30}), $(\eta_i)_{i \in [b]}=L(\pi_{G,n})$ (cf. \eqref{20b}) are the labels 
of the partition $\pi_{G,n}$, and (with a slight abuse of notation) $\pi\colon\,[n]
\to [b]$ is the map from Remark~\ref{rem:coalescent-block-mapping}. In words, the 
functions in \eqref{duality-test-functions} assign the same type to individuals 
that belong to the same family. Note that these functions form a family of functions 
on $\CP(E)^G$,
\be{duality-separating-family}
\left\{
H^{(n)}_\varphi(\cdot, \pi_{G,n})\colon\,
\CP(E)^G\to\R\mid\pi_{G,n} \in \Pi_{G,n},
n \in \N,\varphi \in C_{\mathrm{b}}(E^{n},\R)\right\},
\ee
that separates points. The $C^{\uL}$-process with block resampling and the spatial 
$\uL$-coalescent with non-local coalescence are mutually \emph{dual} w.r.t.\ the duality 
function $H(\cdot,\cdot)$ given by
\be{ag30_b}
H(x,(\varphi, \pi_{G,n})) = H^{(n)}_\varphi (x,  \pi_{G,n}),
\qquad x \in \CE = \CP(E)^G,\,(\varphi, \pi_{G,n}) \in \CE^\prime,
\ee
with $\CE^\prime =\cup_{n\in\N_0}(C_\mathrm{b}(E^n, \R)\times \Pi_{G,n})$.

We proceed with the following observation. Recall the definition of duality in the paragraph including \eqref{ag13}.

\begin{remark}
\label{R2.9}
\begin{itemize}
\item[\textup{(a)}]
Let $X$ and $Y$ be two processes that are dual w.r.t.\ a continuous and
bounded duality function $H(\cdot,\cdot)$. Assume that $X$ and $Y$ are 
solutions to martingale problems corresponding to operators $L_X$, respectively, 
$L_Y$. Then the generator relation
\be{ag24}
[L_X (H(\cdot, Y_0))] (X_0) = [L_Y (H(X_0, \cdot))](Y_0), 
\quad \mbox{for all } (X_0, Y_0) \in \CE \times \CE^\prime,
\ee
is equivalent to the duality relation \eqref{ag13} (see, e.g., Ethier and 
Kurtz~{\rm \cite[Section~4.4]{EK86}}).

\item[\textup{(b)}]
Item \textup{(a)} gives the duality function $H(\cdot,\cdot)$ for all $t \geq 0$ and $n \in \N$, 
as is proved in Proposition~{\rm \ref{prop:spatial-coalescent-duality}} below. In particular, the 
following holds
\be{duality}
\E\left[H^{(n)}_\varphi(X^{(G)}(t),\mathfrak{C}^{(G)}(0)\vert_n)\right]
= \E\left[H^{(n)}_\varphi(X^{(G)}(0),\mathfrak{C}^{(G)}(t)\vert_n)\right],
\ee
with $X^{(G)}$ as below Proposition~{\rm \ref{P.vecLambda}} and $\mathfrak{C}^{(G)}$ as in 
Definition~{\rm \ref{def-spatial-coal}}.
\end{itemize}
\end{remark}

In our context, we have to verify the following relation for the linear operators
in the martingale problem. 

\begin{proposition}
\label{prop:spatial-coalescent-duality} 
{\bf [Operator level duality]}\\
For any of the geographic spaces $G = \Omega_N$, $G= \{0,\ldots,N-1\}^K, K \in \N$ 
and $G = \{0,\ast\}$ the following holds. For all $n \in \N$, for all $H^{(n)}_\varphi$ as in 
\eqref{duality-test-functions}, all $x \in \CP(E)^G$, and 
all $\pi_G \in \Pi_{G}$, 
\be{generator-duality} 
\left(L^{(G)} H^{(n)}_\varphi(\cdot , \pi_G\vert_n) \right)(x) 
= \left( L^{(G)*} H^{(n)}_\varphi(x , \cdot\vert_n) \right)(\pi_G).
\ee
\end{proposition}

\begin{proof}
We check the statement for $G=\Omega_N$. In this case, $L^{(G)}$ is as in
\eqref{ak:multi-level-generator} and $L^{(G)*}$ is as in
\eqref{spatial-block-coalescent-generator}. The proof for the other choices of
$G$ is left to the reader.

The claim follows from a straightforward inspection of 
(\ref{ak:multi-level-migration}--\ref{ak:multi-level-resampling-global}) 
and (\ref{coalescent:multi-level-migration}--\ref{coalescent:multi-level-coal}),
respectively. Indeed, duality of the migration operators in
\eqref{ak:multi-level-migration} and \eqref{coalescent:multi-level-migration} 
is evident:
\be{no5} 
\left( L^{(G)}_{\mathrm{mig}} H^{(n)}_\varphi(\cdot,\pi_G\vert_n)\right)(x) 
= \left( L^{(G)*}_{\mathrm{mig}} H^{(n)}_\varphi(x,\cdot\vert_n)\right)(\pi_G). 
\ee

Let us check the duality of the resampling and coalescence operators in
\eqref{ak:multi-level-resampling-global} and
\eqref{coalescent:multi-level-coal}. It is enough to assume that $d_0 = 0$,
since it is well-known that Fleming-Viot operator $L^d$ (cf.\ \eqref{Ldefs}) is
dual with the generator of the Kingman coalescent which is the special case of
$L^{(G)*}_{\mathrm{coal}}$ (cf.\ \eqref{add2_a}) with $\Lambda_0 = d_0
\delta_0$.

By a standard approximation argument, it is enough to consider the duality test 
functions in \eqref{duality-test-functions} of the product form, i.e., with $\varphi(u) 
= \prod_{i=1}^n \varphi_i(u_i)$, where $u = (u_i)_{i=1}^n \in E^n$ and $\varphi_i
\in C_{\mathrm{b}}(E)$. Using \eqref{ak:multi-level-resampling-global}--\eqref{ak:resampling-mapping-global}, 
\eqref{coalescence-rates}, \eqref{no4} and simple algebra, for $x \in \CP(E)^G$ and 
$\pi_G \in \Pi_{G}$ we can rewrite the action of the resampling operator on 
the duality test function as follows (where for ease of notation we
assume that $\pi_G \in S_G$ (cf. \eqref{eq:initial-state}), i.e., $\pi_G$ has the singleton family structure)
\begin{align}
\label{sk:resampling-action-on-duality-test-function}
&\left(L^{(\Omega_N)}_{\mathrm{res}} H^{(n)}_\varphi(\cdot ,\pi_G\vert_n)\right)(x)
\nonumber
\\
&= \sum_{\eta \in G} \sum_{k \in \N_0} N^{-2k} \int_{[0,1]} 
\Lambda_k^*(\dd r) N^{-k} \sum_{\rho \in B_k(\eta)} \int_E x_\rho(\dd a) 
\nonumber
\\
&\quad\times\left( \prod_{i=1}^{b(\eta)} 
\left\langle \Big( \Phi_{r,a,B_k(\eta)}(x) \Big)_{\eta_{\pi^{-1}(i)}}, 
\prod_{j\colon\,\pi(j)=i} \varphi_j \right\rangle
- \prod_{i=1}^{b(\eta)}
\left\langle x_{\eta_{\pi^{-1}(i)}}, 
\prod_{j\colon\, \pi(j)=i} \varphi_j \right\rangle \right)
\nonumber
\\
&= \sum_{\eta \in G} \sum_{k \in \N_0} N^{-2k} 
\int_{[0,1]} \Lambda_k^*(\dd r) N^{-k} \sum_{\rho \in B_k(\eta)} \int_E x_\rho(\dd a) 
\nonumber
\\
&\quad\quad\times\left(
\sum_{\substack{J \subset [b(\eta)]\\|J| \geq 0}} 
\prod_{i \in [b(\eta)] \backslash J} 
\left\langle (1-r) y_{\eta_{\pi^{-1}(i)},k},
\prod_{j\colon\, \pi(j)=i} \varphi_j \right\rangle 
\prod_{i \in J} \left\langle r \delta_a, 
\prod_{j\colon\, \pi(j)=i} \varphi_j \right\rangle \right.
\nonumber
\\
& \left. \quad\quad\quad - \prod_{i=1}^{b(\eta)}
\left\langle x_{\eta_{\pi^{-1}(i)}},\prod_{j\colon\,\pi(j)=i} 
\varphi_j \right\rangle \right) 
\nonumber
\\
&= \sum_{\eta \in G} \sum_{k \in \N_0} N^{-2k} 
\sum_{\substack{J \subset [b(\eta)],\\|J| \geq 2}}
\lambda^{(\Lambda_k)}_{b(\eta),|J|} 
\nonumber
\\
&\quad\quad\times\left( N^{-k} \sum_{\rho \in B_k(\eta)}
\prod_{i \in [b(\eta)] \backslash J} 
\left\langle N^{-k} \sum_{\xi \in B_k(\eta)} x_\xi,
\prod_{j: \pi(j)=i} \varphi_j \right\rangle 
\prod_{i \in J} \left\langle x_\rho , \prod_{j\colon\,\pi(j)=i} 
\varphi_j \right\rangle \right. 
\nonumber
\\
& \left. \quad\quad\quad - \prod_{i=1}^{b(\eta)}
\left\langle x_{g_{\pi^{-1}(i)}},\prod_{j\colon\,\pi(j)=i} 
\varphi_j \right\rangle \right).
\end{align}
On the other hand, according to \eqref{coalescent:multi-level-coal} (also recall \eqref{labelled-coalescence}, \eqref{reshuffling-mapping}), we have
\begin{align}
\label{sk:coalescent-action-on-duality-test-function}
&\left(L^{(\Omega_N)*}_{\mathrm{coal}} H^{(n)}_\varphi(x,\cdot\vert_n)\right)(\pi_G)
= \sum_{\eta \in \Omega_N} \sum_{k \in \N_0} N^{-2k}
\sum_{\substack{J \subset [b(\eta)],\\|J| \geq 2}}
\lambda^{(\Lambda_k)}_{b(\eta),|J|}
\nonumber
\\
& \times
\left(\sum_{{\xi_1} \in B_k(\eta)} N^{-k}
\sum_{{\xi_2} \in B_k(\eta)} N^{-k} \ldots
\sum_{{\xi_{b(\eta)}} \in B_k(\eta)} N^{-k} \right.
\nonumber
\\
& \quad\quad \times \left( \prod_{i \in [b(\eta)] \backslash J}
\left\langle x_{\xi_i} , \prod_{j\colon\,\pi(j)=i} \varphi_j \right\rangle
\left\langle x_{\xi_{\min\{l\colon\,l \in J\}}} , \prod_{j\colon\, 
\pi(j) \in J} \varphi_j \right\rangle \right.
\nonumber
\\
& \left.\left. \qquad\qquad - \prod_{i=1}^{b(\eta)}
\left\langle x_{g_{\pi^{-1}(i)}},
\prod_{j\colon\, \pi(j)=i} \varphi_j \right\rangle \right) \right).
\end{align}
Comparing \eqref{sk:coalescent-action-on-duality-test-function} with
\eqref{sk:resampling-action-on-duality-test-function}, we get the claim.
\end{proof}

\subsection{The long-time behaviour of the spatial $\Lambda$-coalescent with 
non-local coalescence}
\label{ss.ltbvecL}

We next investigate the long-time behaviour of the spatial $\uL$-coalescent with non-local coalescence. 
Subsequently, the duality relation allows us to translate results on 
the long-time behaviour of the spatial $\uL$-coalescent with non-local coalescence into results on the 
long-time behaviour of the $C^{\uc,\uL}_N$-process.

	\subsubsection{The behaviour as $t\to \infty$}
\label{sss.behtinf}

In this section, we prove the existence and uniqueness of a limiting state 
for the spatial $\uL$-coalescent with non-local coalescence as $t \to \infty$.

\begin{proposition}
\label{P.limstate}
{\bf [Limiting state]}\\
Start the $\mathfrak{C}^{(\Omega_N)}$-process from \eqref{ag23b} in a labelled partition $\{(\pi_i,
\eta_i)\}_{i=1}^{n}$, where $\{\pi_i\}_{i=1}^n$ form a partition of $\N$ and 
$\{\eta_i\}_{i=1}^n$ are the corresponding labels. If $x$ is a translation-invariant 
shift-ergodic random state with mean $\theta \in \CP(E)$, then
\be{ag40}
\CL \left[H^{(n)}_\varphi(x, \mathfrak{C}^{(\Omega_N)}_n(t))\right]
\mathop{\Longrightarrow}_{t \to \infty}
\CL\left[H^{(n)}_\varphi(\underline{\theta},
\mathfrak{C}^{(\Omega_N)}_n(\infty))\right]
\qquad \forall\, n \in \N,
\ee
where $\mathfrak{C}^{(\Omega_N)}_n$ is as in Section~{\rm \ref{sss.evpois}} and 
$H^{(n)}_\varphi$ as in \eqref{duality-test-functions}.
\end{proposition}

\bpr 
We first observe that $|\mathfrak{C}^{(\Omega_N)}_n(t)|$ is monotone non-increasing,  
so that there exists a limit for the number of partition elements. This implies that 
the partition structure converges a.s.\ to a limit partition, which we call 
$\mathfrak{C}^{(\Omega_N,n)}(\infty) \in \Pi_{\Omega_N,n}$ (cf. \eqref{labelled-partitions}). We 
must prove that the locations result in an effective averaging of the configuration  $x$, so that we 
can replace the $|\mathfrak{C}_n^{(\Omega_N)}(t)|$-locations by any tuple for the (constant)
configuration $\underline{\theta}$. This is a standard argument (see, e.g., the proof 
of the ergodic theorem for the voter model in Liggett~\cite{L85}). 
\epr

Recall the definition of the spatial $\Lambda$-coalescent with immigration-emigration introduced in 
Section~\ref{sss.special-spatial-coalescents}.

\begin{corollary}
\label{C.convZ} {\bf [Limiting state of the $\Lambda$-coalescent with
immigration-emigration]} The analogous to \eqref{ag40} statement holds if we
substitute $\mathfrak{C}^{(\Omega_N)}_n$ with the the $\Lambda$-coalescent  with
immigration-emigration (see Section~{\rm \ref{sss.special-spatial-coalescents}}),
i.e., the spatial $\Lambda$-coalescent $\mathfrak{C}^{(G),\mathrm{loc}}_n$ with
geographic space $G = \{0, \ast\}$ and migration kernel $a(0,\ast) = c$,
$a(\ast,0) = 0$.
\begin{equation}
\label{eq:limiting-state-coal-with-im-emi}
\CL \left[H^{(n)}_\varphi(x, \mathfrak{C}^{(\{0, \ast\}),\mathrm{loc}}_n(t))\right]
\mathop{\Longrightarrow}_{t \to \infty}
\CL\left[H^{(n)}_\varphi((y, x_\ast),
\mathfrak{C}^{(\{0, \ast\}),\mathrm{loc}}_n(\infty))\right]
\quad \forall y \in \CP(E), n \in \N
,
\end{equation}
where $H^{(n)}_\varphi$ as in \eqref{duality-test-functions} and $x = (x_0, x_\ast) \in \CP(E)^2$.
Note that the right hand side of \eqref{eq:limiting-state-coal-with-im-emi} does not depend on $y$.
\end{corollary}

\subsubsection{The dichotomy: single ancestor versus multiple ancestors}
\label{sss.dichosv}

A key question is whether the $\mathfrak{C}^{(\Omega_N)}$-process from \eqref{ag23b} converges 
to a single labelled partition element as $t \to \infty$ with probability one. To answer this question, we 
have to investigate whether two tagged partition elements coalesce with probability 
one or not. Recall that, by the projective property of the coalescent, we may focus 
on the subsystem of just two dual individuals, because this translates into the 
same dichotomy for any $\mathfrak{C}_n^{(\Omega_N)}$-coalescent and hence for the entrance law 
starting from countably many individuals. However, there is additional reshuffling 
at all higher levels, which is triggered by a corresponding block-coalescence event. 
Therefore, we consider two coalescing random walks $(Z^1_t, Z^2_t)_{t\geq 0}$ on 
$\Omega_N$ with migration coefficients $(\bar{c}_k(N))_{k \in \N_0}$ (cf.~\eqref{a2}) 
and coalescence at rates $(\lambda_k)_{k\in\N_0}$. Consider the {\it time-$t$ accumulated 
hazard for coalescence} of this pair:
\be{ak12}
H_N(t)  = \sum_{k\in\N_0} \lambda_k N^{-k} \int_0^t 
1\left\{d(Z^1_s, Z^2_s) \leq k\right\}\,\dd s.
\ee
Here, the rate $N^{-2k}$ to choose a $k$-block is multiplied by $N^k$ because all 
partition elements in that block can trigger a coalescence event. This explains the 
factor $N^{-k}$ in \eqref{ak12}. Let
\be{ak123456}
H_N  \equiv  \lim_{t\to\infty} H_N(t).
\ee
We have coalescence of the random walks (= common ancestor) with probability 
one, when $H_N = \infty$ a.s., but separation of the random walks (= different 
ancestors) with positive probability, when $H_N < \infty$ a.s.

\begin{lemma}
\label{ak14b}
{\bf [Zero-one law]}
$H_N = \infty$ a.s.\ if and only if $\bar H_N = \E[H_N] = \infty$. Moreover, under 
the weak regularity condition in \eqref{dichreg} the latter is equivalent to
\be{eq:clcond}
\sum_{k \in \N_0} \frac{1}{c_k} \sum_{l=0}^k \lambda_l = \infty.
\ee
\end{lemma}

\bpr
Write $H_N=\sum_{k\in\N_0} w_k L(k)$ with 
\be{wkLkdef}
w_k(N) = \sum_{j \geq k} \lambda_j N^{-j}, \qquad 
L(k) = \int_0^\infty 1\{d(Z^1_s,Z^2_s)=k\}\dd s.
\ee 
Note that $w_k(N)<\infty$ because of condition \eqref{ak:lambda-growth-condition}. We 
want to show that $\bar H_N = \infty$ implies $H_N = \infty$ (the reverse is immediate).
Recall from Section~\ref{ss.dichoXN} that $P_t(\cdot,\cdot)$ denotes the time-$t$ transition 
kernel of the hierarchical random walk on $\Omega_N$ with migration coefficients 
$(\bar{c}_k(N))_{k \in \N_0}$ given by~\eqref{a2}. In the computations below, we pretend 
that the coefficients are $(c_k)_{k\in\N_0}$.  Afterwards, we can replace $c_k$ by 
$\bar{c}_k(N)$.

Note that $(Z^1_s-Z^2_s)_{s \geq 0}$ has the same law as a single copy $(Z_s)_{s \geq 0}$ 
of the hierarchical random walk but moving at twice the speed. Thus, in law, we may replace 
$L(k)$ by $L(k)=\int_0^\infty 1\{|Z_{2s}|=k\} \dd s$. 

\myparagraph{Step 1.} 
As shown in Dawson, Gorostiza and Wakolbinger~\cite[Eq.~(3.1.5)]{DGW05}, for
the hierarchical random walk with jump rate 1,
\be{ak:dgw-asympt}
P_t(0,\eta) = \sum_{j \geq k} K_{jk}(N)\,\frac{\exp\left[-h_j(N) t\right]}{N^{j}},
\qquad t \geq 0,\, \eta \in \Omega_N\colon\,|\eta|=k \in \N_0,
\ee
where
\be{ak:K-j-k}
K_{jk}(N) =
\begin{cases}
0, & j = k = 0,\\
-1, & j = k > 0, \qquad j,k \in \N_0,\\
N-1, & \text{otherwise,}
\end{cases}
\ee
and
\be{ak:h-j}
h_j(N) = \frac{N}{N-1} r_j(N) + \sum_{i>j}^\infty r_i(N), \quad j \in \N,
\ee
where, for the hierarchical random walk defined in Section~\ref{sss.mighg},
\be{ak:r-j}
r_j(N) = \frac{1}{D(N)}\sum_{i \geq j} \frac{c_{i-1}}{N^{2i-j-1}}, \qquad j \in \N,
\ee
with $D(N)$ the normalising constant such that $\sum_{j\in\N} r_j(N) = 1$. 

The random walk in \cite{DGW05} has jump rate 1, while our hierarchical random walk has
jump rate 
\begin{equation}
\label{Dstardef}
D^*(N) = \sum_{\eta\in\Omega_N} a^{(N)}(0,\eta)
= \sum_{k\in\N} (N^k-N^{k-1}) \sum_{j \geq k} \frac{c_{j-1}}{N^{2j-1}}
= \sum_{m\in\N_0} \frac{c_m}{N^m}\,\left(1-\frac{1}{N^{m+1}}\right).
\end{equation}
Therefore, after computing $H_N$ with the help of the above formulas, we must divide 
$H_N$ by $D^*(N)$ to get the correct expression. 

Note that (\ref{ak:h-j}--\ref{ak:r-j}) simplify considerably when $N\to\infty$, namely,
\be{ak-large-N-rates}
h_j(N) \sim r_j(N) \sim \frac{c_{j-1}}{D(N) N^{j-1}}, \qquad D(N) \sim c_0,
\ee
while also \eqref{wkLkdef} and \eqref{Dstardef} simplify to
\be{wkasymp}
w_k(N) \sim \frac{\lambda_k}{N^k}, \qquad D^*(N) \sim c_0.
\ee
Moreover, because $\limsup_{k\to\infty} \frac{1}{k} \log c_k < \log N$ and $\limsup_{k\to\infty}
\frac{1}{k} \log \lambda_k < \log N$ (see the footnotes in Sections~\ref{sss.mighg}--\ref{sss.reshg}),
the following holds:
\be{fdh-Ncomp}
\begin{array}{ll}
&\text{For every $N \in \Ntwo$ the quantities $h_j(N)$, $r_j(N)$, $D(N)$, $w_k(N)$ and $D^*(N)$}\\ 
&\text{are bounded from above and below by positive finite constants times their}\\
&\text{$N\to\infty$ asymptotics uniformly in the indices $j,k$}. 
\end{array}
\ee

\myparagraph{Step 2.}
For $M \in \N_0$, define the truncated hazard
\be{ak-truncated-hazard}
H^{(M)}_N = \sum_{k=0}^M w_k(N) L(k). 
\ee
For a non-negative random variable $V$ with a finite second moment, Cauchy-Schwarz
gives
\be{ak-inverse-chebyshev}
\P\{V > 0\} \geq (\E[V])^2/\,\E[V^2].
\ee
Therefore,
\be{ak1}
\P \left\{ H^{(M)}_N \bigg/\, \E\left[H^{(M)}_N\right] > 0 \right\} 
\geq \left(\E\left[H^{(M)}_N\right]\right)^2 \bigg/\, \E\left[(H^{(M)}_N)^2\right].
\ee
To compute the quotient in the right-hand side of \eqref{ak1}, we write
\be{ak2alt}
\begin{aligned}
\E\left[H^{(M)}_N\right]
&= \sum_{k=0}^M w_k(N) \int_0^\infty \dd s \, P\{|Z_{2s}| = k\}\\
&= \tfrac12 \sum_{k=0}^M w_k(N) \sum_{\eta \in \partial B_k(0)} G(0,\eta)
\end{aligned}
\ee
and
\be{ak2}
\begin{aligned}
\E\left[\left(H^{(M)}_N\right)^2\right] 
&= \sum_{k,l=0}^M w_k(N) w_l(N) \int_0^\infty \dd s \int_0^\infty \dd t \, P\{|Z_{2s}| = k\}P\{|Z_{2t}| = l\}\\
&= \tfrac12 \sum_{k,l=0}^M w_k(N) w_l(N) \sum_{\substack{\eta \in \partial B_k(0)\\ \eta^\prime \in \partial B_l(0)}} 
G(0,\eta)G(0,\eta^\prime-\eta).
\end{aligned}
\ee
Here, $G$ is the Green function of the hierarchical random walk, which by \eqref{ak:dgw-asympt} 
equals
\be{ak-green}
G(0,\eta) = G_k(N), \qquad \eta \in \Omega_N\colon\,|\eta|=k\in\N_0,
\qquad G_k(N) = \sum_{j \geq k} K_{jk}(N)\,\frac{1}{h_j(N)N^{j}}.
\ee
Let
\be{eq:conting}
N[k] = \left\{\begin{array}{ll}
1, &k=0,\\
N^k-N^{k-1}, &k>0,
\end{array}
\right.
\qquad 
\bar{N}[k] = \left\{\begin{array}{ll}
1, &k=0,\\
N^k-2N^{k-1}, &k>0,
\end{array}
\right.
\ee
denote the number of sites at distance $k$ from the origin, respectively, at distance $k$
from both the origin and a given site itself at distance $k$ from the origin. A straightforward 
counting argument shows that
\be{1st2ndmoment}
\begin{aligned}
\mbox{r.h.s.}\eqref{ak2alt} &= \tfrac12 \sum_{k=0}^M w_k(N)\,N[k]\,G_k(N),\\
\mbox{r.h.s.}\eqref{ak2} &= \tfrac12 \sum_{k,l=0}^M w_k(N)\,w_l(N)\,N[k]\,N[l]\,G_{k \vee l}^2(N)\\
&\qquad + \tfrac12 \sum_{k=0}^M w_k^2(N)\,N[k]\,G_k(N)
\Big\{\big(\bar{N}[k]-N[k]\big)\,G_k(N) + \sum_{m=0}^{k-1} N[m]\,G_m(N) \Big\}. 
\end{aligned}
\ee
For $N\to\infty$, substituting \eqref{ak:K-j-k} and \eqref{ak-large-N-rates} into \eqref{ak-green}
and the resulting expression into \eqref{1st2ndmoment}, we get
\be{ak-h-m-n-1st-moment}
\E \left[ H^{(M)}_N \right] \sim \sum_{k = 0}^M \mu_k
\sum_{m \geq k} \frac{1}{c_m}
\ee
and
\be{ak-h-m-n-2nd-moment}
\E \left[ \left(H^{(M)}_N\right)^2 \right] \sim 2 \sum_{k,l = 0}^M \mu_k \mu_l 
\left( \sum_{m \geq k \vee l} \frac{1}{c_m} \right)^2,
\ee
where we use that the dominant term in the sum defining $G_k(N)$ in \eqref{ak-green} is the 
one with $j=k+1$, and we also use that $\mu_k=\tfrac12\lambda_k$ as in \eqref{ag46}. Thus, 
for every $M$, the right-hand side of \eqref{ak1} is bounded from below by a number that 
tends to $\tfrac12$ as $N\to\infty$. Together with the observation made below 
(\ref{ak-large-N-rates}--\ref{wkasymp}), it therefore follows that there exists a 
$\delta>0$ independent of $M$ and $N$ such that  
\be{ak1MN}
\P \left\{ H^{(M)}_N \bigg/\, \E\left[H^{(M)}_N\right] > 0 \right\} 
\geq \delta.
\ee

\myparagraph{Step 3.}

Since $H^{(M)}_N \leq H_N$ and $H_N = \lim_{M\to\infty} H^{(M)}_N$, it
follows from \eqref{ak1MN} that 
\be{ak1N}
\P \left\{ H_N \big/\, \E[H_N] > 0 \right\} \geq \delta.
\ee
Thus, $\E[H_N]=\infty$ implies $\P\{H_N=\infty\} \geq \delta$. But the event 
$\{H_N=\infty\}$ lies in the tail-sigma-algebra of the hierarchical random 
walk, which is trivial, and therefore this event has probability 0 or 1. Consequently, 
$P\{H_N=\infty\}=1$.

\myparagraph{Step 4.}
Finally, replacing $c_k$ by $\bar{c}_k(N) = c_k + N^{-1}\lambda_{k+1}$ (recall \eqref{a2}), 
noting that \eqref{fdh-Ncomp} continues to apply, and using \eqref{ak-h-m-n-1st-moment} with 
$M=\infty$, we get that $P\{H_N=\infty\}=1$ if and only if 
\begin{equation}
\sum_{k \in \N_0} \lambda_k
\sum_{m \geq k} \frac{1}{c_m + N^{-1} \lambda_{m+1}} = \infty,
\end{equation}
which is the same as
\begin{equation}
\sum_{k \in \N_0} 
\frac{1}{c_k + N^{-1} \lambda_{k+1}} \sum_{l=0}^k \lambda_l = \infty.
\end{equation}
Under the weak regularity condition in \eqref{dichreg} the latter is equivalent to
\eqref{eq:clcond}.
\epr

\section{Well-posedness of martingale problems}
\label{s.wpmp}

Our task in this section is to prove Propositions \ref{P.wp}--\ref{P.vecLambda},
i.e., we have to show that the martingale problem for the single-colony process, 
the McKean-Vlasov process, the multi-colony process and the hierarchically 
interacting Cannings process are all well-posed (= have a unique solution). The 
line of argument is the same for all. In Section~\ref{ss.mp}, we make some preparatory
observations. In Section~\ref{ss.mppropproofs}, we give the proofs.

\subsection{Preparation}
\label{ss.mp}

We first show that the duality relation and the characterisation of the dual
process via a martingale problem allow us to prove the {\em existence} of a
solution to the martingale problem that is strong Markov and has c\`adl\`ag
paths. To this end, observe that via the dual process we can specify a
distribution for every time $t$ and every initial state, since the dual is a
unique solution of its martingale problem (being a projective limit of a Markov
jump process defined for all times $t \geq 0$). Since the family $\{H(\cdot,Y_0) 
\colon\, Y_0 \in \CE^\prime\}$ (cf. \eqref{ag30_b}) separates  points, this uniquely defines a family of
transition kernels $(P_{t,s})_{t \geq s \geq 0}$  satisfying the Kolmogorov
equations, and hence defines uniquely a Markov process. By construction, this
Markov process solves the martingale problem, provided we can verify the
necessary path regularity.

We need to have c\`adl\`ag paths to obtain an admissible solution to the
martingale  problem. For finite geographic space this follows from the theory
of Feller  semigroups (see Ethier and Kurtz~\cite[Chapter 4]{EK86}). For
$\Omega_N$, we consider the exhausting sequence $(B_j(0))_{j\in\N_0}$ and use 
the standard tightness criteria for jump processes to obtain a weak limit point 
solving the martingale problem. The essential step is to control the effect
on a single component of the flow of individuals in and out of $B_j(0)$ in 
finite time as $j\to\infty$.

It is standard to get {\em uniqueness} of the solution from the existence of the
dual process (see, e.g., \cite[Section~1.6]{E00} or \cite[Proposition~4.4.7
and Theorem~4.4.11]{EK86}). Again, this works for all the choices of $G$ in
\eqref{Gchoices}, with a little extra effort when $G=\Omega_N$.

\subsection{Proofs of well-posedness}
\label{ss.mppropproofs}

In this section, we prove Propositions~\ref{P.wp}--\ref{P.vecLambda}. We follow 
the line of argument of Evans~\cite[Theorem~4.1]{E97} and derive existence and 
uniqueness of the spatial Cannings process from the existence of the corresponding 
spatial Cannings-coalescent established in Section~\ref{s.spatcoal}. The main tool 
is duality (cf.\ Proposition~\ref{prop:spatial-coalescent-duality} respectively \eqref{duality}). The proofs of 
Propositions~\ref{P.wp}--\ref{P.vecLambda} follow the same pattern for $G =
\{0,\ldots,N-1\}$, $G = \{0,*\}$ and $G=\Omega_N$.

\begin{proof}[Proof of Propositions~{\em\ref{P.wp}--\ref{P.vecLambda}}]
\mbox{ }\\
\medskip\noindent
$\bullet$ \emph{Well-posedness.} 
First we show that there exists a Markov transition kernel $Q_t$ on $\mathcal{P}(E)^G$ 
such that, for all $\varphi \in C_{\mathrm{b}}(E^{n},\R)$, $\pi \in \Pi_{G,n}$ (cf. \eqref{labelled-partitions}), $X 
\in \mathcal{P}(E)^G$ and $t \geq 0$,
\begin{equation}
\label{eq:forward-proccess-kernel-equation}
\int Q_t(X, \dd X^\prime) H_\varphi^{(n)}(X^\prime,\pi)
= \E\left[ H_\varphi^{(n)}(X,\mathfrak{C}^{(G)}_n(t)) 
\mid \mathfrak{C}^{(G)}_n(0) = {\pi} \right],
\end{equation}
where $H_\varphi^{(n)}$ as in \eqref{duality-test-functions} and
$\mathfrak{C}^{(G)}_n$ as in \eqref{add1} resp. \eqref{ag23b} depending on the
choice of $G$. Once \eqref{eq:forward-proccess-kernel-equation} is established,
the general theory of Markov processes implies the existence of a
\emph{Hunt-process} with the transition kernel $Q_t$ (see, e.g., Blumenthal and
Getoor~\cite[Theorem~I.9.4]{BG68}). This càdlàg process is unique and coincides
with the process $X^{(G)}$ from \eqref{ag7} resp. \eqref{Zdef} resp. from below
Proposition~\ref{P.vecLambda}, since \eqref{eq:forward-proccess-kernel-equation}
implies \eqref{duality}. There can be at most one process satisfying
\eqref{duality}, since the family of duality functions
$H_\varphi^{(n)}(\cdot,\pi)$ separates points on $\mathcal{P}(E)^G$.

Finally, the transition kernel $Q_t$ satisfying
\eqref{eq:forward-proccess-kernel-equation} exists as a solution of the
Hausdorff moment problem \eqref{eq:forward-proccess-kernel-equation} and is
Markov due to the Markov property of the spatial coalescent on the right-hand
side of \eqref{eq:forward-proccess-kernel-equation} (see~\cite[Theorem~4.1]{E97}
for details).

\medskip\noindent
$\bullet$ \emph{Feller property.} 
To show that $X^{(G)}$ is a Feller process we use duality. It is enough to show 
that, for any $F \in \CF$ an appropriate test-function and any $t \geq 0$, the map
\begin{equation}
\begin{aligned}
\label{ak103}
\CP(E)^G \ni x \mapsto \E \left[ F(X^{(G)}(t)) \mid X^{(G)}(0)=x \right] \in \R
\end{aligned}
\end{equation}
is continuous. In \eqref{ak103}, instead of the test functions $F(\cdot) \in
\CF$, it is enough to take the duality test functions $H^{(n)}_\varphi(\cdot ,
\pi_{G,n})$ from \eqref{duality-test-functions}. The duality in \eqref{duality} 
implies that
\begin{equation}
\label{ak104}
\begin{aligned}
\E\left[H^{(n)}_\varphi(X^{(G)}(t),\pi_{G,n}\vert_n)\mid X^{(G)}(0)=x\right]
= \E\left[H^{(n)}_\varphi(x,\mathfrak{C}^{(G)}(t)\vert_n)\right],
\quad
t \geq 0.
\end{aligned}
\end{equation}
Recall that we equip $\CP(E)^G$ with the topology of weak convergence.
Definition \eqref{duality-test-functions} readily implies that the right-hand
side of \eqref{ak104} is continuous in $x$.
\end{proof}

\section{Properties of the McKean-Vlasov process with immigration-emigration}
\label{s.mcv_imem}

The purpose of this section is to show that the $Z_\theta^{c,d,\Lambda}$-process
with immigration-emigration (cf.\ Section~\ref{sss.MV}) is ergodic (Section~\ref{sss.equil}), to identify its 
equilibrium distribution in terms of the dual (Section~\ref{ss.strucmkv}), and 
to calculate its first and second moment measure (Section~\ref{ss.1st2nd}). The
characterisation via the dual will allow us to also show that the equilibrium 
depends continuously on the migration parameter $\theta$ (Section~\ref{sss.conttheta}), 
a key property that will be needed later on and for which we need that the 
$\Lambda$-coalescent is {\em dust-free} (recall \eqref{ag2}).

\subsection{Equilibrium and ergodic theorem}
\label{sss.equil}

The equilibrium $\nu=\nu_\theta^{c,d,\Lambda} \in \CP(\CP(E))$ is the solution of
the equation
\be{cond-invariant}
\left\langle \nu,L_\theta^{c,d,\Lambda} F_\varphi \right\rangle = 0,
\quad
\varphi \in \CC_\mathrm{b}(E^n),\quad n \in \N,
\ee
where we recall (\ref{ak:multi-level-test-functions*}--\ref{Ldefs}) for
the form of $F_\varphi$ and $L_\theta^{c,d,\Lambda}$.

\begin{proposition}
\label{P.ergo}
{\bf [Ergodicity]}\\
For every initial state $Z^{c,d,\Lambda}_\theta (0) \in \CP(E)$,
\be{ag25}
\CL
\left[
Z^{c,d,\Lambda}_\theta (t)
\right]
\mathop{\Longrightarrow}_{t \to \infty}
\nu^{c,d,\Lambda}_\theta
\ee
and the right-hand side is the unique equilibrium of the process. The convergence 
holds uniformly in the initial state.
\end{proposition}

\begin{proof}
We use the  dual process, namely, the $\Lambda$-coalescent  with
immigration-emigration (see Section~\ref{sss.special-spatial-coalescents}), to
show that the expectation in the right-hand side of the duality relation
\eqref{duality} converges. Indeed, we showed in  (\ref{ag40}) in
Proposition~\ref{P.limstate} and its Corollary~\ref{C.convZ} that the state of
the duality function $H(X_0,\cdot)$, cf., \eqref{duality-test-functions},
applied to the dual process converges in law to a limiting random variable as $t
\to \infty$.  The duality function viewed  as a function of the first argument
generates a law-determining family $\{H(\cdot, C_0) \colon\, C_0 \in \CE^\prime\}$ ($\CE^\prime$ as below \eqref{ag30_b}) and
hence \eqref{ag40} proves convergence.

It remains to show that the limit is independent of the initial state. Indeed, this 
is implied by the fact that if we start with {\em finitely many} partition elements,
then all partition elements eventually jump to the cemetery location $\{\ast\}$ 
where all transition rates are zero  and the state is $\theta$. The latter implies 
that the limit is unique. Since $\CP(E)$ is compact and the process is Feller, there 
must exist an equilibrium, and this equilibrium must be equal to the $t \to \infty$ 
limit.
\end{proof}

\subsection{Continuity in the centre of the drift}
\label{sss.conttheta}

We want to prove that
\be{ag26}
\CP(E) \ni \theta \mapsto \nu^{c,d,\Lambda}_\theta \in \CP(\CP(E))
\ee
is uniformly continuous for suitably chosen metrics (in the weak topology 
on the respective metrisable spaces). We will choose the metrics in 
(\ref{ak-metric-p-e}--\ref{ak-metric-cp-cp-e}) below. Recall the definition of 
the duality functions $H$ from (\ref{duality-separating-family}--\ref{ag30_b}). 
Since the family $\{H(\,\cdot\,,C_0) \colon\, C_0 \in \CE^\prime\}$ is dense in
$C_\mathrm{b}(\CP(E),\R)$, we can approximate any function in
$C_\mathrm{b}(\CP(E),\R)$ by duality functions in the supremum norm. 
In fact, even the smaller family $\{H_\varphi(\,\cdot\,,\{ \{1\},\ldots,\{ n\}\}) 
\colon\, n \in \N, \varphi \in C_\mathrm{b}(E)\}$ is dense in $C_\mathrm{b}(\CP(E),\R)$. 
It is enough to prove uniform continuity for the duality function uniformly in the
family, even with the additional restriction $\Vert \varphi \Vert_\infty < 1$.
For this purpose, we analyse the limiting random variable for the corresponding
dual as a function of $\theta$ in the limit as $t \to\infty$.

If $(C^{c,\Lambda}_t)_{t \geq 0}$ denotes the spatial $\Lambda$-coalescent
with immigration-emigration starting from $\{(\{1\},0),\ldots,(\{ n\},0)\}$ and jumping 
to the cemetery state {$\{\ast\}$} at rate $c$, then $H(\theta,C^{c,\Lambda}_\infty)$ 
uniquely determines the McKean-Vlasov limit law $\nu^{c,d,\Lambda}_\theta$ for
$t \to \infty$. Recall that we associate the distribution of types $\theta$ with the 
cemetery state. It is clear that $C^{c,\Lambda}_\infty = \lim_{t \to \infty} C^{c,\Lambda}_t$ 
exists. The random variable $C^{c,\Lambda}_\infty$ has partition elements that are 
all located at the cemetery state.

Let
\be{ag63a}
P_{n,k} = \P \big\{|C^{c,\Lambda}_\infty| = k 
\mid C^{c,\Lambda}_0 =\{\{1\}, \ldots, \{n\}\} \big\}.
\ee
For all $\theta \in \CP(E)$ and all $\varphi \in C_\mathrm{b} (E)$ with $\Vert \varphi 
\Vert_\infty < 1$, taking $H_\varphi(\underline{\theta},(\{1\},\ldots,\{n\}\})=\langle\theta,
\varphi\rangle^n$ we have
\be{ag63b}
\E \left[H(\underline{\theta}, C^{c,\Lambda}_\infty) \mid C^{c,\Lambda}_0 =\{\{1\}, \ldots, \{n\}\} \right] 
= \suml^n_{k=1} P_{n,k} \langle \theta, \varphi \rangle ^k.
\ee
From the right-hand side of \eqref{ag63b}, we read off that the family of functions
\be{ak-uniformity-family-n}
\text{
$\left\{\E \left[H_\varphi(\underline{\theta}, C^{c,\Lambda}_\infty) \mid C^{c,\Lambda}_0 =\{\{1\}, \ldots, \{n\}\} \right] \colon\, n \in \N \right\}$
is uniformly continuous in $\theta$.}
\ee
On $\CP(E)$ we choose the metric
\be{ak-metric-p-e}
\rho_{\CP(E)}(\theta, \theta^\prime)
\equiv
\sum_{k\in\N} 2^{-k} |\langle \theta - \theta^\prime, \varphi_k \rangle|,
\quad \theta, \theta^\prime \in \CP(E),
\ee
where $\{ \varphi_k \in C_\mathbb{b}(E) \colon\, k \in \N \}$ with $ \sup_{k\in\N} 
\Vert \varphi_k \Vert_\infty < 1$ separates points and therefore generates the topology. 
On $\CP(\CP(E))$, we choose the metric
\be{ak-metric-cp-cp-e}
\rho_{\CP(\CP(E))}(X, X^\prime)
\equiv
\sum_{n\in\N} \sum_{k\in\N} 2^{-k-n} 
\E \left[ | H_{\varphi_k}( X-X^\prime, \{ \{ 1\}, \ldots, \{n\} \})| \right],
\quad X, X^\prime \in \CP(\CP(E)).
\ee
Combining (\ref{ak-uniformity-family-n}--\ref{ak-metric-cp-cp-e}), we get the 
uniform continuity of \eqref{ag26}.

\subsection{Structure of the McKean-Vlasov equilibrium}
\label{ss.strucmkv}

In the case of the McKean-Vlasov Fleming-Viot processes, the equilibrium
$\nu^{c,d,0}_\theta$ can be identified as an atomic measure of the form
\be{atm1}
\sum_{i\in\N} \Big[W_i \prod^{i-1}_{j=1} (1-W_j)\Big]\,\delta_{U_i}
\ee
with $(U_i)_{i \in \N}$ i.i.d.~$\theta$-distributed and $(W_i)_{i\in\N}$ i.i.d.\ 
$\mathrm{BETA}(1,\frac{c}{d})$-distributed, independently of each other (cf.\
\cite{DGV95}). What we can say about the equilibrium $\nu^{c,d,\Lambda}_\theta$?

\bp{P.mkv}
{\bf [Towards a representation for McKean-Vlasov equilibrium]}
Let $\nu^{c,d,\Lambda}_\theta$ be the equilibrium of the process $Z^{c,d,\Lambda}_\theta
= (Z^{c,d,\Lambda}_\theta(t))_{t \geq 0}$ with resampling constant $d$ and resampling 
measure $\Lambda \in \CM_f([0,1])$. Assume that $\Lambda$ has the dust-free property (recall
\eqref{ag2}).
\begin{itemize}
\item[\textup{(a)}]
The following decomposition holds:
\be{atm2}
\nu^{c,d,\Lambda}_\theta = 
\CL\left[\sum_{i\in\N} V_i \delta_{U_i}\right].
\ee
Here, $(V_i)_{i \in \N}$ and $(U_i)_{i \in \N}$ are independent sequences of random
variables taking values in $[0,1]$, respectively, $\CP(E)$. Moreover, $(U_i)_{i \in \N}$ 
is i.i.d.\ with distribution $\theta$, $\sum_{i\in\N} V_i=1$ a.s., and 
\be{atm3}
V_i = W_i \prod^{i-1}_{j=1} (1-W_j),
\ee
where
\be{atm4}
(W_j)_{j\in\N} 
\ee
is a sequence of $[0,1]$-valued random variables whose joint distribution is uniquely 
determined by the moment measures of $\nu^{c,d,\Lambda}_\theta$ (which can be 
expressed in terms of the dual coalescent process) and depends on $c,d$ and 
$\Lambda$. (See \textup{Remark~{\rm \ref{Wstructure}}} below.)

\item[\textup{(b)}] If $\theta \notin M=\{\delta_u \colon\, u \in E\}$ and $c,d>0$, then
\be{ag64}
0 \leq \nu^{c,d,\Lambda}_\theta(M) < 1.
\ee
\end{itemize}
\end{proposition}

\begin{proof}
$\mbox{}$

\medskip\noindent
(a) The distribution and the independence of $(U_i)_{i\in \N}$ follow from the 
representation of the state at time $t \in [0,\infty]$ in terms of the entrance 
law of the $\Lambda$-coalescent starting from the partition into singletons: 
$\{ \{1\},\{2\},\ldots \}$. This representation is a consequence of the duality relation 
in (\ref{duality}) and de Finetti's theorem, together with the dust-free condition on 
$\Lambda$ in (\ref{ag2}), which guarantees the existence of the frequencies of the 
partition elements at time $t$. Indeed, every state, including the equilibrium state, can 
be written as the limit of the empirical distribution of the coalescent entrance law 
starting from the partition $\{\{1\},\{2\},\ldots \}$ at site 1, where we assign 
to each dual individual the type of its partition element at time $\infty$, drawn 
independently from $\theta$, the cemetery state. Here, we use the fact that if 
we condition individuals not to coalesce with a given individual, respectively, 
its subsequent partition element, then the process is again a coalescent for the 
smaller (random) subpopulation without that individual, respectively, its subsequent 
partition element. 

The $(V_j)_{j \in \N}$ are the relative frequencies of the partition elements 
ordered according to their smallest element. By construction, $(V_i)_{i \in \N}$ 
and $(U_i)_{i \in \N}$ are independent. 

In principle, via the duality we can express the moments in equilibrium 
\be{momdualquan}
\E_{\nu_\theta^{c,d,\Lambda}} [\langle X,f\rangle^n]
\ee 
in terms of $\langle \theta,f\rangle^k$, $k=1,\dots,n$, and the coalescence probabilities 
before the migration jumps into the cemetery state. The latter in turn can be calculated 
in terms of 
\be{atm6}
c,\,d,\,r^k(1-r)^{n-k} \Lambda (\dd r).
\ee
These relations uniquely determine the distribution of the atom sizes, which in turn 
uniquely determines the marginal distribution of the $W_i$'s via (\ref{atm3}).

\medskip\noindent 
(b) First consider the case $\Lambda=\delta_0$. Let us verify that, for $c>0$ and 
$\theta \notin M$, there can be no mass in $M$. Indeed, if there would be an atom
somewhere in $M$, then there would also be an atom in $M$ after we merge types 
into a finite type set. However, in the latter situation the $W_i$'s are 
$\mathrm{BETA}$-distributed, hence do not have an atom at 0 or 1, and so also 
the law of the $V_i$'s has no atom at 0 or 1. This immediately gives the claim, 
because it means that $\nu^{c,d,\Lambda}_\theta (M)=0$.

Next, consider the case $\Lambda\neq\delta_0$. Then new types keep on coming in. We 
need to prove that the event that $\mathfrak{C}^{(\{0,\ast\})}_\infty$ (the limit of the dual 
coalescent) contains more than one partition element has a positive probability. But this 
is obviously true when $c,d>0$.
\end{proof}

\begin{remark}
\label{Wstructure} It is well known (cf.~\textup{\cite{DGV95}}) that if
$\Lambda=\delta_0$ (the McKean-Vlasov Fleming-Viot process), then the $W_i$'s
are i.i.d.\ with distribution $\mathrm{BETA}(1,\frac{c}{d})$. It remains an open
problem to identify the law of the $W_i$'s for the general Cannings resampling
as function of the ingredients in \eqref{atm6}. We note that if the $W_i$'s
happen to be independent, then $W_i$ has distribution $\mathrm{BETA}(1 - \alpha,
i \alpha + \beta)$ for some $\alpha \in [0,1]$ and $\beta \in [0,\infty)$ (see
\textup{\cite[Theorem~3.4]{P06}}). 
\er

\subsection{First and second moment measure}
\label{ss.1st2nd}

We can identify the first and second moments of the equilibrium explicitly, and 
we can use the outcome to calculate the variance of $M^{(j)}_k$ for $k=0,\ldots,j$,
the interaction chain defined in Section~\ref{sss.extmsp}. Recall the definition 
of $\E_{\nu_\theta}[\var_{\cdot}(\psi)]$ from \eqref{vardef1} and of $\var_x(\psi)$ 
from \eqref{vardef2}. Recall $\lambda = \Lambda([0;1])$.

\bp{prop:2nd}
{\bf [Variance]}
For every $\psi\in \CC_\mathrm{b}(E)$,
\be{varevol}
\E_{\nu_\theta^{c,d,\Lambda}}\,\left[\var_\cdot\left(\psi\right)\right]
= \int_{\CP(E)} \nu_\theta^{c,d,\Lambda}(\dd x) 
\left( \left\langle\psi^2,x\right\rangle
- \left\langle\psi,x\right\rangle^2 \right)
= \frac{2c}{2c+\lambda+2d} \var_\theta\left(\psi\right).
\ee
\ep

\begin{proof}
We calculate the expectation of $\langle\varphi,x\rangle$, $\varphi\in\CC_\mathrm{b}(E)$,
and $\langle\varphi,x^{\otimes 2}\rangle$, $\varphi\in\CC_\mathrm{b}(E^2)$, in equilibrium.

It follows from \eqref{cond-invariant} with $\nu=\nu_\theta^{c,d,\Lambda}$ that
\be{first-moment-invariant-measure}
n=1,\,\varphi\in\CC_\mathrm{b}(E)\colon
\qquad 0 = c \int_{\CP(E)} \nu(\dd x) \left\langle \varphi,(\theta-x)\right\rangle,
\ee
i.e., $\int_{\CP(E)} \nu(\dd x) \left\langle \varphi,x \right\rangle = \left\langle\varphi,
\theta\right\rangle$. It further follows that, for $n=2,\,\varphi\in\CC_\mathrm{b}(E^2)$,
\be{sk:preparation-second-moment-invariant-measure-1}
\begin{aligned}
0 & =  -2c \int_{\CP(E)} \nu(\dd x) \left\langle\varphi,x^{\otimes 2}\right\rangle\\
&\qquad + c \int_{\CP(E)} \nu(\dd x)\left[\left\langle\varphi,\theta\otimes x\right\rangle
+ \left\langle\varphi,x\otimes\theta\right\rangle\right]
\\
&\qquad + 2d \int_{\CP(E)} \nu(\dd x) \left( \int_E x(\dd a) \left\langle \varphi ,
\delta_a^{\otimes 2} \right\rangle - \left\langle \varphi,
x^{\otimes 2} \right\rangle \right) 
\\
&\qquad + \lambda \int_{\CP(E)} \nu(\dd x)
\int_E x(\dd a) \left\langle\varphi,\left(\delta_a-x\right)^{\otimes 2}\right\rangle.
\end{aligned}
\ee
We can rewrite \eqref{sk:preparation-second-moment-invariant-measure-1} as \be{sk:preparation-second-moment-invariant-measure-1b}
\begin{aligned}
&\int_{\CP(E)} \nu(\dd x)
\int_E x(\dd a) \left\langle\varphi,\left(\delta_a-x\right)^{\otimes 2}\right\rangle \\
&\qquad = \int_{\CP(E)} \nu(\dd x)
\left( \int_E x(\dd a) \left\langle\varphi,\delta_a^{\otimes 2}\right\rangle
- \left\langle \varphi,x^{\otimes 2}\right\rangle\right) \\
&\qquad = \frac{c}{\lambda+2d} \left( 2\int_{\CP(E)} \nu(\dd x)
\left\langle\varphi,x^{\otimes 2}\right\rangle - \int_{\CP(E)} \nu(\dd x)
\left[\left\langle\varphi,\theta\otimes x\right\rangle
+ \left\langle\varphi,x\otimes\theta\right\rangle\right] \right).
\end{aligned}
\ee
From this, we see that
\be{sk:preparation-second-moment-invariant-measure-1c}
\begin{aligned}
&\int_{\CP(E)} \nu(\dd x) \left\langle \varphi,x^{\otimes 2}\right\rangle \\
&\qquad = \frac{\lambda+2d}{2c+\lambda+2d} \left( \frac{c}{\lambda+2d}
\int_{\CP(E)} \nu(\dd x)\left[\left\langle\varphi,\theta\otimes x\right\rangle\right.\right.\\
&\qquad \qquad\qquad\qquad\qquad\qquad
+ \left.\left.\left\langle\varphi,x\otimes\theta\right\rangle\right]
+ \int_{\CP(E)} \nu(\dd x) \int_E x(\dd a)
\left\langle\varphi,\delta_a^{\otimes 2}\right\rangle \right) \\
&\qquad = \frac{\lambda+2d}{2c+\lambda+2d} \left( \frac{2c}{\lambda+2d}
\left\langle\varphi,\theta^{\otimes 2} \right\rangle + \int_E \theta(\dd a)
\left\langle\varphi,\delta_a^{\otimes 2}\right\rangle \right),
\end{aligned}
\ee
where we use \eqref{first-moment-invariant-measure} in the last line. Substituting this 
back into \eqref{sk:preparation-second-moment-invariant-measure-1b} and using
\eqref{first-moment-invariant-measure} once more, we get
\be{sk:preparation-second-moment-invariant-measure-1d}
\begin{aligned}
& \int_{\CP(E)} \nu(\dd x)
\int_E \int_E Q_x(\dd u,\dd v)\,\varphi(u,v) \\
&\qquad = \int_{\CP(E)} \nu(\dd x)
\left( \int_E x(\dd a) \left\langle\varphi,\delta_a^{\otimes 2}\right\rangle
- \left\langle \varphi,x^{\otimes 2}\right\rangle\right) \\
&\qquad = \frac{2c}{\lambda+2d} \left( \int_{\CP(E)} \nu(\dd x)
\left\langle\varphi,x^{\otimes 2}\right\rangle
- \left\langle\varphi,\theta^{\otimes 2} \right\rangle \right) \\
&\qquad = \frac{2c}{2c+\lambda+2d} \left( \int_E \theta(\dd a)
\left\langle\varphi,\delta_a^{\otimes 2}\right\rangle
- \left\langle\varphi,\theta^{\otimes 2} \right\rangle \right) \\
&\qquad = \frac{2c}{2c+\lambda+2d} \int_E \int_E Q_\theta(\dd u,\dd v)\,\varphi(u,v).
\end{aligned}
\ee
Pick $\varphi=\psi\times\psi$ in \eqref{sk:preparation-second-moment-invariant-measure-1d}
to get the claim.
\end{proof}

For $\lambda=\Lambda([0,1])=0$, \eqref{varevol} is the same as Dawson, Greven and
Vaillancourt~\cite[Eq.\ (2.5)]{DGV95}.

\begin{corollary}
\label{C.asyvar}
{\bf [Asymptotic variance of entrance law]}\\
For $\varphi \in C_\mathrm{b}(E,\R)$, the interaction chain (cf., Section~\ref{sss.extmsp}) 
satisfies
\be{ag54}
\begin{aligned}
\lim_{j\to \infty} \E_{\CL(M^{(j)}_0)}[\var_\cdot(\varphi)] = 0
\quad
(\text{respectively, }>0),
\end{aligned}
\ee
if $\sum_{k\in\N} m_k=\infty$ (respectively, $\sum_{k\in\N} m_k <\infty$) with $m_k$ 
defined in \eqref{ag46} and $d_k$ in \eqref{diffusion-constants}.
\end{corollary}

\begin{proof}
From \eqref{varevol}, we have the formula
\be{ag55}
\E_{\nu_\theta^{c,d,\Lambda}}[\var_\cdot (\varphi)] 
= \frac{2c}{2c+\lambda+2d} \var_\theta(\varphi).
\ee
Hence, we have the relation (recall \eqref{ag14} for the definition of 
$K_k(\theta, \dd x)$)
\be{ag56}
\int_{\CP(E)} K_k(\theta, \dd x)
\var_x(\varphi) =  \frac{2c_k}{2c_k +\lambda_k +2d_k} \var_\theta(\varphi),
\ee
which says that in one step of the interaction chain the variance is modified by the
factor
\be{ag57}
n_k \equiv \frac{2c_k}{2c_k +\lambda_k + 2d_k} = \frac{1}{1+m_k}.
\ee
Iteration gives
\be{ag58}
\E_{\CL(M^{(j)}_0)}\left[\var_\cdot(\varphi)\right]
= \left(\prod^j_{k=0} n_k\right) \var_\theta(\varphi)
= \left(\prod^j_{k=0} \left(\frac{1}{1+m_k}\right)\right)\var_\theta (\varphi).
\ee
Therefore, taking logarithms, we see that \eqref{ag54} is equivalent to
\be{ag59}
\sum_{k\in\N_0} m_k = \infty \text{ (respectively, } < \infty).
\ee
\end{proof}

We next prove a result that is similar to, but more involved than, \cite{DGV95}, Eq.\ (6.12). 
This result is necessary for the proof of Theorem~\ref{T.dicho2} on diffusive clustering. 

\bp{prop:2nd_int}
{\bf [Variance of the integral against a test function]}
For every $\psi\in \CC_\mathrm{b}(E)$, $j \in \N$ and $0 \leq k \leq j+1$,
\be{varevol_int}
\begin{aligned}
\var_{\CL(M^{(j)}_{-k})}\,\left( \langle \cdot,\psi \rangle \right)
&= \E_{\CL(M^{(j)}_{-k})}[ \langle \cdot,\psi \rangle^2] 
- \left( \E_{\CL(M^{(j)}_{-k})}[ \langle \cdot,\psi \rangle ] \right)^2 \\
&= \left( \sum_{i=k}^j \left( \frac{d_{i+1}}{c_i} 
\prod_{l=i+1}^j \frac{1}{1+m_l} \right) \right) \var_\theta\left(\psi\right).
\end{aligned}
\ee
\ep

\begin{proof}
The proof uses the following two ingredients. Combining \eqref{varevol} and \eqref{ag57}, 
we have
\be{sk101}
\E_{\nu_\theta^{c_k,d_k,\Lambda_k}}[\var_\cdot (\psi)] 
= \frac{1}{1+m_k} \var_\theta(\psi).
\ee
The first and the third line of \eqref{sk:preparation-second-moment-invariant-measure-1d} 
yield
\be{sk102}
\var_{\nu_\theta^{c_k,d_k,\Lambda_k}}(\langle \cdot,\psi\rangle ) 
= \frac{\lambda+2d}{2c}\,\E_{\nu_\theta^{c_k,d_k,\Lambda_k}}[\var_\cdot(\psi)].
\ee
Together with \eqref{varevol} and \eqref{diffusion-constants}, we therefore obtain
\be{sk103}
\var_{\nu_\theta^{c_k,d_k,\Lambda_k}}(\langle \cdot,\psi\rangle) 
= \frac{\lambda_k+2d_k}{2c_k+\lambda_k+2d_k}\, \var_\theta(\psi) 
= \frac{d_{k+1}}{c_k}\, \var_\theta(\psi).
\ee
Fix $j \in \N$. The proof follows by downward induction over $0 \leq k \leq j+1$. The 
initial case $k=j+1$ is obvious because $M^{(j)}_{-(j+1)} = \theta$ by \eqref{ag61}. 
Let us therefore assume that the claim holds for $k+1$. By (\ref{ag61}--\ref{ag14}),
\be{sk104}
\begin{aligned}
& \var_{\CL(M^{(j)}_{-k})}\,\left( \langle \cdot,\psi \rangle \right)
= \E_{\CL(M^{(j)}_{-k})}[ \langle \cdot,\psi \rangle^2] 
- \left( \E_{\CL(M^{(j)}_{-k})}[ \langle \cdot,\psi \rangle ] \right)^2 \\
&= \int_{\CP(E)} \nu_\theta^{c_j,d_j,\Lambda_j}(\dd \theta_j) 
\int_{\CP(E)} \nu_{\theta_j}^{c_{j-1},d_{j-1},\Lambda_{j-1}} 
(\dd \theta_{j-1}) \ldots \int_{\CP(E)} \nu_{\theta_{k+1}}^{c_k,d_k,\Lambda_k} 
(\dd \theta_k) \langle \theta_k,\psi \rangle^2 - \langle \theta,\psi \rangle^2.
\end{aligned}
\ee
Next, use \eqref{sk103} to rewrite the inside integral as
\be{sk105}
\int_{\CP(E)} \nu_{\theta_{k+1}}^{c_k,d_k,\Lambda_k} 
(\dd \theta_k) \langle \theta_k,\psi \rangle^2 
= \E_{\nu_{\theta_{k+1}}^{c_k,d_k,\Lambda_k}}( \langle \cdot,\psi \rangle^2 )
= \langle \theta_{k+1},\psi \rangle^2 + \frac{d_{k+1}}{c_k} \var_{\theta_{k+1}}(\psi).
\ee
Substitute this back into \eqref{sk104}, to obtain
\be{sk106}
\begin{aligned}
&\var_{\CL(M^{(j)}_{-k})}\,\left( \langle \cdot,\psi \rangle \right)
= \var_{\CL(M^{(j)}_{-(k+1)})}\,\left( \langle \cdot,\psi \rangle \right) \\
&\quad + \frac{d_{k+1}}{c_k} \int_{\CP(E)} \nu_\theta^{c_j,d_j,\Lambda_j}
(\dd \theta_j) \int_{\CP(E)} \nu_{\theta_j}^{c_{j-1},d_{j-1},\Lambda_{j-1}} 
(\dd \theta_{j-1})\\
&\qquad\qquad\qquad \cdots 
\int_{\CP(E)} \nu_{\theta_{k+2}}^{c_{k+1},d_{k+1},\Lambda_{k+1}} 
(\dd \theta_{k+1}) \var_{\theta_{k+1}}(\psi).
\end{aligned}
\ee
The first term is given by the induction hypothesis. For the second term we 
use \eqref{sk101}, to see that the inside integral equals
\be{sk107}
\int_{\CP(E)} \nu_{\theta_{k+2}}^{c_{k+1},d_{k+1},\Lambda_{k+1}} 
(\dd \theta_{k+1}) \var_{\theta_{k+1}}(\psi)
= \E_{\nu_{\theta_{k+2}}^{c_{k+1},d_{k+1},\Lambda_{k+1}}}( \var_\cdot(\psi) )
= \frac{1}{1+m_{k+1}} \var_{\theta_{k+2}}(\psi).
\ee
Iteration of this reasoning for the second term in \eqref{sk106} leads to 
\be{sk108}
\begin{aligned}
\var_{\CL(M^{(j)}_{-k})}\,\left( \langle \cdot,\psi \rangle \right)
&= \var_{\CL(M^{(j)}_{-(k+1)})}\,\left( \langle \cdot,\psi \rangle \right) 
+ \frac{d_{k+1}}{c_k} \prod_{l=k+1}^j \frac{1}{1+m_l} \var_\theta(\psi) \\
&= \left( \sum_{i=k+1}^j \left( \frac{d_{i+1}}{c_i} 
\prod_{l=i+1}^j \frac{1}{1+m_l} \right) \right) \var_\theta\left(\psi\right) 
+ \frac{d_{k+1}}{c_k} \prod_{l=k+1}^j \frac{1}{1+m_l} \var_\theta(\psi),
\end{aligned}
\ee
which proves the claim.
\end{proof}

If $\lambda_k=\Lambda_k([0,1])=0$, $k \in \N_0$, then \eqref{varevol_int} reduces 
to \cite[Eq.\ (6.12)]{DGV95}. Indeed, in that case $d_{i+1} \prod_{l=i+1}^j \frac{1}{1+m_l}$ 
is equal to $d_{i+1}$. (Note the typo in \cite[Eq.\ (6.12)]{DGV95}: $d_k$ should be 
replaced by $d_{k+1}$.) 

\begin{remark}
The results in this section can alternatively be inferred from the long-time behaviour of 
the spatial $\Lambda$-coalescent with $G = \{0,*\}$.
\end{remark}

\section{Strategy of the proof of the main scaling theorem}
\label{s.stratth}

The proof of Theorem~\ref{mainth} will be carried out in
Sections~\ref{s.mflambda}--\ref{s.proof1stth}. In this section we explain 
the main line of the argument.

\subsection{General scheme and three main steps}
\label{ss.gen3}

In Dawson, Greven and Vaillancourt~\cite{DGV95}, a general scheme was developed
to derive the scaling behaviour of space-time block averages as in \eqref{a3} for hierarchically
interacting Fleming-Viot processes, with the interaction coming from migration, 
i.e., a system similar to ours but \emph{without} $\Lambda$-Cannings block 
resampling (so for $\Lambda =\delta_0$, which results in diffusion processes 
rather than jump processes). Nevertheless, this scheme is widely applicable 
and indicates what estimates have to be established in a concrete model (with 
methods that may be specific to that model).

For our model, the difficulty sits in the fact that {\em diffusions} are replaced by 
{\em jump processes}, even in the many-individuals-per-site limit. Below we 
explain how we can use the special properties of the dual process derived in 
Section~\ref{s.spatcoal} to deal with this difficulty. In 
Sections~\ref{s.mflambda}--\ref{s.proof1stth} the various steps will be carried 
out in detail to prove our scaling result in Theorem~\ref{mainth}. In these 
sections, we focus on the new features coming from the $\Lambda$-Cannings 
block resampling. The refined multi-scale result in Theorem~\ref{Tscb} will be 
proved in Section~\ref{s.msana}. The line of argument can be largely based 
on the work in \cite[Section 4]{DGV95}, where it was developed in detail for 
Fleming-Viot. No new ideas are needed for the Cannings process: only a new 
moment calculation is required.

The analysis in Sections~\ref{s.mflambda}-\ref{s.proof1stth} proceeds in three main 
steps:
\begin{itemize}
\item 
Show that for the mean-field system from Section~\ref{sss.mcolonycannings}, i.e., $G=G_{N,1}=\{0,1,\ldots,N-1\}$, in the
limit as $N\to\infty$ we obtain for single sites on time scale $t$ independent 
McKean-Vlasov processes (recall Section~\ref{sss.MV}), and for block averages on time scale $Nt$ Fleming-Viot 
processes with a resampling constant $d_1$ corresponding to $\Lambda_0$ and $c_0$.
With an additional $\Lambda_1$-block resampling at rate $N^{-2}$ there is no effect 
on time scale $t$, and so on time scale $Nt$ we obtain a $C^{\wt \Lambda}$-process 
with $\wt \Lambda =d_1 \delta_0 + \Lambda_1$. This is done in Section~\ref{s.mflambda}.
\item 
Consider the $C^{\uc,\uL}_N$-process from Section~\ref{sss.hierarCan} restricted to $G_{N,K}$ as in \eqref{ag17}. More precisely, 
study its components and its $k$-block averages \eqref{k-block-average} for $1 \leq k \leq j<K$ on 
time scales $ N^j + t N^k$. This is done in Section~\ref{s.hiermod}.
\item 
Treat the $(j,k)$ renormalised systems for $1\leq k\leq j<K$ via an approximation of 
the $C^{\uc,\uL}_N$-process on $\Omega_N$ by the process on $G_{N,K}$ from 
the previous step, in the limit as $N \to \infty$ and on time scales at most $N^K t$ 
for a fixed but otherwise arbitrary $K \in \N$. This is done in Section~\ref{s.proof1stth}.
\end{itemize}

The three steps above are carried out following the scheme of proof developed 
in \cite{DGV95}. What is new for {\em jump} processes? We are dealing with 
sequences of measure-valued processes $X=(X_t)_{t \geq 0}$, and the key 
difference is that now semi-martingales arising from functionals of the process 
of the form $\langle X_t,f\rangle^n$ with $f \in C_\mathrm{b}(E)$ are no longer 
controlled just by the {\em compensator} and the {\em increasing process} of the 
linear functional $\langle X_t,f \rangle$. This is different from the case of diffusions, 
where linear and quadratic functions $\langle X_t, f\rangle$ and $\langle X_t,f \rangle^2$ 
in a set $\CF$ of test-functions suffice to establish both {\em tightness in path space} and {\em convergence 
of finite-dimensional distributions (f.d.d.s)}. 

The new ingredients are the analysis of the linear operators of the martingale 
problem acting on all of $\CF$, and the extension of the tightness arguments 
necessary to handle the jumps. We explain the basic structure of the argument 
in the next section.

\subsection{Convergence criteria}
\label{ss.embedding}

In the proofs, we view the process with $G=\{0,1,\ldots,N-1\}$, 
$G=G_{N,K}=\{0,1,\ldots,N-1\}^K$ and $G=\Omega_N$ (cf. \eqref{ag30_a}) as embedded in the
process with $G=\N$, $G=\N^K$ and $G=\Omega_\infty$, where
\be{ak4}
\Omega_\infty =\bigcup\limits_{M\in \N} \Omega_M \subseteq \N^\N.
\ee
Note that $\Omega_\infty$ is {\em countable}, but that the $\Omega_M$'s are 
{\em not} subgroups of $\Omega_\infty$. The embedding requires us to embed the 
test functions and the generators on $\Omega_M$ into those on $\Omega_\infty$.
In the calculations in Sections~\ref{s.mflambda}--\ref{s.proof1stth}, we use 
this embedding without writing it out formally.

The claims we have to prove require us to show that certain sequences of probability 
measures $(P_n)_{n \in \N}$ on $D([0,\infty),E)$ converge to a specified limit $P$. 
Therefore we have to show
\begin{itemize}
\item
tightness on $D([0,\infty),E)$,
\item
convergence of the f.d.d.'s to the ones of the claimed limit.
\end{itemize}
What we will use to establish tightness (and later also f.d.d.-convergence) is that the 
$P_n$'s and $P$ are solutions to martingale problems for measure-valued processes.
We write $X^{(N)},X$ to denote realisations of these processes.

The states of our processes are {\em probability} measures on the type space (recall \eqref{k-block-average} and \eqref{macroscopic-behaviour}).
We use Jakubowski's criterion for measure-valued processes (see~\cite[Theorem 3.6.4]{D93}). 
This requires us to prove: (1) a compact containment condition for the path, i.e., for all
$\epsilon,T>0$ there exists a $K_{T,\epsilon}$ compact such that
\be{ag43-0}
\P(\{X^{(N)}(t) \in K_{T,\epsilon} \mbox{ for all } t \in [0,T]\}) \geq 1-\ve;
\ee
(2) tightness of evaluation processes $(F(X^{(N)}(t)))_{t \geq 0}$ in path space
for all $F \in \CD$, with $\CD$ a dense subspace of continuous functions on 
type space. We will use for $\CD$ the set
\be{ag80}	
\CD = \{\langle X,f \rangle^n | f \in C_b (E, \R), \quad n \in \N\} \subseteq C_b (\CP(E),\R).
\ee
	 
In our setting, the compact containment condition in (1) is immediate, because we have a 
compact type space and the probability measures on it form a compact set in the weak 
topology. Condition (2) can be verified by using a criterion for tightness by Kurtz (see 
Dawson~\cite[Corollary 3.6.3]{D93}). (Alternatively, we could use a tightness criterion 
by Joffe-M\'etivier \cite[Theorem 3.6.6 and Corollary 3.6.7]{D93}.) In particular, we get 
that (2) follows from
\be{ag51}
\sup\limits_{N \in \N} \|L^{(N)} F\|_\infty < \infty, \quad \forall \; F \in \CD.
\ee
Thus, to conclude tightness, we have to calculate $L^{(N)} F$, for $F \in \CD$, and 
bound it in the supremum norm.

In order to show f.d.d.-convergence of $X^{(N)}$ to the claimed limit $X$, we use that these 
measure-valued processes arise as the solution to the $(L^{(N)}, \CD,\delta_{X^N_0})$-martingale
problem, respectively, the $(L,\CD, \delta_{X_0})$-martingale problem, where the latter is well-posed.
It then suffices to show that, for a dense subset $\mathcal{A}$ of $C_\mathrm{b}((\CP(E))^\N,\R)$ and all all $F \in \mathcal{A}$,  
the compensator terms satisfy:
\be{ak4b}
\CL\left[\left(\int^t_0 L^{(G_N)} F((X^N_s))\dd s\right)_{t \geq 0}\right]
\mathop{\Longrightarrow}_{N \to \infty}
\CL\left[\left(\int^t_0 (L^{(G)}F)(X_s)\dd s\right)_{t \geq 0}\right]
\ee
and the initial laws satisfy
\be{ag50a}
\CL[X^{(N)}_0] 
\mathop{\Longrightarrow}_{N \to \infty}
\CL[X_0].
\ee
This allows us to conclude that $X^{(N)}$ converges in f.d.d.\ to $X$, so that we get (2).

Thus, to prove the convergence as claimed, we have to verify (\ref{ag51}) and
(\ref{ak4b}) for each of the three processes mentioned in Section \ref{ss.gen3}.
For the proof of (\ref{ak4b}), it is necessary to use the {\em duality
relation}, in order to establish certain properties of the  process $X^{(N)}$ in
the limit as $N \to \infty$ that allow us to draw more information from  the
generator calculation. This includes a proof that certain higher-order terms can
be  discounted, or an argument that establishes independence over sufficiently
large distances. 

The averaging arguments we will use in the following sections are close in
spirit to those in~\cite{K92}. In our case, however, the latter work does not
apply immediately, in particular, because we deal with $N$-dependent state
space.

In summary, the role of Sections~\ref{s.mflambda}--\ref{s.proof1stth} is to first carry out 
some generator calculations, leading to the bound in (\ref{ag51}), and then an asymptotic 
evaluation of the resulting generator expressions, leading to a limiting form that uniquely 
determines the limiting process in (\ref{ak4b}). The latter will be based on a direct calculation. 
In view of the large time scales involved, we can use an {\em averaging principle} for local 
variables, based on the {\em local equilibria} dictated by the macroscopic slowly changing 
variables. The properties of the limiting process are established in Section \ref{s.mcv_imem}.

\section{The mean-field limit of $C^\Lambda$-processes}
\label{s.mflambda}

This section deals with the case $G=\{0,1,\ldots,N-1\}$ for a model that includes 
mean-field migration and Cannings reproduction at rate 1 with resampling measure 
$\Lambda_0$ in single colonies (cf. Section~\ref{sss.mcolonycannings}). We analyse the single components and the block 
averages on time scales $t$, $Nt$ and $Nt+u$ with $u\in\R$. The key results are 
formulated in Propositions~\ref{P.mkvl} and \ref{P.mffsys} below. We will see that 
we can also incorporate block resampling at rate $N^{-2} \Lambda_1$ and still get 
the same results.

The analysis for mean-field interacting Fleming-Viot processes with drift is given 
in detail in \cite[Section~4]{DGV95}. The reader unfamiliar with the arguments 
involved is referred to this paper (see, in particular, the outline of the abstract 
scheme in \cite[Section 4(b)(i), pp.~2314--2315]{DGV95}). In what follows, we provide 
the main ideas again, and focus on the changes arising from the replacement of the 
Fleming-Viot process by the $\Lambda$-Cannings resampling process, i.e., the change 
from continuous to càdlàg semi-martingales.

We always start the process in a product state with law $\chi^{\otimes N}$ with $\chi\in\CP(\CP(E))$ 
satisfying
\be{add2_b}
\int_{\CP(E)} x \chi( \dd x) = \theta \in \CP(E).
\ee
The system will be analysed in the limit as $N\to\infty$ in two steps: (1) component-wise 
on time scale $t$ (Section~\ref{ss.singcol}); (2) block-wise on time scale $N t$ and 
component-wise on time scale $N t+u$ with $u\in\R$ (Section~\ref{ss.mffinite}).

\subsection{Propagation of chaos: Single colonies and the McKean-Vlasov process}
\label{ss.singcol}

In this section, we consider the $C^\Lambda$-mean-field model from
Section~\ref{sss.mcolonycannings} with $G=\{0,1,\ldots,N-1\}$. We prove
propagation of chaos for the collection 
\be{ak:201402081423}
(\{X_0^{(N)}(t),
\ldots,X_{N-1}^{(N)}(t)\})_{t\geq 0}
\ee 
in the limit as $N\to\infty$, i.e., we
prove asymptotic independence of the components via duality as well as
component-wise convergence to the McKean-Vlasov process with parameters
$d_0=0,c_0,\Lambda_0,\theta$ (cf.\ \eqref{Ldefs}).

\begin{proposition}
\label{P.mkvl}
{\bf [McKean-Vlasov limit, propagation of chaos]}\\
Under assumption \eqref{add2_b}, for any $L \in \N$ fixed,
\be{ag43}
\CL\left[(X^{(N)}_0 (t), \ldots, X^{(N)}_L (t))_{t \geq 0}\right]
\mathop{\Longrightarrow}_{N \to \infty} 
\bigotimes_{i=0}^{L} \CL\left[Z^{c_0,d_0,\Lambda_0}_{i,\theta} \right],
\ee
where $Z^{c_0,d_0,\Lambda_0}_{i,\theta}$ solves the martingale problem for
$(L_\theta^{c_0,d_0,\Lambda_0},\mathcal{F}, \chi)$.
\ep

\begin{corollary}
\label{Cmkvbs}
{\bf [McKean-Vlasov limit with block resampling]}\\
Consider the system above with an additional rate $N^{-2}\Lambda_1$ of block 
resampling per site. Then \eqref{ag43} continues to hold.
\end{corollary}

In order to prove (\ref{ag43}), we will argue that the laws
$\CL[(\{X^{(N)}_\xi(t), \xi=0,\ldots,L\})_{t \geq 0}]$, $N \in \N$, are tight.
We show this first for components (Section \ref{sss.tightN}). Then, we verify
asymptotic independence (Section~\ref{sss.asyind}), calculate explicitly the
action of the generator on the test functions in the martingale problem of
$X^{(N)}$ (Section~\ref{sss.genconv}), and show, for functions depending on one
component, uniform convergence to the generator of the McKean-Vlasov operator
with parameter  $\theta=\E[X^{(N)}_0(0)]$
(Section~\ref{sss.single-mckean-vlasov}).

\subsubsection{Tightness on path space in $N$}
\label{sss.tightN}

Since we have a state in $(\CP(E))^\N$ equipped with the product topology, it suffices to
establish tightness for $L$-tuples of components.
We focus first on one component $(X_\xi(t))_{t \geq 0}$ and conclude later the result for
tuples of $L$-components.

Here, we use test functions as in 
\eqref{ak:test-functions} that only depend on the first $L$ coordinates. We further 
make use of the boundedness of the characteristics of the generator as a function of 
$N$ when acting on a test function (recall \eqref{ag11}, \eqref{ak:migration-operator} and \eqref{ak:one-level-generator}). 
Namely, we will see in Section~\ref{sss.genconv} (in (\ref{sk_15c_1}), (\ref{sk_15c_3}) 
and (\ref{sk_15c_4}) below) that the generator $L^{(N)}F$ satisfies
\be{ag43b}
\sup_{N \in \N} \|L^{(N)} F\|_\infty < \infty, \quad \text{for all } F \in C^2_b(\CP(E), \R).
\ee
As we outlined in Section \ref{ss.embedding}, this guarantees tightness.

\subsubsection{Asymptotic independence}
\label{sss.asyind}

In this section, we use duality to prove the factorisation of spatial mixed moments
(including the case with non-local coalescence at rate $N^{-2}\Lambda_1$). Namely, 
we show that for any $L \in \N$, any $k_\xi \in \N$, $\xi \in [L]$,
\be{ag44}
\limsup\limits_{N \to \infty} \left|\E \left[\prodl^L_{\xi=0} 
\left(\langle X^{(N)}_\xi (t), f_\xi\rangle\right)^{k_\xi}\right]
- \prodl^L_{\xi=0} \E\left[\left(\langle X^{(N)}_\xi (t), 
f_\xi\rangle\right)^{k_\xi}\right]\right| = 0,
\quad \text{for all } t \geq 0.
\ee
Similar to \eqref{ag44} decorrelation holds also for mixed moments at different time points.

\begin{proof}[Proof of {\rm(\ref{ag44})}]
Obviously, no non-local coalescence takes place in the time interval $[0,T]$ in the limit 
as $N \to \infty$. We verify the remaining claim by showing that any two partition 
elements of the dual process starting at different sites never meet, so that for $n$ 
partition elements none of the possible pairs will ever meet. Indeed, the probability 
for two random walks to meet is the waiting time for the rate-$2c_0$ random walk to 
hit 2 starting from 1. This waiting time is the sum of a geometrically distributed number 
of jumps with parameter $N^{-1}$, each occurring after an $\exp(2c_0)$-distributed 
waiting time. By explicit calculation, the probability for this event to occur before time 
$t$ is $O(N^{-1})$, which gives the claim.
\end{proof}

\subsubsection{Generator convergence}
\label{sss.genconv}

In order to show the convergence of $L^{(N)} F$, we investigate the migration and the 
resampling part separately.

\medskip\noindent
$\bullet$ {\bf Migration part.}
Recall from \eqref{ak:migration-operator} that the migration operator for the geographic 
space $G=G_{N,1}=\{0,1,\ldots,N-1\}$ is
\be{sk_15c_1}
(L^{(N)}_{\mathrm{mig}}F)(x)
= \frac{c_0}{N} \sum_{\xi,\zeta \in G_{N,1}} \int_E (x_\zeta - x_\xi)(\dd a)\,
\frac{\partial F(x)}{\partial x_\xi}[\delta_a],
\ee
where $F \in \CF \subset C_\mathrm{b}(\CP(E)^N,\R)$, with $\CF$ the algebra of functions 
of the form \eqref{ak:test-functions}. We rewrite \eqref{sk_15c_1} as
\be{sk_15c_2}
\begin{aligned}
(L^{(N)}_{\mathrm{mig}} F)(x)
&= c_0 \sum_{\xi \in G_{N,1}} \int_E \frac{1}{N} \sum_{\zeta \in G_{N,1}}
\left(x_{\zeta}-x_{\xi} \right)(\dd a)\,\frac{\partial F(x)}{\partial x_\xi}[\delta_a]\\
&= c_0 \sum_{\xi \in G_{N,1}} \int_E
\left(y-x_{\xi}\right)(\dd a)\,\frac{\partial F(x)}{\partial x_\xi}[\delta_a],
\end{aligned}
\ee
where $y=N^{-1} \sum_{\zeta=0}^{N-1} x_\zeta =N^{-1} \sum_{\zeta \in G_{N,1}} x_\zeta$ 
denotes the block average. We will show that, in the limit $N\to\infty$,
$(L^{(N)}_{\mathrm{mig}} F)(x)$ only depends on the mean type measure $\theta$ of the 
initial state, i.e., it converges to
\be{ak3}
(L^{c_0}_\theta F)(x) \equiv c_0 \sum_{\xi \in \N_0}  \int (\theta - x_\xi)(\dd a)
\frac{\partial F(x)}{\partial x_\xi} [\delta_a],
\ee
where we use for this generator acting on $C_\mathrm{b} (\CP(E))^\N, \R)$ the same notation
we used for the McKean-Vlasov process with immigration-emigration on $\CP(E)$ (cf. \eqref{generic-interaction-operator}).
Furthermore, we show that
\be{ak3b}
\theta \mapsto L^{c_0}_\theta F \in C_\mathrm{b}(\CP(E),\R) \mbox{ is continuous for all } 
\theta \in \CP(E).
\ee

To show the convergence, define
\be{ak5}
\B_\theta = \Big\{x \in (\CP(E))^{\N_0} \colon\,
N^{-1} \suml_{\xi \in G_{N,1}} x_\xi  \mathop{\longrightarrow}_{N \to \infty}
\theta \Big\} \subseteq (\CP(E))^\N,
\ee
and
\be{ak5b}
\B =\bigcup\limits_{\theta \in \CP(E)} \B_\theta.
\ee
For $x \in \CP(E)^{\N_0}$ and $n \in \N$, denote $x\vert_n = (x_0, x_1, \ldots, x_{n-1})$.

If we have an i.i.d.\ initial law (respectively, an exchangeable law) with mean measure
$\theta$, then the process $X^{(N)}$ satisfies
\be{ak6}
\CL[X^{(N)}(t)] (\B \vert_{N}) = 1 \quad (\text{respectively, } \CL[X^{(N)}(t)(\B_\theta\vert_N)]=1).
\ee
Indeed, as we will see in Section \ref{ss.mffinite}, the $1$-block average $Y^{(N)}_{\xi,1}$ (recall 
\eqref{k-block-average}) evolves on time scale $Nt$. More precisely, $(Y^{(N)}_{\xi,1}
(tN))_{t \geq 0}$ is tight in path space and therefore converges over a finite time 
horizon to the mean type measure $\theta$ of the initial state. In a formula (the 
right-hand side means a constant path):
\be{ag60}
\CL[(Y^{(N)}_{\xi,1} (t))_{t \in [0,T]}] \Ntoo \CL [(\uth)_{t \in [0,T]}].
\ee
Therefore, we have
\be{ak7}
|(L^{(N)}_{\mathrm{mig}} F)(x\vert_N) - (L^{c_0}_\theta F)(x) 
| \mathop{\longrightarrow}_{N \to \infty} 0,
\quad \text{for all } x \in \B_\theta,
\ee
Hence, on the path space, by dominated convergence, we have
\be{ak7b}
\CL \left[\left(\left| \intl^t_0 (L^{(N)}_{\textrm{mig}}F)(X^{(N)}(s))\dd s
- \intl^t_0 (L^{c_0}_{Y^{(N)}_{\xi,1} (s)} F)(X^{(N)}(s))\dd s \right|\right)_{t \geq 0} \right]
 \mathop{\Longrightarrow}_{N \to \infty} \delta_{\unu}.
\ee

\medskip\noindent
$\bullet$ {\bf Resampling part.}
The action of the resampling term on each component (recall \eqref{ak:one-level-generator}) 
does not depend on $N$ and hence we obtain, by the law of large numbers for the marking 
operation (recall that $F$ as in \eqref{ak:test-functions} depends on finitely many coordinates only)
\be{sk_15c_3}
| (L^{(N)}_{\mathrm{res}} F)(x\vert_N) - (L^{\Lambda_0} F)(x) | 
\mathop{\longrightarrow}_{N \to \infty} 0, \quad \text{for all }x \in (\CP(E)^\N),
\ee
where
\be{sk_15c_4}
\begin{aligned}
&(L^{\Lambda_0} F)(x)\\
&\equiv \sum_{\xi \in \N_0} \int_{[0,1]} \Lambda_0^*(\dd r) \int_E x_\xi(\dd a)
\Big[F\big(x_0,\ldots,x_{\xi-1},(1-r)x_\xi+r\delta_a,x_{\xi+1},\ldots,x_{N-1} \big)
-F(x)\Big].
\end{aligned}
\ee
Again, we use for this generator acting on $C_\mathrm{b} (\CP(E))^\N, \R)$ the same notation
we used for the McKean-Vlasov process with immigration-emigration on $\CP(E)$ (cf. \eqref{generic-interaction-operator}).

\subsubsection{Convergence to the McKean-Vlasov process}
\label{sss.single-mckean-vlasov}

In this section, we finally show the convergence of the mean-field
C$^{\Lambda}$-process (see Section~\ref{sss.mcolonycannings}) to the
McKean-Vlasov process (see Section~\ref{sss.MV}) which was claimed 
in Proposition~\ref{P.mkvl}.

In what follows, we fix $\xi\in \N_0$ and let
\be{one-level-test-function}
G\!\left( x_{\xi} \right)
= \int_{E^n} x_{\xi}^{\otimes n}(\dd u)\,\varphi(u)
= \left\langle \varphi , x_{\xi}^{\otimes n} \right\rangle, 
\qquad n \in \N, \varphi \in C_{\mathrm{b}}(E^n,\R).
\ee
We know that $(X^{(N)}_{\xi}(t))_{\xi \in \N_0}$ is tight and that all weak limit points 
are systems of independent random processes (i.e, that propagation of chaos 
holds). It remains to identify the unique marginal law.

Let the initial condition $(X^{(\infty)}_{\xi}(0))_{\xi\in\N_0}$ be i.i.d.\ $\CP(E)$-valued 
random variables with mean $\theta$. Then each single component converges 
and the limiting coordinate process has generator (recall \eqref{generic-interaction-operator})
\be{sk:generator-one-dim-limit-1}
\begin{aligned}
(L_{\theta}^{c_0,0,\Lambda_0} G)\!\left(x_{\xi}\right)
=& c_0 \int_E \left(\theta-x_{\xi} \right)\!(\dd a)\,
\frac{\partial G(x_\xi)}{\partial x_\xi}[\delta_a]\\
&+ \int_{[0,1]} \Lambda^*_0 (\dd r) \int_E x_{\xi}(\dd a)\,
\big[G\big((1-r) x_{\xi} + r \delta_a \big)
- G\!\left(x_{\xi}\right)\big],
\end{aligned}
\ee
where $\theta \in \CP(E)$ is the initial mean measure. Indeed, we may now reason 
as in \cite[second part of Section~2.9]{D93}. Tightness of the processes 
$(X^{(N)}(t))_{t\geq 0}$ was shown in Section~\ref{sss.tightN}.  Fix $\xi \in \N_0$ 
and consider a convergent subsequence $(X^{(N_k)}_\xi(t))_{t\geq 0}$, $k\in\N$. 
We claim that the limiting process is the unique solution to the well-posed 
martingale problem with corresponding generator $L_{\theta}^{c_0,0,\Lambda_0}$ 
and initial distribution $\mathcal{L}\left[ X_\xi(0)\right]$. Recall from Section~\ref{sss.genconv} that, for all test 
functions $F \in \CF$,
\be{tf1}
\CL\left[ \left(\intl^t_0 (L^{(N)}_{\mathrm{mig}} +L^{(N)}_{\mathrm{res}})(F) 
(X^N(s))\dd s \right)_{t\geq 0}\right]
\Ntoo
\CL\left[  \left(\intl^t_0 L_\theta^{c_0, d_0,\Lambda_0} 
(X^\infty(s)) \dd s\right)_{t \geq 0} \right]
\ee
Hence, all weak limit points of $X^{(N)}$ solve the $L_\theta^{c_0,d_0,\Lambda_0}$-martingale 
problem of Section \ref{sss.MV}. The right-hand side of \eqref{tf1} is the compensator of a 
well-posed martingale problem (recall Proposition~\ref{prop:McKean-Vlasov-well-posedness}), and hence we have convergence \eqref{ag43}.

\subsection{The mean-field finite-system scheme}
\label{ss.mffinite}

In this section, we verify the mean-field ``finite system scheme'' for the 
$C^\Lambda$-process, i.e., we consider $L+1$ tagged sites $\{X^{(N)}_0(t),
\ldots,X^{(N)}_L(t)\}$ evolving as in Section~\ref{sss.mcolonycannings} and the
corresponding block average $Y^{(N)}(t) = N^{-1}\sum_{\xi \in G_{N,1}}
X_\xi^{(N)}(t)$. We prove:
\begin{itemize}
\item 
convergence of $(Y^{(N)}(Nt))_{t\geq 0}$ to the Fleming-Viot diffusion $Y(t) 
= Z^{0,d_1,0}_\theta (t)$ with parameter $d_1=\frac{c_0 \lambda_0}{2c_0+\lambda_0}$ 
and initial state $\theta$ (cf. Section~\ref{sss.MV} and recall \eqref{diffusion-constants} with $d_0=0$); 
\item 
convergence of the components $(\{X^{(N)}_\xi (Nt +u),\xi=0,\ldots,L\})_{u \geq 0}$ 
to the equilibrium McKean-Vlasov process with immigration-emigration 
$(Z^{c_0,d_0,\Lambda_0}_{\theta(t)}(u))_{u \geq 0}$ starting from distribution
$\nu^{c_0,d_0,\Lambda_0}_{\theta(t)}$ (recall \eqref{cond-invariant})  with
$\theta(t)=Y(t)$ (recall that $d_0=0$).
\end{itemize}

\begin{proposition}
\label{P.mffsys}  
{\bf [Mean-field finite system scheme]}\\
For initial laws with i.i.d.\ initial configuration and mean measure $\theta$,
\be{ak10}
\CL[(Y^{(N)}(Nt))_{t \geq 0}] 
\mathop{\Longrightarrow}_{N \to \infty}
\CL[(Z^{0,d_1,0}_\theta (t))_{t \geq 0}]
\ee
with $d_1 = \frac{c_0 \lambda_0}{2c_0+\lambda_0}$. Moreover, for every $u \in \R$
and $L\in\N$,
\be{ag42}
\CL[(X^{(N)}_\xi(Nt+u))_{\xi=0,\ldots,L}]
\mathop{\Longrightarrow}_{N \to \infty}
\intl_{\CP(E)} P_t (\dd \theta^\prime)  
\left(\nu^{c_0,d_0,\Lambda_0}_{\theta^\prime}\right)^{\otimes (L+1)} 
\mbox{ with } P_t =\CL[Z^{0,d_1,0}_\theta(t)].
\ee
\end{proposition}

\begin{corollary}
\label{C.mffin}
{\bf [Mean-field finite system scheme with $\Lambda_1$-block resampling]}\\
Consider the model above with additional block resampling at rate $N^{-2}\Lambda_1$.
Then, in the right-hand side of \eqref{ak10}, $Z_\theta^{0, d_1,0}$ must
be replaced by $Z_\theta^{0,d_1,\Lambda_1}$, and similarly in the definition of 
$P_t$ in \eqref{ag42}.
\end{corollary}

The proof of the mean-field finite system scheme follows the abstract argument developed 
in~\cite{DGV95}. Namely, we first establish tightness of the sequence of processes 
$(Y^{(N)}(Nt))_{t \geq 0}$, $N\in\N$, which can be done as in Section~\ref{sss.tightN} 
for $(X^{(N)}_0 (t), \ldots, X^{(N)}_L (t))_{t \geq 0}$, $N\in\N$, once we have calculated 
the generators. A representation for the generator of the process is found in Sections
\ref{sss.1blockmig}--\ref{sss.1blockres-part-one} below. With the help
of the idea of local equilibria based on the ergodic theorems of Section \ref{s.mcv_imem}, 
we obtain first (\ref{ag42}) and then (\ref{ak10}) in Section~\ref{sss.1blockop}. 

In Sections~\ref{sss.1blockmig}-\ref{sss.1blockres-part-one}, we calculate the action of 
the generator of the martingale problem on the test functions induced by the functions 
necessary to arrive at the action of the generator of the limiting process. In 
Section~\ref{sss.1blockop}, we pass to the limit $N \to \infty$, where as in
Section~\ref{ss.singcol}, we have to use an averaging principle. However, instead of a 
simple law of large numbers, this now is a {\em dynamical averaging principle} with 
{\em local equilibria} for the single components necessary to obtain the expression for 
the limiting block-average process.

By the definition of the generator of a process, $M^{x,F} = (M^{x,F}_t)_{t \geq 0}$,
\be{sk_15c_9}
M_t^{x,F} = F(x_t) - F(x_0)
- \int_0^t \dd s \left(L_{\mathrm{mig}}^{(N)} F
+ L_{\mathrm{res}}^{(N)} F \right)\!(x_s)
\ee
is a martingale for all $F$, as in \eqref{one-level-test-function}. The same holds with 
$x$ replaced by the block averages $y$ (by the definition of $y$). Once again, we 
will investigate the migration and the resampling operator separately, this time 
for the block average.

\subsubsection{Migration}
\label{sss.1blockmig}

In this section, we consider functions $F \circ y$ with $F$ as in
\eqref{one-level-test-function} and 
\be{ak-block-avg}
y=N^{-1} \sum_{\xi \in G_{N,1}} x_\xi
\ee
a
block average (with $G_{N,1}=\{0,1,\ldots,N-1\}$). We will show below that
$L^{(N)}_{\mathrm{mig}}(F \circ y)=0$, so that migration has no effect.

Recall $(L_{\mathrm{mig}}^{(N)} F)(x)$ as rewritten in \eqref{sk_15c_2}. For the block 
averages $y$, the migration operator can be calculated as follows. Since $y = y(x)$ and
$F(y) = (F \circ y)(x)$ can be seen as functions of $x$ in the algebra $\CF$ of functions 
in $x$ of the form \eqref{one-level-test-function}, we have
\be{sk_15c_5}
(L_{\mathrm{mig}}^{(N)} F)(y)
= \left( L_{\mathrm{mig}}^{(N)}( F \circ y ) \right)\!(x)
= \sum_{\xi \in G_{N,1}} c_0
\int_E \left( y - x_{\xi} \right)\!(\dd a)
\frac{\partial (F \circ y)(x) }{\partial x_\xi}[\delta_a].
\ee
For $y=N^{-1} \sum_{\xi \in G_{N,1}} x_\xi$ this yields
\be{sk_15c_6}
\frac{\partial \!\left( F \circ y \right)\!(x) }{\partial x_\xi}[\delta_a]
=\frac{\partial F(y)}{\partial y}\!\left[ \frac{\delta_a}{N} \right]
\ee
and hence
\be{sk_15c_7}
(L_{\mathrm{mig}}^{(N)} F)(y)
= \sum_{\xi \in G_{N,1}} c_0
\int_E \left( y - x_{\xi} \right)\!(\dd a)
\frac{\partial F(y)}{\partial y}\!\left[ \frac{\delta_a}{N} \right]
=0.
\ee

\subsubsection{From $\Lambda$-Cannings to Fleming-Viot}
\label{sss.1blockres-part-one}

Next, we evaluate the moment measures of the average \eqref{ak-block-avg} in the limit as $N \to
\infty$ and show convergence of the terms to the Fleming-Viot second order term.

\begin{remark}[Notation for the rescaled generators]
Given a generator $L$ of a Markov process, we denote by $L^{[k]}$ (for $k \in \N$)
the generator of the Markov process on time scale $N^k t$. Evidently, this time speed-up 
simply amounts to multiplication of the original generator $L$ by $N^k$. 
\end{remark}

We are interested in the action of the
rescaled generator $L_{\mathrm{res}}^{(N)[1]}$ on the functions of the corresponding $1$-block averages \eqref{ak-block-avg}.

\begin{lemma}
\label{L.genres}
{\bf [Generator convergence: resampling]}\\
On time scale $Nt$, in the limit as  $N \to \infty$,
\be{sk_15c_8}
\begin{aligned}
&(L_{\mathrm{res}}^{(N)[1]} F)(y)\\
&= \frac{1}{N} \sum_{\xi \in G_{N,1}} \int_{[0,1]} \Lambda^*_0 (\dd r) \int_E
x_{\xi}(\dd a) \frac{1}{2} \frac{\partial^2 F(y)}{\partial y^2}[r(- x_{\xi}
+ \delta_a),r(- x_{\xi} + \delta_a)] + O\!\left( N^{-1} \right).
\end{aligned}
\ee
\end{lemma}

\begin{proof}[Proof of Lemma~\ref{L.genres}]
We first rewrite $F(y_t)$ in terms of $x_t$:
\be{sk_15c_10}
\begin{aligned}
F(y_t)
&= \left\langle \varphi, y_t^{\otimes n} \right\rangle
= \left\langle \varphi, \left( \frac{1}{N}
\sum_{\xi \in G_{N,1}} x_{\xi}(t) \right)^{\otimes n} \right\rangle\\
&= \frac{1}{N^n} \sum_{\xi_1 \in G_{N,1}} \ldots 
\sum_{\xi_n \in G_{N,1}} \left\langle \varphi , 
x_{\xi_1}(t) \otimes \ldots \otimes x_{\xi_n}(t)
\right\rangle \\
&= \frac{1}{N^n} \left( \bigotimes_{i=1}^n 
\sum_{\xi_i \in G_{N,1}} \right) \left\langle \varphi, 
x_{\xi_1}(t) \otimes \cdots \otimes x_{\xi_n}(t)
\right\rangle.
\end{aligned}
\ee
Abbreviate
\be{sk_15c_11}
F^{(\xi_1,\ldots, \xi_n)}(x)
= \int_{E^n} \left( \bigotimes_{i=1}^n x_{\xi_i}\!\left( \dd u^{(i)} \right) \right)
\varphi\!\left(u^{(1)}, \ldots, u^{(n)} \right)
= \left\langle \varphi , \bigotimes_{i=1}^n x_{\xi_i} \right\rangle.
\ee
Note that, in this notation, $\xi_i=\xi_j$ for $i \neq j$ is possible. Recall that 
$(x_t)_{t \geq 0}$ has generator $L^{(N)}$ and is the unique solution of the martingale 
problem \eqref{sk_15c_9}. If we use \eqref{sk_15c_10} in \eqref{sk_15c_9} with $x$ 
replaced by $y$, then we obtain that $(y_t)_{t\geq 0}$ solves the martingale problem 
with generator
\be{sk_15c_12}
(L_{\mathrm{res}}^{(N)} F)(y)
= \frac{1}{N^n} \left( \bigotimes_{i=1}^n 
\sum_{\xi_i \in G_{N,1}} \right) L_{\mathrm{res}}^{(N)}\!
\left( F^{\left( \xi_1, \ldots, \xi_n \right)} \right)(x)
\ee
for the resampling part. Together with \eqref{ak:one-level-generator} this yields the 
expression
\be{sk_15c_13}
\begin{aligned}
(L_{\mathrm{res}}^{(N)} F)(y)
&= \frac{1}{N^n} \left( \bigotimes_{i=1}^n \sum_{\xi_i \in G_{N,1}} \right) 
\sum_{\xi \in G_{N,1}} \int_{[0,1]} \Lambda^*_0 (\dd r) \int_E x_\xi(\dd a)\\
&\quad \times
\left[F^{\left(\xi_1,\ldots,\xi_n \right)}\!
\big(x_0,\ldots,x_{\xi-1},(1-r)x_\xi+r\delta_a,x_{\xi+1},\ldots,x_{N-1}\big)
-F^{\left(\xi_1,\ldots,\xi_n \right)}(x)\right].
\end{aligned}
\ee
We must analyse this expression in the limit as $N \to \infty$. To do so, we collect 
the leading order terms. The key quantity is the cardinality of the set $\{\xi_1,
\ldots,\xi_n\}$, for which we distinguish three cases.

\medskip\noindent
{\bf Case 1:} 
$\left| \left\{\xi_1,\ldots,\xi_n\right\} \right| = n$, i.e., all $\xi_i, 1 \leq i \leq n$ 
are distinct.

\medskip\noindent
The contribution to \eqref{sk_15c_13} is zero. For $\xi \not \in \left\{\xi_1,\ldots,\xi_n
\right\}$ this is obvious by the definition of $F^{\left( \xi_1, \ldots, \xi_n \right)}(x)$ 
in \eqref{sk_15c_11}. Otherwise, we have
\be{sk_15c_14}
\begin{aligned}
&\int_E x_\xi(\dd a) \left[F^{\left( \xi_1, \ldots, \xi_n \right)}\!
\big( x_0,\ldots,x_{\xi-1},(1-r)x_\xi+r\delta_a,x_{\xi+1},\ldots,x_{N-1} \big)
- F^{\left( \xi_1, \ldots, \xi_n \right)}(x) \right]\\
&= \int_E x_{\xi}(\dd a)\\
&\qquad \times \Big[ \left\langle \varphi , x_{\xi_1} \otimes \cdots \otimes
\underbrace{\left( (1-r) x_{\xi} + r \delta_a \right)}_{\substack{
\text{only change (unique)}\\
\text{position with } \xi_i=\xi}} \otimes \cdots \otimes x_{\xi_n} \right\rangle
- \left\langle \varphi , x_{\xi_1} \otimes \cdots \otimes x_{\xi_n} \right\rangle \Big] \\
& = 0,
\end{aligned}
\ee
where in the last line we use that $\left\langle x_{\xi},1\right\rangle=1$.

\medskip\noindent
{\bf Case 2:} 
$\left| \left\{\xi_1,\ldots,\xi_n\right\}\right| \leq n-2$.

\medskip\noindent
The contribution to \eqref{sk_15c_13} is of order $N^{-2}$. Indeed, the contribution 
is bounded from above by
\be{sk_15e_17}
\frac{1}{N^n} \left( \bigotimes_{i=1}^n \sum_{\xi_i \in G_{N,1}} \right) 
1_{\left\{ \left| \left\{ \xi_1, \ldots, \xi_n \right\} \right| 
\leq n-2 \right\}} \lambda_0 C_{F} = N^{-2} \lambda_0 C_{F},
\ee
where $C_F$ denotes a generic constant that depends on $F$ (as in 
\eqref{one-level-test-function}) only, and thereby on $\varphi$ and $n$. Here we use 
\eqref{ak:resampling-mapping-global} and the fact that the sum $\sum_{\xi \in G_{N,1}}$ 
yields at most $n$ non-zero summands by the definition of $F^{\left(\xi_1,\ldots,\xi_n
\right)}(x)$ in \eqref{sk_15c_11}.

\medskip\noindent
{\bf Case 3:} $\left| \left\{ \xi_1,\ldots,\xi_n \right\} \right| = n-1$.

\medskip\noindent
There exist $1 \leq m_1 < m_2 \leq n$ such that $\xi_{m_1}=\xi_{m_2}$ while all 
other $\xi_i, 1 \leq i \leq n$, are different. By the reasoning as in \eqref{sk_15c_14}, 
we see that the only non-zero contribution of the sum $\sum_{\xi \in G_{N,1}}$ to the 
generator in \eqref{sk_15c_13} comes from the case where $\xi=\xi_{m_1}=\xi_{m_2}$. 
We therefore obtain
\be{sk_15c_15}
\begin{aligned}
& (L_{\mathrm{res}}^{(N)} F)(y) \\
&= \frac{1}{N^n} \left( \bigotimes_{i=1}^n \sum_{\xi_i \in G_{N,1}} \right) 
1_{\left\{ \left| \left\{ \xi_1,\ldots,\xi_n \right\} \right| = n-1 \right\}}
\sum_{1 \leq m_1 < m_2 \leq n} 1_{\left\{ \xi_{m_1}=\xi_{m_2} = \xi \right\}} 
\int_{[0,1]} \Lambda^*_0 (\dd r) \int_E x_\xi(\dd a) \\
&\quad \times
\left[F^{\left( \xi_1, \ldots, \xi_n \right)}\!
\big( x_0,\ldots,x_{\xi-1},(1-r)x_\xi+r\delta_a,x_{\xi+1},\ldots,x_{N-1} \big)
-F^{\left( \xi_1, \ldots, \xi_n \right)}(x) \right] + O\!\left( N^{-2} \right).
\end{aligned}
\ee
Reasoning similarly to \eqref{sk_15e_17},  we see that extending
\be{sk_15c_16}
\left( \bigotimes_{i=1}^n \sum_{\xi_i \in G_{N,1}} \right) 
1_{\left\{ \left| \left\{ \xi_1,\ldots,\xi_n \right\} \right| = n-1 \right\}} 
\sum_{1 \leq m_1 < m_2 \leq n} 1_{\left\{ \xi_{m_1}=\xi_{m_2} \right\}}
\ee
in \eqref{sk_15c_15} to
\be{sk_15c_17}
\sum_{1 \leq m_1 < m_2 \leq n} \sum_{\xi_{m_1} \in G_{N,1}} 
1_{\left\{ \xi_{m_1}=\xi_{m_2} \right\}} 
\left( \bigotimes_{i \in \{1,\ldots,n\} \backslash \{m_1,m_2\}} 
\sum_{\xi_i \in G_{N,1}} \right)
\ee
only produces an additional error of order $N^{-2}$. Using this observation in 
\eqref{sk_15c_15}, we get
\be{sk_15c_18}
\begin{aligned}
& (L_{\mathrm{res}}^{(N)} F)(y) \\
=& \frac{1}{N^2} \sum_{1 \leq m_1 < m_2 \leq n}
\sum_{\xi \in G_{N,1}} \int_{[0,1]} \Lambda^*_0 (\dd r)
\int_E x_{\xi}(\dd a)\\
& \times \left[ \left\langle \varphi , y_{\xi_1} \otimes \cdots \otimes
\underbrace{\left( (1-r) x_{\xi} + r \delta_a \right)}_{\text{only change position }
\xi_{m_1}} \otimes \cdots \otimes
\underbrace{\left( (1-r) x_{\xi} + r \delta_a \right)}_{\text{and position }
\xi_{m_2}} \otimes \cdots \otimes y_{\xi_n} \right\rangle \right.\\
& \left. - \left\langle \varphi , y_{\xi_1} \otimes \cdots \otimes
\underbrace{x_{\xi}}_{\text{only change position } \xi_{m_1}} \otimes \cdots \otimes
\underbrace{x_{\xi}}_{\text{and position } \xi_{m_2}} \otimes \cdots \otimes
y_{\xi_n} \right\rangle \right]\\
& + O\!\left( N^{-2} \right).
\end{aligned}
\ee

Now use that
\be{sk_15e_10}
\int_E x_{\xi}(\dd a) \left\langle \varphi , y_{\xi_1} \otimes \cdots \otimes
\underbrace{\left( x_{\xi} \right)}_{\substack{\text{only change} \\ \text{position }
\xi_{m_1}}} \otimes \cdots \otimes
\underbrace{\left( -r x_{\xi}
+ r \delta_a \right)}_{\substack{\text{and position } \xi_{m_2} \\
\text{for } m_1,m_2 \text{ fixed}}} \otimes \cdots \otimes y_{\xi_n} \right\rangle = 0
\ee
to obtain from \eqref{sk_15c_18}, for $F(y)
= \langle \varphi, y^{\otimes n} \rangle$, that
\be{sk_15c_19}
\begin{aligned}
& (L_{\mathrm{res}}^{(N)} F)(y)\\
&= \frac{1}{N^2} \sum_{1 \leq m_1 < m_2 \leq n}
\sum_{\xi \in G_{N,1}} \int_{[0,1]} \Lambda^*_0 (\dd r)
\int_E x_{\xi}(\dd a)\\
&\quad \times \left\langle \varphi , y_{\xi_1} \otimes \cdots \otimes
\underbrace{\left( r(- x_{\xi} + \delta_a) \right)}_{\text{only change position }
\xi_{m_1}} \otimes \cdots \otimes \underbrace{\left( r(- x_{\xi} + \delta_a)
\right)}_{\text{and position } \xi_{m_2}} \otimes
\cdots \otimes y_{\xi_n} \right\rangle\\
&\quad + O\!\left( N^{-2} \right) \\
&= \frac{1}{N^2} \sum_{\xi \in G_{N,1}} \int_{[0,1]} \Lambda^*_0 (\dd r) 
\int_E x_{\xi}(\dd a) \frac{1}{2}
\frac{\partial^2 F(y)}{\partial y^2}[r(- x_{\xi} + \delta_a),
r(- x_{\xi} + \delta_a)] + O\!\left( N^{-2} \right).
\end{aligned}
\ee
Comparing Cases 1--3, we see that only the latter contributes to the leading 
term. Changing to time scale $Nt$ in \eqref{sk_15c_19}, i.e., multiplying 
$L^{(N)}_{\mathrm{res}}$ by $N$, we complete the proof.
\end{proof}

\subsubsection{A comment on coupling and duality}
\label{sss.coupling}

The techniques of coupling and duality are of major importance. One application 
can be found in \cite[Section 4]{DGV95}, namely, to prove Equation (4.17) therein. 
The key point is to obtain control on the difference between $\CL[Z_t]$ and 
$\CL[Z^\prime_t]$ for two Markov processes with {\em identical dynamics} but 
{\em different initial states}. Such estimates can be derived via coupling of 
the two dynamics, or alternatively, via dual processes that are based on finite 
particle systems with non-increasing particle numbers, allowing for an entrance 
law starting from a countably infinite number of particles. Both these properties 
hold in our model. This fact is used to argue that the configuration locally 
converges on time scale $Nt$ to an equilibrium by the following restart argument.

At times $Nt$ and $Nt-t_N$, with $\lim_{N\to\infty}t_N=\infty$ and
$\lim_{N\to\infty} t_N/N=0$, the empirical mean remains constant. Hence, we can
argue that, in the limit as $N\to\infty$, a system started at time $Nt-t_N$
converges over time $t_N$ to the equilibrium dictated by the current mean. Two
facts are needed to make this rigorous: (1) the map $\theta \mapsto
\nu^{c,d,\Lambda}_\theta$ must be continuous (recall
Section~\ref{sss.conttheta}); (2) the ergodic theorem must hold uniformly in the
initial state. Both coupling and duality do the job, which is why both work in 
\cite{DGV95}.

\subsubsection{McKean-Vlasov process of the $1$-block averages on time scale $Nt$}
\label{sss.1blockop}

Recall the definition of the Fleming-Viot diffusion operator $Q$ in
\eqref{flemming-viot-kernel} and the equilibrium $\nu$ of the McKean-Vlasov
process in the line preceding \eqref{cond-invariant}. Observe that the
compensators of $M^{x,F}$, see \eqref{sk_15c_9} are functionals of the empirical
measure of the configuration. The set of configurations on which $X^{(N)}$
concentrates in the limit  as $N \to \infty$ turns out to be 
\be{mcv1}
\B^\ast_\theta = \B_\theta \cap \left\{\ux \in (\CP(E))^\N \colon\, \frac{1}{N}
\suml^N_{\xi=1} \delta_{(x_\xi)} \Ntoo \nu^{c_0,0,\Lambda_0}_\theta\right\}, 
\ee
where $\theta$ is called the {\it intensity of the configuration} and 
\be{mcv2}
\B^\ast = \bigcup\limits_{\theta \in \CP(E)} \B^\ast_\theta.
\ee

\begin{lemma}
\label{L.locequ}
{\bf [Local equilibrium]}
\begin{itemize}
\item[\textup{(a)}] 
The block resampling term satisfies, with $y$ the intensity of the configuration 
$\ux$ for $\ux \in \B^\ast$,
\be{sk_15c_20}
\begin{aligned}
\lim_{N \rightarrow \infty} (L_{\mathrm{res}}^{(N)[1]} F)(y)
&= \frac{\lambda_0}{2} \int_{\CP(E)} \nu_y^{c_0,0,\Lambda_0}(\dd \wt x) 
\int_E \int_E Q_{\wt x}(\dd u,\dd v)\,
\frac{\partial^2 F(y)}{\partial y^2}[\delta_u, \delta_v] \\
&= \frac{c_0 \lambda_0}{2 c_0 + \lambda_0} \int_E \int_E Q_y(\dd u,\dd v)\,
\frac{\partial^2 F(y)}{\partial y^2}[\delta_u, \delta_v].
\end{aligned}
\ee
\item[\textup{(b)}] 
If the system starts i.i.d.\ with  some finite intensity measure, then every 
weak limit point of $\CL[(X^{(N)}(Nt+u))_{u \in \R}]$ as $N\to\infty$ has 
paths that satisfy
\be{mcv3}
\P(X^{(\infty)} (t,u) \in \B^\ast) = 1,
\qquad \text{for all } t \in [0,\infty),\,u \in \R.
\ee
\end{itemize}
\end{lemma}

\begin{proof}
(a) The proof uses the line of argument in \cite[Section 4(d)]{DGV95} (recall the 
comment in Section~\ref{sss.coupling}), together with 
\eqref{sk:preparation-second-moment-invariant-measure-1d} and the definition 
of $Q$. In what follows, two observations are important:
\begin{itemize}
\item[(i)] 
We use the results on the existence and uniqueness of a stationary distribution 
to \eqref{sk:generator-one-dim-limit-1} on the time scale $t$ with $N \rightarrow 
\infty$, including the convergence to the stationary distribution uniformly in the 
initial state, combined with the Feller property of the limiting dynamics (see 
Section~\ref{s.mcv_imem}). Note, in particular, that with
\eqref{sk:preparation-second-moment-invariant-measure-1d} we get the second
assertion in \eqref{sk_15c_20} from the first assertion.
\item[(ii)] 
We use the property that the laws of the processes $(Y^{(N)}(Nt))_{t \geq 0}$,
$N\in\N$, are tight in path space.
\end{itemize} 
The combination of (i) and (ii) will allow us to derive the claim. 

To verify (ii), use (\ref{sk_15c_19}) together with \eqref{sk_15c_7} to
establish that $\|L^{(N)[1]}_{\mathrm{res}} (F)\|_\infty$ is bounded in $N$,
which gives the tightness (recall Section~\ref{ss.embedding}). To verify (i), we
want to show that the weak limit points satisfy the $(\delta_\theta, 
L_\theta^{0,d_1,0})$-martingale problem. For that, we have to show that
\be{mcv4}
\begin{aligned}
&\CL \left[ \left( F\big(Y^{(N)}(tN)\big) - F\big(Y^{(N)}(0)\big)
- \intl^t_0 (L^{(N),[1]} F)\big(Y^{(N)} (sN)\big)\dd s \right)_{t \geq 0}\right]\\
& \hspace{3cm} \Ntoo
\CL \left[ \left( F\big(Z^{0,\dd_1,0)}(t)\big)-F(\theta)
- \intl^t_0 (L^{0,\dd_1,0} F)\big(Z^{0,\dd_1,0} (s)\big) \dd s \right)_{t \geq 0}\right].
\end{aligned}
\ee
In order to do so, we first need some information on $L^{(N),[1]}$. Since we are on 
time scale $Nt$ with $N \to \infty$, we get 
\be{sk_15c_21}
\begin{aligned}
&\lim_{N \rightarrow \infty} (L_{\mathrm{res}}^{(N)[1]} F)(y)\\
&= \int_{[0,1]} \Lambda^*_0 (\dd r) \int_{\CP(E)} \nu_y^{c_0,0,\Lambda_0}(\dd x) 
\int_E x(\dd a) \frac{1}{2} \frac{\partial^2 F(y)}{\partial y^2}
[r(- x + \delta_a),r(- x + \delta_a)] \\
&= \frac{\lambda_0}{2} \int_{\CP(E)} \nu_y^{c_0,d_0,\Lambda_0}(\dd x)
 \int_E x(\dd a) \frac{\partial^2 F(y)}{\partial y^2}
[- x + \delta_a,- x + \delta_a]
\qquad \forall\,\ux \in \B^\ast_y,\,y \in \CP(E).
\end{aligned}
\ee
Use the definition of the Fleming-Viot diffusion operator $Q$ from 
\eqref{flemming-viot-kernel} to obtain the first line of the claim in
(\ref{sk_15c_20}). The second line follows with the help of
\eqref{sk:preparation-second-moment-invariant-measure-1d} (recall $d_0=0$ in
this section).

\medskip\noindent
(b) To show that the relevant configurations (under the limiting laws) are
in $\B^\ast$, we use a restart argument in combination with the ergodic theorem 
for the McKean-Vlasov process. Namely, to study the process at time $Nt+u$ we 
consider the time $Nt+u-t_N$ with $\lim_{N\to\infty} t_N = \infty$ and 
$\lim_{N\to\infty} t_N/N=0$. We know that the density process $Y^{(N)}$ 
at times $Nt+u-t_N$ and $Nt+u$ is the same in the limit $N \to \infty$, say 
equal to $\theta$, and so over the time stretch $t_N$ the process converges 
to the equilibrium $(\nu^{c_0,0,\Lambda_0}_\theta)^{\otimes \N}$. By the law 
of large numbers, this gives the claim. Therefore, all possible limiting dynamics 
allow for an averaging principle with the local equilibrium.
\end{proof}

\subsubsection*{Conclusion of the proof of Proposition~\ref{P.mffsys}}

Recall from \eqref{sk_15c_7} that migration has no effect. Lemma \ref{L.locequ}
shows the effect of the block resampling term on time scale $Nt$ for $N
\to \infty$. Adding both effects together, we have that all weak limit points 
of $\CL [(Y^{(N)}(Nt))_{t \geq 0}]$, $N \in \N$, satisfy
\be{sk_15c_30}
\mbox{the } (\delta_\theta, L_\theta^{0,\dd_1,0})\mbox{-martingale problem with } 
d_1 = \frac{c_0 \lambda_0}{2c_0+\lambda_0}.
\ee

\section{Hierarchical $C^\Lambda$-process}
\label{s.hiermod}

The next step in our construction is to consider finite spatial systems with a 
hierarchical structure of $K$ levels and to study the $k$-block averages with 
$k=0,1,\ldots,K$ on their natural time scales $N^k t$ and $N^k t+u$. This section 
therefore deals with the geographic space 
\be{mcv5}
G=G_{N,K}=\{0,1,\ldots,N-1\}^K, \qquad  N,K\in\N.
\ee
Define the Cannings process on $G_{N,K}$ by restricting $X^{(\Omega_N)}$ 
from Section~\ref{sss.hierarCan} to $B_K(0)$ and putting
\be{mcv6}
c_k, \lambda_k=0, \quad \text{for all }k \geq K. 
\ee
The corresponding process will be denoted by $X^{(N,K)}$ and its generator by 
$L^{(N,K)}$, etc. It is straightforward to include also a block resampling at 
rate $N^{-2K}$ with resampling measure $\Lambda_K$ (compare Corollary \ref{Cmkvbs}). 

In this section, our principal goal is to understand how we move up $0 \leq k 
\leq K$ levels when starting from level 0. However, in order to also understand 
a system with $k$ levels starting from level, say, $L$ and moving up to level 
$L+k$, we will add a Fleming-Viot term to the generator of $X^{(N)}$, i.e., we 
consider the case $d_0>0$. We do not need to add Fleming-Viot terms acting on 
higher blocks. As we saw in Lemma~\ref{L.locequ}, a resampling term can result, on a higher time 
scale and in the limit as $N \to \infty$, in a Fleming-Viot term. For instance, if 
we choose $d_0=0$ in the beginning, then we obtain $d_1= \frac{c_0 \lambda_0}{2c_0
+\lambda_0} > 0$ on time scale $Nt$ for the $1$-block average (recall \eqref{sk_15c_30}).

We look at the block averages on space scales $N^k$ and time scales $N^kt$ with 
$k=1,\ldots,K$. In Section \ref{ss.K=2}, we will focus on the case $K=2$, where 
most of the difficulties for general $K$ are already present. Many features from
Section~\ref{ss.mffinite} reappear here, but we have to be aware that level-one averages are
forming only asymptotically a mean-field system of the type we had in Section~\ref{s.mflambda}
and we have to prove that we can in fact ignore this perturbation. For $K>2$, lower 
order perturbations arise, which we will discuss only briefly in Section~\ref{ss.K>3}
because they can be treated similarly as in \cite{DGV95}. In Section~\ref{s.proof1stth},
we will take the limit $K\to\infty$ and show how this approximates the model 
with $G=\Omega_N$ on all the time scales we are interested in for our main theorem.

\subsection{Two-level systems}
\label{ss.K=2}

The geographic space is $G_{N,2}=\{0,1,\ldots,N-1\}^2$, we pick $d_0,c_0,c_1, 
\lambda_0,\lambda_1>0$ and put $c_k,\lambda_k$ to zero for $k \geq 2$. We 
will prove the following: (1) On time scales $t$ and $Nt$ we obtain the same 
limiting objects as described in Section \ref{s.mflambda}, but with an additional 
Fleming-Viot term ($d_0>0$) and with block resampling via $\Lambda_1$; (2) For 1-block averages (each belonging to an address 
$\eta \in \{0,1,\cdots,N-1\}$) we introduce the notation
\be{add21} 
Y^{(N)}_\eta(t) = N^{-1}\sum_{\xi \in G_{N,1}} X_{\xi,\eta}^{(N)}(t).
\ee
Next, we consider the {\it total average} 
\be{add20}
Z^{(N)}(t) = N^{-2}\sum_{\zeta \in G_{N,2}} X_\zeta^{(N)}(t).
\ee
We get a similar structure to the one in Section~\ref{s.mflambda}. 
Namely, we can replace the system $(Y^{(N)},Z^{(N)})$ for $N \to \infty$ by a 
system of the type in Section~\ref{s.mflambda}, where the role of components 
on time scale $t$ is taken over by $1$-block averages on time scale $Nt$ 
and the role of the total ($1$-block) average on time scale $Nt$ taken over 
by the $2$-block average on time scale $N^2t$. Once again, we only focus on the
new features arising in our model. The general scheme of the proof for the two-level 
system can be found in \cite[Section 5(a), pp.~2328--2337]{DGV95}. The calculations 
in Sections~\ref{sss.tsO}--\ref{sss.totalN2t} correspond to Steps 4--5 in 
\cite[Section 5(a)]{DGV95}, with the focus now shifted from the characteristics of 
diffusions to the {\em full} generator because we are dealing with jump processes.

\begin{proposition}
\label{P.2lresc}
{\bf [Two-level rescaling]}\\
Under the assumptions made above,
\be{ak15}
\CL[(X^{(N)}_\zeta(t))_{t \geq 0}] 
\mathop{\Longrightarrow}_{N \to \infty}
\CL[(Z^{c_0, d_0, \Lambda_0}_\theta (t))_{t \geq 0}]
\qquad \forall\, \zeta \in G_{N,2},
\ee
and
\be{ak16}
\CL[(Y^{(N)}_\xi (Nt))_{t \geq 0}] 
\mathop{\Longrightarrow}_{N \to \infty}
\CL[(Z^{c_1, d_1, \Lambda_1}_\theta (t))_{t \geq 0}] \quad \mbox{ with } \quad 
d_1 = \frac{c_0 (\lambda_0 + 2d_0)}{2 c_0 + \lambda_0 + 2 d_0}, \quad \xi \in G_{N,1}
\ee
and
\be{ak17}
\CL[(Z^{(N)} (N^2 t))_{t \geq 0}] 
\mathop{\Longrightarrow}_{N \to \infty}
\CL[(Z^{0,d_2,0}_\theta (t))_{t \geq 0}]
\quad \mbox{ with } d_2 = \frac{c_1 (\lambda_1 + 2d_1)}{2 c_1 + \lambda_1 + 2 d_1}.
\ee
\end{proposition}

\noindent
The proof of (\ref{ak15}--\ref{ak17}) is carried out in 
Sections~\ref{sss.tsO}--\ref{sss.totalN2t}.

\subsubsection{The single components on time scale $t$}
\label{sss.tsO}

In this section, our main goal is to argue that the components of $X^{(N)}$ change on 
time scale $t$ as before, and that the same holds on time scales $N t+u$ and $N^2 t+u$ 
with $u \in \R$, provided we use the appropriate value for the $1$-block average as 
the centre of drift.

We first look at the components on time scale $t$. Due to the Markov property
and the  continuity in $\theta$ of the law of the McKean-Vlasov process (cf., Section~\ref{sss.conttheta}), the
behaviour of the  components on time scales $Nt+u$ and $N^2 t+u$ with $u \in \R$
is immediate once we  have the tightness of $Y^{(N)}$ and $Z^{(N)}$ on these
scales.   Again, our convergence results are obtained by: (1) establishing {\em
tightness} in path space; (2) verifying convergence of the finite-dimensional 
distributions by means of establishing asymptotic independence and the {\em generator
calculation} for the martingale problem. Since the latter is key also for the
tightness arguments (recall \eqref{ag51}), we give the analysis of the generator terms first. In
fact, the rest of the argument is the same as in Section~\ref{ss.singcol}.

\myparagraph{Migration part.}
Consider the migration operator in \eqref{ak:multi-level-migration} with \eqref{32b} 
applied to functions $F \in \CF$, the algebra of functions in 
\eqref{ak:multi-level-test-functions}. The migration operator can be rewritten as 
(recall that the upper index $2$ in $L^{(N,2)}$ indicates that we consider $K=2$ levels)
\be{sk_15c_22}
\begin{aligned}
(L^{(N,2)}_{\mathrm{mig}} F)(x)
&= \sum_{\xi,\zeta \in G_{N,2}} a^{(N)}_{\xi,\zeta} \int_E
\left(x_{\zeta}-x_{\xi}\right)(\dd a)\, 
\frac{\partial F(x)}{\partial x_\xi}[\delta_a] \\
&= \sum_{\xi,\zeta \in G_{N,2}} \sum_{ d(\xi,\zeta) \leq k \leq 2}
c_{k-1} N^{1-2k} \int_E \left(x_{\zeta}-x_{\xi}\right)(\dd a)\,
\frac{\partial F(x)}{\partial x_\xi}[\delta_a] \\
&= \sum_{\xi \in G_{N,2}} \sum_{k \leq 2} c_{k-1} N^{1-2k}
\sum_{\zeta \in B_k(\xi)} \int_E
\left(x_{\zeta}-x_{\xi} \right)(\dd a)\,\frac{\partial F(x)}{\partial x_\xi}[\delta_a] \\
&= \sum_{\xi \in G_{N,2}} \sum_{k \leq 2} c_{k-1} N^{1-k} \int_E
\left(y_{\xi,k}-x_{\xi}\right)(\dd a)\,\frac{\partial F(x)}{\partial x_\xi}[\delta_a],
\end{aligned}
\ee
where we use \eqref{def-blocks} in the last line. Thus, for $F$ as in 
\eqref{ak:multi-level-test-functions}, we obtain
\be{sk_15c_23}
(L^{(N,2)}_{\mathrm{mig}} F)(x) = \sum_{\xi \in G_{N,2}} c_0
\int_E \left(y_{\xi,1}-x_{\xi}\right)(\dd a)\,
\frac{\partial F(x)}{\partial x_\xi}[\delta_a] + E^{(N)},
\ee
where
\be{sk_15c_24}
|E^{(N)}| \leq N^{-1} c_1 C_F = O\!\left( N^{-1} \right)
\ee
with $C_F$ a generic constant depending on the choice of $F$ only. Here we 
use that, by the definition of $F$ in \eqref{ak:multi-level-test-functions}, the sum 
over $\xi \in G_{N,2}$ is a sum over finitely many coordinates only, with the number 
depending on $F$ only.

\myparagraph{Resampling part.}
Recall \eqref{ak:multi-level-test-functions}. For $F \in \mathcal{F}$, consider the resampling 
operator $(L_{\mathrm{res}}^{(N,2)} F)(x)$ in \eqref{ak:multi-level-resampling-global}--\eqref{ak:resampling-mapping-global}. We have
\be{sk_15c_26}
(L^{(N,2)}_{\mathrm{res}} F)(x) =
\sum_{\xi \in G_{N,2}} \int_{[0,1]} \Lambda^*_0 (\dd r)
\int_E x_{\xi}(\dd a)
\left[ F\left( \Phi_{r,a,B_0(\xi)}(x) \right) - F(x) \right]
+ E^{(N)}
\ee
with
\be{sk_15c_27}
|E^{(N)}| \leq N^{-2} \int_{[0,1]} \Lambda^*_1 (\dd r) C_{F} r^2 N 
= C_{F} N^{-1} \lambda_1 = O\!\left( N^{-1} \right).
\ee
Here we use \eqref{ak:resampling-mapping-global} in the first inequality, together 
with the fact that $F(\Phi_{r,a,B_1(\xi)}(x))-F(x)$ is non-zero for at most $C_F N$ 
different values of $\xi \in G_{N,2}$.

\myparagraph{Additional Fleming-Viot part.}
Recall that in this section we consider the case $d_0>0$, i.e., we add the 
Fleming-Viot generator
\be{sk_15c_28}
(L_{\mathrm{FV}}^{(N,2)} F)(x)
= d_0 \sum_{\xi \in G_{N,2}} \int_E \int_E Q_{x_\xi}(\dd u,\dd v)\,
\frac{\partial^2 F(x)}{\partial x_\xi^2}[\delta_u,\delta_v]
\ee
with $Q_{x_\xi}$ as in \eqref{flemming-viot-kernel}.
Contrary to the migration and the resampling operator, the Fleming-Viot operator does 
not act on higher block levels.

\myparagraph{The resulting generator.}
Combining the migration parts \eqref{sk_15c_23} and \eqref{sk_15c_24}, the resampling 
parts \eqref{sk_15c_26} and \eqref{sk_15c_27}, and the Fleming-Viot part \eqref{sk_15c_28}, 
we obtain
\be{sk_15c_29}
\begin{aligned}
(L^{(N,2)} F)(x)
=& \sum_{\xi \in G_{N,2}} c_0
\int_E \left(y_{\xi,1}-x_{\xi}\right)(\dd a)\,
\frac{\partial F(x)}{\partial x_\xi}[\delta_a] \\
&+ \sum_{\xi \in G_{N,2}} \int_{[0,1]} \Lambda^*_0 (\dd r)
\int_E x_{\xi}(\dd a)
\left[ F\left( \Phi_{r,a,B_0(\xi)}(x) \right) - F(x) \right] \\
&+ d_0 \sum_{\xi \in G_{N,2}} \int_E \int_E Q_{x_\xi}(\dd u,\dd v)\,
\frac{\partial^2 F(x)}{\partial x_\xi^2}[\delta_u,\delta_v] 
+ O\!\left( N^{-1} \right),
\end{aligned}
\ee
where $O\!\left( N^{-1} \right)$ is uniform in $x$.

\myparagraph{Convergence to McKean-Vlasov process.}
We can use \eqref{sk_15c_29} to argue that
\be{add22}
\|L^{(N,2)} F - L^{c_0,d_0,\Lambda_0}_{y_{\xi,1}} F\|_\infty \leq C_F N^{-1},
\quad \|L^{c_0,d_0,\Lambda_0}_{y_{\xi,1}} F\| \leq C(F), 
\quad n \in \N, \, F\in \mathcal{F}
\ee
with $\mathcal{F}$ as in \eqref{ak:multi-level-test-functions}.
Next, following again the line of argument in Section~\ref{ss.embedding}, we see 
that $\CL[X^{(N)}]$ is tight in path space and, following the argument as in 
Section~\ref{ss.singcol}, we obtain that $X^{(N)}$ converges as a process to 
the McKean-Vlasov limit, which is an i.i.d.\ collection of single components 
indexed by $\N_0$ with generator
\be{sk_15c_31}
\begin{aligned}
(L_{\theta}^{c_0,d_0,\Lambda_0} G)\!\left(x_{\xi}\right)
=& c_0 \int_E \left(\theta-x_{\xi} \right)\!(\dd a)\,
\frac{\partial G(x_\xi)}{\partial x_\xi}[\delta_a] \\
&+ \int_{[0,1]} \Lambda^*_0 (\dd r) \int_E x_{\xi}(\dd a)\,
\big[G\big((1-r) x_{\xi} + r \delta_a \big)
- G\!\left(x_{\xi}\right)\big] \\
&+ d_0 \int_E \int_E Q_{x_\xi}(\dd u,\dd v)\,
\frac{\partial^2 G(x)}{\partial x_\xi^2}[\delta_u,\delta_v],
\end{aligned}
\ee
where $\theta \in \CP(E)$ is the initial mean measure. This completes the 
proof of \eqref{ak15}.

\subsubsection{The 1-block averages on time scale $Nt$}
\label{sss.1block2l}

Again, we need to prove: (1) uniform boundedness (in $N$) of the generator in the 
supremum norm for test-functions in $\mathcal{F}$ to get tightness in path 
space of $(Y^{(N)}_\xi (Nt))_{t \geq 0}$ (cf. \eqref{ag51}); (2) convergence of finite-dimensional 
distributions via asymptotic independence and generator convergence. As we saw 
in Section~\ref{s.mflambda}, the latter is also the key to tightness. Therefore, we 
proceed by first calculating the generator of $1$-block averages on time scale $Nt$ 
and then using this generator to show convergence of the process. At that point we 
need that the total average over the full space (cf. \eqref{add20}) remains $\theta$
on time scale $Nt$, in the sense of a constant 
path on time scale $Nt$. The latter property will be proved in Section~\ref{sss.totalN2t}.

\myparagraph{Basic generator formula.}
We proceed as in Section~\ref{ss.mffinite}. Since $G=G_{N,2}$, the $1$-block 
averages are now indexed too. We use the following notation for the indexing of 
$1$-block averages. Recall the notation $y_{\zeta,1}=N^{-1} \sum_{\xi \in B_1(\zeta)} 
x_\xi$ from \eqref{def-blocks}, which is the $1$-block around $\zeta$. This $1$-block 
coincides with the $1$-block around $\xi$ if and only if $d(\zeta,\xi) \leq 1$. To endow 
every $1$-block with a unique label, we proceed as follows. Let $\phi$ be the shift-operator 
\be{add22b}
\phi\colon\,G_{N,K} \to G_{N,K-1}, (\phi \xi)_i = \xi_{i+1}, 
\qquad 0 \leq i \leq K-1,\, K \in \N.
\ee
We consider the evolution in time of the $1$-block averages indexed block-wise, i.e.,
\be{sk_15c_32}
y_{\eta}^{[1]} \equiv N^{-1} \sum_{\xi \in G_{N,2}, \phi \xi = \eta} x_{\xi},
\ee
where we suppress the dependence of $y_\eta^{[1]}$ on $N$. Note in particular 
that
\be{sk_15c_33}
y_{\xi,1} = y_\eta^{[1]} \text{ for all } \xi \text{ such that } \phi \xi = \eta.
\ee
We will often drop the superscript $[1]$ to lighten the notation.

This time, we consider functions $F \in \CF$ (see
\eqref{ak:multi-level-test-functions}) applied to $y^{[1]} \equiv y^{[1]}(x)$,
where $y^{[1]} = ( y^{[1]}_\eta )_{\eta \in G_{N,1}}$. Recall the ${}^{[k]}$-notation for
the rescaled generators from Section~\ref{sss.1blockres-part-one}. By explicit
calculation of the different terms below, we will obtain the following
expression (recall $\Phi_{r,a,\eta}$ from \eqref{ak:resampling-mapping-global}
and $Q_{x_\xi}$ from \eqref{flemming-viot-kernel}):
\be{sk_15c_52}
\begin{aligned}
& (L^{(N,2)[1]} F)(y) = \left(L^{(N,2)[1]}_{\mathrm{mig}} 
+ L^{(N,2)[1]}_{\mathrm{res},0}
+ L^{(N,2)[1]}_{\mathrm{res},1} + L^{(N,2)[1]}_{\mathrm{FV}}\right) (F)(y)\\
&= \sum_{\eta \in G_{N,1}} c_1 \int_E 
\left( y_{\phi \eta}^{[1]} - y_{\eta} \right)\!(\dd a)
\frac{\partial F(y)}{\partial y_\eta}[\delta_a]\\
&\quad + \sum_{m=1}^q \frac{1}{N} 
\sum_{ \xi\colon\,\phi \xi = \eta^{(m)} } 
\int_{[0,1]} \Lambda^*_0 (\dd r) \int_E x_{\xi}(\dd a) 
\frac{1}{2} \frac{\partial^2 F(y)}{\partial y_{{\eta}^{(m)}}^2}[r(- x_{\xi} + \delta_a),r(- x_{\xi} + \delta_a)]\\
&\quad + \sum_{\eta \in G_{N,1}} \int_{[0,1]} \Lambda^*_1 (\dd r)
\int_E y_{\eta}(\dd a) \left[ F\left( \Phi_{r,a,\eta}(y) \right) - F(y) \right]\\
&\quad + d_0 \sum_{\eta \in G_{N,1}} \frac{1}{N} 
\sum_{\xi\colon\,\phi \xi = \eta} \int_E \int_E Q_{x_\xi}(\dd u,\dd v)\,
\frac{\partial^2 F(y)}{\partial y_\eta^2}[\delta_u, \delta_v] 
+ O\!\left( N^{-1} \right).
\end{aligned}
\ee
Here, we assumed that $F$ can be written as follows: $F(y_t) = F(y_t^{[1]}) = \langle \varphi,\bigotimes_{l=1}^q y_{\eta^{(l)}}^{\otimes n_l}
\rangle$ with $y=y^{[1]} = (y^{[1]}_\eta )_{\eta \in G_{N,1}}$, $\eta^{(l)} \in G_{N,1}$, 
$q \in \{1,\ldots,N\}$ and $n_l \in \N, 1 \leq l \leq q$. We give more detail in \eqref{sk_15c_38} below.

\myparagraph{Convergence to McKean-Vlasov process.}
We first argue how to conclude the argument, and then further below we carry out
the necessary generator calculations.

We have to argue first that the $N$ different $1$-blocks satisfy the propagation of 
chaos property (recall \eqref{ag44}, where we had this for components). The proof 
again uses duality, namely, dual particles from different $1$-blocks need a time of 
order $N^2$ to meet and hence do not meet on time scale $Nt$. We do not repeat the 
details here. 

Once we have the propagation of chaos property, it suffices to consider {\em single 
blocks}, which we do next. We have to verify tightness in path space and convergence 
of the finite-dimensional distributions. As we saw before, this reduces to showing that
the action of the generators is uniformly bounded in $N$ in the sup-norm on $\CF$, 
so that we have convergence of the generator on $\CF$ by the same tightness argument
as used in Section~\ref{sss.1blockop}, but now based on \eqref{sk_15c_52}. Consider
the resampling and Fleming-Viot parts of the generator in \eqref{sk_15c_52} separately.

Reason as in the proof of Lemma~\ref{L.locequ} to see that (recall the
definition of $\nu_{y_\eta}^{c_0,d_0,\Lambda_0}$ from \eqref{cond-invariant})
\be{sk_15c_53}
\begin{aligned}
& \lim_{N \to \infty} (L_{\mathrm{res},0}^{(N,2)[1]} F)(y)\\
& = \lim_{N \rightarrow \infty} \sum_{m=1}^q \frac{1}{N} 
\sum_{ \xi\colon\, \phi \xi = \eta^{(m)} } \int_{[0,1]} 
\Lambda^*_0 (\dd r) \int_E x_{\xi}(\dd a) \frac{1}{2} 
\frac{\partial^2 F(y)}{\partial y_{{\eta}^{(m)}}^2}[r(- x_{\xi} + \delta_a),r(- x_{\xi} + \delta_a)]\\
&= \frac{\lambda_0}{2} \sum_{\eta \in \N_0} \int_{\CP(E)} 
\nu_{y_\eta}^{c_0,d_0,\Lambda_0}(\dd x) \int_E \int_E Q_x(\dd u,\dd v)\,
\frac{\partial^2 F(y)}{\partial y_\eta^2}[\delta_u, \delta_v] \\
&= \frac{c_0 \lambda_0}{2 c_0 + \lambda_0 + 2d_0} \sum_{\eta \in \N_0} 
\int_E \int_E Q_{y_\eta}(\dd u,\dd v)\,
\frac{\partial^2 F(y)}{\partial y_\eta^2}[\delta_u, \delta_v],
\end{aligned}
\ee
where by \eqref{sk:preparation-second-moment-invariant-measure-1d} the 
second assertion follows from the first. Recall \eqref{sk_15c_28}. Similarly, 
we have
\be{sk_15c_54}
\lim_{N \rightarrow \infty} (L_{\mathrm{FV}}^{(N,2)[1]} F)(y)
= d_0 \sum_{\eta \in \N_0} \int_{\CP(E)} \nu_{y_{\eta}}^{c_0,d_0,\Lambda_0}(\dd x) 
\int_E \int_E Q_x(\dd u,\dd v)\,
\frac{\partial^2 F(y)}{\partial y_\eta^2}[\delta_u, \delta_v].
\ee
Using \eqref{sk:preparation-second-moment-invariant-measure-1d} once more, 
we get
\be{sk_15c_55}
\mbox{r.h.s.\ of \eqref{sk_15c_54}}
= \frac{2 c_0 d_0}{2 c_0 + \lambda_0 + 2 d_0} \sum_{\eta \in \N_0} 
\int_E \int_E Q_{y_{\eta}}(\dd u,\dd v)\,
\frac{\partial^2 F(y)}{\partial y_\eta^2}[\delta_u, \delta_v].
\ee

Combine \eqref{sk_15c_53} with \eqref{sk_15c_55} and argue as in 
Section~\ref{sss.single-mckean-vlasov}, to see that each single component of the 
$1$-block averages $y=y^{[1]}=( y^{[1]}_\eta )_{\eta \in G_{N,1}}$ converges and 
the limiting coordinate process has generator
\be{sk_15c_56}
\begin{aligned}
(L_{\theta}^{c_1,d_1,\Lambda_1} G)\!\left(y_{\eta}\right)
&= c_1 \int_E \left(\theta-y_{\eta} \right)\!(\dd a)\,
\frac{\partial G(y_\eta)}{\partial y_\eta}[\delta_a]\\
&\quad + d_1 \int_E \int_E Q_{y_\eta}(\dd u,\dd v)\,
\frac{\partial^2 G(y)}{\partial y_\eta^2}[\delta_u,\delta_v]\\
&\quad + \int_{[0,1]} \Lambda^*_1 (\dd r) \int_E y_{\eta}(\dd a)\,
\big[G\big((1-r) y_{\eta} + r \delta_a \big)
- G\!\left(y_{\eta}\right)\big],
\end{aligned}
\ee
for test-functions $G$ of the form \eqref{one-level-test-function}.
Note that $\theta \in \CP(E)$ is the initial mean measure of a component and $d_1 = 
\frac{c_0 (\lambda_0 + 2d_0)}{2 c_0 + \lambda_0 + 2 d_0}$. At this point we use that 
the average over the complete population remains the path that stands still at $\theta$
on time scale $Nt$.
  
\myparagraph{Generator calculation: proof of (\ref{sk_15c_52}).}
We next verify the expression given in \eqref{sk_15c_52}. We calculate separately the 
action of the various terms in the generator on the function $F$. In what follows a
change to time scale $N^kt$ is denoted by an additional superscript $[k]$.

\myparagraph{Migration part.} Recall $(L_{\mathrm{mig}}^{(N,2)} F)(x)$ from
\eqref{sk_15c_22} and that the upper index $2$ in $L^{(N,2)}$ indicates that we
consider $K=2$ levels. Let $F$ be as in \eqref{ak:multi-level-test-functions}.
Denote $\phi^k \equiv \underbrace{\phi \circ \phi \circ \ldots \circ \phi}_\text{$k$
times}$. Proceeding along the lines of (\ref{sk_15c_5}--\ref{sk_15c_7}), we get
\be{sk_15c_34}
\begin{aligned}
(L_{\mathrm{mig}}^{(N,2)} F)(y)
&= \sum_{\xi \in G_{N,2}} \sum_{k \leq 2} c_{k-1} N^{1-k}
\int_E \left( y_{\xi,k} - x_{\xi} \right)\!(\dd a)
\frac{\partial \!\left( F \circ y \right)\!(x) }{\partial x_\xi}[\delta_a] \\
&= \sum_{\xi \in G_{N,2}} \sum_{k \leq 2} c_{k-1} N^{1-k}
\int_E \left( y_{\phi^k \xi}^{[1]} - x_{\xi} \right)\!(\dd a)
\frac{\partial F(y)}{\partial y_{\phi \xi}}\!\left[ \frac{\delta_a}{N} \right] \\
&= N \sum_{\eta \in G_{N,1}} \sum_{k \leq 2} c_{k-1} N^{1-k}
\int_E \left( y_{\phi^{k-1} \eta}^{[1]} - y_{\eta} \right)\!(\dd a)
\frac{\partial F(y)}{\partial y_\eta}\!\left[ \frac{\delta_a}{N} \right] \\
&= \sum_{\eta \in G_{N,1}} \sum_{k \leq 1} c_k N^{1-k} \int_E
\left( y_{\phi^k \eta}^{[1]} - y_{\eta} \right)\!(\dd a)
\frac{\partial F(y)}{\partial y_\eta}\!\left[ \frac{\delta_a}{N} \right].
\end{aligned}
\ee
Next, for functions $F$ that are linear combinations of functions in 
\eqref{ak:multi-level-test-functions}, we have
\be{sk_15c_35}
N \frac{\partial F(y)}{\partial y_\eta}\!\left[ \frac{\delta_a}{N} \right] 
= \frac{\partial F(y)}{\partial y_\eta}[\delta_a].
\ee
On the time scale $Nt$, we have (recall that the upper index $[1]$ indicates time scale $N^1 t$)
\be{sk_15c_36}
(L_{\mathrm{mig}}^{(N,2)[1]} F)(y)
= \sum_{\eta \in G_{N,1}} c_1
\int_E \left( y_{\phi \eta}^{[1]} - y_{\eta} \right)\!(\dd a)
\frac{\partial F(y)}{\partial y_\eta}[\delta_a].
\ee

\myparagraph{Resampling part.}
The calculations proceed along the same lines as in Section~\ref{sss.1blockres-part-one}. 
Apart from an additional higher-order term, the main extension is that we consider 
$F(y_t) = F(y_t^{[1]}) = \langle \varphi,\bigotimes_{l=1}^q y_{\eta^{(l)}}^{\otimes n_l}
\rangle$ with $y=y^{[1]} = (y^{[1]}_\eta )_{\eta \in G_{N,1}}$, $\eta^{(l)} \in G_{N,1}$, 
$q \in \{1,\ldots,N\}$ and $n_l \in \N, 1 \leq l \leq q$, instead of 
restricting ourselves to test-functions of the form \eqref{sk_15c_10} 
(which corresponds to the case $q=1$). We will now use functions $F$ 
of the form
\be{sk_15c_38}
\begin{aligned}
&F(y) = \int_{E^{n_1 + \ldots + n_q}}
\left( \bigotimes_{l=1}^q y_{\eta^{(l)}}^{\otimes n_l}
\big(\dd u^{(l)}\big) \right) \varphi \big(u^{(1)},\ldots,u^{(q)}\big),
\quad y = (y_{\eta})_{\eta \in G_{N,1}} \in \CP(E)^N, \\
&q \in \{1,\ldots,N\},\,n_l \in \N, \eta^{(l)} \in
G_{N,1},\,l\in\{1,\ldots,q\},\, \\ &\eta^{(l)} \neq \eta^{(l')}
\text{, for all } l \neq l',\,u^{(l)} \in
E^{n_l},\, \varphi \in C_{\mathrm{b}}(E^{n_1+\ldots+n_q},\R).
\end{aligned}
\ee
The only difference with \eqref{ak:multi-level-test-functions} is the restriction of 
the ordering of the entries. This facilitates the notation in the computation below,
but is no loss of generality because the set of functions in \eqref{sk_15c_38} 
generates the same algebra $\CF$. We will now show that 
\be{sk_15c_47}
\begin{aligned}
&(L_{\mathrm{res}}^{(N,2)[1]} F)(y)\\
&= \sum_{m=1}^q \frac{1}{N} \sum_{ \xi\colon\, \phi \xi = \eta^{(m)} } 
\int_{[0,1]} \Lambda^*_0 (\dd r) \int_E x_{\xi}(\dd a) \frac{1}{2} 
\frac{\partial^2 F(y)}{\partial y_{{\eta}^{(m)}}^2}[r(- x_{\xi} + \delta_a),r(- x_{\xi} + \delta_a)] \\
&\quad + \sum_{\eta \in G_{N,1}} \int_{[0,1]} \Lambda^*_1 (\dd r)
\int_E y_{\eta}(\dd a) \left[ F\left( \Phi_{r,a,\eta}(y) \right)
- F(y) \right] + O\!\left( N^{-1} \right)
\end{aligned}
\ee
with $\Phi_{r,a,\eta}$ as in \eqref{ak:resampling-mapping-global}.

Recall the notation in \eqref{sk_15c_38} and set
\be{sk_15c_39}
L = \sum_{l=1}^q n_l.
\ee
Proceeding as in (\ref{sk_15c_10}-\ref{sk_15c_12}), we obtain
\be{sk_15c_40}
(L_{\mathrm{res}}^{(N,2)} F)(y)
= \frac{1}{N^L} \left( \bigotimes_{l=1}^q \bigotimes_{i=1}^{n_l} 
\sum_{\xi_i^l\colon\, \phi \xi_i^l = \eta^{(l)}} \right) L_{\mathrm{res}}\!
\left( F^{\left( \xi_1^1, \ldots, \xi_{n_q}^q \right)} \right)(x)
\ee
with $F^{\left( \xi_1^1, \ldots, \xi_{n_q}^q \right)}$ as in \eqref{sk_15c_11}.
As in Section~\ref{sss.1blockres-part-one}, we distinguish between the different cases for the 
structure of the set $\{\xi^1_1,\cdots,\xi^1_{n_q}\}$ and we obtain, using the definition 
of the resampling operator in \eqref{ak:multi-level-resampling-global}--\eqref{ak:resampling-mapping-global}, 
\be{sk_15c_41}
\begin{aligned}
&(L_{\mathrm{res}}^{(N,2)} F)(y)\\
&= \frac{1}{N^L} \left( \bigotimes_{l=1}^q \bigotimes_{i=1}^{n_l} 
\sum_{\xi_i^l\colon\,\phi \xi_i^l = \eta^{(l)}} \right) \sum_{\xi \in G_{N,2}}
\int_{[0,1]} \Lambda^*_0 (\dd r) \int_E x_\xi(\dd a)\\
&\quad \times
\left[F^{\left( \xi_1^1, \ldots, \xi_{n_q}^q \right)}\!
\left(\Phi_{r,a,B_0(\xi)}(x)\right)
-F^{\left( \xi_1^1, \ldots, \xi_{n_q}^q \right)}(x)\right] \\
&\quad + \frac{1}{N^L} \left( \bigotimes_{l=1}^q \bigotimes_{i=1}^{n_l} 
\sum_{\xi_i^l\colon\, \phi \xi_i^l
= \eta^{(l)}} \right) \sum_{\xi \in G_{N,2}} N^{-2}
\int_{[0,1]} \Lambda^*_1 (\dd r) \int_E y_{\xi,1}(\dd a)\\
&\quad \times
\left[F^{\left( \xi_1^1, \ldots, \xi_{n_q}^q \right)}\!
\left(\Phi_{r,a,B_1(\xi)}(x)\right)
-F^{\left( \xi_1^1, \ldots, \xi_{n_q}^q \right)}(x) \right] \\
&= I_0 + I_1.
\end{aligned}
\ee

For the first term $I_0$ in \eqref{sk_15c_41} we proceed along the lines of 
(\ref{sk_15c_14}--\ref{sk_15e_17}) to conclude that the only non-negligible 
contribution to the sum in $I_0$ comes from terms with $|\{ \xi_i^l,1 \leq l 
\leq q, 1 \leq i \leq n_l\}| = L-1$. It remains to investigate the terms
with $|\{\xi_i^l, 1 \leq l \leq q, 1 \leq i \leq n_l\}| = L-1$. Since $\phi \xi_i^l 
= \eta^{(l)}$, this implies that there exist $1 \leq m \leq q$ and $1 \leq m_1 
< m_2 \leq n_m$ such that $\xi_{m_1}^m=\xi_{m_2}^m$ and all other 
$\xi_i^l$ different. By the same reasoning as in \eqref{sk_15c_14}, we
see that the only non-zero contribution of the sum $\sum_{\xi \in G_{N,2}}$ 
comes from $\xi=\xi_{m_1}^m=\xi_{m_2}^m$. We therefore obtain
\be{sk_15c_42}
\begin{aligned}
I_0 =&  \frac{1}{N^L} \left( \bigotimes_{l=1}^q \bigotimes_{i=1}^{n_l} 
\sum_{\xi_i^l\colon\, \phi \xi_i^l
= \eta^{(l)}} \right) 1_{\left\{ \left| \left\{ \xi_i^l, 1 \leq l \leq q, 
1 \leq i \leq n_l \right\} \right| = L-1 \right\}}
\sum_{m=1}^q \sum_{1 \leq m_1 < m_2 \leq n_m} 
1_{\left\{ \xi_{m_1}^m=\xi_{m_2}^m=\xi \right\}} \\
&\quad \times \int_{[0,1]} \Lambda^*_0 (\dd r) \int_E x_\xi(\dd a)
\left[F^{\left( \xi_1^1, \ldots, \xi_{n_q}^q \right)}\!
\left(\Phi_{r,a,B_0(\xi)}(x)\right)
-F^{\left( \xi_1^1, \ldots, \xi_{n_q}^q \right)}(x)\right] + O\!\left( N^{-2} \right).
\end{aligned}
\ee
Now follow the reasoning from \eqref{sk_15c_15} to \eqref{sk_15c_19}, to get
\be{sk_15c_43}
\begin{aligned}
I_0 &= \frac{1}{N^2} \sum_{m=1}^q \sum_{ \xi\colon\, \phi \xi = \eta^{(m)} } 
\int_{[0,1]} \Lambda^*_0 (\dd r) \int_E x_{\xi}(\dd a) \frac{1}{2}
\frac{\partial^2 F(y)}{\partial y_{{\eta}^{(m)}}^2}[r(- x_{\xi} + \delta_a),r(- x_{\xi} + \delta_a)] \\
&\quad + O\!\left( N^{-2} \right).
\end{aligned}
\ee

For the second term $I_1$ in \eqref{sk_15c_41}, we obtain, by the definition of 
$\Phi_{r,a,B_1(\xi)}(x)$ in \eqref{ak:resampling-mapping-global} and using 
\eqref{sk_15c_33},
\be{sk_15c_44}
\begin{aligned}
I_1 =& \frac{1}{N^L} \left( \bigotimes_{l=1}^q \bigotimes_{i=1}^{n_l} 
\sum_{\xi_i^l\colon\, \phi \xi_i^l = \eta^{(l)}} \right) \sum_{\xi \in G_{N,2}} N^{-2}
\int_{[0,1]} \Lambda^*_1 (\dd r) \int_E y_{\xi,1}(\dd a)\\
&\quad \times
\left[F^{\left( \xi_1^1, \ldots, \xi_{n_q}^q \right)}\!
\left(\Phi_{r,a,B_1(\xi)}(x)\right)
-F^{\left( \xi_1^1, \ldots, \xi_{n_q}^q \right)}(x) \right] \\
=& \frac{1}{N^L} \left( \bigotimes_{l=1}^q \bigotimes_{i=1}^{n_l} 
\sum_{\xi_i^l\colon\, \phi \xi_i^l
= \eta^{(l)}} \right) \sum_{\eta \in G_{N,1}} N^{-1}
\int_{[0,1]} \Lambda^*_1 (\dd r) \int_E y_\eta(\dd a)\\
&\quad \times
\left[F^{\left( \xi_1^1, \ldots, \xi_{n_q}^q \right)}\!
\left(\Phi_{r,a,\eta}^{[1]}(x)\right)
-F^{\left( \xi_1^1, \ldots, \xi_{n_q}^q \right)}(x) \right]
\end{aligned}
\ee
with
\be{sk_15c_45}
\left[ \Phi_{r,a,\eta}^{[1]}(x) \right]_\xi
= \begin{cases}
(1-r) y_{\eta} + r \delta_a, & \phi \xi = \eta, \\
x_{\xi}, & \text{otherwise}.
\end{cases}
\ee
Now observe that the sum $\sum_{\eta \in G_{N,1}}$ in \eqref{sk_15c_44} yields 
non-zero contributions only for $\eta \in \left\{ \eta^{(1)}, \ldots, \eta^{(q)} \right\}$, 
and so we can rewrite $I_1$ as
\be{sk_15c_46}
\begin{aligned}
I_1 =& \frac{1}{N^L} \left( \bigotimes_{l=1}^q \bigotimes_{i=1}^{n_l} 
\sum_{\xi_i^l\colon\, \phi \xi_i^l
= \eta^{(l)}} \right) \sum_{l=1}^q N^{-1}
\int_{[0,1]} \Lambda^*_1 (\dd r) \int_E y_{\eta^{(l)}}(\dd a)\\
& \times \left[ \left\langle \varphi , x_{\xi_1^1} \otimes \cdots
\otimes x_{\xi_{n_{l-1}}^{l-1}} \otimes
\underbrace{\left( (1-r) y_{\eta^{(l)}} 
+ r \delta_a \right)}_{\text{change from position }
\xi_1^l}\right.\right.\\
&\left.\left.\qquad \otimes \cdots \otimes \underbrace{\left( (1-r) y_{\eta^{(l)}}
+ r \delta_a \right)}_{\text{to position } \xi_{n_l}^l} \otimes x_{\xi_1^{l+1}} 
\otimes \cdots \otimes x_{\xi_{n_q}^q} \right\rangle 
- \left\langle \varphi, x_{\xi_1^1} \otimes \cdots \otimes x_{\xi_{n_q}^q}
\right\rangle \right]\\
=& \sum_{l=1}^q N^{-1} \int_{[0,1]} \Lambda^*_1 (\dd r)
\int_E y_{\eta^{(l)}}(\dd a)\\
& \times \left[ \left\langle \varphi , y_{\eta^{(1)}}^{\otimes n_1} 
\otimes \cdots \otimes
y_{\eta^{(l-1)}}^{\otimes n_{l-1}} \otimes \left( (1-r) y_{\eta^{(l)}}
+ r \delta_a \right)^{\otimes n_l} 
\otimes y_{\eta^{(l+1)}}^{\otimes n_{l+1}}\right.\right.\\
&\left.\left.\qquad \otimes \cdots \otimes y_{\eta^{(q)}}^{\otimes n_q} \right\rangle
- \left\langle \varphi,\bigotimes_{l=1}^q y_{\eta^{(l)}}^{\otimes n_l} 
\right\rangle \right]\\
=& \sum_{\eta \in G_{N,1}} N^{-1} \int_{[0,1]} \Lambda^*_1 (\dd r)
\int_E y_{\eta}(\dd a) \left[ F\left( \Phi_{r,a,\eta}(y) \right)
- F(y) \right].
\end{aligned}
\ee
Combining \eqref{sk_15c_41}, \eqref{sk_15c_43} and \eqref{sk_15c_46}, we obtain 
\eqref{sk_15c_47} on time scale $Nt$.

\myparagraph{Additional Fleming-Viot part.}
We proceed as with the migration operator (recall that in the present Section~\ref{s.hiermod}
we added a Fleming-Viot term to the generator, i.e., we 
consider the case $d_0>0$) and write
\be{sk_15c_48}
\begin{aligned}
(L_{\mathrm{FV}}^{(N,2)} F)(y)
&= \left( L_{\mathrm{FV}}^{(N,2)} \!\left( F \circ y \right) \right)\!(x) \\
&= d_0 \sum_{\xi \in G_{N,2}} \int_E \int_E Q_{x_\xi}(\dd u,\dd v)\,
\frac{\partial^2\!\left( F \circ y \right)\!(x)}
{\partial x_\xi^2}[\delta_u,\delta_v],
\end{aligned}
\ee
with $Q_{x_\xi}$ as in \eqref{flemming-viot-kernel} and
where the definition of $y=y^{[1]}$ in \eqref{sk_15c_32} yields
\be{sk_15c_49}
\frac{\partial^2 \left( F \circ y \right)\!(x)}
{\partial x_\xi^2}[\delta_u,\delta_v]
= \frac{\partial^2 F(y)}{\partial y_{\phi \xi}^2}\!\left[ \frac{\delta_u}{N}, \frac{\delta_v}{N} \right].
\ee
Hence, on time scale $Nt$,

\be{sk_15c_50}
\begin{aligned}
(L_{\mathrm{FV}}^{(N,2)[1]} F)(y)
&= d_0 N \sum_{\eta \in G_{N,1}} \sum_{\xi\colon\,\phi \xi = \eta} 
\int_E \int_E Q_{x_\xi}(\dd u,\dd v)\,
\frac{\partial^2 F(y)}{\partial y_\eta^2}\!
\left[ \frac{\delta_u}{N}, \frac{\delta_v}{N} \right] \\
&= d_0 \sum_{\eta \in G_{N,1}} \frac{1}{N} \sum_{\xi\colon\,\phi \xi = \eta} 
\int_E \int_E Q_{x_\xi}(\dd u,\dd v)\,
\frac{\partial^2 F(y)}{\partial y_\eta^2}[\delta_u, \delta_v],
\end{aligned}
\ee
where in the last line we use that, for $F$ a linear combination of the functions 
in \eqref{ak:multi-level-test-functions},
\be{sk_15c_51}
N^2 \frac{\partial^2 F(y)}{\partial y_\eta^2}\!
\left[ \frac{\delta_u}{N}, \frac{\delta_v}{N} \right] 
= \frac{\partial^2 F(y)}{\partial y_\eta^2}[\delta_u, \delta_v].
\ee

\myparagraph{The resulting generator.}
Combining the migration \eqref{sk_15c_36}, resampling \eqref{sk_15c_47} and 
Fleming-Viot \eqref{sk_15c_50} parts for the $1$-block averages on time scale $Nt$, 
we obtain (\ref{sk_15c_52}). This completes the proof of \eqref{ak16}.

\subsubsection{The total average on time scale $N^2t$}
\label{sss.totalN2t}

Denote the total average by (recall $y_\eta^{[1]}$ from \eqref{sk_15c_32})
\be{add4}
z = N^{-1} \sum_{\eta \in G_{N,1}} y_\eta^{[1]} = N^{-2} 
\sum_{\xi \in G_{N,2}} x_\xi.
\ee
(This is a $2$-block average because we are considering the case $K=2$.) 
Recall notation \eqref{add20}.
We must prove: (1) the sequence of laws $\{\CL [(Z^{(N)} (t N^2))_{t \geq 0},\,N \in \N\}$ is {\em tight} 
in path space; (2) the weak limit points of this sequence are solutions of the
martingale problem for $Z^{0,d_2,0}_\theta$ (cf. \eqref{ak17}) by showing 
\eqref{ak4b} (recall Section~\ref{ss.embedding}). 
From the uniqueness of the solution to the martingale problem, we get the claim. 

We now verify these points by calculating the generator.
Recall the ${}^{[k]}$-notation from Section~\ref{sss.1blockres-part-one} for
the rescaled generators.

\myparagraph{Migration part.}
For the total average, the migration operator can be obtained from \eqref{sk_15c_36} 
by writing $z=z(y)$ and using the analogue to \eqref{sk_15c_6}, (cf., \eqref{add22b} 
for the definition of $\phi$)
\be{sk_15e_1}
(L_{\mathrm{mig}}^{(N,2)[1]} F)(z)
= (L_{\mathrm{mig}}^{(N,2)[1]} (F \circ z))(y)
= \sum_{\eta \in G_{N,1}} c_1
\int_E \left( y_{\phi \eta}^{[1]} - y_{\eta} \right)\!(\dd a)
\frac{\partial F(z)}{\partial z}\left[ \frac{\delta_a}{N} \right].
\ee
Using that $z=y_{\phi \eta}^{[1]} = N^{-1} \sum_{\eta \in G_{N,1}} y_\eta^{[1]}$, for all 
$\eta \in G_{N,1}$, we get
\be{sk_15e_2}
(L_{\mathrm{mig}}^{(N,2)[1]} F)(z) = (L_{\mathrm{mig}}^{(N,2)[2]} F)(z) = 0.
\ee

\myparagraph{Resampling part.}
Consider $F(z)= \langle \varphi,z^{\otimes n} \rangle$. Follow the derivation of 
\eqref{sk_15c_12} to obtain
\be{sk_15e_6}
(L_{\mathrm{res}}^{(N,2)} F)(z)
= \frac{1}{N^n} \left( \bigotimes_{i=1}^n \sum_{\eta_i \in G_{N,1}} \right) 
L_{\mathrm{res}}^{(N)}\!\left( F^{\left( \eta_1, \ldots, \eta_n \right)} \right)(y)
= I_0' + I_1'
\ee
with $F^{(\eta_1,\ldots, \eta_n)}(y) = \langle \varphi,\bigotimes_{i=1}^n y_{\eta_i}
\rangle$ as in \eqref{sk_15c_11}, where we recall from \eqref{sk_15c_41} that
\be{sk_15e_7}
\begin{aligned}
&(L_{\mathrm{res}}^{(N,2)} F^{(\eta_1,\ldots, \eta_n)})(y)\\
&= \frac{1}{N^n} \left( \bigotimes_{l=1}^n \sum_{\xi_l\colon\,\phi \xi_l
= \eta_l} \right) \sum_{\xi \in G_{N,2}}
\int_{[0,1]} \Lambda^*_0 (\dd r) \int_E x_\xi(\dd a)\\
&\quad \times
\left[F^{\left( \xi_1, \ldots, \xi_n \right)}\!
\left(\Phi_{r,a,B_0(\xi)}(x)\right)
-F^{\left( \xi_1, \ldots, \xi_n \right)}(x)\right]\\
&\quad + \frac{1}{N^n} \left( \bigotimes_{l=1}^n \sum_{\xi_l\colon\,\phi \xi_l
= \eta_l} \right) \sum_{\xi \in G_{N,2}} N^{-2}
\int_{[0,1]} \Lambda^*_1 (\dd r) \int_E y_{\xi,1}(\dd a)\\
&\quad \times
\left[F^{\left( \xi_1, \ldots, \xi_n \right)}\!
\left(\Phi_{r,a,B_1(\xi)}(x)\right)-F^{\left( \xi_1^1, \ldots, \xi_n \right)}(x) \right]\\
&= I_0'' + I_1''
\end{aligned}
\ee
with $\Phi_{r,a,B_k(\xi)}$ as in \eqref{ak:resampling-mapping-global}.

Let us begin with the second term $I_1''$ in \eqref{sk_15e_7}, which corresponds to $I_1$ 
in \eqref{sk_15c_41} and was rewritten in (\ref{sk_15c_44}--\ref{sk_15c_46}) as
\be{sk_15e_8}
I_1'' = \sum_{\eta \in G_{N,1}} N^{-1} \int_{[0,1]} \Lambda^*_1 (\dd r)
\int_E y_{\eta}(\dd a) \left[ F^{(\eta_1,\ldots, \eta_n)}\left( \Phi_{r,a,\eta}(y) \right)
- F^{(\eta_1,\ldots, \eta_n)}(y) \right].
\ee
Combine \eqref{sk_15e_6} and \eqref{sk_15e_8}, change to timescale $Nt$ and compare the 
result to \eqref{sk_15c_13}. We obtain that $I_1'$ on time scale $Nt$ behaves analogously 
to \eqref{sk_15c_13} on time scale $t$. By moving one time scale upwards, we obtain as 
in \eqref{sk_15c_20} (respectively, \eqref{sk_15c_53} with $d_1=\frac{c_0(\lambda_0+2d_0)}
{2c_0+\lambda_0+2d_0}>0$) that
\be{sk_15e_9}
\lim_{N \rightarrow \infty} (I_1')^{[2]} 
= \frac{c_1 \lambda_1}{2 c_1 + \lambda_1 + 2d_1} \int_E \int_E Q_z(\dd u,\dd v)\,
\frac{\partial^2 F(z)}{\partial z^2}[\delta_u, \delta_v].
\ee

The term $I_0'$ can be handled in the same spirit as $I_0$ in \eqref{sk_15c_41}. To obtain 
non-zero contributions in $I_0''$, we need to have $|\{ \xi_l, \phi \xi_l = \eta_l, 1 \leq 
l \leq n \}|<n$ (recall \eqref{sk_15c_14}). This is possible only if $|\eta_1,\ldots,\eta_n|
<n$. Reasoning similarly as in \eqref{sk_15e_17}, we obtain negligible terms if 
$|\{\xi_l, \phi \xi_l = \eta_l, 1 \leq l \leq n \}|<n-1$. Indeed, two sites residing in a common $1$-block
already result in a factor of $O(N^{-2})$ (on time scale $t$): first a common block 
has to be chosen ($|\eta_1,\ldots, \eta_n|=n-1$), which contributes a factor $N^{-2} 
\sum_{\eta \in G_{N,1}}$, and subsequently a common site has to be chosen, which 
contributes a factor $N^{-2} \sum_{\xi: \phi \xi=\eta}$. Any additional choice results in 
terms that vanish for $N \to \infty$ on time scale $N^2t$. Consequently, we  can reason 
as in (\ref{sk_15c_15}--\ref{sk_15c_19}) to obtain on time scale $t$
\be{sk_15e_11}
\begin{aligned}
(I_0')^{[0]} &= \frac{1}{N^2} \sum_{\eta \in G_{N,1}} \frac{1}{N^2} 
\sum_{\xi\colon\,\phi \xi=\eta} \int_{[0,1]} \Lambda^*_0 (\dd r) \\
&\quad \times \int_E x_{\xi}(\dd a) \frac{1}{2}
\frac{\partial^2 F(z)}{\partial z^2}
[r(- x_{\xi} + \delta_a),r(- x_{\xi} + \delta_a)] + O\!\left( N^{-3} \right).
\end{aligned}
\ee

\myparagraph{Additional Fleming-Viot part.}
We proceed as for the migration operator. Recall \eqref{sk_15c_50}, to get
\be{sk_15e_3}
\begin{aligned}
(L_{\mathrm{FV}}^{(N,2)[1]} F)(z)
&= d_0 \sum_{\eta \in G_{N,1}} \frac{1}{N} 
\sum_{\xi\colon\, \phi \xi = \eta} \int_E \int_E Q_{x_\xi}(\dd u,\dd v)\,
\frac{\partial^2 (F \circ z)(y)}{\partial y_\eta^2}[\delta_u, \delta_v].
\end{aligned}
\ee
Now use the analogue to \eqref{sk_15c_49}, to obtain
\be{sk_15e_4}
\begin{aligned}
(L_{\mathrm{FV}}^{(N,2)[1]} F)(z)
&= d_0 \sum_{\eta \in G_{N,1}} \frac{1}{N} 
\sum_{\xi\colon\,\phi \xi = \eta} \int_E \int_E Q_{x_\xi}(\dd u,\dd v)\,
\frac{\partial^2 F(z)}{\partial z^2}\!
\left[ \frac{\delta_u}{N}, \frac{\delta_v}{N} \right].
\end{aligned}
\ee
After changing to time scale $N^2t$, we have
\be{sk_15e_5}
\begin{aligned}
(L_{\mathrm{FV}}^{(N,2)[2]} F)(z)
&= d_0 \frac{1}{N} \sum_{\eta \in G_{N,1}} \frac{1}{N} 
\sum_{\xi\colon\,\phi \xi = \eta} \int_E \int_E Q_{x_\xi}(\dd u,\dd v)\,
\frac{\partial^2 F(z)}{\partial z^2}[\delta_u, \delta_v].
\end{aligned}
\ee

\myparagraph{Tightness.}
We have to bound the generator, i.e., show that $\sup_N \| L^{(N,2)[1]} (F)\| 
< \infty$, in order to apply the tightness criterion, as explained in 
Section~\ref{ss.embedding}. (Recall that the upper index $[1]$ indicates time scale 
$N^1 t$ and that the upper index $2$ indicates that we consider $K=2$ levels.) 
This we read off from \eqref{sk_15e_2}, \eqref{sk_15e_7}, \eqref{sk_15e_8}, 
\eqref{sk_15e_11} and \eqref{sk_15e_5}.

\myparagraph{Convergence to McKean-Vlasov process.}
We have to identify the limiting generator. One approach would be to try and make 
the following heuristics rigorous.

\medskip\noindent
{\em Begin heuristics.}
On time scale $N^2t$, we obtain, by reasoning as in \eqref{sk_15c_53}, 
using \eqref{sk_15e_11}, now on time scale $tN^2$, together with 
\eqref{sk:preparation-second-moment-invariant-measure-1d} in the second and fourth equation,
\be{sk_15e_12}
\begin{aligned}
\lim_{N \rightarrow \infty}(I_0')^{[2]}
&= \frac{\lambda_0}{2} \lim_{N \rightarrow \infty} \frac{1}{N} 
\sum_{\eta \in G_{N,1}} \int_{\CP(E)} 
\nu_{y_\eta}^{c_0,d_0,\Lambda_0}(\dd x) \int_E \int_E Q_x(\dd u,\dd v)\,
\frac{\partial^2 F(z)}{\partial z^2}[\delta_u, \delta_v] \\
&= \frac{c_0 \lambda_0}{2 c_0 + \lambda_0 + 2d_0} \lim_{N \rightarrow \infty} 
\frac{1}{N} \sum_{\eta \in G_{N,1}} \int_E \int_E Q_{y_\eta}(\dd u,\dd v)\,
\frac{\partial^2 F(z)}{\partial z^2}[\delta_u, \delta_v] \\
&= \frac{c_0 \lambda_0}{2 c_0 + \lambda_0 + 2 d_0} \int_{\CP(E)} 
\nu_z^{c_1,d_1,\Lambda_1}(\dd y) \int_E \int_E Q_y(\dd u,\dd v)\,
\frac{\partial^2 F(z)}{\partial z^2}[\delta_u, \delta_v] \\
&= \frac{2 c_1}{2 c_1 + \lambda_1 + 2 d_1} 
\frac{c_0 \lambda_0}{2 c_0 + \lambda_0 + 2 d_0} \int_E \int_E Q_z(\dd u,\dd v)\,
\frac{\partial^2 F(z)}{\partial z^2}[\delta_u, \delta_v].
\end{aligned}
\ee
Combine \eqref{sk_15e_9} with \eqref{sk_15e_12}, to get from \eqref{sk_15e_6}
\be{sk_15e_13}
\begin{aligned}
&\lim_{N \rightarrow \infty} (L_{\mathrm{res}}^{(N,2)[2]} F)(z)\\
&= \frac{2 c_1}{2 c_1 + \lambda_1 + 2 d_1} 
\left( \frac{\lambda_1}{2} + \frac{c_0 \lambda_0}{2 c_0 + \lambda_0 + 2 d_0} \right) 
\int_E \int_E Q_z(\dd u,\dd v)\,
\frac{\partial^2 F(z)}{\partial z^2}[\delta_u, \delta_v].
\end{aligned}
\ee
For the Fleming-Viot part in \eqref{sk_15e_5}, we obtain, by reasoning once more as in \eqref{sk_15c_53}, using \eqref{sk:preparation-second-moment-invariant-measure-1d},
\be{sk_15e_14}
\begin{aligned}
&\lim_{N \rightarrow \infty}(L_{\mathrm{FV}}^{(N,2)[2]} F)(z)\\
&= d_0 \lim_{N \rightarrow \infty} \frac{1}{N} \sum_{\eta \in G_{N,1}} 
\int_{\CP(E)} \nu_{y_\eta}^{c_0,d_0,\Lambda_0}(\dd x) \int_E \int_E Q_x(\dd u,\dd v)\,
\frac{\partial^2 F(z)}{\partial_z^2}[\delta_u, \delta_v] \\
&= \frac{2 c_0 d_0}{2 c_0 + \lambda_0 + 2d_0} \lim_{N \rightarrow \infty} 
\frac{1}{N} \sum_{\eta \in G_{N,1}} \int_E \int_E Q_{y_\eta}(\dd u,\dd v)\,
\frac{\partial^2 F(z)}{\partial_z^2}[\delta_u, \delta_v] \\
&= \frac{2 c_0 d_0}{2 c_0 + \lambda_0 + 2 d_0} \int_{\CP(E)} 
\nu_z^{c_1,d_1,\Lambda_1}(\dd y) \int_E \int_E Q_y(\dd u,\dd v)\,
\frac{\partial^2 F(z)}{\partial_z^2}[\delta_u, \delta_v] \\
&= \frac{2 c_1}{2 c_1 + \lambda_1 + 2 d_1} \frac{2 c_0 d_0}{2 c_0 + \lambda_0 + 2 d_0} 
\int_E \int_E Q_z(\dd u,\dd v)\,
\frac{\partial^2 F(z)}{\partial_z^2}[\delta_u, \delta_v].
\end{aligned}
\ee
Collecting the limiting terms as $N \to \infty$ on time scale $N^2t$ for migration 
\eqref{sk_15e_2}, resampling \eqref{sk_15e_13} and Fleming-Viot \eqref{sk_15e_14},
we obtain
\be{sk_15e_15}
\begin{aligned}
&\lim_{N \rightarrow \infty}(L^{(N,2)[2]} F)(z)\\
&= \frac{2 c_1}{2 c_1 + \lambda_1 + 2 d_1} 
\left( \frac{\lambda_1}{2} 
+ \frac{c_0 \lambda_0 + 2 c_0 d_0}{2 c_0 + \lambda_0 + 2 d_0} \right) 
\int_E \int_E Q_z(\dd u,\dd v)\,
\frac{\partial^2 F(z)}{\partial_z^2}[\delta_u, \delta_v].
\end{aligned}
\ee
In order to obtain the convergence in (\ref{sk_15e_12}--\ref{sk_15e_14}), we would need 
to restrict the set of configurations, argue that the law of the process lives on that
set of configurations, and show that therefore the compensators of the martingale problems 
converge to the compensator of the limit process. However, it is technically easier to 
follow a different route, as we do below. {\em End heuristics.}

\medskip
We want to view the expression for the generator of the total average on time scale 
$tN^2$ with $K=2$ levels, $(L^{(N,2),[2]} F)(z)$, as an average 
over $N$ different $1$-block averages. 
If we replace the $\big(L^{(N,2)[1]}_{\mathrm{res},0} + L^{(N,2)[1]}_{\mathrm{FV}} \big)$-part
of the $1$-block averages (cf.\ \eqref{sk_15c_52}) by a system of $N$ exchangeable Fleming-Viot 
diffusions with resampling constant $d_1$ (for which we have a formula in terms of 
$c_0, d_0$ and $\lambda_0$, cf.\ \eqref{ak16}), which on time scale $Nt$ lead to the generator
\be{add25}
L^{(N,2),[1]}_{\mathrm{mig}} (F)(y)
+ \frac{c_0 (\lambda_0 +2 d_0)}{2c_0+\lambda_0 +2d_0} \intl_E \intl_E Q_{y} (\dd u,\dd v)
\frac{\partial^2 F(y)}{\partial y^2} [\delta_u, \delta_v]
+ (L^{(N,2),[1]}_{\mathrm{res},1} F)(y),
\ee
then we can apply the analysis of Section \ref{s.mflambda} to this new collection of 
processes, denoted by
\be{add29}
\left\{\wt Y^{(N)}_i (tN)\colon i=1,\ldots,N\right\},
\ee
to conclude that  on time scale $tN^2$ the block average 
$\wt Z^{(N)} (tN) =N^{-1} \suml^N_{i=1} \wt Y^N_i(Nt)$ 
satisfies,
\be{add26}
\CL[(\wt Z^{(N)} (tN^2))_{t \geq 0}] \mathop{\Longrightarrow}_{N \to \infty} 
\CL [(\wt Z(t))_{t \geq 0}],
\ee
where $\wt Z$ is a Fleming-Viot diffusion with resampling constant
\be{add27}
\frac{c_1}{2c_1 +\lambda_1 +2 d_1} (\lambda_1 + 2d_1), \mbox{ where } d_1
= \frac{c_0(\lambda_0+2 d_0)}{2c_0+\lambda_0 + 2 d_0}. 
\ee
Hence, we obtain a limit process with a generator acting on $F$ as
\be{sk_15e_16}
\frac{c_1 (\lambda_1 + 2d_1)}{2 c_1 + \lambda_1 + 2 d_1}
 \int_E \int_E Q_z(\dd u,\dd v)\,
\frac{\partial^2 F(z)}{\partial_z^2}[\delta_u, \delta_v].
\ee
Hence, the weak limit points of the laws $\{\CL [(\wt Z^{(N)} (t N^2))_{t \geq 0}],\,N \in \N\}$
satisfy the martingale problem with generator $(L_{\theta}^{0,d_2,0} G)(z)$ with 
$d_2 = \frac{c_1 (\lambda_1 + 2d_1)}{2 c_1 + \lambda_1 + 2 d_1}$. 

Since we know that the martingale problem for the generator $L^{0,d_2,0}$ and for 
the test functions given in \eqref{ak:multi-level-test-functions} is well-posed 
(recall Proposition~\ref{prop:McKean-Vlasov-well-posedness}), we have the claimed 
convergence in (\ref{ak17}) on path space if $Z$ (a weak limit point for the original 
problem) and $\wt Z$ agree. Thus, we have to argue that it is legitimate to
\be{add27b}
\mbox{replace }
\{((Y^{(N)}_i (Nt))_{i=1,\ldots,N})_{t \geq 0}\} 
\mbox{ by }
\{(\wt Y^{(N)}_i (Nt)_{i=1,\ldots,N})_{t \geq 0}\}.
\ee

For that purpose, observe that we know from Section~\ref{s.mflambda} that, for a suitable 
subsequence along which $\CL[(Z^{(N)}(sN^2))_{s \geq 0}]$ converges to $Z(s)$,
\be{add28}
\CL [((Y^{(N)}_i (N^2 s+Nt))_{i=1,\ldots,N})_{t \geq 0}]
\mathop{\Longrightarrow}_{N \to \infty}
\CL[((Y^{(\infty)}_i(s,t))_{i \in \N})_{t \geq 0}],
\ee
where the right-hand side is the McKean-Vlasov process with Fleming-Viot part at rate 
$d_1$, Cannings part $\Lambda_1$, and immigration-emigration at rate $c_1$ from the 
random source $Z(s)$. We need to argue that the latter implies that $Z$ and $\wt Z$ 
agree.

For $F \in C^2_\mathrm{b} (\CP(E),\R)$, define $G_N \in C^2_\mathrm{b} ((\CP(E))^N, \R)$ 
and $H_N \in C^2_\mathrm{b} ((\CP(E))^{N^2}, \R)$ by
\be{add30b}
F(z) = G_N(y) = H_N(x), \quad x \in (\CP(E))^{N^2},
\quad y \in (\CP(E))^N, \quad z \in \CP(E),
\ee
with
\be{add31}
z=\frac{1}{N} \sum_{i \in \{1,\ldots,N\}} y_i, \quad y_i =\frac{1}{N} 
\suml_{j \in \{1,\ldots,N\}} x_{j,i}.
\ee
In order to verify that $Z$ and $\wt Z$ agree, it suffices to show that the compensator 
processes for $\wt Z$ and $Z$ agree for a measure-determining family of functions
$F \in C^2_\mathrm{b} (\CP(E),\R)$, namely,
\be{add32}
\begin{aligned}
&\CL \Big[\Big( \int^{tN^2}_0 \dd s\, 
\Big[\int_{E \times E}  d_1 \sum^N_{i=1} Q_{y_i (s)} (\dd u,\dd v)
\frac{\partial^2 G_N(y (s))}{\partial y_i^2} [\delta_u, \delta_v]
 + L^{(N,2)[1]}_{\mathrm{res},1} G_N (y (s))\Big]\Big)_{t \geq 0}\Big]\\ 
&\quad 
- \CL\Big[ \Big(\int^{t N^2}_0 \dd s\, \Big[L^{(N,2)[1]}_{\mathrm{res},1} G_N (y_j(s)) 
+ \frac{1}{N^2} \sum^N_{i=1} \sum^N_{j=1} 
\left( L^{(N,2)}_{\mathrm{res},0} + L^{(N,2)}_{\mathrm{FV}} \right) H_N (x_{j,i}(s)) \Big]\Big)_{t \geq 0}\Big]\\
&\quad\quad\quad \mathop{\Longrightarrow}_{N \to \infty} \mbox{ Zero measure}.
\end{aligned}
\ee
To that end, first note that the two terms with $L^{(N,2),[1]}_{\mathrm{res},1}$ cancel
each other out. Regarding the remaining terms, after we transform $s$ to $sN^2$, we must
show that for each $s\in [0,t]$ the term in the second line converges weakly to the term 
in the first line (the joint law of the density and the empirical measure converges).
When worked out in detail, this requires a somewhat subtle argument. However, nothing 
is specific to our model: a detailed argument along these lines can be found in \cite{DGV95}, 
pp.\ 2322-2339.

\subsection{Finite-level systems}
\label{ss.K>3}

The next step is to consider general $K \geq 3$ (recall the beginning of 
Section~\ref{s.hiermod}). We can copy the arguments used for $K=2$, and then
argue recursively. Namely, we can view the $(j-1),j,(j+1)$-block averages as a 
{\em two-level system} on time scales $tN^{j-1},N(tN^{j-1}),N^2(tN^{j-1})$. The 
limit as $N\to\infty$ is a two-level system with migration rates $c_{j-1},c_j,c_{j+1}$ 
instead of $c_0,c_1,c_2$, resampling measures $\Lambda_{j-1},\Lambda_j,
\Lambda_{j+1}$ instead of $\Lambda_0,\Lambda_1,\Lambda_2$, and volatility 
$\dd_{j-1}$ instead of  $\dd_0$. If we would have $c_0=c_1 = \cdots = c_{j-2} = 0$ 
and $\lambda_0 = \cdots  = \lambda_{j-2} =0$, then this would be literally the 
case. Hence, the key point is to show that the lower-order perturbation terms 
play no role in the renormalised dynamics after they have played their role in 
determining the coefficients $d_{j-1}, d_j,d_{j+1}$.

The argument has again a tightness part, which is the same as before and which 
we do not discuss, and a finite-dimensional distributions part. Since the solution of 
the martingale problem is uniquely determined by the marginal distributions (see 
\cite[Theorem~4.4.2]{EK86}), this part is best based on duality, which determines 
the transition kernel of the process as follows.

We have to verify that the dual of the $(j+1)$-level system on the time scales 
$N^{j-1}t,N^j t$ behaves like the dual process of a two-level system. This means
that the dual process can be replaced by the system where the locations up to level 
$j-2$ are uniformly distributed and all partition elements originally within that 
distance have coalesced. This can be obtained by showing that the dual system with 
the lower-order terms is instantaneously uniformly distributed in small balls,
and that within that distance coalescence is instantaneous, since we are working 
with times at least $tN^{j-1}$. Therefore, the dynamics as $N \to \infty$ results
{\em effectively} in a coalescent corresponding to a two-level system.

\section{Proof of the hierarchical mean-field scaling limit}
\label{s.proof1stth}

We are finally ready to prove Theorem~\ref{mainth}. Recall 
the $C_N^{\uc,\uL}$-process on $\Omega_N$, denoted $X^{(\Omega_N)}$ 
from Section~\ref{sss.hierarCan} and \eqref{k-block-average}. Also recall 
the discussion on convergence criteria from Section~\ref{ss.embedding}. 
We establish the tightness by checking the bound on the
generator action. Having Section \ref{s.hiermod}, all we need is to show that the
higher-order term action on monomials is bounded in $N$ in the considered time
scale. This is readily checked from the explicit form of the terms. In order to
show convergence of the finite dimensional distribution,  we approximate our
infinite spatial system by finite spatial systems of the type studied in 
Section~\ref{s.hiermod}. As before, we denote the finite system with geographic
space $G_{N,K}$ by $X^{(N,K)}$ and the one with $G=\Omega_N$ by
$X^{(\Omega_N)}$.

\begin{proposition}
\label{P.Kapprox}
{\bf [$K$-level approximation]}\\
For $t \in (0,\infty)$ and $s_N \in (0,\infty)$ with $\lim_{N\to\infty} s_N = \infty$ 
and $\lim_{N\to\infty} s_N/N = 0$, consider the $k$-block averages
$Y^{(\Omega_N)}_{\xi,k}$ and 
$Y^{(N,K)}_{\xi,k}$ on time scale $tN^j + s_N N^k$ for $0 \leq k \leq j<K$. 
Then
\be{ak18}
d_{\mathrm{Prokh}} 
\left(\CL\left[\left(Y^{(\Omega_N)}_{\xi,k} (tN^j + s_N N^k)\right)\right],
\CL\left[\left(Y^{(N,K)}_{\xi,k}(tN^j + s_N N^k)\right]\right)\right)
\mathop{\Longrightarrow}_{N \to \infty} 0,
\ee
where $d_{\mathrm{Prokh}}$ is the Prokhorov metric.
\end{proposition}

Once we have proved this proposition, we obtain Theorem~\ref{mainth} by observing that 
(\ref{ak18}) allows us to replace our system on $\Omega_N$ by the one on $G^{N,K}$ 
when we are interested only in block averages of order $\leq K$ on time scales of order 
$<N^K$. In that case, we can use the result of Section~\ref{s.hiermod} to obtain the 
claim of the theorem for $(j,k)$ with $k \leq j <K$. Thus, it remains only to prove 
Proposition~\ref{P.Kapprox}. We give the proof for $K=2$, and later indicate how to 
extend it to $K \in \N$.

The main idea is the following. We want to compare the laws of the solution of two 
martingale problems at a fixed time and show that their difference goes to zero in 
the weak topology. To this end, it suffices to show that the difference of the action 
of the two generators in the martingale problems on the functions in the algebra $\CF$ 
tends to zero. Indeed, we then easily get the claim with the help of the formula of 
partial integration for two semigroups $(V_t)_{t \geq 0}$ and $(U_t)_{t \geq 0}$ (see, 
e.g., Ethier and Kurtz~\cite[Section 1, (5.19)]{EK86}):
\be{add32b}
V_t = U_t +\int^t_0 U_{t-s} (L_V-L_U) V_s \dd s.
\ee

In Sections~\ref{sss.sts}--\ref{sss.avlts}, we calculate and asymptotically evaluate 
the difference of the generator acting on $\CF$ on the two spatial and temporal scales.

\subsection{The single components on time scale $t$}
\label{sss.sts}

For an $F \in \mathcal{F}$ (cf.\ \eqref{ak:multi-level-test-functions}) that depends only on 
$\{x_\xi$, $\xi \in B_1(0)\}$ (cf., \eqref{block-definition}), we have (as we will see below)
\be{ak19}
(L^{(\Omega_N)}F)(x) = (L^{(N,2)} F)(x) + (L^{\mathrm{err}} F) (x) ,
\ee
where $\Vert L^{\mathrm{err}} \Vert = O(N^{-1})$ ($\Vert \cdot \Vert$ is the operator norm 
generated by the $\sup$-norm). By the formula of partial integration for semigroups, it follows 
that 
\be{add23}
\left| \E \left[F(X^{(\Omega_N)}(t))\right] 
- \E \left[F(X^{(N,2)} (t))\right] \right| \leq t O(N^{-1}).
\ee
Since our test functions are measure-determining, the claim follows for any finite 
time horizon. To prove \eqref{ak19}, we discuss the different parts of the generators 
separately.

Consider the migration operator in \eqref{ak:multi-level-migration} applied to functions 
$F \in \CF$. 
The migration operator can be rewritten, similarly as in \eqref{sk_15c_22},
\be{sk_15f_1}
(L^{(\Omega_N)}_{\mathrm{mig}} F)(x)
= \sum_{\xi \in \Omega_N} \sum_{k \in \N} c_{k-1} N^{1-k} \int_E
\left(y_{\xi,k}-x_{\xi}\right)(\dd a)\,\frac{\partial F(x)}{\partial x_\xi}[\delta_a].
\ee
We obtain
\be{sk_15f_2}
(L^{(\Omega_N)}_{\mathrm{mig}} F)(x) = \sum_{\xi \in \Omega_N} c_0
\int_E \left(y_{\xi,1}-x_{\xi}\right)(\dd a)\,
\frac{\partial F(x)}{\partial x_\xi}[\delta_a] + E^{(N)},
\ee
where
\be{sk_15f_3}
|E^{(N)}| \leq N^{-1} C_F \sum_{k \in \Ntwo} c_{k-1} N^{2-k},
\ee
with $C_{F}$ a generic constant depending on the choice of $F$ only. Here we use 
that, by the definition of $F$ in \eqref{ak:multi-level-test-functions}, the sum over $\xi 
\in \Omega_N$ is a sum over finitely many coordinates only, with the number depending 
on $F$ only. By \eqref{ak:recurrence-cond} we get
\be{sk_15f_4}
|E^{(N)}| \leq O\!\left( N^{-1} \right).
\ee

For the resampling operator in \eqref{ak:multi-level-resampling-global}, 
applying first \eqref{ak:resampling-mapping-global} 
and then \eqref{ak:lambda-growth-condition}, we obtain, 
\be{sk_15f_5}
(L^{(\Omega_N)}_{\mathrm{res}} F)(x) =
\sum_{\xi\in\Omega_N} \int_{[0,1]} \Lambda^*_0 (\dd r)
\int_E x_{\xi}(\dd a)
\left[ F\left( \Phi_{r,a,B_0(\xi)}(x) \right) - F(x) \right]
+ E^{(N)}
\ee
with
\be{sk_15f_6}
|E^{(N)}| \leq \sum_{k \in \N} N^{-2k} \int_{[0,1]} 
\Lambda^*_k (\dd r) C_{F}  N^k r^2 = C_{F} \sum_{k \in \N} N^{-k} 
\lambda_k = O\!\left( N^{-1} \right).
\ee

Finally, the Fleming-Viot operator reads as in \eqref{sk_15c_28}:
\be{sk_15f_7}
(L_{\mathrm{FV}}^{(\Omega_N)} F)(x)
= d_0 \sum_{\xi \in \Omega_N} \int_E \int_E Q_{x_\xi}(\dd u,\dd v)\,
\frac{\partial^2 F(x)}{\partial x_\xi^2}[\delta_u,\delta_v].
\ee

Combining the migration parts in \eqref{sk_15f_2} and \eqref{sk_15f_4}, the 
resampling parts in \eqref{sk_15f_5} and \eqref{sk_15f_6}, and the Fleming-Viot 
part in \eqref{sk_15f_7}, we obtain
\be{sk_15f_8}
\begin{aligned}
(L^{(\Omega_N)} F)(x)
=& \sum_{\xi \in \Omega_N} c_0
\int_E \left(y_{\xi,1}-x_{\xi}\right)(\dd a)\,
\frac{\partial F(x)}{\partial x_\xi}[\delta_a] +O(N^{-1}) \\
&+ \sum_{\xi\in\Omega_N} \int_{[0,1]} \Lambda^*_0 (\dd r)
\int_E x_{\xi}(\dd a)
\left[ F\left( \Phi_{r,a,B_0(\xi)}(x) \right) - F(x) \right] 
+ O\!\left( N^{-1} \right) \\
&+ d_0 \sum_{\xi \in \Omega_N} \int_E \int_E Q_{x_\xi}(\dd u,\dd v)\,
\frac{\partial^2 F(x)}{\partial x_\xi^2}[\delta_u,\delta_v].
\end{aligned}
\ee
Combining (\ref{sk_15f_8}) with (\ref{sk_15f_1}--\ref{sk_15f_7}) and (\ref{sk_15c_29}) 
(also recall the discussion on embeddings from Section \ref{ss.embedding}), we get 
\eqref{ak19}.

\subsection{The $1$-block averages on time scale $Nt$}
\label{sss.avlts}

As before, we prove, for $F \in \mathcal{F}$ depending on $\{x_\xi$, $\xi \in B_1(0)\}$ 
only (recall that the upper index $[1]$ indicates time scale $N^1 t$ and that 
the upper index $2$ indicates that we consider $K=2$ levels),
\be{ak20}
(L^{(\Omega_N)[1]})(y) = (L^{(N,2)[1]}F)(y)+O(N^{-1})
\ee
after which the claim follows in the limit as $N \to \infty$ by the same argument as 
in Section~\ref{sss.sts}. We prove (\ref{ak20}) by considering separately the different 
parts of the generator.

For the $1$-block averages $y=y^{[1]}$, the migration operator can be calculated as 
in \eqref{sk_15c_34}. Using \eqref{sk_15c_35}, we get
\be{sk_15f_9}
(L_{\mathrm{mig}}^{(\Omega_N)} F)(y)
= \frac{1}{N} \sum_{\eta \in \Omega_N} \sum_{k \in \N} c_k N^{1-k} \int_E
\left( y_{\phi^k \eta}^{[1]} - y_{\eta} \right)\!(\dd a)
\frac{\partial F(y)}{\partial y_\eta}[\delta_a].
\ee
We obtain on the time scale $Nt$
\be{sk_15f_10}
(L_{\mathrm{mig}}^{(\Omega_N)[1]} F)(y)
= \sum_{\eta \in \Omega_N} c_1 
\int_E \left( y_{\phi \eta}^{[1]} - y_{\eta} \right)\!(\dd a)
\frac{\partial F(y)}{\partial y_\eta}[\delta_a] + E^{(N)},
\ee
where
\be{sk_15f_11}
\left| E^{(N)} \right| \leq C_{F} \sum_{k \in \Ntwo} c_k N^{1-k} 
= O\!\left( N^{-1} \right).
\ee
Note that, by \eqref{sk_15c_36},
\be{sk_15f_17}
(L_{\mathrm{mig}}^{(\Omega_N)[1]} F)(y)
= (L_{\mathrm{mig}}^{(N,2)[1]} F)(y) + O\!\left( N^{-1} \right).
\ee

For the resampling operator, the only change to \eqref{sk_15c_40} is that
\eqref{sk_15c_41} gets replaced by
\be{sk_15f_12}
(L_{\mathrm{res}}^{(\Omega_N)} F)(y)
= I_0 + I_1 + E^{(N)}
\ee
with $I_0, I_1$ as in \eqref{sk_15c_41} (with $G_{N,2}$ replaced by $\Omega_N$) 
and
\be{sk_15f_13}
\begin{aligned}
\left| E^{(N)} \right|
&\leq \frac{1}{N^L} \left( \bigotimes_{l=1}^q \bigotimes_{i=1}^{n_l} 
\sum_{\xi_i^l\colon\, \phi \xi_i^l= \eta^{(l)}} \right) 
\sum_{k \in \Ntwo} N^{-2k} \int_{[0,1]} \Lambda^*_k (\dd r) L N^k C_{F} r^2 \\
&= C_F \sum_{k \in \Ntwo} N^{-k}
\lambda_k = O\!\left( N^{-2} \right).
\end{aligned}
\ee
After a change to time scale $Nt$, we therefore have
\be{sk_15f_14}
(L_{\mathrm{res}}^{(\Omega_N)[1]} F)(y) 
= (L_{\mathrm{res}}^{(N,2)[1]} F)(y) + O\!\left( N^{-1} \right)
\ee
with $(L_{\mathrm{res}}^{(N,2)} F)(y)$ as in \eqref{sk_15c_40}.

The Fleming-Viot operator on time scale $t$ reads as in \eqref{sk_15c_48}, respectively, 
on time scale $Nt$ as in \eqref{sk_15c_50},
\be{sk_15f_15}
(L_{\mathrm{FV}}^{(\Omega_N)[1]} F)(y)
= (L_{\mathrm{FV}}^{(N,2)[1]} F)(y).
\ee

\subsection{Arbitrary truncation level}
\label{sss.gencase}

For every $K \in\N$, consider the block averages up to level $K-1$ on time scales 
up to $N^K t$, estimate the generator difference, bound this by an $O(N^{-1})$-term 
and get the same conclusion as above. There are more indices involved in the 
notation, but the argument is the same. The details are left to the interested reader.

\section{Multiscale analysis}
\label{s.msana}

\subsection{The interaction chain}
\label{ss.intchain}

In this section, we prove Theorem~\ref{Tscb}. In addition to Theorem~\ref{mainth},
what is needed is the convergence of the {\em joint law} of the collection of 
$k$-level block averages for $k=0, \ldots, j+1$ on the corresponding time scales 
$N^j t_N+N^k t$, with $\lim_{N\to\infty} t_N =\infty$ and $\lim_{N\to\infty} 
t_N/N=0$. We already know that the $\ell$-block averages for $\ell >k$ do not 
change on time scale $t N^k$ and that this holds in path space as well. Hence, 
in particular, the $(j+1)$-block average converges to a constant path at times 
$N^j t_N+N^k t$ for all $0\leq k \leq j$. We also have the convergence of the 
marginal distributions for each $k=0,\ldots,j+1$, namely, we know that the process 
on level $k$ solves a martingale problem on time scale $tN^k$, which we have 
identified and where only the block average on the next level appears as a 
parameter. Therefore, arguing downward from level $j+1$ to level $j$, we 
see that the {\em Markov property} holds for the limiting law. It therefore 
only remains to identify the transition probability.

We saw in Section~\ref{s.hiermod} that when going from level $k+1$ to level $k$, 
we get the corresponding equilibrium law of the level-$k$ limiting dynamics as 
a McKean-Vlasov process with parameters $(c_k,\theta,d_k,\Lambda_k)$ with $\theta$ 
equal to the limiting state on level $k+1$. Note here that, instead of $N^{k+1}s 
+ N^k t$, we can write $N^{k+1}s + N^kt_N$ with $\lim_{N\to\infty} t_N = \infty$ 
and $\lim_{N\to\infty} t_N/N = 0$, since an $o(1)$ perturbation of $s$ has no 
effect as $N \to \infty$. For more details, consult \cite[Section 5(f)]{DGV95}.

In the remainder of this section, we prove the implications of the scaling
results of  $(d_k)_{k \in \N}$ for the hierarchical multiscale analysis of the
process $X^{(\Omega_N)}$, involving clustering versus coexistence
(Section~\ref{ss.dichointeract}), related  phase transitions (Section~\ref{ss.ich}), 
as well as a more detailed description of the properties of the different regimes 
(Section~\ref{ss.cluscoex}), as discussed in Section~\ref{sss.extmsp}.

\subsection{Dichotomy for the interaction chain}
\label{ss.dichointeract}

In this section, we prove Theorem~\ref{T.dicho}.

\noindent\textit{Proof of Theorem~\ref{T.dicho}.} Fix $j\in\N_0$. The first observation 
is that the interaction chain $(M^{(j)}_k)_{k=-(j+1),\ldots,0}$ from Section~\ref{sss.extmsp} 
is a $\CP(E)$-valued Markov chain such that 
\be{add33}
\big(\langle M^{(j)}_k,\varphi \rangle\big)_{k=-(j+1),\ldots,0} 
\text{ is a {\em square-integrable martingale}, for any $\varphi \in C_\mathrm{b}(E)$}
\ee
(because it is bounded). For the analysis of the interaction chain for Fleming-Viot 
diffusions, carried out in \cite[Section~6]{DGV95}, this fact was central in combination 
with the formula for the variance of evaluations analogous to Proposition~\ref{prop:2nd}.  
We argue as follows.

Since the map $\theta \mapsto \nu^{c,d,\Lambda}_\theta$ is continuous (cf.\ 
Section~\ref{sss.conttheta}), the convergence as $j \to \infty$ in the \emph{local coexistence} 
regime is a standard argument (see \cite[Section 6a]{DGV95}). In the \emph{clustering} 
regime, the convergence to the mono-type state follows by showing, with the help of the 
variance formula \eqref{ag58}, that $\lim_{j\to\infty}\E_{\CL(M^{(j)}_0)}[\var_\cdot(\varphi)]=0$ 
for all $\varphi \in C_\mathrm{b}(E)$ (cf., Corollary~\ref{C.asyvar}), so that all limit points of 
$\CL[M^{(j)}]$ are concentrated on $\delta$-measures on $E$ (recall that $\CP(E)$ is 
compact). This argument is identical to the one in \cite[Section 6a]{DGV95}. The mixing 
measure for the value of the mono-type state can be identified via the martingale property.

It remains to show that in the case where $\E_{\CL(M^{(j)}_0)}[\var_\cdot(\varphi)]$
is bounded away from zero, the limit points allow for the coexistence of types. 
The argument in \cite[Section 6a]{DGV95} shows that for $\Lambda=0$,
\be{ag69}
\nu^{c,d,\Lambda}_\theta (M) = 0 \mbox{ if } d>0, 
\quad M = \{\delta_u \colon\, u \in E\}.
\ee
This is no longer true for $\Lambda\neq 0$. Instead, we have $\nu_\theta^{c,d,\Lambda}
(M) \in [0,1)$, as proven in Section~\ref{ss.strucmkv} (see (\ref{ag64})), and hence the 
variance is $>0$.

\subsection{Scaling for the interaction chain}
\label{ss.ich}

In this section, we prove Theorems~\ref{T.dicho2} and \ref{T.clustering-exp} 

The proof of the scaling result in the regime of diffusive clustering in 
\cite[Section~6(b), Steps~1--3]{DGV95} uses two ingredients:
\begin{itemize}

\item[(I)]
Assertion \eqref{add33}.

\item[(II)] 
For $c_k \to c \in (0,\infty)$ as $k \to \infty$, by \cite[Eq.\ (6.12)]{DGV95},
\be{ag66}
\var\big(\langle M^{(j)}_{k_2},f\rangle \mid M^{(j)}_{k_1}=\theta\big) 
= \frac{(-k_1)-(-k_2)+1}{c+(-k_1)}\, \var_\theta (f), 
\qquad \forall\,f \in C_\mathrm{b}(E,\R).
\ee
\end{itemize}
In \cite[Section 6(b)]{DGV95}, (I--II) led to the conclusion that if $\lim_{j \to \infty} 
(-k_j)/j = \bar{\beta}_i \in [0,1]$, $i=1,2$, with $\bar{\beta}_1>\bar{\beta}_2$, then
\be{add34}
\lim_{j \to \infty} \var\big(\langle M^{(j)}_{k_2},f\rangle 
\mid M^{(j)}_{k_1}=\theta\big) 
= \frac{\bar{\beta}_1-\bar{\beta}_2}{\bar{\beta}_1}\,\var_\theta (f).
\ee
Thus, as soon as we have these formulae, we get the claim by repeating the argument 
in \cite[Section 6(b)]{DGV95}, which includes the time transformation $\bar{\beta}
=e^{-s}$ in Step 3 to obtain a time-homogeneous expression from \eqref{add34}. 

We know the necessary first and second moment formulae from Section~\ref{ss.1st2nd}. 
Replace \cite[Eq.\ (6.12)]{DGV95} by \eqref{varevol_int}, to see that we must make 
sure that
\be{sk109}
\lim_{j\to\infty} 
\sum_{i=\lfloor \bar{\beta}_2 j \rfloor}^{ \lfloor \bar{\beta}_1 j \rfloor} 
\left( \frac{d_{i+1}}{c_i} 
\prod_{l=i+1}^{\lfloor \bar{\beta}_1 j \rfloor} \frac{1}{1+m_l} \right)
= 1 - \left(\frac{\bar{\beta}_2}{\bar{\beta}_1}\right)^R
\ee
(recall \eqref{diffusion-constants} and \eqref{ag46} for the definition of $d_k$ and $m_k$).
Note that \eqref{sk109} remains valid also for $\bar{\beta}_2 = 0$.

Moreover, by following the reasoning in \cite[Section 6(b), Step 4]{DGV95}, we obtain
by using \eqref{varevol_int} instead of \cite[(6.34)]{DGV95} that
\be{sk110}
\left\{
\begin{array}{c}
\text{ fast growing clusters} \\
\text{ slowly growing clusters} \\
\end{array}
\right\}
\quad 
\text{if}
\quad
 \sum_{i=n}^m \left( \frac{d_{i+1}}{c_i} \prod_{l=i+1}^m \frac{1}{1+m_l} \right)
\left\{
\begin{array}{c}
\rightarrow 0 \\
\rightarrow 1 \\
\end{array}
\right\}
\ee
when $m,n\to\infty$ such that $n/m \to \alpha$, for all $\alpha \in (0,1)$.

\begin{proof}[Proof of Theorem~\ref{T.dicho2}]
The proof follows by inserting the asymptotics of $c_k$,  $d_k$ and $m_k$ obtained 
in Theorem~\ref{dcases} and Corollary~\ref{C.scalemk} into \eqref{sk109} or 
\eqref{sk110}. 

\begin{itemize}
\item[(i)] 
In Cases (a) and (b), the asymptotics in (\ref{ag49}--\ref{ag48}) and (\ref{dkmk}) imply
\begin{equation}
\sum_{i= \lfloor \alpha m \rfloor}^m \left( \frac{d_{i+1}}{c_i} 
\prod_{l=i+1}^m \frac{1}{1+m_l} \right) 
= O\left(\eee^{-C  m}\right), \quad C > 0.
\end{equation}
In Case (c), using the fact that $d_{i+1}/ c_i \sim m_i \to 0$ and $\sum_{l \in \N_0} 
m_l = \infty$, we obtain
\begin{equation}
\sum_{i= \lfloor \alpha m \rfloor}^m \left( \frac{d_{i+1}}{c_i} 
\prod_{l=i+1}^m \frac{1}{1+m_l} \right) \to 0.
\end{equation}
\item[(ii)]
In Case (d), for any $\varepsilon > 0$ and $l$ large enough we have $|m_l- R/l| \leq 
\varepsilon R/l$. This implies
\be{ak4000}
\begin{aligned}
\prod_{l=i+1}^{\lfloor \bar{\beta}_2 j \rfloor} \frac{1}{1+m_l}
&= \exp\left[-\sum_{l=i+1}^ {\bar{\beta}_1 j} \left(\frac{R}{l} + O(m_l^2)\right) \right].
\end{aligned}
\ee
Since $d_{i+1}/c_i \sim R/i$ and $m_l=O(1/l)$, it follows that
\begin{equation}
\label{ak4001}
\begin{aligned}
\sum_{i= \lfloor \bar{\beta}_2 j \rfloor}^{ \lfloor\bar{\beta}_1 j \rfloor} 
\left( \frac{d_{i+1}}{c_i} \prod_{l=i+1}^{\lfloor \bar{\beta}_2 j \rfloor} 
\frac{1}{1+m_l} \right) 
\sim
\sum_{i= \lfloor \bar{\beta}_2 j \rfloor}^{ \lfloor \bar{\beta}_1 j \rfloor}
\frac{R}{i}\left(\frac{\bar{\beta}_1 j}{i}\right)^{-R}
\to
1-\left(\frac{\bar{\beta}_2}{\bar{\beta}_1}\right)^R.
\end{aligned}
\end{equation}
\end{itemize}
\end{proof}

\begin{proof}[Proof of Theorem~\ref{T.clustering-exp}]
In Case (A), $m_k \to \infty$, which by \eqref{sk110} implies fast clustering.
In Case (B), $m_k \to \bar{K} + \bar{M} > 0 $, which also implies fast clustering.
In Case (C1), $m_k \sim (c_k \sigma_k)^{-1} \to C > 0$, which implies fast clustering. 
In Case (C2), $d_k/c_k \sim m_k \sim (1-c)/c > 0$,  which implies fast clustering. 
In Case (C3),  $d_k/c_k \sim m_k \sim \mu_k/(c_k (\mu -1))$, which implies fast, 
diffusive and slow clustering depending on the asymptotic behaviour of $k\mu_k / c_k$.
\end{proof}

\section{Dichotomy between clustering and coexistence for finite $N$}
\label{ss.cluscoex}

In this section, we prove Theorems~\ref{dich_coclu}--\ref{P.ergodic}.

\begin{proof}[Proof of Theorem {\rm \ref{dich_coclu}}]
The key is the spatial version of the formulae for the first and second moments
in  terms of the coalescent process. The variance tends to zero for all
evaluations if  and only if the coalescent started from two individuals at a
single site coalesces  into one partition element. Therefore, all we have to
show is that the hazard function  for the time to coalesce is $H_N$, and then
show that $\lim_{N\to\infty} H_N = \infty$  a.s.\ if and only if
$\lim_{N\to\infty} \bar H_N=\infty$. The latter was already carried out in
Section~\ref{sss.dichosv}.
\end{proof}

\begin{proof}[Proof of Theorem {\rm \ref{P.ergodic}}]
We first note that the set of functions
\begin{equation}
\big\{H^{(n)}_\varphi (\cdot, \pi_{G,n})\colon\, n \in \N,\, 
\varphi \in C_\mathrm{b} (E^n, \R),\, \pi_{G,n} \in \Pi_{G,n}\big\},
\end{equation}
(recall the definition of $H^{(n)}_\varphi$ from \eqref{duality-test-functions} and 
of $\Pi_{G,n}$ from \eqref{labelled-partitions}) is a distribution-determining subset 
of the set of bounded continuous functions on $\CP(\CP(E))^G$. It therefore suffices 
to establish the following: 
\begin{itemize}
\item[(1)]
For all initial laws $\CL[X^{(\Omega_N)}(0)]$, where $X^{(\Omega_N)}$ is the 
$C_N^{\uc,\uL}$-process on $\Omega_N$ satisfying our assumptions for a given 
parameter $\theta \in \CP(E)$ (see below Proposition~\ref{P.vecLambda}), and 
all admissible $n,\varphi,\pi_{G,n}$, we have
\be{ag70}
\E\left[ H^{(n)}_\varphi (X^{(\Omega_N)}(t), \pi_{G,n})\right]
\ttO F((\varphi, n, \pi_{G,n}), \theta),
\ee
which implies that $\CL[X^{(\Omega_N)}(t)]$ converges to a limit law as $t\to\infty$ 
that depends on the initial law only through the parameter $\theta$.
\item[(2)] 
Depending on whether $\bar H_N < \infty$ or $ \bar H_N = \infty$, with $\bar H_N$ as 
in Section~\ref{sss.dichosv}, the quantity in the right-hand side of \eqref{ag70} 
corresponds to the form of the limit claimed in (\ref{stable-ergodicity}--\ref{nonstable-ergodicity}).
\end{itemize}

Item (2) follows from Theorem~\ref{dich_coclu} once we have proved the convergence 
result in (\ref{ag70}), since (\ref{ag52b}) implies that the marginal law of the limiting state 
is $\delta_\theta$, and we will see in (\ref{ag74}) below that recurrence of the migration mechanism $a$ (recall \eqref{32b}) implies that
\be{ag71}
\E_{\nu_\theta,\uc, \uL}^{(\Omega_N)}  
\left[\langle \varphi, \bigotimes^n_{i=1} x_{\eta_i} \rangle\right]
= \langle f^n(u), \theta \rangle, \quad \text{for } 
\varphi(u_1, \cdots, u_n) = \prod^n_{i=1} f(u_i),
\ee
which in turn implies
\be{ag72}
\nu_{\theta, \uc, \uL}^{(\Omega_N)} = \int_K (\delta_u)^{\otimes \Omega_N} \theta(\dd u)
.
\ee

In order to prove item (1), we use duality and express the expectation in the left-hand 
side of (\ref{ag70}) as an expectation over a coalescent $\mathfrak{C}^{(\Omega_N)}_t$ 
as in \eqref{ag23b} starting with $n$ partition elements. We therefore know that the number 
of partition elements, which is nonincreasing in $t$, converges to a limit as $t\to\infty$, which is 
1 for $\bar H_N = \infty$ and a random number in $\{1,\ldots,n\}$ for $\bar H_N < \infty$. 
This means that there exists a finite random time after which the partition elements 
never meet again, and keep on moving by migration only. For such a scenario, it was proven 
in \cite{DGV95}, Lemma 3.2, that the positions of the partition elements are given, 
asymptotically, by $k=1,\dots,n$ random walks, all starting at the origin. Using that 
the initial state is ergodic, we can then calculate, for $\varphi (u_1, \cdots, u_n) 
= \Pi^n_{k=1} f(u_k)$,
\be{ag74}
\lim_{t\to\infty} E\left[H^{(n)}_\varphi \left(X^{(\Omega_N)}(0),
\mathfrak{C}^{(\Omega_N)}_t\right)\right]
= \sum^n_{k=1} \langle f, \theta \rangle^k q^{(\pi_{G,n})}_k,
\ee
with $q^{(\pi_{G,n})}_k$ the probability that the coalescent starting in $\pi_{G,n}$ in 
the limit has $k$ remaining partition elements. Furthermore, if the initial positions of 
a sequence $(\pi^{(m)}_{G,n})_{m\in\N}$ of initial states satisfies $\lim_{m\to\infty}
d(\eta^{(m)}_i,\eta^{(m)}_j)=\infty$ for $i \neq j$, then for transient $a$ we know 
that 
\be{agextext}
\lim_{m\to\infty} q^{(\pi^{(m)}_{G,n})}_k=0, \quad \forall\,k=1,\dots,n-1 \mbox{ and }
\liml_{m \to \infty} q^{(\Pi^{(m)}_{G,n})}_n = 1. 
\ee
In view of (\ref{ag74}), this proves that the law on $(\CP(E))^G$ defined by the 
right-hand side of (\ref{ag70}) is a translation-invariant and ergodic probability 
measure, with mean measure $\theta$ (see \cite{DGV95}, p.\ 2310, for details).
\end{proof}

\section{Scaling of the volatility in the clustering regime}
\label{s.prcor}

In Section~\ref{ss.gencomp}, we prove Theorems~\ref{dichodual} and \ref{T.beh}, in 
Section~\ref{ss.dcases} we prove Theorem~\ref{dcases}.

\subsection{Comparison with the hierarchical Fleming-Viot process}
\label{ss.gencomp}

\begin{proof}[Proof of Theorem~\ref{T.beh}]
(a) Rewrite the recursion relation in \eqref{diffusion-constants} as
\be{recrew1}
d_0=0, \qquad \frac{1}{d_{k+1}} = \frac{1}{c_k} + \frac{1}{\mu_k+d_k}, \quad k\in\N_0.
\ee
From \eqref{recrew1}, it is immediate that $\uc\mapsto \ud$ and $\umuu \mapsto \ud$ 
are component-wise non-decreasing. 

\medskip\noindent
(b) To compare $\ud$ with $\ud^*$, the solution of the recursion relation in 
\eqref{dbarkrec} when $\mu_0>0$ and $\mu_k=0$ for all $k\in\N$, simply note 
that $d_1=d^*_1=c_0\mu_0/(c_0+\mu_0)$. This gives 
\be{dkd*k}
d_k \geq d^*_k, \quad k\in\N,
\ee
with $d^*_k$ given by \eqref{dbarksol}.

\medskip\noindent
(c) Inserting the definition $m_k= (\mu_k+d_k)/c_k$ into \eqref{recrew1}, we get
the recursion relation 
\be{recrew1alt}
c_0m_0=\mu_0, \qquad
c_{k+1}m_{k+1} = \mu_{k+1} + \frac{c_km_k}{1+m_k}, \quad k\in\N_0.
\ee
Iterating \eqref{recrew1alt}, we get
\be{recrew1alt*}
c_km_k = \sum_{l=0}^k \frac{\mu_l}{\prod_{j=l}^k (1+m_j)}.
\ee
Ignoring the terms in the denominator, we get
\be{mkub}
m_k \leq \frac{1}{c_k} \sum_{l=0}^k \mu_l,
\ee
which proves that $\sum_{k\in\N_0} (1/c_k) \sum_{l=0}^k \mu_l<\infty$ implies 
$\sum_{k\in\N_0} m_k<\infty$. To prove the reverse, suppose that $\sum_{k\in\N_0}
m_k<\infty$. Then $\prod_{j\in\N_0} (1+m_j) = C < \infty$. Hence
\eqref{recrew1alt*} gives
\be{recrew1alt**}
m_k \geq \frac{1}{C}\,\frac{1}{c_k} \sum_{l=0}^k \mu_l,
\ee
which after summation over $k\in\N_0$ proves the claim. 

\medskip\noindent
(d) We know from \eqref{dbarksol} that $d_k \geq d^*_k = \mu_0/(1+\mu_0\sigma_k)$ for 
$k\in\N$. Hence, if $\lim_{k\to\infty} \sigma_k = \infty$, then $\liminf_{k\to\infty} 
\sigma_k d_k \geq 1$. To get the reverse, note that iteration of \eqref{recrew1} gives
\be{dkformsol}
\begin{aligned}
\frac{1}{d_k} 
&= \sum_{l=0}^{k-1} \frac{1}{c_l \prod_{j=l+1}^{k-1} (1+\frac{\mu_j}{d_j})}
\geq \sum_{l=0}^{k-1} \frac{1}{c_l \prod_{j=l+1}^{k-1} (1+\frac{\mu_j}{d^*_j})}\\
&\geq \sum_{l=0}^{k-1} \frac{1}{c_l \prod_{j=l+1}^\infty 
(1+\frac{\mu_j}{\mu_0}[1+\mu_0\sigma_j])}. 
\end{aligned}
\ee
If $\sum_{j\in\N} \sigma_j\mu_j < \infty$, then the product in the last line tends 
to 1 as $l\to\infty$. Hence, if also $\lim_{k\to\infty} \sigma_k = \infty$, then 
it follows that $\liminf_{k\to\infty} (1/\sigma_kd_k) \geq 1$.   

\medskip
Note from the proof of (c) and (d) that in the \emph{local coexistence regime} $d_k 
\sim \sum_{l=0}^k \mu_l$ as $k\to\infty$ when this sum diverges and $d_k \to 
\sum_{l\in\N_0} \mu_l/\prod_{j=l}^\infty (1+m_j) \in (0,\infty)$ when it converges. 
  
We close with the following observation. Since $1/c_k\sigma_k= (\sigma_{k+1}-\sigma_k)
/\sigma_k$, $k\in\N$, and 
\be{sigmasandw}
\frac{\sigma_{k+1}-\sigma_k}{\sigma_1} \geq \frac{\sigma_{k+1}-\sigma_k}{\sigma_k} 
\geq \int_{\sigma_k}^{\sigma_{k+1}} \frac{\dd x}{x}, \quad k\in\N,
\ee
we have
\be{iffimpl}
\lim_{k\to\infty} \sigma_k = \infty \quad \Longleftrightarrow \quad 
\sum_{k\in\N} \frac{1}{c_k\sigma_k} = \infty. 
\ee
\end{proof}

\begin{proof}[Proof of Theorem~\ref{dichodual}.]
Combining Lemma~\ref{ak14b} with Theorem~\ref{T.beh}(c), we get the claim.
\end{proof}

\subsection{Preparation: M\"obius-transformations}
\label{ss.Mobtr}

To draw the scaling behaviour of $d_k$ as $k\to\infty$ from \eqref{recrew1}, we need 
to analyse the recursion relation
\be{recrew}
x_0=0, \qquad x_{k+1} = f_k(x_k), \quad k\in\N_0,
\ee
where
\be{coeff}
f_k(x) = \frac{c_kx+c_k\mu_k}{x + (c_k + \mu_k)}, \qquad x \neq -(c_k+\mu_k).
\ee
The map $x \mapsto f_k(x)$ is a M\"obius-transformation on $\R^*$, the one-point 
compactification of $\R$. It has determinant $c_k(c_k+\mu_k)-c_k\mu_k = c_k^2>0$
and therefore is hyperbolic (see Kooman~\cite{K98}; a M\"obius-transformation $f$ 
on $\R^*$ is called \emph{hyperbolic} when it has two distinct fixed points at 
which the derivatives are not equal to $-1$ or $+1$.) Since
\be{fkder}
f_k'(x) = \left(\frac{c_k}{x+(c_k+\mu_k)}\right)^2, \qquad x \neq -(c_k+\mu_k),
\ee
it is strictly increasing except at $x=-(c_k+\mu_k)$, is strictly convex for 
$x<-(c_k+\mu_k)$ and strictly concave for $x>-(c_k+\mu_k)$, has horizontal 
asymptotes at height $c_k$ at $x=\pm\infty$ and vertical asymptotes at 
$x=-(c_k+\mu_k)$, and has two fixed points
\be{fpM}
x_k^+ = \tfrac12\mu_k[-1 + \sqrt{1+4c_k/\mu_k}] \in (0,\infty), 
\qquad x_k^- = \tfrac12\mu_k[-1 - \sqrt{1+4c_k/\mu_k}] \in (-\infty,0), 
\ee
of which the first is attractive ($f_k'(x_k^+)<1$) and the second is repulsive 
($f_k'(x_k^-)>1$). For us, only $x_k^+$ is relevant because, as is clear from 
\eqref{recrew}, our iterations take place on $(0,\infty)$. See Fig.~\ref{fig-Mobius} 
for a picture of $f_k$.

\begin{figure}[htbp]
\begin{center}
\setlength{\unitlength}{0.5cm}
\begin{picture}(14,10)(-4,-2)
\put(-6,0){\line(15,0){13}}
\put(1,-5){\line(0,10){11}}
{\thicklines
\qbezier(-1.9,-5)(-1.9,2.7)(5,2.7)
\qbezier(-2.4,6)(-2.4,3.3)(-5,3.3)
}
\qbezier[80](-4,-5)(1,0)(6,5)
\qbezier[60](-5,3)(0,3)(6,3)
\qbezier[80](-2.2,-5)(-2.2,0)(-2.2,6)
\qbezier[20](3.6,0)(3.6,1.3)(3.6,2.6)
\qbezier[20](-1.75,0)(-1.75,-1.3)(-1.75,-2.75)
\put(-.3,6.5){$f_k(x)$}
\put(7.5,-.3){$x$}
\put(3.15,-.8){$x_k^+$}
\put(-1.95,.4){$x_k^-$}
\put(3.6,2.6){\circle*{.3}}
\put(-1.75,-2.75){\circle*{.3}}
\end{picture}
\vspace{2cm}
\caption{\small The M\"obius-transformation $x \mapsto f_k(x)$.}
\label{fig-Mobius}
\end{center}
\end{figure}
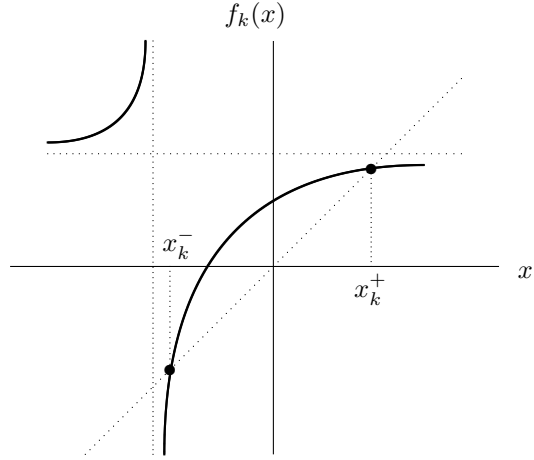

In what follows, we will use the following two theorems of Kooman~\cite{K98}. 
We state the version of these theorems for $\R$, although they apply for $\C$ 
as well.

\begin{theorem}
\label{Kooman1}
{\rm [Kooman~\cite{K98}, Corollary~6.5]}\\
Given a sequence of M\"obius-transforma\-tions $(f_k)_{k\in\N_0}$ on $\R^*$ that 
converges point-wise to a M\"obius-transformation $f$ that is hyperbolic. Then, for 
one choice of $x_0\in\R^*$ the solution of the recursion relation $x_{k+1}=f_k(x_k)$, 
$k\in\N_0$, converges to the repulsive fixed point $x^-$ of $f$, while for all other 
choices of $x_0$ it converges to the attractive fixed point $x^+$ of $f$. 
\end{theorem}
 
\begin{theorem}
\label{Kooman2}
{\rm [Kooman~\cite{K98}, Theorem 7.1]}\\
Given a sequence of M\"obius-transforma\-tions $(f_k)_{k\in\N_0}$ on $\R^*$ whose 
fixed points are of bounded variation and converge to (necessarily finite) distinct 
limits, i.e., 
\be{bdvardislim}
\begin{aligned}
&\sum_{k\in\N_0} |x_{k+1}^+-x_k^+|<\infty, \quad 
\sum_{k\in\N_0} |x_{k+1}^--x_k^-|<\infty,\\ 
&x^+=\lim_{k\to\infty} x_k^+ \in \R^*, \quad x^- = \lim_{k\to\infty} x_k^- \in \R^*,
\quad x^+ \neq x^-.
\end{aligned} 
\ee
If
\be{sloplim}
\prod_{k\in\N_0} |f_k'(x_k^+)| = 0, 
\ee
then, for one choice of $x_0\in\R^*$, the solution of the recursion relation 
$x_{k+1}=f_k(x_k)$, $k\in\N_0$, converges to $x^-$, while for all other choices 
of $x_0$ it converges to $x^+$. If, on the other hand,
\be{sloplimalt}
\prod_{k\in\N_0} |f_k'(x_k^+)| > 0, 
\ee
then all choices of $x_0\in\R^*$ lead to different limits.    
\end{theorem}

\noindent
Theorem~\ref{Kooman1} deals with the situation in which there is a limiting hyperbolic 
M\"obius-transfor\-mation, while Theorem~\ref{Kooman2} deals with the more general situation 
in which the limiting M\"obius-transformation may not exist or may not be hyperbolic, 
but the fixed points do converge to distinct finite limits and they do so in a summable 
manner. (In Theorem~\ref{Kooman1}, it is automatic that the fixed points of $f_k$ converge 
to the fixed points of $f$.) The conditions in (\ref{bdvardislim}--\ref{sloplim}) are 
necessary to ensure that the solutions of the recursion relation can reach the limits 
of the fixed points. Indeed, condition (\ref{sloplimalt}) prevents precisely that. As 
is evident from Fig.~\ref{fig-Mobius}, the single value of $x_0$ for which the solution 
converges to the limit of the repulsive fixed point must satisfy $x_0<0$, which is excluded 
in our case because $x_0=0$. We therefore also do not need the bounded variation
condition in the second part of the first line of \eqref{bdvardislim}.

\subsection{Scaling of the volatility for polynomial coefficients}
\label{ss.dcases}

\noindent\textit{Proof of Theorem~\ref{dcases}.}
Theorem~\ref{dcases} shows \emph{four regimes}. Our key assumptions are 
(\ref{regvarass}--\ref{regvarex}). For the scaling behaviour as $k\to\infty$ 
of the attractive fixed point $x_k^+$ given in \eqref{fpM}, there are three 
regimes depending on the value of $K$:
\be{threereg}
x_k^+ \sim \left\{\begin{array}{ll}
c_k, &\mbox{if } K=\infty,\\
M^+c_k,  &\mbox{if } K \in (0,\infty) \mbox{ with } M^+=\tfrac12K[-1+\sqrt{1+(4/K)}],\\
\sqrt{c_k\mu_k}, &\mbox{if } K=0.
\end{array}
\right.
\ee
Our target will be to show that (recall $x_k$ from \eqref{recrew})
\be{xkxk+target}
x_k \sim x_k^+  \quad \mbox{ as } \quad k\to\infty,
\ee 
which is the scaling we are after in Theorems~\ref{dcases}(a--c). We will see that 
\eqref{xkxk+target} holds for $K \in (0,\infty]$, and also for $K=0$ when $L=\infty$. 
A different situation arises for $K=0$ when $L<\infty$, namely, $x_k \sim 1/\sigma_k$, 
which is the scaling we are after in Theorem~\ref{dcases}(d). 

For the proofs given in Sections~\ref{ss.1}--\ref{ss.4}, below we make use of 
Theorems~\ref{Kooman1}--\ref{Kooman2} after doing the appropriate change of 
variables. Along the way, we need the following elementary facts: 
\begin{itemize}
\item[(I)] 
If $(a_k)$ and $(b_k)$ have bounded variation, then both $(a_k+b_k)$ and $(a_kb_k)$
have bounded variation.
\item[(II)]
If $(a_k)$ has bounded variation and $h\colon\,\R \to \R$ is globally Lipschitz on 
a compact interval containing the tail of $(a_k)$, then $(h(a_k))$ has bounded 
variation.
\item[(III)]
If $(a_k)$ is bounded and is asymptotically monotone, then it has bounded variation.
\end{itemize}
Moreover, the following notion will turn out to be useful. According to Bingham, Goldie
and Teugels~\cite[Section 1.8]{BGT87}, a strictly positive sequence $(a_k)$ is said to 
be \emph{smoothly varying} with index $\rho\in\R$ if 
\be{smoothvar}
\lim_{k\to\infty} k^na_k^{[n]}/a_k = \rho(\rho-1)\times\dots\times(\rho-n+1),
\qquad n\in\N,
\ee
where $a_k^{[n]}$ is the $n$-th order discrete derivative, i.e., $a_k^{[0]} = a_k$ and
$a_k^{[n+1]}=a_{k+1}^{[n]}-a_k^{[n]}$, $k,n\in\N_0$.
\begin{itemize}
\item[(IV)] 
If $(a_k)$ is smoothly varying with index $\rho\notin\N_0$, then $(a_k^{[n]})$ is 
asymptotically monotone for all $n\in\N$, while if $\rho\in\N$, then the same is 
true for all $n\in\N$ with $n\leq\rho$.
\end{itemize} 
This observation will be useful in combination with (I--III).

According to \cite[Theorem 1.8.2]{BGT87}, if $(a_k)$ is regularly varying with index 
$\rho\in\R$, then there exist smoothly varying $(a'_k)$ and $(a''_k)$ with index $\rho$ 
such that $a'_k \leq a_k \leq a''_k$ and $a'_k \sim a''_k$. In words, any regularly 
varying function can be sandwiched between two smoothly varying functions with the same 
asymptotic behaviour. In view of the monotonicity property in Theorem~\ref{T.beh}(a), 
it therefore suffices to prove Theorem~\ref{dcases} under the following assumption, 
which is stronger than \eqref{regvarass}:
\be{regvarassextra}
\begin{aligned}
&(c_k), \, (\mu_k), \, (\mu_k/c_k), \, (k^2\mu_k/c_k) \mbox{ are smoothly varying}\\ 
&\mbox{(with index $a$, $b$, $a-b$, respectively, $2+a-b$)}.
\end{aligned}
\ee

\subsubsection{Case (b)}
\label{ss.1}

Let $K\in (0,\infty)$. Put $y_k = x_k/c_k$. Then the recursion relation in \eqref{recrew} 
becomes
\be{grecKfin}
y_0=0, \qquad y_{k+1}=g_k(y_k), \quad k\in\N_0,
\ee
where
\be{gdefKfin}
g_k(y) = \frac{A_k y+B_k}{C_k y+D_k}, \qquad y \in \R^*,
\ee
with coefficients
\be{coeffrec} 
A_k=\frac{c_k^2}{c_{k+1}},
\quad 
B_k = \frac{c_k\mu_k}{c_{k+1}},
\quad 
C_k=c_k,
\quad D_k=c_k+\mu_k.
\ee
By \eqref{regvarass}, we have $c_k/c_{k+1} \sim 1$, and hence $A_k\sim C_k \sim c_k$, 
$B_k \sim Kc_k$, $D_k \sim (K+1)c_k$. Therefore, \eqref{gdefKfin} yields
\be{glimKfin}
\lim_{k\to\infty} g_k(y) = g(y) = \frac{y + K}{y + (K+1)}, \qquad y\in\R^*.
\ee
Since $g$ is hyperbolic with fixed points $y^\pm=M^\pm=\tfrac12 K[-1\pm\sqrt{1+(4/K)}]$, 
we can apply Theorem~\ref{Kooman1} and conclude that
\be{yklimKfin}
\lim_{k\to\infty} y_k = M^+. 
\ee

\subsubsection{Case (a)}
\label{ss.2}

Let $K=\infty$. Again put $y_k = x_k/c_k$. Then the same recursion relation as in 
(\ref{grecKfin}--\ref{gdefKfin}) holds with the same coefficients as in \eqref{coeffrec}, 
but this time $c_k/c_{k+1} \sim 1$ gives $A_k \sim C_k \sim c_k$, $B_k \sim D_k \sim 
\mu_k$, and 
\be{glimKzero_a}
\lim_{k\to\infty} g_k(y) = g(y) = 1, \qquad y\in\R^*.
\ee
Since $g$ is not hyperbolic, we cannot apply Theorem~\ref{Kooman1}. To compute $y^\pm
= \lim_{k\to\infty} y_k^\pm$, we note that $g_k$ has fixed points
\be{yk+}
y_k^{\pm} = \frac{1}{a_k} h^\pm(b_k/a_k^2) \, \mbox{ with } \,
h^\pm(x) = \frac{1}{2x}\big(\,1 \mp \sqrt{1+4x}\,\big),\,\,
a_k = \frac{A_k-D_k}{B_k},\,\, b_k = \frac{C_k}{B_k}
\ee
(use that $a_k<0$ for $k$ large enough). Since $c_k/\mu_k\to 0$, we have $a_k \to -1$ 
and $b_k \to 0$. It follows that $y_k^+ \to y^+ =1$ and $y_k^- \to y^-=-\infty$, so 
that we can apply Theorem~\ref{Kooman2}. To prove that $y_k \to y^+ = 1$, we need to 
check that (recall (\ref{bdvardislim}--\ref{sloplim}))
\begin{itemize}
\item[(1)]
$(y_k^+)_{k\in\N_0}$ has bounded variation.
\item[(2)] 
$\prod_{k\in\N_0} g_k'(y_k^+)=0$.
\end{itemize}
(What happens near $y_k^-$ is irrelevant because $x_k>0$ for all $k$.) 

To prove (1), note that $h^+$ is globally Lipschitz near zero. Since, by \eqref{coeffrec} 
and 
\eqref{yk+},
\be{coefwr1}
a_k = \frac{c_k}{\mu_k}\left(1-\frac{c_{k+1}}{c_k}\right) - \frac{c_{k+1}}{c_k},
\quad b_k = \frac{c_k}{\mu_k}\,\frac{c_{k+1}}{c_k},
\ee
it follows from \eqref{regvarassmon}, (I), (III--IV) and \eqref{regvarassextra} 
that $(a_k)$ and $(b_k)$ have bounded variation. Since $a_k \to -1$ and $b_k \to 0$, 
it in turn follows from (I--II) that $(1/a_k)$ and $(b_k/a_k^2)$ have bounded variation. 
Via (I--II) this settles (1).     

To prove (2), note that
\be{gkslope1Kzero}
g_k'(y_k^+) = \frac{\Delta_k}{(C_ky_k^+ + D_k)^2} \,\,\mbox{ with } \,\,
\Delta_k = A_kD_k-B_kC_k.
\ee
Since $y_k^+>0$ and $D_k>\mu_k$, we have
\be{gkslope2Kzero}
\prod_{k\in\N_0} g_k'(y_k^+) \leq \prod_{k\in\N_0} \frac{\Delta_k}{\mu_k^2}.
\ee
But $\Delta_k = c_k^3/c_{k+1}$ and so, because $c_k/c_{k+1} \sim 1$, we have 
$\Delta_k/\mu_k^2= c_k^3/c_{k+1}\mu_k^2 \sim (c_k/\mu_k)^2 \to 0$. Hence (2) 
indeed holds.

\subsubsection{Case (c)}
\label{ss.3}

Let $K=0$ and $L=\infty$. Put $y_k = x_k/\sqrt{c_k\mu_k}$. Then the same recursion 
relation as in (\ref{grecKfin}--\ref{gdefKfin}) holds with coefficients
\be{coefKzero_a}
A_k = c_k \sqrt{\frac{c_k\mu_k}{c_{k+1}\mu_{k+1}}},
\quad
B_k = c_k\mu_k \sqrt{\frac{1}{c_{k+1}\mu_{k+1}}},
\quad 
C_k = \sqrt{c_k\mu_k},
\quad 
D_k = c_k + \mu_k.
\ee
By \eqref{regvarass}, $c_{k+1}/c_k \sim 1$ and $\mu_{k+1}/\mu_k \sim 1$, and hence
$A_k \sim D_k \sim c_k$, $B_k \sim C_k \sim \sqrt{c_k\mu_k}$. Therefore \eqref{gdefKfin} 
yields
\be{glimKzero_b}
\lim_{k\to\infty} g_k(y) = g(y) = y, \qquad y\in\R^*.
\ee
Since $g$ is not hyperbolic, we cannot apply Theorem~\ref{Kooman1}. To compute 
$y^\pm=\lim_{k\to\infty} y_k^\pm$ from \eqref{yk+}, we abbreviate
\be{coefasympKinf} 
\alpha_k = \frac{c_{k+1}}{c_k}-1, \quad \beta_k = \frac{\mu_{k+1}}{\mu_k}-1,
\quad \gamma_k = \frac{\mu_k}{c_k},
\ee
and write
\be{coefasymp}
a_k = \frac{1}{\sqrt{\gamma_k}}
\left[1-(1+\gamma_k)\sqrt{(1+\alpha_k)(1+\beta_k)}\right],
\quad 
b_k = \sqrt{(1+\alpha_k)(1+\beta_k)}. 
\ee
We have $\alpha_k \to 0$, $\beta_k \to 0$, $\gamma_k \to 0$. Moreover, 
(\ref{regvarassmon}--\ref{regvarex}), (IV) and \eqref{regvarassextra} imply 
that $(k\alpha_k)$ and $(k\beta_k)$ are asymptotically monotone and bounded. Together 
with $\lim_{k\to\infty} k^2\gamma_k=\infty$ this in turn implies that $\alpha_k/\sqrt{\gamma_k}
\to 0$ and $\beta_k/\sqrt{\gamma_k} \to 0$. Hence $a_k \to 0$ and $b_k \to 1$, and therefore 
\eqref{yk+} yields $y^\pm = \pm 1$, so that we can apply Theorem~\ref{Kooman2}. 

To prove (1), note that (\ref{regvarassmon}--\ref{regvarex}), (IV) and 
\eqref{regvarassextra} also imply that $(\sqrt{\gamma_k})$ and $(1/\sqrt{k^2\gamma_k})$, 
are asymptotically monotone and bounded. By \eqref{coefasymp} and (I--III), this in 
turn implies that $(a_k)$ and $(b_k)$ have bounded variation. Indeed, the first equality 
in \eqref{coefasymp} can be rewritten as 
\be{coefasymprew}
a_k = \frac{1}{\sqrt{\gamma_k}}\,\frac{1-(1+\gamma_k)^2(1+\alpha_k)(1+\beta_k)}
{1+(1+\gamma_k)\sqrt{(1+\alpha_k)(1+\beta_k)}}.
\ee
The denominator tends to 2, is Lipschitz near 2, and has bounded variation because 
$(\alpha_k)$, $(\beta_k)$, $(\gamma_k)$ have bounded variation. The numerator equals 
$-\alpha_k-\beta_k-2\gamma_k$ plus terms that are products of $\alpha_k$, $\beta_k$ 
and $\gamma_k$. Writing $\alpha_k/\sqrt{\gamma_k} = k\alpha_k/\sqrt{k^2\gamma_k}$ 
and $\beta_k/\sqrt{\gamma_k} = k\beta_k/\sqrt{k^2\gamma_k}$ and using that 
$\sqrt{k^2\gamma_k}\to\infty$, we therefore easily get the claim.

To prove (2), note that 
\be{Dkrep1}
\Delta_k = c_k^2\sqrt{\frac{c_k\mu_k}{c_{k+1}\mu_{k+1}}} 
= c_k^2/\sqrt{(1+\alpha_k)(1+\beta_k)},
\quad 
C_ky_k^+ + D_k = c_k(1+y_k^+\sqrt{\gamma_k}+\gamma_k),
\ee 
and hence
\be{gprodest2}
\prod_{k\in\N_0} g_k'(y_k^+) \leq \prod_{k\in\N_0} 
\frac{1}{\sqrt{(1+\alpha_k)(1+\beta_k)}(1+y_k^+\sqrt{\gamma_k})^2}.
\ee
The term under the product equals
\be{prodbd_a}
1 - 2y^+\sqrt{\gamma_k}\,[1+o(1)],
\ee 
which yields (2) because $\sqrt{k^2\gamma_k}\to\infty$.

\subsubsection{Case (d)}
\label{ss.4}

Let $K=0$ and $L<\infty$. Put $y_k=\sigma_kx_k$. Then the same recursion relation as in 
(\ref{grecKfin}--\ref{gdefKfin}) holds with coefficients
\be{coefKzero_b}
A_k = c_k\frac{\sigma_{k+1}}{\sigma_k},
\quad
B_k = c_k\mu_k\sigma_{k+1},
\quad 
C_k = \frac{1}{\sigma_k},
\quad 
D_k = c_k + \mu_k.
\ee
Abbreviate
\be{deltakdef}
\delta_k = \frac{\sigma_{k+1}}{\sigma_k}-1 = \frac{1}{c_k\sigma_k}.
\ee
We have $k\mu_k/c_k \to 0$ and, by \eqref{regvarass}, $c_{k+1}/c_k \sim 1$, $\sigma_{k+1}
/\sigma_k \sim 1$ and $k\delta_k \to 1-a$ with $a \in (-\infty,1)$ the exponent in
\eqref{regvarass}. It therefore follows that 
\be{coefasympalt}
\frac{A_k}{D_k} \to 1, 
\quad 
\frac{B_k}{D_k} \sim \mu_k\sigma_k 
= \frac{k\mu_k}{c_k}\,\frac{1}{k\delta_k} \to 0, 
\quad
\frac{C_k}{D_k} \sim \frac{1}{c_k\sigma_k} = \delta_k \to 0.
\ee 
Hence, \eqref{gdefKfin} yields
\be{glimKzeroext}
\lim_{k\to\infty} g_k(y) = g(y) = y, \qquad y\in\R^*.
\ee
Since $g$ is not hyperbolic, we cannot apply Theorem~\ref{Kooman1}. To compute 
$y^\pm=\lim_{k\to\infty} y_k^\pm$, we rewrite \eqref{yk+} as
\be{yk+alt}
y_k^\pm = \tfrac12\left(\bar a_k \pm \sqrt{\bar a_k^2 + 4\bar b_k} \right) 
\quad \mbox{ with } \quad \bar a_k = \frac{A_k-D_k}{C_k}, \quad \bar b_k = \frac{B_k}{C_k},
\ee
and note that
\be{coeffrepalt}
\begin{aligned}
\bar a_k &= \frac{c_k}{c_{k+1}}-\mu_k\sigma_k
= \frac{c_k}{c_{k+1}} - \frac{k\mu_k}{c_k}\,\frac{1}{k\delta_k},\\
\bar b_k &= c_k\mu_k\sigma_k\sigma_{k+1} 
= \frac{k^2\mu_k}{c_k}\,\frac{\sigma_{k+1}}{\sigma_k}\,\frac{1}{(k\delta_k)^2}.
\end{aligned}
\ee
Since $k^2\mu_k/c_k \to L<\infty$ and $k\delta_k \to 1-a$ with $a \in (-\infty,1)$ the 
exponent in \eqref{regvarass}, it follows that $\bar a_k \to 1$ and $\bar b_k \to L/(1-a)^2$. 
Hence $y_k^\pm \to y^\pm = \tfrac12(1\pm\sqrt{1+4L/(1-a)^2})$, so that we can apply 
Theorem~\ref{Kooman2}. 

To prove (1), note that (\ref{regvarassmon}--\ref{regvarex}), (I--IV) and 
\eqref{regvarassextra} imply that $(\bar a_k)$ and $(\bar b_k)$ have bounded 
variation. This yields the claim via \eqref{yk+alt}.

To prove (2), note that 
\be{detform}
\begin{aligned}
\Delta_k &= c_k^2\,\frac{\sigma_{k+1}}{\sigma_k}
= c_k^2(1+\delta_k),\\
C_ky_k^+ + D_k &= \frac{y_k^+}{\sigma_k} + c_k + \mu_k 
= c_k\left(1+\delta_ky_k^+ + \frac{\mu_k}{c_k}\right),
\end{aligned}
\ee
and, hence,
\be{gprodest3}
\prod_{k\in\N_0} g_k'(y_k^+) \leq \prod_{k\in\N_0} 
\frac{1+\delta_k}{(1+\delta_ky_k^+)^2}.
\ee
The term under the product equals
\be{prodbd_b}
1 - (2y^+-1)\delta_k\,[1+o(1)],
\ee 
Since $y^+ \geq 1$, it follows that (2) holds if and only if $\sum_{k\in\N_0} 
\delta_k = \infty$, which by \eqref{iffimpl} and \eqref{deltakdef} holds if and 
only $\lim_{k\to\infty} \sigma_k = \infty$. Theorem~\ref{Kooman2} shows that 
failure of (2) implies that $y_k$ converges to a limit different from 1.

\subsection{Scaling of the volatility for exponential coefficients}
\label{ss.exp}

\textit{Proof of Theorem~\ref{dcases-exp}.} In this section, we briefly comment
on how to extend the proof of Theorem~\ref{dcases} to cover the case of
Theorem~\ref{dcases-exp}.

The claims made for Cases (A) and (B) follow from minor adaptations of the 
arguments for Cases (a) and (b) in Sections~\ref{ss.2} and \ref{ss.1}. The claim 
made for Case (C1) follows from Theorem~\ref{T.beh}(d). The claims made for 
Cases (C2) and (C3) follow from minor adaptations of the arguments for Cases 
(b) and (c) in Sections~\ref{ss.1} and \ref{ss.3}. The details are left to the 
reader.

\section{Notation index}
\label{s.notation}

\subsection{General notation}
\begin{itemize}
\item
$E$ $\leadsto$ compact Polish space of types.

\item
$\CP(E)$ $\leadsto$ set of probability measures on $E$.

\item
$M(E)$ $\leadsto$ set of measurable functions on $E$.

\item
$\CM([0,1])$ $\leadsto$ set of non-negative measures on $[0,1]$.

\item
$\CM_f([0,1])$ $\leadsto$ set of finite non-negative measures on $[0,1]$.

\item
$\mathcal{L}$ $\leadsto$ law.

\item 
$\Longrightarrow$ $\leadsto$ weak convergence on path space.

\item
$\Lambda^*\in\CM([0,1])$ $\leadsto$ 
(cf.\ \eqref{ag3}). 

\item
$\Lambda \in \CM_f([0,1])$ $\leadsto$ 
(cf.\ Section~\ref{ss.cannings}).  

\item
$\frac{\partial F(x)}{\partial x_i}[\delta_a]$ $\leadsto$ G\^ateaux-derivative 
of $F$ with respect to $x_i$ in the direction $\delta_a$ 
(cf.\ \eqref{Gatder}).

\item $\frac{\partial^2 F(x)}{\partial x^2 }[\delta_u,\delta_v]$ $\leadsto$
second G\^ateaux-derivative  of $F$ with respect to $x$ in the directions
$\delta_u$,  $\delta_v$ (cf.~\eqref{ak-second-variation}).

\item
$D(T,\CE)$ $\leadsto$ set of càdlàg paths in $\CE$ indexed by the elements of 
$T \subset \R$ and equipped with the Skorokhod $J_1$-topology.

\item
$C_\mathrm{b}\big(\CE,\CE^\prime)$ $\leadsto$ set of continuous bounded mappings 
from $\CE$ to $\CE^\prime$.
\end{itemize}

\subsection{Interacting $\Lambda$-Cannings processes}
\begin{itemize}
\item
$\Omega_N$ $\leadsto$ hierarchical group of order $N$ 
(cf.\ \eqref{ag30_a}).

\item 
$\uc=(c_k)_{k\in\N_0} \in (0,\infty)^{\N_0}$ $\leadsto$ migration coefficients 
(cf.\ \eqref{ckdef}).

\item
$\uL=\big(\Lambda_k)_{k\in\N_0} \in \CM_f([0,1])^{\N_0}$ $\leadsto$
offspring measures (cf.\ \eqref{ag41}).

\item
$\lambda_k=\Lambda_k([0,1])$ $\leadsto$ resampling rates 
(cf.\ \eqref{lambda-total-masses}).

\item 
$\ud=(d_k)_{k\in\N_0}$ $\leadsto$ volatility constants 
(cf.\ \eqref{diffusion-constants}).

\item 
$\umu=(m_k)_{k\in\N_0}$ $\leadsto$
(cf.\ \eqref{ag46}).

\item
$\mu_k=\tfrac12\lambda_k$ $\leadsto$ 
(cf.\ \eqref{ag46}).

\item
$\sigma_k$ $\leadsto$ 
(cf.\ \eqref{dbarksol}).

\item
$B_k(\eta)$ $\leadsto$ $k$-macro-colony around $\eta$ 
(cf.\ \eqref{block-definition}).

\item
$y_{\eta,k}$ $\leadsto$ type distribution in $B_k(\eta)$ 
(cf. \eqref{def-blocks}).

\item
$C^{\Lambda}$-process $\leadsto$ non-spatial continuum-mass $\Lambda$-Cannings process 
(cf.\ Section~\ref{sss.scolonycannings}).

\item
$a^{(N)}(\cdot,\cdot)$ $\leadsto$ hierarchical random walk kernel on $\Omega_N$ 
(cf.\ \eqref{32b}).

\item
$C_N^{\uc,\uL}$-process $\leadsto$ hierarchically interacting Cannings process on $\Omega_N$ 
(cf.\ Section~\ref{sss.hierarCan}).

\item
$\mathcal{F}$  $\leadsto$ algebra of test functions on $\CP(E)^{\Omega_N}$ (cf., \eqref{ak:multi-level-test-functions}).

\item $L^{(N)}$, $L^{(N)}_{\mathrm{mig}}$, $L^{(N)}_{\mathrm{res}}$ $\leadsto$
generators of the mean-field Cannings process 
(cf.\ \eqref{ag20a}).

\item
$L^{(\Omega_N)}$, $L^{(\Omega_N)}_{\mathrm{mig}}$, $L^{(\Omega_N)}_{\mathrm{res}}$ 
$\leadsto$ generators of the hierarchical Cannings process 
(cf.\ \eqref{ak:multi-level-generator-decomposition}).

\item
$\Phi_{r,a,B_k(\eta)}$ $\leadsto$ reshuffling-resampling map 
(cf.\ \eqref{ak:resampling-mapping-global}). 

\item 
$X^{(\Omega_N)}$ $\leadsto$ $C^{\uc,\uL}_N$-process 
(cf.\ Section~\ref{sss.hierarCan}).

\item 
$Y_{\eta,k}^{(\Omega_N)}(\cdot)$ $\leadsto$ macroscopic observables (= block averages) 
of $X^{(\Omega_N)}$ 
(cf.\ \eqref{k-block-average}).

\item
$y_{\eta}^{[1]}$ $\leadsto$ $1$-block averages indexed block-wise 
(cf.\ \eqref{sk_15c_32}).

\item $L_{\mathrm{res}}^{(N)[k]}$, $L_{\mathrm{mig}}^{(N)[k]}$ $\leadsto$
generators of the $k$-block averaged hierarchically interacting Cannings process
at the time scale $t^k N$ (cf.\ \ref{sss.1blockres-part-one}).

\item
$G_{N,K}$ $\leadsto$ $K$-level truncation of $\Omega_N$ 
(cf.\ \eqref{ag17}).

\item
$X^{(N)}$ $\leadsto$ mean-field interacting Cannings process 
(cf.\ Section~\ref{sss.mcolonycannings}).

\item
$Q_x(\dd u,\dd v)$ $\leadsto$ Fleming-Viot diffusion function 
(cf.\ \eqref{flemming-viot-kernel}). 

\item 
$L_\theta^{c,d,\Lambda}$, $L^c_\theta$, $L^d$, $L^\Lambda$ $\leadsto$ generators of
the McKean-Vlasov process
(cf.\ \eqref{generic-interaction-operator}).

\item
$Z_\theta^{c,d,\Lambda}$ $\leadsto$ McKean-Vlasov process with immigration-emigration 
(cf.\ Section~\ref{sss.MV}).

\item
$\nu^{c,d,\Lambda}_\theta$ $\leadsto$ unique equilibrium of $Z$ 
(cf.\ \eqref{cond-invariant}).

\item 
$(M^{(j)}_k)_{k=-(j+1),\ldots,0}$ $\leadsto$ interaction chain 
(cf.\ Section~\ref{sss.extmsp}).

\end{itemize}

\subsection{Spatial $\Lambda$-coalescents}
\begin{itemize}
\item
$[n] = \{1,\ldots,n\}$.

\item
$\Pi_n$ $\leadsto$ set of all partitions of $[n]$ into disjoint families 
(cf.\ \eqref{ag19}).

\item
$\Pi_{G,n}$ $\leadsto$ set of $G$-labelled partitions of $[n]$ 
(cf.\ \eqref{labelled-partitions}).

\item
$S_{G,n} \in \Pi_{G,n}$ $\leadsto$ $G$-labelled partition into singletons 
(cf.\ \eqref{labelled-singletons}).

\item
$\Pi$, $\Pi_G$ $\leadsto$ partitions of $\N$, $G$-labelled partitions of $\N$ 
(cf.\ \eqref{ag20}).

\item 
$L(\pi_G)$ $\leadsto$ set of labels of partition $\pi_G$ 
(cf.\ \eqref{20b}).

\item
$\lambda^{(\Lambda)}_{b,i}$ $\leadsto$ coalescence-rates 
(cf.\ \eqref{coalescence-rates}).

\item
$\cdot\vert_n$ $\leadsto$ operation of projection from $[m]$
(respectively, $\N$) onto $[n]$.

\item
$L^{(G)*}, L^{(G)*}_{\mathrm{mig}}, L^{(G)*}_{\mathrm{coal}}$ generators of 
the spatial coalescent on $G$ (cf., \eqref{spatial-coalescent-generator}).

\item $L^{(\Omega_N)*}$, $L^{(\Omega_N)*}_{\mathrm{mig}}$, $L^{(\Omega_N)*}_{\mathrm{coal}}$ 
$\leadsto$ generators of the spatial $\Lambda$-coalescent with non-local coalescence
(cf.\ \eqref{spatial-block-coalescent-generator}).

\item 
$\mathfrak{P}$ $\leadsto$ field of Poisson point processes driving the spatial 
$\Lambda$-coalescent 
(cf.\ \eqref{ppp-coalescence}).

\item
$\mathfrak{P}^{(\Omega_N)}$ $\leadsto$ driving Poisson point process for 
the spatial $n$-$\Lambda$-coalescent with non-local coalescence 
(cf.\ \eqref{ag22b}).

\item
$\mathfrak{C}^{(G)}_n$ $\leadsto$ spatial finite $n$-$\Lambda$-coalescent on $G$
(cf.\ \eqref{no1}).

\item 
$\mathfrak{C}^{(G)}$ $\leadsto$ spatial $\Lambda$-coalescent on $G$
(cf.\ \eqref{add1}).

\item 
$\mathfrak{C}^{(\Omega_N)}$ $\leadsto$ spatial $\uL$-coalescent with non-local coalescence
(cf.\ \eqref{ag23b}).

\end{itemize}

\bibliographystyle{alea3}

\end{document}